\documentclass[12pt]{amsart}

\usepackage{
	amsmath,
 	amsfonts,
  	amssymb,
 	amsthm,
        datetime,
	}

\usepackage{tikz-cd}
\usepackage{tikz}

\usepackage{rotating}

\usetikzlibrary{decorations.pathmorphing, decorations.pathreplacing}
 \usetikzlibrary{decorations.pathreplacing,backgrounds,decorations.markings}

\usepackage{color}

\usepackage{stmaryrd}
\usepackage[cal=boondoxo]{mathalfa}
\usepackage{mathtools}
\usepackage{bbm}

\usepackage{hyperref}
\usepackage[capitalise]{cleveref}

\usepackage{a4wide}

\usepackage{mathabx}

\theoremstyle{plain}
\newtheorem{thm}{Theorem}[section]
\newtheorem{cor}[thm]{Corollary}
\newtheorem{lem}[thm]{Lemma}
\newtheorem{prop}[thm]{Proposition}
\newtheorem{conj}[thm]{Conjecture}

\theoremstyle{definition}
\newtheorem{rem}[thm]{Remark}

\newtheorem{ex}[thm]{Example}
\newtheorem{exe}[thm]{Example}

\theoremstyle{definition}
\newtheorem{defn}[thm]{Definition}

\newenvironment{citethm}[1]{%
	\renewcommand\thethm{\ref{#1}}\thm}{\endthm\addtocounter{thm}{-1}}
	
\newenvironment{citecor}[1]{%
	\renewcommand\thethm{\ref{#1}}\cor}{\endthm\addtocounter{thm}{-1}}
	
\newenvironment{citelem}[1]{%
	\renewcommand\thethm{\ref{#1}}\lem}{\endthm\addtocounter{thm}{-1}}

\def\makeautorefname#1#2{\expandafter\def\csname#1autorefname\endcsname{#2}}
\makeautorefname{lem}{Lemma}%
\makeautorefname{prop}{Proposition}%
\makeautorefname{rem}{Remark}%
\makeautorefname{section}{Section}%

\crefname{equation}{}{}

\newcommand{\cA}{\mathcal{A}}
\newcommand{\cB}{\mathcal{B}}
\newcommand{\cC}{\mathcal{C}}
\newcommand{\cD}{\mathcal{D}}

\newcommand{\cF}{\mathcal{F}}

\newcommand{\cL}{\mathcal{L}}
\newcommand{\cM}{\mathcal{M}}

\newcommand{\cP}{\mathcal{P}}
\newcommand{\cQ}{\mathcal{Q}}

\newcommand{\cT}{\mathcal{T}}

\newcommand{\cV}{\mathcal{V}}

\newcommand{\bN}{\mathbb{N}}

\newcommand{\bQ}{\mathbb{Q}}
\newcommand{\bR}{\mathbb{R}}

\newcommand{\bZ}{\mathbb{Z}}

\newcommand{\be}{{\boldsymbol{e}}}
\newcommand{\bg}{{\boldsymbol{g}}}

\newcommand{\BP}{{\boldsymbol{P}}}

\newcommand{\vsimeq}{\rotatebox{-90}{\(\simeq\)}}

\newcommand{\pp}[1]{(\!( #1 )\!)}
\newcommand{\opalg}[1]{{#1}^{\opp}}
\newcommand{\und}[1]{{\underline{#1}}}

\newcommand{\xrightarrowdbl}[2][]{%
  \xrightarrow[#1]{#2}\mathrel{\mkern-14mu}\rightarrow
}

\newcommand{\sssum}[1]{{\sum\limits_{\substack{#1}}}}
\newcommand{\ssbigoplus}[1]{{\bigoplus\limits_{\substack{#1}}}}

\newcommand{\bV}{\raisebox{0.03cm}{\mbox{\footnotesize$\textstyle{\bigwedge}$}}}

\DeclareMathOperator{\id}{Id}
\DeclareMathOperator{\Hom}{Hom}
\DeclareMathOperator{\End}{End}

\DeclareMathOperator{\Ind}{Ind}
\DeclareMathOperator{\Res}{Res}

\DeclareMathOperator{\HOM}{HOM}
\DeclareMathOperator{\END}{END}
\DeclareMathOperator{\RHOM}{RHOM}

\DeclareMathOperator{\cRHom}{\mathcal{RH}\mathit{om}}

\DeclareMathOperator{\Lderiv}{L}

\newcommand{\Lotimes}{\otimes^{\Lderiv}}

\DeclareMathOperator{\opp}{{op}}

\DeclareMathOperator{\gdim}{gdim}

\DeclareMathOperator{\Image}{im}

\DeclareMathOperator{\mcolim}{MColim}

\DeclareMathOperator{\cone}{Cone}

\newcommand{\bKO}{\boldsymbol{K}_0}

\newcommand{\slt}{\mathfrak{sl}_2}
\newcommand{\g}{{\mathfrak{g}}}
\newcommand{\p}{{\mathfrak{p}}}
\newcommand{\bo}{{\mathfrak{b}}}

\newcommand\B{{\sf{B}}}
\newcommand\E{{\sf{E}}}
\newcommand\F{{\sf{F}}}

\newcommand{\oB}{\overline{B}}

\newcommand{\cyclicMod}{Y}
\newcommand{\dgCyclicMod}{G}
\newcommand{\predgCyclicMod}{\tilde G}

\newcommand{\quiverSchur}{Q}
\newcommand{\minquiverSchur}{{}^{red}Q}
\newcommand{\dgQuiverSchur}{{}_{dg}Q}
\newcommand{\mindgQuiverSchur}{{}_{dg}^{red}Q}

\newcommand{\antimapslt}{\bar \tau}

\newcommand{\qbinom}[2]{\genfrac{[}{]}{0pt}{}{#1}{#2}_q}

\newcommand{\br}{\mathrm{ \mathbf p}} 
\newcommand{\rb}{\mathrm{ \mathbf q}} 
\newcommand{\fr}{\mathrm{ \mathbf i}} 

\DeclareMathOperator{\amod}{\mathrm{-}mod}

\DeclareMathOperator{\Hqe}{Hqe}

\DeclareMathOperator{\nh}{NH}

\newcommand{\MV}{\cM\cV}
\newcommand{\TL}{\cT\cL}

\DeclareMathOperator{\Pol}{Pol} 

\definecolor{myblue}{rgb}{0,.5,1}
\definecolor{mygreen}{rgb}{.3,.75,.1}

\newcommand{\stikzdiag}[1][]{\tikz[#1, thick, scale=.25,baseline={([yshift=-.5ex]current bounding box.center)}]}
\newcommand{\tikzdiagh}[2][]{\tikz[#1,very thick,baseline={([yshift=1ex+#2]current bounding box.center)}]}
\newcommand{\tikzdiag}[1][]{\tikzdiagh[#1]{-1.5ex}}
\tikzstyle{tikzdot}=[fill, circle, inner sep=2pt]

\tikzstyle{stdhl}=[red,double=red!30,double distance=1pt]
\tikzstyle{vstdhl}=[myblue,double=myblue!30,double distance=1pt]
\tikzstyle{nail}=[draw=myblue,fill=white!30,circle,inner sep=2pt]

\newcommand{\tikzRBR}{{\stikzdiag{
\draw[stdhl](0,0) -- (0,1);\draw (1,0) -- (1,1);\draw[stdhl](2,0) -- (2,1);
}}}
\newcommand{\tikzBRR}{{\stikzdiag{
\draw(0,0) -- (0,1);\draw[stdhl] (1,0) -- (1,1);\draw[stdhl](2,0) -- (2,1);
}}}
\newcommand{\tikzRRB}{{\stikzdiag{
\draw[stdhl](0,0) -- (0,1);\draw[stdhl] (1,0) -- (1,1);\draw(2,0) -- (2,1);
}}}

\newcommand{\tikzbrace}[4]{\draw[decoration={brace,mirror,raise=-8pt},decorate]  (#1-.1,#3 -.35) -- node {#4} (#2+.1,#3-.35)}
\newcommand{\tikzbraceop}[4]{\draw[decoration={brace,raise=-8pt},decorate]  (#1-.1,#3 +.35) -- node {#4} (#2+.1,#3+.35)}

\newcommand{\thmasympKO}{Theorem 9.15}
\newcommand{\propcblfbim}{Proposition 9.18}

\title[Tensor product categorifications and the blob 2-category]{Tensor product categorifications, Verma modules and the blob 2-category}

\author{Abel Lacabanne}
\address{Institut de Recherche en Math\'ematique et Physique\\
Universit\'e Catholique de Louvain\\ 
Chemin du Cyclotron 2\\ 
1348 Louvain-la-Neuve\\ 
Belgium}
\email{abel.lacabanne@uclouvain.be}
\author{Gr\'egoire Naisse}
\address{Max-Planck Institute for Mathematics\\
 Vivatsgasse 7 \\ 
53111 Bonn\\ 
Germany}
\email{gregoire.naisse@gmail.com}
\author{Pedro Vaz}
\address{Institut de Recherche en Math\'ematique et Physique\\
Universit\'e Catholique de Louvain\\ 
Chemin du Cyclotron 2\\ 
1348 Louvain-la-Neuve\\ 
Belgium}
\email{pedro.vaz@uclouvain.be}
\begin{document}


\begin{abstract}
  We construct a dg-enhancement of KLRW algebras that categorifies the tensor product of a universal $\mathfrak{sl}_2$ Verma module and several integrable irreducible modules. 
  When the integrable modules are two-dimensional, we construct
a categorical action of the blob algebra on derived categories of these dg-algebras which
intertwines the categorical action of $\mathfrak{sl}_2$.
  From the above we derive a categorification of the blob algebra.
\end{abstract}


\maketitle

\setcounter{tocdepth}{1}
\tableofcontents



\section{Introduction}\label{sec:intro}

Dualities are fundamental tools in mathematics in general and in higher representation theory in particular.
For example, Stroppel's version of Khovanov homology~\cite{Str05,Str09} and Khovanov's HOMFLY--PT homology~\cite{Kh-soergel} can be seen as instances of higher Schur--Weyl duality (see also~\cite{Str-ICM} for further explanations).  
In this paper we construct an instance of higher Schur--Weyl duality between $U_q(\slt)$ and the blob algebra of Martin and Saleur~\cite{Martin-Saleur} by using a categorification of the tensor product of a Verma module and several two-dimensional irreducibles.

\subsection{State of the art}

\subsubsection{Schur--Weyl duality, $U_q(\mathfrak{sl}_2)$ and the Temperley--Lieb algebra }
Schur--Weyl duality connects finite-dimensional modules of the general linear and symmetric groups.    
In particular, it states that over an algebraically closed field the actions of $GL_m$ and $\mathfrak{S}_r$ on the $r$-folded tensor power of the natural module $V$ of $GL_m$ commute and are the centralizers of each other. In the quantum version, $GL_m$ and $\mathfrak{S}_r$ are replaced respectively by the quantum general linear algebra $U_q(\mathfrak{gl}_m)$ and the Hecke algebra $\mathcal{H}_r(q)$.
We note that these consequences of (quantum) Schur--Weyl duality remain true if one replaces the general linear with the special linear group. 
For example, in the case of $m=2$, the centralizer of the action of $U_q(\slt)$ on $V^{\otimes r}$ is the Temperley--Lieb algebra $TL_r$, a well-known quotient of the Hecke algebra. One of the applications of this connection is the construction of the Jones--Witten--Reshetikhin--Turaev $U_q(\slt)$-tangle invariant as a state-sum model (a linear combination of elements of $TL_r$) called the Kauffman bracket, which was the version categorified by Khovanov~\cite{Kh-jones} in the particular case of links.

\subsubsection{The blob algebra}
It was shown in~\cite{swTLB} that, for a projective $U_q(\slt)$ Verma module $M$ with highest weight $\lambda$ (in the sense that $\lambda$ is the eigenvalue), the endomorphism algebra of $M\otimes V^{\otimes r}$ is the blob algebra $\cB_r=\mathcal{B}_r(q,\lambda)$ of Martin--Saleur~\cite{Martin-Saleur}. 
This algebra $\cB_r$, which was unfortunately called Temperley--Lieb algebra of type B in~\cite{swTLB}, is in fact a quotient of the Temperley--Lieb algebra of type B~\cite{Graham,Green}. 
Note that the parameters $\lambda$ and $q$ in~\cite{swTLB} are not algebraically independent but can be easily made independent by working with a universal Verma module as in~\cite{LV}. 

The blob algebra $\mathcal{B}_r$ can be given a diagrammatic presentation in terms of $\bQ(q,\lambda)$-linear combinations of flat tangle diagrams~\cite{swTLB} on $r+1$ strands, with generators 
\begin{align*}
u_i &:= \ 
\tikzdiagh{0}{
	\draw[ultra thick,myblue] (-.5,0) -- (-.5,1);
	\draw[red] (0,0) -- (0,1);
	\node at (.5,.5){\small $\dots$};
	\draw[red] (1,0)  -- (1,1);
	\draw[red] (1.5, 0)  .. controls (1.5,.5) and (2,.5) .. (2,0)  ;
	\draw[red] (1.5, 1) .. controls (1.5,.5) and (2,.5) .. (2,1);
	\draw[red] (2.5,0) -- (2.5,1);
	\node at (3,.5){\small $\dots$};
	\draw[red] (3.5,0)  -- (3.5,1);
	\tikzbrace{-.5}{1}{0}{$i$};
}
\intertext{for $i=1,\dotsc ,r-1$, and}
\xi &:= \ 
\tikzdiag{
      \draw[red] (0,0) .. controls (0,.25) .. (-.5,.5) .. controls (0,.75) ..  (0,1);
      \draw[fill=white, color=white] (-.52,.5) circle (.02cm);
      \draw[red] (.5,0)    -- (.5,1);
      \node at (1,.5){\small $\dots$};
      \draw[red] (1.5,0) -- (1.5,1);
      \draw[ultra thick,myblue] (-.5,0)  -- (-.5,1);
}
\end{align*}
taken up to planar isotopy fixing the endpoints, and subject to the usual Temperley--Lieb relation of type A:
\begin{align*}\allowdisplaybreaks
\tikzdiag[yscale=.75]{
	\draw[red] (0, 0)  .. controls (0,.5) and (.5,.5) .. (.5,0) .. controls (.5,-.5) and (0,-.5) .. (0,0);
}
\ &= \ -(q+q^{-1}),
  \intertext{and the blob relations:}
    \tikzdiag[yscale=.75]{
	\draw[red] (0,0)   .. controls (0,.25) .. (-.5,.5) .. controls (0,.75) ..  
		(0,1) .. controls (0,1.5) and  (.5,1.5) .. 
		(.5,1) --
		(.5,0) .. controls (.5,-.5) and (0,-.5) ..
  (0,0);
  \draw[fill=white, color=white] (-.52,.5) circle (.02cm);
	\draw[ultra thick,myblue] (-.5,-.5) -- (-.5,1.5);
}
\ &= -(\lambda q+\lambda^{-1}q^{-1}) \ 
\tikzdiag[yscale=.75]{
	\draw[ultra thick,myblue] (-.5,-.5)  -- (-.5,1.5);
    }
    \\
q^{-1} \ 
\tikzdiag{
	\draw[red] (0,0) .. controls (0,.25) .. (-.5,.5) .. controls (0,.75) ..  
		(0,1) .. controls (0,1.25) .. (-.5,1.5) .. controls (0,1.75) ..
    (0,2);
    \draw[fill=white, color=white] (-.52,.5) circle (.02cm);
    \draw[fill=white, color=white] (-.52,1.5) circle (.02cm);
    \draw[ultra thick,myblue] (-.5,0) -- (-.5,2);
}
\ &= \ (\lambda q+\lambda^{-1}q^{-1}) \ 
\tikzdiag{
	\draw[red] (0,0)  --
		(0,.5) .. controls (0,.75) .. (-.5,1) .. controls (0,1.25) ..  
		(0,1.5)  --
    (0,2);
    \draw[fill=white, color=white] (-.52,1) circle (.02cm);
	\draw[ultra thick,myblue] (-.5,0) -- (-.5,2);
}
\ - q \ \tikzdiag{
	\draw[red] (0,0)  -- (0,2);
	\draw[ultra thick,myblue] (-.5,0)-- (-.5,2);
}  
\end{align*}
Note that this generators-relations definition of the blob algebra makes also sense over $\bZ[q^{\pm 1},\lambda^{\pm 1}]$.

\begin{rem}
In \cite{Martin-Saleur}, the blob algebra is given a different presentation, where the generator of type B is pictured as a dot on the left-most strand, and is an idempotent. We use the presentation given in \cite{swTLB}, which is isomorphic to the one in \cite{Martin-Saleur} over $\bZ(q, \lambda)$ (but not over $\bZ[q^{\pm1}, \lambda^{\pm 1}]$). 
This presentation is closer to the representation theory of $U_q(\slt)$ and is the one that arises from our categorification construction. 
\end{rem}

More generally, we consider the category $\cB$ with objects given by $M \otimes V^{\otimes r}$ for various $r \in \bN$, and hom-spaces given by $U_q(\slt)$-intertwiners. 
This category, that we call the \emph{blob category}, has a very similar diagrammatic description as the blob algebra, where objects are collections of $r+1$ points on the horizontal line. The hom-spaces are presented by flat tangles connecting these points, with the left-most point of the source always connected to the left-most point of the target, allowing 4-valent intersections between the first two strands. These diagrams are subject to the same relations as the blob algebra.
We stress that, in contrast to the Temperley--Lieb category of type A, the blob category is not monoidal w.r.t. juxtaposition of diagrams since the blue strand in the pictures above needs to be on the left-hand side of any diagram.

\subsubsection{Webster categorification}
In a seminal paper~\cite{webster}, Webster has constructed categorifications of tensor products of integrable modules for symmetrizable Kac--Moody algebras, generalizing Lauda's~\cite{L1}, Khovanov--Lauda~\cite{KL1,KL2} and Chuang--Rouquier~\cite{CR} and Rouquier's~\cite{rouquier} categorification of quantum groups, and their integrable modules. 
Webster further used his categorifications to give a link homology theory categorifying the Witten--Reshetikhin--Turaev invariant of tangles.  
The construction in~\cite{webster} involves algebras, called KLRW algebras (or tensor product algebras), that are finite-dimensional algebras presented diagrammatically, generalizing cyclotomic KLR algebras.  
Categories of  finitely generated modules over KLRW algebras come equipped with an action of  Khovanov--Lauda--Rouquier's 2-Kac--Moody category, and their Grothendieck groups are isomorphic to tensor products of integrable modules. 
Link invariants and categorifications of intertwiners are
constructed using functors given by the derived tensor product with certain bimodules over KLRW algebras.

\subsubsection{Verma categorification: dg-enhancements}

In~\cite{naissevaz1,naissevaz2,naissevaz3}, the second and third authors have given a categorification of (universal, parabolic) Verma modules for (quantized) symmetrizable Kac--Moody algebras. In its more general form~\cite{naissevaz3}, the categorification is given as a derived category of dg-modules over a certain dg-algebra, similar to a KLR algebra but containing an extra generator in homological degree $1$. 
 This dg-algebra can also be endowed with a collection of different differentials, each of them turning it into a dg-algebra whose homology is isomorphic to a cyclotomic KLR algebra.
This can be interpreted as a categorification of the projection of a universal Verma module onto an integrable module. Categorification of Verma modules was used by the second and third authors in~\cite{naissevaz4} to give a quantum group higher representation theory construction of Khovanov--Rozansky's HOMFLY--PT link homology.

\subsection{The work in this paper}

For $\lambda$ a formal parameter, let $M(\lambda)$ be the universal $U_q(\slt)$-Verma module with highest weight $\lambda$, and
$V(\und{N}):=V(N_1)\otimes \dotsm \otimes V(N_r)$, where $V(N_j)$ is the irreducible of highest weight $q^{N_j}$, $N_j \in\bN$. In this paper we combine Webster's categorification with the Verma categorification to give a categorification of $M(\lambda)\otimes V(\und{N})$. 
Then we construct a categorification of the blob algebra by 
categorifying the intertwiners of $M(\lambda)\otimes V(\und{N})$  where all the $N_j$ are 1. 

\subsubsection{Dg-enhanced KLRW algebras and categorification of tensor products
(Sections~\ref{sec:dgWebster} and~\ref{sec:catTensProd})}

Fix a commutative unital ring $\Bbbk$.
The KLRW algebra is the $\Bbbk$-algebra spanned by planar isotopy classes of braid-like diagrams whose strands are of two types: there are black strands labeled by simple roots of a symmetrizable Kac--Moody algebra $\g$ and carrying dots, and there are red strands labeled by dominant integral weights. 
KLRW algebras are cyclotomic algebras in the sense that they generalize cyclotomic KLR algebras to a string of dominant integral weights, 
where the ``\emph{violating condition}''~\cite[Definition 4.3]{webster} plays the role of the cyclotomic condition. 
KLRW algebras were also defined without the violating condition, in which case we call them non-cyclotomic or affine KLRW algebras. 
In the case of $\slt$, for $b\in\bN$ and $\und{N}\in\bN^r$, we denote by $T_b^{\und{N}}$ (resp. $\widetilde T_b^{\und N}$) 
 the (resp. affine) KLRW algebra spanned by $b$ black strands (all labeled by the simple root of $\slt$) and $r$ red strands, labeled in order $N_1,\dotsc, N_r$ from left to right.

Following a procedure analogous to~\cite{naissevaz2,naissevaz3}, we construct in~\cref{sec:dgWebster} an algebra $T_b^{\lambda,\und{N}}$, with $\lambda$ a formal parameter, that contains the affine KLRW algebra $\widetilde T_b^{\und N}$ as a subalgebra.
In a nutshell, $T_b^{\lambda,\und{N}}$ is defined by 
putting a vertical blue strand labeled by $\lambda$ on the left of the diagrams of $\widetilde T_b^{\und N}$, and adding a new generator that we call a nail (this corresponds with the ``tight floating dots'' of \cite{naissevaz2,naissevaz3}). We draw this new generator as:
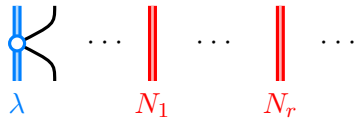
\[
	\tikzdiagh{0}{
		\draw (.5,-.5) .. controls (.5,-.25)  .. 
			(0,0) .. controls (.5,.25)  .. (.5,.5);
  \draw[vstdhl] (0,-.5) node[below]{\small $\lambda$} -- (0,.5) node [midway,nail]{};
\node at (1.2,0) {$\dotsm$};
  \draw[stdhl] (1.8,-.5) node[below]{\small $N_1$} -- (1.8,.5) ;
  \node at (2.65,0) {$\dotsm$};
 \draw[stdhl] (3.5,-.5) node[below]{\small $N_r$} -- (3.5,.5) ;
 \node at (4.3,0) {$\dotsm$};
}
\]
Note that a nail can only be placed on the left-most strand, which is always blue. The nails are subject to the following local relations:
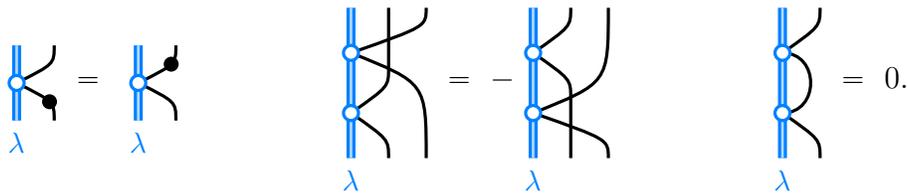
\begin{align*}
	\tikzdiagh{0}{
		\draw (.5,-.5) .. controls (.5,-.25)  ..  
			(0,0) node[midway, tikzdot]{}
			 .. controls (.5,.25)  .. (.5,.5);
	           \draw[vstdhl] (0,-.5) node[below]{\small $\lambda$} -- (0,.5) node [midway,nail]{};
  	}
  	\ &= \ 
	\tikzdiagh{0}{
		\draw (.5,-.5) .. controls (.5,-.25)  .. 
			(0,0) 
			.. controls (.5,.25)  .. (.5,.5) node[midway, tikzdot]{};
	           \draw[vstdhl] (0,-.5) node[below]{\small $\lambda$} -- (0,.5) node [midway,nail]{};
            }
          &
	\tikzdiagh{0}{
		\draw (.5,-1) .. controls (.5,-.75) .. (0,-.4)
				.. controls (.5,-.05) .. (.5,.2)
				-- (.5,1);
		\draw (1,-1) .. controls (1,0) .. (0, .4)
			.. controls (1,.75) .. (1,1); 
	     	\draw [vstdhl]  (0,-1) node[below]{\small $\lambda$} --  (0,1) node [pos=.3,nail] {} node [pos=.7,nail] {} ;
	 } \ &= \ -
	\tikzdiagh[yscale=-1]{0}{
		\draw (.5,-1) .. controls (.5,-.75) .. (0,-.4)
				.. controls (.5,-.05) .. (.5,.2)
				-- (.5,1);
		\draw (1,-1) .. controls (1,0) .. (0, .4)
			.. controls (1,.75) .. (1,1); 
	     	\draw [vstdhl]  (0,-1)  --  (0,1) node[below]{\small $\lambda$} node [pos=.3,nail] {} node [pos=.7,nail] {} ;
	 }
&
	\tikzdiagh{0}{
		\draw (.5,-1) .. controls (.5,-.75) .. (0,-.4)
			.. controls (.5,-.4) and (.5,.4) ..
			(0,.4) .. controls (.5,.75) .. (.5,1);
	     	\draw [vstdhl]  (0,-1) node[below]{\small $\lambda$} --  (0,1) node [pos=.3,nail] {} node [pos=.7,nail] {} ;
	 }
	\ &= \ 
	0. 
\end{align*}
When $\und{N}=\varnothing$ is the empty sequence, we recover the dg-enhanced nilHecke algebra from~\cite{naissevaz2}. 
The subalgebra spanned by all diagrams without a nail is isomorphic to the affine KLRW algebra $\widetilde T_b^{\und N}$.

As we will see, the algebra $T^{\lambda, \und N}_b$ can be equipped with three ($\bZ$-)gradings: two internal gradings, one as in Webster's original definition and an additional grading (see~\cref{def:dgwebsteralg}), as well as a homological grading. The first two of these gradings categorify the parameters $q$ and $\lambda$ respectively, and we call them $q$- and $\lambda$-gradings. As usual, the homological grading allows us to categorify relations involving minus signs. 
We write $q^k$ (resp. $\lambda^k$) for a grading shift up by $k$ in the $q$- (resp. $\lambda$-)grading, and $[k]$ for a grading shift up by $k$ in the homological grading, for $k \in \bZ$. 

We let the nail be in homological degree $1$, while diagrams without a nail are in homological degree $0$. 
As in the categorification of Verma modules, if we endow the algebra $T^{\lambda,\underline{N}}_b$ with a trivial differential, then it becomes a dg-algebra categorifying $M(\lambda)\otimes V(\und{N})$ (see below). 
We can also equip $T^{\lambda,\underline{N}}_b$ with a differential $d_N$, for $N\geq 0$, which acts trivially on diagrams without a nail, while 
\[
d_N\left(
	\tikzdiagh{-1ex}{
		\draw (.5,-.5) .. controls (.5,-.25)  .. 
			(0,0) .. controls (.5,.25)  .. (.5,.5);
	           \draw[vstdhl] (0,-.5) node[below]{\small $\lambda$} -- (0,.5) node [midway,nail]{};
  	}
  	\right) 
\ := \ 
	\tikzdiagh{-1ex}{
		\draw (.5,-.5) -- (.5,.5) node[midway,tikzdot]{} node[midway,xshift=1.75ex,yshift=.75ex]{\small $N$};
	           \draw[vstdhl] (0,-.5) node[below]{\small $\lambda$} -- (0,.5);
  	}
\]
and extending using the graded Leibniz rule.
The dg-algebra $(T^{\lambda,\underline{N}}_b,d_N)$ is formal with
homology isomorphic to the KLRW algebra $T^{(N,\underline{N})}_b$ (see~\cref{thm:dNformal}). 

The usual framework using the algebra map $T^{\lambda,\underline{N}}_b \rightarrow T^{\lambda,\underline{N}}_{b+1}$ that adds a black strand at the right of a diagram gives rise to induction and restriction dg-functors $\E_b$ and $\F_b$ between the derived dg-categories 
$\cD_{dg}(T^{\lambda,\underline{N}}_{b},0)$ and $\cD_{dg}(T^{\lambda,\underline{N}}_{b+1},0)$. 
The following describes the categorical $U_q(\slt)$-action:
\newtheorem*{thma}{\cref{thm:sl2comqi}}
\begin{thma}
There is a quasi-isomorphism
\[
\cone(\F_{b-1}\E_{b-1} \rightarrow \E_b\F_b) \xrightarrow{\cong} \oplus_{[\beta+|\underline{N}|-2b]_q} \id_b,
\]
of dg-functors.
\end{thma}
\noindent As usual in the context of categorification, the notation $\oplus_{[\beta+|\underline{N}|-2b]_q}$ on the right-hand side is an infinite coproduct categorifying multiplication by the rational fraction $(\lambda q^{|\underline{N}|-2b}-\lambda^{-1} q^{-|\underline{N}|+2b})/(q-q^{-1})$ interpreted as a Laurent series.

Turning on the differential $d_N$ gives functors $\E_b^N$, $\F_b^N$ on $\cD_{dg}(\oplus_{b\geq 0}T^{\lambda,\underline{n}}_{b},d_N)$. In this case, the right-hand side in \cref{thm:sl2comqi} becomes quasi-isomorphic to a finite sum and we recover the usual action on categories of modules over KLRW algebras (see~\cref{prop:actionN}).

In~\cite{asympK0}, the second author introduced the notion of an asymptotic Grothendieck group, which is a notion of a Grothendieck group for (multi)graded categories of objects admitting 
infinite iterated extensions (like infinite composition series or infinite resolutions) whose gradings satisfy some mild conditions. Denote by  ${}_\bQ\bKO^\Delta(-)$ the asymptotic Grothendieck group (tensored over $\bZ\pp{q,\lambda}$ with $\bQ\pp{q,\lambda}$). 
The categorical $U_q(\slt)$-actions on the derived categories $\cD_{dg}(\oplus_{b\geq 0}T^{\lambda,\underline{N}}_{b},0)$ and
$\cD_{dg}(\oplus_{b\geq 0}T^{\lambda,\underline{N}}_{b},d_N)$ descend to the asymptotic Grothendieck group and we have the main result of~\cref{sec:catTensProd}, which  reads as following: 
\newtheorem*{thmb}{\cref{thm:K0}}
\begin{thmb}
There are isomorphisms of $U_q(\slt)$-modules
\[
{}_\bQ\bKO^\Delta(T^{\lambda,\underline{N}},0) \cong M(\lambda) \otimes V(\underline{N}),
\]
and
\[
{}_\bQ\bKO^\Delta(T^{\lambda,\underline{N}},d_N)\cong V(N) \otimes V(\underline{N}),
\]
for all $N \in \bN$. 
\end{thmb}

In~\cref{sec:zigazag} we prove that in the case of $b=1$, $\und{N}={1,\dotsc ,1}$ and $N=1$, the dg-algebra $(T^{\lambda,1,\dotsc ,1}_{1},d_1)$ is isomorphic to a dg-enhanced zigzag algebra, generalizing~\cite[\S4]{qi-sussan}.

\subsubsection{The blob 2-category (Sections~\ref{sec:bimod} and~\ref{sec:catTLB})}

We study the case of $\und{N}=1,\dotsc ,1$ in more detail.
We define several functors on $\cD_{dg}(T^{\lambda,\und{N}},0)$ commuting with the categorical action of $U_q(\slt)$.
As in~\cite{webster}, these are defined as a first step via (dg-)bimodules over the abovementioned dg-enhancements of KLRW-like algebras.
To simplify matters, let $T^{\lambda,r}$ be the dg-enhanced KLRW algebra with $r$ strands labeled 1 and a blue strand labeled $\lambda$.
  The categorical Temperley--Lieb action is realized 
by a pair of biadjoint functors, constructed in the same way as in~\cite{webster}. 
They are given by derived tensoring with the $(T^{\lambda,r},T^{\lambda,r\pm 2})$-bimodules $B_i$ and $\overline{B}_i$ generated respectively by 
the diagram 
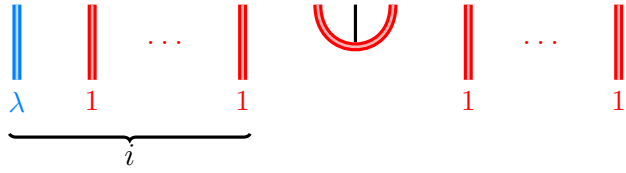
\[
	\tikzdiag{
		\draw[vstdhl] (0,0) node[below]{\small $\lambda$} -- (0,1);
		\draw[stdhl] (1,0) node[below]{\small $1$} -- (1,1);
		\node[red] at(2,.5) { $\dots$};
		\draw[stdhl] (3,0) node[below]{\small $1$} -- (3,1);
		\draw[decoration={brace,mirror,raise=-8pt},decorate]  (-.1,-1) -- node {$i$} (3.1,-1);
		\draw (4.5,.5) -- (4.5,1);
		\draw [stdhl] (4,1) .. controls (4,.25) and (5,.25) .. (5,1);
		\draw[stdhl] (6,0) node[below]{\small $1$} -- (6,1);
		\node[red] at(7,.5) { $\dots$};
		\draw[stdhl] (8,0) node[below]{\small $1$} -- (8,1);
	}
\]
and its mirror along a horizontal axis. We stress again that the blue strand is on the left. Moreover, these diagrams are subjected to some local relations (see \cref{sec:Tlaction}).
Taking the derived tensor product with these bimodules defines the coevaluation and evaluation dg-functors as 
\begin{align*}
\B_i := B_i \Lotimes_T - : \cD_{dg}(T^{\lambda,r},0) \rightarrow \cD_{dg}(T^{\lambda,r+2},0), \\
\overline{\B}_i := \overline{B}_i \Lotimes_T - : \cD_{dg}(T^{\lambda,r+2},0) \rightarrow \cD_{dg}(T^{\lambda,r},0). 
\end{align*}
In~\cref{sec:catTLaction} we extend~\cite{webster} and prove that these functors satisfy the relations of the Temperley--Lieb algebra: 

\newtheorem*{thmtla}{Corollaries~\ref{cor:TLbiadj} and~\ref{cor:TLloop}}
\begin{thmtla}
There are natural isomorphisms
\begin{align*}
\bar \B_{i \pm 1} \circ \B_i &\cong \id,
&
\overline \B_i \circ \B_i  &\cong q \id [1] \oplus q^{-1} \id [-1].
\end{align*}
\end{thmtla}

We define the double braiding functor in the same vein,  using the $(T^{\lambda,r},T^{\lambda,r})$-bimodule $X$ generated by the diagram 
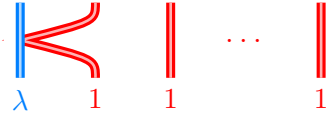
\[
	\tikzdiag{
		\draw[stdhl] (1,0) node[below]{\small $1$} .. controls (1,.25) .. (0,.5)
				.. controls (1,.75) .. (1,1);
		\draw[fill=white, color=white] (-.1,.5) circle (.1cm);
		\draw[vstdhl] (0,0)  node[below]{\small $\lambda$} -- (0,1);
		\draw[stdhl] (2,0) node[below]{\small $1$} -- (2,1);
		\node[red] at(3,.5) { $\dots$};
		\draw[stdhl] (4,0) node[below]{\small $1$} -- (4,1);
	}
\]
modulo the defining relations of $T^{\lambda,r}$, and the extra local relations
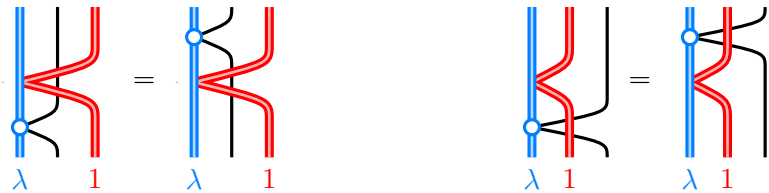
\begin{align*}
	\tikzdiagh{0}{
		\draw (.5,-.5) .. controls (.5,-.3) .. (0,-.1) .. controls (.5,.1) ..  (.5,.3) -- (.5,1.5);
		\draw[stdhl] (1,-.5) node[below]{\small $1$} -- (1,0) .. controls (1,.25) .. (0,.5)
				.. controls (1,.75) .. (1,1) -- (1,1.5);
		\draw[fill=white, color=white] (-.1,.5) circle (.1cm);
		\draw[vstdhl] (0,-.5)  node[below]{\small $\lambda$} -- (0,1.5) node[pos=.2,nail]{};
	}
\ &= \ 
	\tikzdiagh[yscale=-1]{0}{
		\draw (.5,-.5) .. controls (.5,-.3) .. (0,-.1) .. controls (.5,.1) ..  (.5,.3) -- (.5,1.5);
		\draw[stdhl] (1,-.5)-- (1,0) .. controls (1,.25) .. (0,.5)
				.. controls (1,.75) .. (1,1) -- (1,1.5) node[below]{\small $1$} ;
		\draw[fill=white, color=white] (-.1,.5) circle (.1cm);
		\draw[vstdhl] (0,-.5)   -- (0,1.5) node[pos=.2,nail]{} node[below]{\small $\lambda$};
	}
&
	\tikzdiagh{0}{
		\draw (1,-.5) .. controls (1,-.3) .. (0,-.1) .. controls (1,.1) ..  (1,.3) -- (1,1.5);
		\draw[stdhl] (.5,-.5) node[below]{\small $1$} -- (.5,0) .. controls (.5,.25) .. (0,.5)
				.. controls (.5,.75) .. (.5,1) -- (.5,1.5);
		\draw[fill=white, color=white] (-.1,.5) circle (.1cm);
		\draw[vstdhl] (0,-.5)  node[below]{\small $\lambda$} -- (0,1.5) node[pos=.2,nail]{};
	}
\ &= \ 
	\tikzdiagh[yscale=-1]{0}{
		\draw (1,-.5) .. controls (1,-.3) .. (0,-.1) .. controls (1,.1) ..  (1,.3) -- (1,1.5);
		\draw[stdhl] (.5,-.5) -- (.5,0) .. controls (.5,.25) .. (0,.5)
				.. controls (.5,.75) .. (.5,1) -- (.5,1.5)  node[below]{\small $1$};
		\draw[fill=white, color=white] (-.1,.5) circle (.1cm);
		\draw[vstdhl] (0,-.5)  -- (0,1.5) node[pos=.2,nail]{}  node[below]{\small $\lambda$};
	}
\end{align*}
The \emph{double braiding functor} is then defined as the derived tensor product 
\[
\Xi := X \Lotimes_T - : \cD_{dg}(T^{\lambda,r},0) \rightarrow \cD_{dg}(T^{\lambda,r},0). 
\]

The functors $B_i$, $\overline{B}_i$ and $\Xi$ intertwine the
categorical $U_q(\slt)$-action on $\cD_{dg}(T^{\lambda,r},0)$:  
\newtheorem*{thmc}{\cref{prop:catactioncommutes}}
\begin{thmc}
  We have natural isomorphisms $\E \circ \Xi \cong \Xi \circ \E$ and $\F \circ \Xi \cong \Xi \circ \F$, and also $\E \circ\B_i  \cong \B_i \circ\E$, $\F \circ\B_i \cong \B_i \circ\F$, and similarly for $\overline\B_i$.
\end{thmc}

The first  main result of~\cref{sec:catTLB}  is that the blob algebra acts on $\cD_{dg}(T^{\lambda,r},0)$. This follows from the Temperley--Lieb action in~\cref{sec:catTLaction} and~\cref{cor:qi-Xquadratic},~\cref{prop:Xi-autoequiv} and~\cref{cor:qi-bubbleremv}, summarized below. 
\newtheorem*{thmd}{\cref{cor:qi-Xquadratic},  \cref{prop:Xi-autoequiv} and \cref{cor:qi-bubbleremv}}
\begin{thmd}
  The functor $\Xi :  \cD_{dg}(T^{\lambda,r},0) \rightarrow \cD_{dg}(T^{\lambda,r},0)$ is an autoequivalence, with inverse given by
  \[
    \Xi^{-1} := \RHOM_T(X,-):  \cD_{dg}(T^{\lambda,r},0) \rightarrow \cD_{dg}(T^{\lambda,r},0) .
  \]
There are quasi-isomorphisms 
\begin{align*}
\cone\bigl(\lambda q^2 \Xi [1] \rightarrow q^2 \id [1]\bigr)[1] &\xrightarrow{\simeq} \cone( \Xi \circ \Xi \rightarrow \lambda^{-1} \Xi),
\intertext{and}
\lambda q (\id)[1] \oplus \lambda^{-1} q^{-1} (\id) [-1] &\xrightarrow{\simeq} \bar \B_1 \circ \Xi \circ \B_1,
\end{align*}
of dg-functors
\end{thmd}

 One of the main difficulties in establishing the results above is that,  in order to compute derived tensor products,  we have to take left (resp. right) cofibrant replacements of several dg-bimodules. 
As observed in~\cite[\S 2.3]{MW}, while the left (resp. right) module structure remains unchanged when passing to the  left (resp. right) cofibrant replacement, the right (resp. left) module structure is preserved only in the $A_\infty$ sense. 
As a consequence, constructing natural transformations between compositions of derived tensor product functors often requires to use $A_\infty$-bimodules maps.
We have tried to avoid as much as possible to end up in this situation, replacing the potentially unwieldy $A_\infty$-bimodules by quasi-isomorphic dg-bimodules. 

Let $\mathfrak B_r$ be a certain subcategory (see \cref{sec:blob2cat}) of the derived dg-category of $(T^{\lambda,r},0)$-$(T^{\lambda,r},0)$-bimodules generated by the dg-bimodules corresponding with the dg-functors identity, $\Xi^{\pm 1}$ and $\B_i \circ \overline{\B}_i $. 
Given two dg-bimodules in $\mathfrak B_r$, we can compose them in the derived sense by replacing both of them with a bimodule cofibrant replacement (i.e. a cofibrant replacement as dg-bimodule, and not only left or right dg-module), and taking the usual tensor product. This gives a dg-bimodule, isomorphic to the derived tensor product of the two initial dg-bimodules. In particular, it equips ${}_\bQ\bKO^\Delta(\mathfrak{B}_r)$ with a ring structure. 
We show that $\mathfrak B_r$ is a categorification of the blob algebra $\cB_r$ with ground ring extended to $\bQ\pp{q,\lambda}$: 
\begin{citecor}{cor:twoblobalg} 
There is an isomorphism of $\bQ\pp{q,\lambda}$-algebras
\[
{}_\bQ\bKO^\Delta(\mathfrak{B}_r) \cong \cB_r(q,\lambda).
\]
\end{citecor}
This result generalizes to the blob category. However, a technical issue we find here is that dg-categories up to quasi-equivalence do not form a 2-category, but rather an $(\infty,2)$-category~\cite{faonte}. Concretely in our case, we consider a sub-$(\infty,2)$-category of this $(\infty,2)$-category, where the objects are the derived dg-categories $\cD_{dg}(T^{\lambda,r},0)$ for various $r \in \bN$, and the $1$-hom are generated by the dg-functors  identity, $\Xi^{\pm 1}$, $\overline{\B}_i $ and $\B_i$. 
Moreover, these 1-hom are stable $(\infty,1)$-categories, and thus their homotopy categories are triangulated (see \cite{lurie}). 
In particular, we write ${}_\bQ\bKO^\Delta(\mathfrak{B}) $ for the category with the same objects as $\mathfrak{B}$ and with hom-spaces given by the asymptotic Grothendieck groups of the homotopy category of the $1$-hom of $\mathfrak{B}$. 
By~\cite{faonte} and~\cite{toen}, we can compute these hom-spaces by considering usual derived categories of dg-bimodules, and we obtain the following, again after extending the ground rings to $\bQ\pp{q,\lambda}$:
\begin{citecor}{cor:twoblob} 
There is an equivalence of categories
\[
{}_\bQ\bKO^\Delta(\mathfrak{B}) \cong \cB.
\]
\end{citecor}

\subsubsection{The general case: symmetrizable $\g$}

The definition of dg-enhanced KLRW algebras in~\cref{sec:dgWebster} generalizes immediately to any symmetrizable $\g$. We indicate this generalization in~\cref{sec:dgWebstergeneral}. We expect that the results of~\cref{sec:dgWebster} and \cref{sec:catTensProd} extend to this case without difficulty. 

\subsubsection{Quiver Schur algebras}

Quiver Schur algebras were introduced geometrically by Stroppel and Webster in~\cite{SW} to give a graded version of
the cyclotomic $q$-Schur algebras of Dipper, James and Mathas~\cite{DJM}.
Independently, Hu and Mathas~\cite{HM} constructed a graded Morita equivalent variant of the quiver Schur algebras in~\cite{SW} as graded quasi-hereditary covers of cyclotomic KLR algebras for linear quivers.
While the construction in~\cite{SW} is geometric, the construction in~\cite{HM} is combinatorial/algebraic.

More recently, Khovanov, Qi and Sussan~\cite{KSY} gave a variant of the quiver Schur algebras in~\cite{HM} for the case of cyclotomic nilHecke algebras, and showed that Grothendieck groups of their algebras can be identified with tensor products of integrable modules of $U_q(\mathfrak{sl}_2)$. 
Following similar ideas, in~\cref{sec:dgqSchur} we construct a dg-algebra, which we conjecture to be the quiver Schur variant of the dg-enhanced KLRW algebra of~\cref{sec:dgWebster}
(Conjectures~\ref{conj:dgSchurCycSchur} and~\ref{conj:WebQSchur}).

\subsubsection{Appendix}

We have moved the most computational proofs to \cref{sec:computations}, leaving only a sketch of some of the proofs in the main text.
The reader can also find in \cref{sec:dgcat} some explanations and results about homological algebra, $A_\infty$-structures and asymptotic Grothendieck groups.

\subsection{Possible future directions and applications}

\subsubsection{Khovanov homology for tangles of type B}

The topological interpretation of the blob algebra in~\cite[\S3.4]{swTLB} gives rise to a Jones polynomial for tangles of type B (i.e. tangles in the annulus). 
We expect that by introducing braiding functors as in \cite{webster}, we obtain a link homology of type B, yielding invariants of links in the annulus akin to ones introduced by Asaeda--Przytycki--Sikora~\cite{APS} (see also~\cite{AGW,GLW,QR-annular}).

Given a link in the annulus, the invariant obtained from our construction would be a dg-endofunctor of the derived dg-category of dg-modules over the dg-enhanced KLRW algebra $(T^{\lambda,\emptyset},0)$. This means that the empty link is sent to the dg-endomorphism space of the identity functor, which coincides with the Hochschild cohomology of $T^{\lambda,\emptyset}$, and is infinite-dimensional (the center of $T^{\lambda,\emptyset}$ is already infinite-dimensional). 
By  restricting to the subcategory of dg-modules over $(T^{\lambda,\emptyset}_0,0)$, it becomes $1$-dimensional since $T^{\lambda,\emptyset}_0 \cong \Bbbk$. 
With this restriction, we conjecture that our ``would-be'' invariant coincides with the usual annular Khovanov homology.

The following is a work in progress with A. Wilbert. As it is the case of using Webster's machinery~\cite{webster}, computing the tangle invariant of type B using our framework could be unwieldy. 
A more computation-friendly alternative could be 
to use dg-bimodules over annular arc algebras constructed using the annular TQFT of~\cite{APS}, as done in~\cite[\S5.3]{annulararcalg} (see also \cite[\S5]{ehrigtubbenhauer}). 
Furthermore, evidences show there is a (at least weak) categorical action of the blob algebra on the derived category of dg-modules over these annular arc algebras. 

In a different direction, one could try to extend our results to construct a Khovanov invariant for links in handlebodies, in the spirit of the handlebody HOMFLY--PT-link homology of Rose--Tubbenhauer in~\cite{RT}.

\subsubsection{Constructions using homotopy categories}

KLRW algebras are given diagrammatically, which is the often an appropriate framework for constructions with an additive flavor. 
Nevertheless, the various functors realizing the various intertwiners and the braiding need to pass to derived categories of modules. 
This makes it harder to describe explicitly the 2-categories realizing these symmetries since a bimodule for two of those algebras induces an $A_\infty$-bimodule on the level of derived categories. This was pointed out by Mackaay and Webster in~\cite{MW}, who gave explicit constructions of categorified intertwiners in order to prove the equivalence between the several existing $\mathfrak{gl}_n$-link homologies. 
One of the things~\cite{MW} tells us is how to construct homotopy versions of Webster's categorifications. 

A construction using homotopy categories for the results in this paper seems desirable from our point of view. We hope it can be done either by mimicking~\cite{MW}, which can turn out to be a technically challenging problem, or alternatively, by a construction of dg-enhancements for redotted Webster algebras, as considered in~\cite{KhS} and \cite{KhLSY} to give a homotopical version of some of the above, but whose low-tech presentation might hide difficulties.

\subsubsection{Generalized blob algebras and variants}

The results of~\cite{swTLB} were extended in~\cite{LV}, where the first and third authors have computed the endomorphism algebra of the $U_q(\mathfrak{gl}_m)$-module $M^{\p}(\Lambda)\otimes V^{\otimes n}$ for $M^{\p}(\Lambda)$ a parabolic universal Verma modules and $V$ the natural module of $U_q(\mathfrak{gl}_m)$, 
which is always a quotient of an Ariki--Koike algebra. 
As particular cases (depending on $\p$ and the relation between $n$ and $m$) we obtain Hecke algebras of type $B$ with two parameters, the generalized blob algebra of Martin and Woodcock~\cite{generalized_blob} or the Ariki--Koike algebra itself. 
With this result in mind it is tantalizing to ask for an extension to $\mathfrak{gl}_m$ of the work in this paper.
Modulo technical difficulties the methods in this paper could work for $\mathfrak{gl}_m$ in the case of a parabolic Verma module for a 2-block parabolic subalgebra, which is the case where the generators of the endomorphism algebra satisfy a quadratic relation.
Constructing a categorification of the Ariki--Koike algebra or the generalized blob algebra as the blob 2-category in~\cref{sec:catTLB} looks quite challenging at the moment, in particular for a functor-realization of the cyclotomic relation and the relation $\tau=0$ (for the generalized blob algebra in the presentation given in~\cite[Theorem 2.24]{LV}).

\smallskip


\subsection*{Acknowledgments}
The authors thank Catharina Stroppel for interesting discussions, and for pointing us~\cite{Martin-Saleur}, helping to clarify the confusion with the terminology of ``blob algebra'' and ``Temperley--Lieb algebra of type B''. 
The authors would also like to thank the referee for his/her numerous, detailed and helpful comments. 
A.L. was supported by the Fonds de la Recherche Scientifique - FNRS under Grant no.~MIS-F.4536.19. 
G.N. was a Research Fellow of the Fonds de la Recherche Scientifique - FNRS, under Grant no.~1.A310.16 when starting working on this project. G.N. is also grateful to the Max Planck Institute for Mathematics in Bonn for its hospitality and financial support.
P.V. was supported by the Fonds de la Recherche Scientifique - FNRS under Grant no.~MIS-F.4536.19.
%




\section{Quantum $\slt$ 
and the blob algebra}\label{sec:qsltTLB}

\subsection{Quantum $\slt$} 

Recall that \emph{quantum $\slt$} can be defined as the $\bQ(q)$-algebra $U_q(\slt)$, with generic $q$, generated by $K,K^{-1}, E$ and $F$ with relations
\begin{align*}
&KE = q^2EK, &  &KK^{-1} = 1 = K^{-1}K, \\
&KF = q^{-2}FK, & &EF - FE = \frac{K-K^{-1}}{q-q^{-1}}.
\end{align*}
Quantum $\slt$ becomes a bialgebra when endowed with comultiplication
\begin{align*}
\Delta(K^{\pm 1}) &:= K^{\pm 1} \otimes K^{\pm 1}, &
\Delta(E) &:= E \otimes 1 + K^{-1} \otimes E, &
 \Delta(F) &:= F \otimes K + 1 \otimes F,
\end{align*}
and with counit $\varepsilon(K^{\pm 1}) := 1$, $\varepsilon(E) := \varepsilon(F) := 0$.

There is a $\mathbb{Q}(q)$-linear anti-involution $\antimapslt$ of $U_q(\slt)$ defined on the generators by
\begin{align}
  \antimapslt(E) &:= q^{-1}K^{-1}F, & \antimapslt(F) &:= q^{-1}EK, & \antimapslt(K) &:= K.
\end{align}
It is easily checked that
\begin{equation}
  \label{eq:tauCoprod}
  \Delta\circ \antimapslt = (\antimapslt \otimes \antimapslt) \circ \Delta.
\end{equation}

\subsubsection{Integrable modules}
For each $N \in \bN$, there is a finite-dimensional irreducible $U_q(\slt)$-module $V(N)$, called \emph{integrable module},  with basis  $v_{N,0}, v_{N,1}, \dots, v_{N,N}$ and 
\begin{align*}
K \cdot v_{N,i} &:= q^{N-2i} v_{N,i}, \\
F \cdot v_{N,i} &:= v_{N,i+1}, \\
E \cdot v_{N,i} &:= [i]_q [N-i+1]_q v_{N,i-1},
\end{align*}
where $[n]_q$ is the $n$-th \emph{quantum integer}
\[
[n]_q := \frac{q^n-q^{-n}}{q-q^{-1}} = q^{n-1} + q^{n-1-2} + \cdots + q^{1-n}.
\]
In particular, let $V := V(1)$ be the \emph{fundamental $U_q(\slt)$-module}. 

\smallskip

The module $V(N)$ can be equipped with the \emph{Shapovalov form} 
\[
(-,- )_N : V(N) \times V(N) \rightarrow \bQ(q),
\]
 which is a non-degenerate bilinear form such that $(v_{N,0}, v_{N,0})_N = 1$ and which is $\antimapslt$-Hermitian: for any $v,v' \in V(N)$ and $u \in U_q(\slt)$, we have $(u\cdot v, v')_N= (v, \antimapslt(u)\cdot v')_N$. A computation shows that
\[
  (v_{N,i}, v_{N,j})_N = \delta_{i,j}q^{i(N-i)}\frac{[i]_q![N]_q!}{[N-i]_q!},
\]
where $[0]_q! := 1$ and $[n]_q! := [n]_q[n-1]_q\ldots [2]_q[1]_q$.

\subsubsection{Verma modules}

Let $\beta$ be a formal parameter and write $\lambda := q^{\beta}$ as a formal variable. Let $\bo$ be the standard upper Borel subalgebra of $\slt$ and $U_q(\bo)$ be its quantum version. It is the $U_q(\slt)$-subalgebra generated by $K,K^{-1}$ and $E$. 
Let $K_\lambda$ be a $1$-dimensional $\bQ(\lambda,q)$-vector space, with fixed basis element $v_\lambda$. 
We endow $K_\lambda$ with an  $U_q(\bo)$-action by declaring that: 
\begin{align*}
K^{\pm 1} v_\lambda &:= \lambda^{\pm 1} v_\lambda, & E v_\lambda &:= 0,
\end{align*}
extending linearly through the obvious inclusion $\bQ(q) \hookrightarrow \bQ(q,\lambda)$. 
The universal \emph{Verma module $M(\lambda)$} is the induced module
\[
M(\lambda) := U_q(\slt) \otimes_{U_q(\bo)} K_\lambda. 
\]
It is irreducible and infinite-dimensional with $\bQ(q,\lambda)$-basis $v_{\lambda,0} := v_\lambda, v_{\lambda,1}, \dots , v_{\lambda,i} , \dots$ and
\begin{align*}
K \cdot v_{\lambda,i} &:= \lambda q^{-2i} v_{\lambda,i}, \\
F \cdot v_{\lambda,i} &:= v_{\lambda, i+1}, \\
E \cdot v_{\lambda,i} &:=  [i]_q [\beta-i+1]_q v_{\lambda, i-1},
\end{align*}
where
\[
[\beta+k]_q := \frac{\lambda q^k - \lambda^{-1} q^{-k}}{q-q^{-1}}.
\]

\smallskip

The Verma module $M(\lambda)$ can also be equipped with a Shapovalov form $(\cdot,\cdot)_\lambda$, which is again a non-degenerate bilinear form such that $(v_\lambda, v_\lambda)_\lambda = 1$ and which is $\antimapslt$-Hermitian: for any $v,v' \in M(\lambda)$ and $u \in U_q(\slt)$, we have $(u\cdot v, v')_\lambda = (v, \antimapslt(u)\cdot v')_\lambda$. One easily calculates that
\[
  (v_{\lambda,i}, v_{\lambda,j})_\lambda = \delta_{i,j}\lambda^iq^{-i^2}[i]_q![\beta]_q[\beta-1]_q\cdots[\beta-i+1]_q.
\]

\subsubsection{Tensor products}

Given $W$ and $W'$ two $U_q(\slt)$-modules, their tensor product $W \otimes W'$ is again a $U_q(\slt)$-module with the action induced by $\Delta$. Explicitly, 
\begin{align*}
K^{\pm 1} \cdot (w \otimes w') &:= K^{\pm 1} w \otimes K^{\pm 1} w', \\
F \cdot (w \otimes w') &:= F w \otimes K w' + w \otimes Fw', \\
E \cdot (w \otimes w') &:= E w \otimes w' + K^{-1} w \otimes E w',
\end{align*}
for all $w \in W$ and $w' \in W'$.

\smallskip

For $\underline{N}=(N_1,\ldots,N_r)\in \mathbb{N}^r$ we write $V(\underline{N}) :=
V(N_1) \otimes \cdots \otimes V(N_r)$ and $M\otimes V(\underline{N}):=M(\lambda)\otimes V(N_1) \otimes \cdots \otimes V(N_r)$. 
In the particular case $N_1=\dotsm  =N_r=1$, we write
$V^r$ for the $r$-th folded tensor product $V \otimes V \otimes \cdots \otimes V$.

\subsubsection{Weight spaces} 

The module $M\otimes V(\underline{N})$ decomposes into \emph{weight spaces}
\[
M\otimes V(\underline{N})_{\lambda q^{k}} := \{ v \in M\otimes V(\underline{N}) | Kv = \lambda q^k v \}. 
\]
Note that we have $M\otimes V(\underline{N})  \cong \bigoplus_{\ell \geq 0} M\otimes V(\underline{N})_{\lambda q^{|\und N| - 2\ell}} $, where $|\und N| := \sum_i N_i$. 

\subsubsection{Basis} 

Let $\mathcal{P}_{b}^{r}$ be the set of weak  compositions of $b$ into $r+1$ parts, that is:
\[
  \mathcal{P}_{b}^{r}:=\left\{(b_0,b_1,\dots,b_r)\in\mathbb{N}^{r+1}\ \middle\vert\ \sum_{i=0}^r b_i = b\right\}.
\]
Consider also 
\[
\mathcal{P}_{b}^{r, \und N} := \left\{ (b_0, b_1, \dots, b_r)\in\mathcal{P}_{b}^{r} | b_i \leq N_i  \text{ for $1\leq i \leq r$}\right\} \subset \mathcal{P}_{b}^{r}.
\]

In addition to the induced basis by the tensor product, the space $M\otimes V(\underline{N})$ admits a basis that will be particularly useful for categorification. For $\rho = (b_0, \dots, b_r) \in  \mathcal{P}_{b}^{r}$, we write
\[
 v_\rho := F^{b_r}\left( \cdots F^{b_1} \left( F^{b_0}(v_\lambda) \otimes v_{N_1,0}  \right) \cdots \otimes v_{N_r,0}  \right).
\]
Then, $M\otimes V(\underline{N})$ has a basis given by
\[
\left\{
v_\rho 
| \rho \in \mathcal{P}_{b}^{r, \und N}, b \geq 0
\right\}.
\]
In particular, we have that $M\otimes V(\underline{N})_{\lambda q^{|\und N|- 2b}}$ has a basis given by $\{ v_\rho \}_{\rho \in \mathcal{P}_{b}^{r, \und N}}$. 

One can describe inductively the change of basis from $\{ v_\rho \}_{\rho \in \mathcal{P}_{b}^{r, \und N}}$ to the induced basis as follows:
\[
  v_{(b_0,\ldots,b_r)}=\sum_{k=0}^{\min(b_r,N_r)}q^{(1-k)(b_r-k)}\qbinom{b_r}{k}v_{(b_0,\ldots,b_{r-1}+b_r-k)}\otimes v_{N_r,k},
\]
for any $(b_0,\ldots,b_r)\in\mathcal{P}_{b}^{r}$ and
\[
  v_{(b_0,\ldots,b_{r-1})}\otimes v_{N_r,n} = \sum_{k=0}^n (-1)^{n-k}q^{(n-k)(n-2)}\qbinom{n}{k}v_{(b_0,\ldots,b_{r-1}+n-k,k)},
\]
for any $(b_0,\ldots,b_{r-1})\in \mathcal{P}_{b}^{r-1}$ and $0\leq n \leq N_r$, with $\qbinom{n}{k}:=\frac{[n]_q!}{[k]_q![n-k]_q!}$. 

We can also use these formulas to inductively rewrite a vector $v_{\rho}$ with $\rho\in \mathcal{P}_{b}^{r}$ in terms of various $v_{\kappa}$ for $\kappa\in\mathcal{P}_{b}^{r,\und N}$. Indeed, we have
\[
  v_{(b_0,\ldots,b_r)} = \sum_{k=0}^{\min(b_r,N_r)}q^{(1-N_r)(b_r-k)}\frac{\displaystyle\prod_{j=1,j\neq k}^{N_r}[b_r-j]_q}{\displaystyle\prod_{j=1,j\neq k}^{N_r}[k-j]_q}v_{(b_0,\ldots,b_{r-1}+b_r-k,k)},
\]
for any $(b_0,\ldots,b_r)\in\mathcal{P}_{b}^{r}$.

\subsubsection{Shapovalov forms for tensor products}\label{sec:shepfortensor}

Following \cite[\S4.7]{webster}, we consider a family of bilinear forms $(\cdot,\cdot)_{\lambda,\underline{N}}$ on tensor products of the form $M(\lambda)\otimes V(\underline{N})$  satisfying the following properties:
\begin{enumerate}
\item each form $(\cdot,\cdot)_{\lambda,\underline{N}}$ is non-degenerate;
\item for any $v,v'\in M(\lambda)\otimes V(\underline{N})$ and $u\in U_q(\slt)$ we have $(u \cdot v,v')_{\lambda,\underline{N}} = (v, \antimapslt(u)\cdot v')_{\lambda,\underline{N}}$;
\item for any $f\in \mathbb{Q}(q,\lambda)$ and $v,v'\in M(\lambda)\otimes V(\underline{N})$, we have $(f v,v')_{\lambda,\underline{N}} = (v,fv')_{\lambda,\underline{N}} = f(v,v')_{\lambda,\underline{N}}$;
\item if $v,v'\in M(\lambda)\otimes V(\underline{N})$, then we have $(v,v')_{\lambda,\underline{N}} = (v\otimes v_{N,0},v'\otimes v_{N,0})_{\lambda,\underline{N'}}$ where $\underline{N'}=(N_1,\ldots,N_r,N)$.
\end{enumerate}

Similarly to \cite[Proposition 4.33]{webster} we have:

\begin{prop}
  There exists a unique system of such bilinear forms which are given by
  \[
    (v , v')_{\lambda,\underline{N}} = (v, v')_{\lambda,\underline{N}}^{\Pi}, 
  \]
for every $v,v'\in M(\lambda)\otimes V(\underline{N})$ where $(\cdot,\cdot)_{\lambda,\underline{N}}^{\Pi}$ is the product of the universal Shapovalov form on $M(\lambda)$ and of the Shapovalov forms on the various $V(N_i)$.
\end{prop}

\subsection{The blob algebra}\label{sec:blobalgebra}

Recall that the \emph{blob algebra} $\cB_r$ is the $\bQ(\lambda,q)$-algebra with generators $u_1, \dots, u_{r-1}$ and $\xi$, and with the relations of type A:
\begin{align}
u_i u_j &= u_j u_i,  & \text{for $|i-j| > 1$,}& \label{eq:TLrels}\\
u_i u_{i+1} u_i &= u_i, &  \text{for $1 \leq i \leq r-2$,}& \label{eq:TLrels2} \\
u_i u_{i-1} u_i &= u_i, &  \text{for $2 \leq i \leq r-1$,}& \label{eq:TLrels3} \\
u_i^2 &= -(q+q^{-1}), & \text{for $1 \leq i \leq r-1$,}& \label{eq:TLloopremov}
\end{align}
and the blob relations:
\begin{align}
\xi u_i &= u_i \xi, &\text{for $2 \leq i \leq r$,} \label{eq:TLBrels}  \\
u_1 \xi u_1 &= -(\lambda q + \lambda^{-1} q^{-1}) u_1, \label{eq:TLBloopremov} \\
q^{-1} \xi^2 &= (\lambda q + \lambda^{-1} q^{-1}) \xi - q. \label{eq:TLBdoublebraid}
\end{align}
Note that $\xi$ is invertible, with inverse given by $\xi^{-1} = \lambda+q^{-2}\lambda^{-1} -q^{-2}\xi$, and that the relations \eqref{eq:TLrels}-\eqref{eq:TLloopremov} imply that the generators $u_1,\dotsc, u_{r-1}$ generate a subalgebra isomorphic to the Temperley--Lieb algebra of type $A$.

The blob algebra has several well-known diagrammatic presentations. The most classical one already appeared in~\cite{Martin-Saleur}, but (a slight modification of) the one in~\cite{swTLB} is more convenient for our purposes. This presentation is given by setting
\begin{align*}
u_i &= 
\tikzdiagh{0}{
      \draw[ultra thick,myblue] (-.5,0) -- (-.5,1);
      \draw[red] (0,0) -- (0,1);
      \node at (.5,.5){\small $\dots$};
      \draw[red] (1,0)  -- (1,1);
      \draw[red] (1.5, 0)  .. controls (1.5,.5) and (2,.5) .. (2,0);
      \draw[red] (1.5, 1) .. controls (1.5,.5) and (2,.5) .. (2,1);
      \draw[red] (2.5,0) -- (2.5,1);
      \node at (3,.5){\small $\dots$};
      \draw[red] (3.5,0) -- (3.5,1);
      \tikzbrace{-.5}{1}{0}{$i$};
}
\\
\xi &=
\tikzdiag{
      \draw[red] (0,0) .. controls (0,.25) .. (-.5,.5) .. controls (0,.75) ..  (0,1);
      \draw[fill=white, color=white] (-.52,.5) circle (.02cm);
      \draw[red] (.5,0) -- (.5,1);
      \node at (1,.5){\small $\dots$};
      \draw[red] (1.5,0) -- (1.5,1);
      \draw[ultra thick,myblue] (-.5,0) -- (-.5,1);
}
\end{align*}
where diagrams are taken up to planar isotopy and read from bottom to top, 
and with local relations
\begin{align}\allowdisplaybreaks
\tikzdiag{
	\draw[red] (0, 0)  .. controls (0,.5) and (.5,.5) .. (.5,0) .. controls (.5,-.5) and (0,-.5) .. (0,0);
}
\ &= \ -(q+q^{-1}), \tag{\ref{eq:TLloopremov}} \\
\tikzdiag{
  \draw[red] (0,0)   .. controls (0,.25) .. (-.5,.5) .. controls (0,.75) ..  
  (0,1) .. controls (0,1.5) and  (.5,1.5) .. 
  (.5,1) --
  (.5,0) .. controls (.5,-.5) and (0,-.5) ..
  (0,0);
  \draw[fill=white, color=white] (-.52,.5) circle (.02cm);
  \draw[ultra thick,myblue] (-.5,-.5) -- (-.5,1.5);
}
\ &= -(\lambda q+\lambda^{-1}q^{-1}) \ 
\tikzdiag{
	\draw[ultra thick,myblue] (-.5,-.5) -- (-.5,1.5);
} \tag{\ref{eq:TLBloopremov}}  \\
q^{-1}
\tikzdiag{
  \draw[red] (0,0) .. controls (0,.25) .. (-.5,.5) .. controls (0,.75) ..  
  (0,1) .. controls (0,1.25) .. (-.5,1.5) .. controls (0,1.75) ..
  (0,2);
  \draw[fill=white, color=white] (-.52,.5) circle (.02cm);
  \draw[fill=white, color=white] (-.52,1.5) circle (.02cm);
  \draw[ultra thick,myblue] (-.5,0) -- (-.5,2);
}
\ &= \ (\lambda q+\lambda^{-1}q^{-1})
\tikzdiag{
    \draw[red] (0,0) --
    (0,.5) .. controls (0,.75) .. (-.5,1) .. controls (0,1.25) ..  
    (0,1.5)  --
    (0,2); 
    \draw[fill=white, color=white] (-.52,1) circle (.02cm);
    \draw[ultra thick,myblue] (-.5,0) -- (-.5,2);
}
\ - q \ \tikzdiag{
	\draw[red] (0,0) -- (0,2);
	\draw[ultra thick,myblue] (-.5,0) -- (-.5,2);
} \tag{\ref{eq:TLBdoublebraid}} 
\end{align}
corresponding to \eqref{eq:TLloopremov}, \eqref{eq:TLBloopremov} and \eqref{eq:TLBdoublebraid} (explaining why we kept the same numbering). 
Note that the relations  \eqref{eq:TLrels} -- \eqref{eq:TLrels3}  and \eqref{eq:TLBrels} are encoded by the planar isotopies.

\begin{rem}\label{rem:dbbraiding}
  In the graphical description of $\cB_r$ given in~\cite{swTLB} the generator $\xi$ is presented as a double braiding (see~\cite[Figure 1]{swTLB}). We don't follow that interpretation in our diagrammatics in order to simplify pictures, but we keep the terminology (see~\S\ref{def:dbbraiding} ahead). 
\end{rem}

\begin{rem}
With respect to~\cite{swTLB} our conventions switch $(\lambda,q)$ and $(\lambda^{-1},q^{-1})$, which can be interpreted as exchanging the double braiding by the double inverse braiding. 
\end{rem}

There is an action of $\cB_r$ on $M \otimes V^r$ that intertwines with the quantum $\slt$-action. 
This action can be described locally, identifying the first vertical strand in $\cB_r$ with the identity on $M(\lambda)$, and the $i$th vertical strand with the identity on the $i$-th copy of $V$ in $M \otimes V^r$.  Then the action is given using the following maps
\begin{align*}
\tikzdiag{
	\draw[red] (1.5, 0)  .. controls (1.5,.5) and (2,.5) .. (2,0);
} 
 &: V \otimes V \rightarrow \bQ(q,\lambda),\  
\begin{cases}
v_{1,0} \otimes v_{1,0} &\mapsto 0, \\
v_{1,0} \otimes v_{1,1} &\mapsto 1, \\
v_{1,1} \otimes v_{1,0} &\mapsto -q^{-1}, \\
v_{1,1} \otimes v_{1,1} &\mapsto 0, \\
\end{cases}
\\
\tikzdiag[yscale=-1]{
	\draw[red] (1.5, 0)  .. controls (1.5,.5) and (2,.5) .. (2,0);
} 
 &:  \bQ(q,\lambda)\rightarrow V \otimes V, \ 
1 \mapsto -q v_{1,0} \otimes v_{1,1} + v_{1,1} \otimes v_{1,0},
\\
\tikzdiag{
  \draw[red] (0,0) .. controls (0,.25) .. (-.5,.5) .. controls (0,.75) ..  (0,1);
  \draw[fill=white, color=white] (-.52,.5) circle (.02cm);
  \draw[ultra thick,myblue] (-.5,0) -- (-.5,1);
}
 &: M \otimes V \rightarrow M \otimes V,\ 
 \begin{cases}
 v_{\lambda,k} \otimes v_{1,0} &\mapsto \lambda^{-1}q^{2k}v_{\lambda,k} \otimes v_{1,0}\\ &\qquad - q(q-q^{-1})[k]_q [\beta-k+1]_q v_{\lambda,k-1} \otimes v_{1,1},\\
 v_{\lambda,k} \otimes v_{1,1} &\mapsto (\lambda^{-1}+\lambda q^2-\lambda^{-1}q^{2(k+1)})v_{\lambda,k} \otimes v_{1,1}\\
 &\qquad-\lambda^{-1}q^{2(k+1)}(q-q^{-1})v_{\lambda,k+1} \otimes v_{1,0},
 \end{cases}
\end{align*}
where the formula for $\xi$ is obtained by acting twice with an $R$-matrix. 
In our conventions, we have $\xi=f\circ \Theta_{21} \circ f \circ \Theta$ where $\Theta$ is given by the action of
  \[
    \sum_{n=0}^{+\infty}(-1)^nq^{-n(n-1)/2}\frac{(q-q^{-1})^n}{[n]_q!}F^n\otimes E^n,
  \]
  $\Theta_{21}$ by the action of
  \[
    \sum_{n=0}^{+\infty}(-1)^nq^{-n(n-1)/2}\frac{(q-q^{-1})^n}{[n]_q!}E^n\otimes F^n,
  \]
  $f(v_{\lambda,k}\otimes v_{1,0}) := \lambda^{-1/2}q^{k}v_{\lambda,k}\otimes v_{1,0}$ and $f(v_{\lambda,k}\otimes v_{1,1}) := \lambda^{1/2}q^{-k}v_{\lambda,k}\otimes v_{1,1}$ for any $k\in\mathbb{N}$.
  
  The following will be useful later:
  
\begin{lem}\label{lem:explicitaction}
  The action of $\cB_r$ translates in terms of $v_\rho$-vectors of $M \otimes V^r$ as
  \begin{align}
    \label{eq:caponk}
    \tikzdiagh[scale=0.75]{2}{
	\draw[ultra thick,myblue] (-.5,0) -- (-.5,1);
	\draw[red] (0,0) -- (0,1);
	\node at (.5,.5){\small $\dots$};
	\draw[red] (1,0)  -- (1,1);
	\draw[red] (1.5, 0)  .. controls (1.5,.5) and (2,.5) .. (2,0)  ;
	\draw[red] (2.5,0) -- (2.5,1);
	\node at (3,.5){\small $\dots$};
	\draw[red] (3.5,0)  -- (3.5,1);
	\tikzbrace{-.5}{1}{-0.2}{$i$};
}
    :& v_{(\dots, b_{i-1}, b_i, b_{i+1}, b_{i+2}, \dots)} \mapsto -q^{-1} [b_i]_q  v_{(\dots, b_{i-1} + b_i + b_{i+1} - 1, b_{i+2}, \dots)},
    \\      
    \label{eq:cuponk}
    \tikzdiagh[scale=0.75,yscale=-1]{2}{
	\draw[ultra thick,myblue] (-.5,0) -- (-.5,1);
	\draw[red] (0,0) -- (0,1);
	\node at (.5,.5){\small $\dots$};
	\draw[red] (1,0)  -- (1,1);
	\draw[red] (1.5, 0)  .. controls (1.5,.5) and (2,.5) .. (2,0)  ;
	\draw[red] (2.5,0) -- (2.5,1);
	\node at (3,.5){\small $\dots$};
	\draw[red] (3.5,0) -- (3.5,1);
	\tikzbrace{-.5}{1}{1.9}{$i$};
}
    :&v_{\rho} \mapsto q[2]_q v_{(\dots, b_{i-1},1 ,0 ,b_{i}, \dots)} - q v_{(\dots, b_{i-1}+1, 0, 0,b_{i}, \dots)} -q v_{(\dots, b_{i-1},0 , 1,b_{i}, \dots)},
    \\
    \label{eq:xionk}
    \tikzdiag[scale=0.75]{
    \draw[red] (0,0) .. controls (0,.25) .. (-.5,.5) .. controls (0,.75) ..  (0,1);
    \draw[fill=white, color=white] (-.52,.5) circle (.02cm);
    \draw[red] (0.5,0) -- (0.5,1);
    \node at (1,0.5){\small $\dots$};
    \draw[red] (1.5,0) -- (1.5,1); 
    \draw[ultra thick,myblue] (-.5,0) -- (-.5,1);
    }
    :&v_{(b_0 , b_1 ,\dots)} \mapsto (\lambda^{-1}q^{b_0} - \lambda q [b_0]_q) v_{(0,b_0+b_1,\dots)} + \lambda q^2 [b_0]_q v_{(1,b_0+b_1-1,\dots)}.
  \end{align}
\end{lem}

\begin{proof}
A computational proof is given in \cref{sec:computations}.
\end{proof}

As a matter of fact, this completely determines $ \End_{U_q(\slt)}(M \otimes V^r)$:

\begin{thm}[{\cite[Theorem 4.9]{swTLB}}]\label{thm:BcongMV}
There is an isomorphism
\begin{equation}\label{eq:BcongMV}
\cB_r \cong \End_{U_q(\slt)}(M \otimes V^r).
\end{equation}
\end{thm}

The \emph{blob category $\cB$} is the $\bQ(\lambda,q)$-linear category given by 
\begin{itemize}
\item objects are non-negative integers $r \in \bN$;
\item $\Hom_{\cB}(r,r')$ is given by $\bQ(\lambda,q)$-linear combinations of string diagrams connecting $r+1$ points on the bottom to $r'+1$ points on the top, with the first strand always connecting the left-most point to the left-most point, where the strings cannot intersect each other except for diagrams like $\xi$. Diagrams are considered up to planar isotopy and subject to the relations \eqref{eq:TLloopremov}, \eqref{eq:TLBloopremov} and \eqref{eq:TLBdoublebraid}.
\end{itemize}
Let $\TL$ be the Temperley--Lieb category of type $A$, defined diagrammatically. It is a $\bQ(q)$-linear monoidal category equivalent to $\cF und(\mathfrak{sl}_2)$, the full monoidal subcategory of $ U_q(\slt)\amod$ generated by $V$. 
Note that $\cB$ can be endowed with a structure of module category over $\TL$, by gluing diagrams on the right. 

Also consider the full subcategory $\MV \subset U_q(\slt)\amod$ given by the modules $M(\lambda) \otimes V^{\otimes r}$ for all $r \in \bN$. It is a module category over $\cF und(\mathfrak{sl}_2)$ by acting on the right with tensor product of $U_q(\slt)$-modules. 

\begin{thm}[{\cite[Theorem 4.9]{swTLB}}]
There are equivalences of categories such that
\[
\begin{tikzcd}[ampersand replacement=\&]
\cB  \ar{d}{\vsimeq} \& \ar[swap]{l}{\otimes \text{acts}} \cT\cL \ar{d}{\vsimeq}
\\
\cM\cV \&  \ar{l}{\otimes \text{acts}} \cF und(\mathfrak{sl}_2)
\end{tikzcd}
\]
commutes.
\end{thm}

\begin{rem}
 Note that~\cite{swTLB} considers projective Verma modules with integral highest weight.
 The case of universal Verma modules was studied in~\cite{LV}, albeit not in the categorical setup. 
\end{rem}




\section{Dg-enhanced  KLRW algebras}\label{sec:dgWebster}

In~ \cite{naissevaz2} and~\cite{naissevaz3} it was explained how to construct a `dg-enhancement' of cyclotomic nilHecke algebras to pass from a categorification of the integrable module $V(N)$ to a categorification of the Verma module $M(\lambda)$.
This suggests that one might try to go from a categorification of $V(N) \otimes V(\underline{N})$ to a categorification of $M(\lambda) \otimes V(\underline{N})$ by constructing a dg-enhancement of KLRW algebras~\cite[\S4]{webster}, which we do next.

\subsection{Preliminaries and conventions} \label{sec:conventions}

Before defining the various algebras, we fix some conventions, and we recall some common facts about dg-structures (a reference for this is~\cite{keller}). 
First, let $\Bbbk$ be a commutative unital ring for the remaining of the paper. 

\subsubsection{Dg-algebras}

A \emph{$\bZ^n$-graded dg-($\Bbbk$-)algebra} $(A,d_A)$ is a unital $\bZ \times \bZ^n$-graded ($\Bbbk$-)algebra $A = \bigoplus_{(h,\bg) \in \bZ \times \bZ^n} A_\bg^h$, where we refer to the $\bZ$-grading as homological (or $h$-degree) and the $\bZ^n$-grading as $\bg$-degree, with a differential $d : A \rightarrow A$ such that:
\begin{itemize}
\item $d_A(A_\bg^h) \subset A_{\bg}^{h-1}$ for all $\bg \in \bZ^n, h \in \bZ$;
\item $d_A(xy) = d_A(x)y + (-1)^{\deg_h(x)} x d_A(y)$;
\item $d_A^2 = 0$. 
\end{itemize}
The \emph{homology} of $(A,d_A)$ is $H(A,d_A) := \ker(d)/\Image(d)$, 
which is a $\bZ \times \bZ^n$-graded algebra that decomposes as $\bigoplus_{h \in \bZ,\bg \in \bZ^n} H^h_\bg(A, d_A)  := H^h(A_\bg, d_A)$. 
A morphism of dg-algebras $f: (A,d_A) \rightarrow (A', d_{A'})$ is a morphism of algebras that preserves the $\bZ \times \bZ^n$-grading and commutes with the differentials. 
Such a morphism induces a morphism $f^* : H(A,d_A) \rightarrow H(A',d_{A'})$. We say that $f$ is a \emph{quasi-isomorphism} whenever $f^*$ is an isomorphism. Also, we say that $(A,d_A)$ is formal if there is a quasi-isomorphism $(A,d_A) \xrightarrow{\simeq} (H(A,d_A), 0)$. 

\begin{rem}
Note that, in contrast to \cite{keller}, the differential decreases the homological degree instead of increasing it.
\end{rem}

Similarly, a $\bZ^n$-graded left dg-module is a $\bZ \times \bZ^{n}$-graded module $M$ with a differential $d_M$ such that:
\begin{itemize}
\item $d_M(M_\bg^h) \subset M_{\bg}^{h-1}$ for all $\bg \in \bZ^n, h \in \bZ$;
\item $d_M(x \cdot m) = d_A(x) \cdot y + (-1)^{\deg_h(x)} x  \cdot d_M(y)$;
\item $d_M^2 = 0$. 
\end{itemize}
Homology, maps between dg-modules and quasi-isomorphisms are defined as above, and there are similar notions of $\bZ^n$-graded right dg-modules and dg-bimodules.

In our convention,  a $\bZ^m$-graded category is a category with a collection of $m$ autoequivalences, strictly commuting with each others. 
The category $(A,d_A)\amod$ of (left) $\bZ^n$-graded dg-modules over a dg-algebra $(A,d_A)$ is a $\bZ \times \bZ^n$-graded abelian category, with kernels and cokernels defined as usual. The action of $\bZ$ is given by the \emph{homological shift functor} $[1] : (A,d_A)\amod \rightarrow (A,d_A)\amod$ acting by:
\begin{itemize}
\item increasing the degree of all elements in a module $M$ up by $1$, i.e. $\deg_h(m[1]) = \deg_h(m) + 1$;
\item switching the sign of the differential $d_{M[1]} := -d_M$;
\item introducing a sign in the left-action $r \cdot (m[1])  := (-1)^{\deg_h(r)} (r \cdot m)[1]$.
\end{itemize}
The action of $\bg \in \bZ^n$ is given by increasing the $\bZ^n$-degree of elements up by $\bg$, in the sense that
\[
(\bg M)_{\bg_0 + \bg} := (M)_{\bg_0},
\]
or in other terms, an element $x \in M$ with degree $\bg_0$ becomes of degree $\bg_0+\bg$ in $\bg M$.
There are similar definitions for categories of right dg-modules and dg-bimodules, with the subtlety that the homological shift functor does not twist the right-action:
\[
(m[1]) \cdot r := (m \cdot r)[1].
\]
As usual, a short exact sequence of dg-(bi)modules induces a long exact sequence in homology. 

\smallskip

Let $f : (M,d_M) \rightarrow (N,d_N)$ be a morphism of dg-(bi)modules. Then, one constructs the \emph{mapping cone} of $f$ as 
\begin{align} \label{eq:cone}
\cone(f) &:= (M[1] \oplus N, d_C), & 
d_C &:= \begin{pmatrix} -d_M & 0 \\ f & d_N \end{pmatrix}.
\end{align}
The mapping cone is a dg-(bi)module, and it fits in a short exact sequence:
\[
0 \rightarrow N \xrightarrow{\imath_N} \cone(f) \xrightarrow{\pi_{M[1]}} M[1] \rightarrow 0,
\]
where $\imath_N$ and $\pi_{M[1]}$ are the inclusion and projection $N \xrightarrow{\imath_N} M[1] \oplus N \xrightarrow{\pi_{M[1]}} M[1]$. 

\subsubsection{Hom and tensor functors}\label{sec:classicalhomandtensor}

Given a left dg-module $(M,d_M)$ and a right dg-module $(N,d_N)$, one constructs the tensor product
\begin{equation}\label{eq:dgtens}
\begin{split}
(N,d_N) \otimes_{(A,d_A)} (M,d_M) &:= \bigl( (M \otimes_A N), d_{M \otimes N} \bigr), \\
d_{M \otimes N}(m \otimes n) &:= d_M(m) \otimes n + (-1)^{\deg_h(m)} m \otimes d_N(n).
\end{split}
\end{equation}
If $(N,d_N)$ (resp. $(M,d_M)$) has the structure of a dg-bimodule, then the tensor product inherits a left (resp. right) dg-module structure. 

Given a pair of left dg-modules $(M,d_M)$ and $(N,d_N)$, one constructs the dg-hom space
\begin{equation}\label{eq:dghom}
\begin{split}
\HOM_{(A,d_A)}\bigl( (M,d_M), (N,d_N) \bigr) &:= \bigl( \HOM_A(M,N), d_{\HOM(M,N)} \bigr), \\
d_{\HOM(M,N)}(f) &:= d_N \circ f - (-1)^{\deg_h(f)} f \circ d_M,
\end{split}
\end{equation}
where $\HOM_A$ is the $\bZ\times \bZ^n$-graded hom space of maps between $\bZ\times \bZ^n$-graded $A$-modules. Again, if $(M,d_M)$ (resp. $(N,d_N)$) has the structure of a dg-bimodule, then it inherits a left (resp. right) dg-module structure.

In particular, given a dg-bimodule $(B,d_B)$ over a pair of dg-algebras $(S,d_{S})$-$(R,d_R)$, we obtain tensor and hom functors
\begin{align*}
B \otimes_{(R,d_R)} (-) &: (R,d_R)\amod \rightarrow(S,d_S)\amod, \\
\HOM_{(S, d_{S})}(B, -) &: (S,d_{S})\amod \rightarrow(R,d_R)\amod,
\end{align*}
which form a adjoint pair $(B \otimes_{(R,d_R)} -) \vdash \HOM_{(S,d_{S})}(B, -)$. 
Explicitly, the natural bijection
\begin{equation}\label{eq:homtensajd}
\Phi_{M,N}^B : \Hom_{(S,d_S)}( B \otimes_{(R,d_R)} M, N  ) \xrightarrow{\simeq} \Hom_{(R,d_R)}(M, \HOM_{(S,d_S)}(B,N)), 
\end{equation}
is given by $(f : B \otimes_{(R,d_R)} M \rightarrow  N) \mapsto \bigl(m \mapsto (b \mapsto f(b \otimes m))) \bigr)$.

\subsubsection{Diagrammatic algebras}

We always read diagram from bottom to top. We say that a diagram is braid-like when it is given by strands connecting a collection of points on the bottom to a collection of points on the top, without being able the turnback. Suppose these diagrams can have singularities (like dots, 4-valent crossings, or other similar decorations). 

A \emph{braid-like planar isotopy} is an isotopy fixing the endpoints and that does not create any critical point, in particular it means we can exchange distant singularities $f$ and $g$:
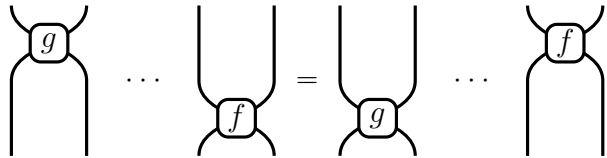
 \[
 \tikzdiag{
	\draw (0,-1) -- (0,0) ..controls (0,.5) and (1,.5) .. (1,1);
	\draw (1,-1) -- (1,0) ..controls (1,.5) and (0,.5) .. (0,1);
		\filldraw [fill=white, draw=black,rounded corners] (.5-.25,.5-.25) rectangle (.5+.25,.5+.25) node[midway] { $g$};
}
\quad
\cdots
\quad
\tikzdiag{
	\draw (0,0) ..controls (0,.5) and (1,.5) .. (1,1) -- (1,2);
	\draw (1,0) ..controls (1,.5) and (0,.5) .. (0,1) -- (0,2);
		\filldraw [fill=white, draw=black,rounded corners] (.5-.25,.5-.25) rectangle (.5+.25,.5+.25) node[midway] { $f$};
}
\ =  \ 
\tikzdiag{
	\draw (0,0) ..controls (0,.5) and (1,.5) .. (1,1) -- (1,2);
	\draw (1,0) ..controls (1,.5) and (0,.5) .. (0,1) -- (0,2);
		\filldraw [fill=white, draw=black,rounded corners] (.5-.25,.5-.25) rectangle (.5+.25,.5+.25) node[midway] { $g$};
}
\quad
\cdots
\quad
\tikzdiag{
	\draw (0,-1) -- (0,0) ..controls (0,.5) and (1,.5) .. (1,1);
	\draw (1,-1) -- (1,0) ..controls (1,.5) and (0,.5) .. (0,1);
		\filldraw [fill=white, draw=black,rounded corners] (.5-.25,.5-.25) rectangle (.5+.25,.5+.25) node[midway] { $f$};
}
 \]

Suppose that the diagrams carry a homological degree (associated to singularities), and consider linear combination of such diagrams. Then, a \emph{graded braid-like planar isotopy} is an isotopy fixing the endpoints, that does not create any critical point and such that we get a sign whenever we exchange two distant singularities $f$ and $g$:
 \[
 \tikzdiag{
	\draw (0,-1) -- (0,0) ..controls (0,.5) and (1,.5) .. (1,1);
	\draw (1,-1) -- (1,0) ..controls (1,.5) and (0,.5) .. (0,1);
		\filldraw [fill=white, draw=black,rounded corners] (.5-.25,.5-.25) rectangle (.5+.25,.5+.25) node[midway] { $g$};
}
\quad
\cdots
\quad
\tikzdiag{
	\draw (0,0) ..controls (0,.5) and (1,.5) .. (1,1) -- (1,2);
	\draw (1,0) ..controls (1,.5) and (0,.5) .. (0,1) -- (0,2);
		\filldraw [fill=white, draw=black,rounded corners] (.5-.25,.5-.25) rectangle (.5+.25,.5+.25) node[midway] { $f$};
}
\ = (-1)^{|f||g|} \ 
\tikzdiag{
	\draw (0,0) ..controls (0,.5) and (1,.5) .. (1,1) -- (1,2);
	\draw (1,0) ..controls (1,.5) and (0,.5) .. (0,1) -- (0,2);
		\filldraw [fill=white, draw=black,rounded corners] (.5-.25,.5-.25) rectangle (.5+.25,.5+.25) node[midway] { $g$};
}
\quad
\cdots
\quad
\tikzdiag{
	\draw (0,-1) -- (0,0) ..controls (0,.5) and (1,.5) .. (1,1);
	\draw (1,-1) -- (1,0) ..controls (1,.5) and (0,.5) .. (0,1);
		\filldraw [fill=white, draw=black,rounded corners] (.5-.25,.5-.25) rectangle (.5+.25,.5+.25) node[midway] { $f$};
}
 \]
 where $|f|$ (resp. $|g|$) is the homological degree of $f$ (resp. $g$).


\subsection{Dg-enhanced KLRW algebras}

Let $\underline{N} =  (N_1,\ldots,N_r)$.  Recall the KLRW algebra~\cite[\S4]{webster}  (also called tensor product algebra) on $b$ strands $T_b^{\underline{N}}$ is the diagrammatic $\Bbbk$-algebra generated by braid-like diagrams on $b$ black strands and $r$ red strands.
Red strands are labeled from left to right by $N_1, \dots, N_r$ and cannot intersect each other, while black strands can intersect red strands transversely, they can intersect transversely among themselves and can carry dots.
Diagrams are taken up to braid-like planar isotopy, and satisfy local relations~\eqref{eq:nhR2andR3}-\eqref{eq:redR3}
which are given below, together with the \emph{violating condition} that a black strand in the leftmost region is $0$:
\[
\tikzdiagh{0}{
	\draw (0,0) -- (0,1)  ;
	\draw[stdhl] (1,0) node[below]{\small $N_1$}-- (1,1)  ;
} 
\quad \cdots \quad
\ = 0.
\]
We write $\widetilde T_b^{\underline{N}}$ for the same construction but without the violating condition.

The following are the defining (local) relations of $T_b^{\underline{N}}$:
\begin{itemize}
\item The \emph{nilHecke relations}:
\begin{align}
\label{eq:nhR2andR3}
\tikzdiag{
	\draw (0,0) ..controls (0,.25) and (1,.25) .. (1,.5) ..controls (1,.75) and (0,.75) .. (0,1)  ;
	\draw (1,0) ..controls (1,.25) and (0,.25) .. (0,.5)..controls (0,.75) and (1,.75) .. (1,1)  ;
} 
\ &=\  
0
&
\tikzdiag{
	\draw  (0,0) .. controls (0,0.25) and (1, 0.5) ..  (1,1);
	\draw  (1,0) .. controls (1,0.5) and (0, 0.75) ..  (0,1);
	\draw  (0.5,0) .. controls (0.5,0.25) and (0, 0.25) ..  (0,0.5)
		 	  .. controls (0,0.75) and (0.5, 0.75) ..  (0.5,1);
} 
\ &= \ 
\tikzdiag[xscale=-1]{
	\draw  (0,0) .. controls (0,0.25) and (1, 0.5) ..  (1,1);
	\draw  (1,0) .. controls (1,0.5) and (0, 0.75) ..  (0,1);
	\draw  (0.5,0) .. controls (0.5,0.25) and (0, 0.25) ..  (0,0.5)
		 	  .. controls (0,0.75) and (0.5, 0.75) ..  (0.5,1);
} \\
\label{eq:nhdotslide}
\tikzdiag{
	\draw (0,0) ..controls (0,.5) and (1,.5) .. (1,1) node [near start,tikzdot]{};
	\draw (1,0) ..controls (1,.5) and (0,.5) .. (0,1);
}
\ &= \ 
\tikzdiag{
	\draw (0,0) ..controls (0,.5) and (1,.5) .. (1,1) node [near end,tikzdot]{};
	\draw (1,0) ..controls (1,.5) and (0,.5) .. (0,1);
}
\ + \ 
\tikzdiag{
	\draw (0,0) -- (0,1)  ;
	\draw (1,0)-- (1,1)  ;
}
&
\tikzdiag{
	\draw (0,0) ..controls (0,.5) and (1,.5) .. (1,1);
	\draw (1,0) ..controls (1,.5) and (0,.5) .. (0,1) node [near end,tikzdot]{};
}
\ &= \ 
\tikzdiag{
	\draw (0,0) ..controls (0,.5) and (1,.5) .. (1,1);
	\draw (1,0) ..controls (1,.5) and (0,.5) .. (0,1) node [near start,tikzdot]{};
}
\ + \ 
\tikzdiag{
	\draw (0,0) -- (0,1)  ;
	\draw (1,0)-- (1,1)  ;
} 
\end{align}

\item The \emph{black/red relations}:
\begin{align}
\tikzdiagh{0}{
	\draw (1,0) ..controls (1,.5) and (0,.5) .. (0,1) node [near end,tikzdot]{};
	\draw[stdhl] (0,0) node[below]{\small $N_i$} ..controls (0,.5) and (1,.5) .. (1,1);
}
\ &= \ 
\tikzdiagh{0}{
	\draw (1,0) ..controls (1,.5) and (0,.5) .. (0,1) node [near start,tikzdot]{};
	\draw[stdhl] (0,0) node[below]{\small $N_i$} ..controls (0,.5) and (1,.5) .. (1,1);
}
&
\tikzdiagh{0}{
	\draw (0,0) ..controls (0,.5) and (1,.5) .. (1,1) node [near start,tikzdot]{};
	\draw[stdhl] (1,0) node[below]{\small $N_i$} ..controls (1,.5) and (0,.5) .. (0,1);
}
\ &= \ 
\tikzdiagh{0}{
	\draw (0,0) ..controls (0,.5) and (1,.5) .. (1,1) node [near end,tikzdot]{};
	\draw[stdhl] (1,0) node[below]{\small $N_i$} ..controls (1,.5) and (0,.5) .. (0,1);
} 
\label{eq:dotredstrand}
\\
\tikzdiagh{0}{
	\draw (1,0) ..controls (1,.25) and (0,.25) .. (0,.5)..controls (0,.75) and (1,.75) .. (1,1)  ;
	\draw[stdhl] (0,0) node[below]{\small $N_i$} ..controls (0,.25) and (1,.25) .. (1,.5) ..controls (1,.75) and (0,.75) .. (0,1)  ;
} 
\ &= \ 
\tikzdiagh{0}{
	\draw[stdhl] (0,0) node[below]{\small $N_i$} -- (0,1)  ;
	\draw (1,0) -- (1,1)  node[midway,tikzdot]{}  node[midway,xshift=1.75ex,yshift=.75ex]{\small $N_i$} ;
} 
&
\tikzdiagh{0}{
	\draw (0,0) ..controls (0,.25) and (1,.25) .. (1,.5) ..controls (1,.75) and (0,.75) .. (0,1)  ;
	\draw[stdhl] (1,0) node[below]{\small $N_i$} ..controls (1,.25) and (0,.25) .. (0,.5)..controls (0,.75) and (1,.75) .. (1,1)  ;
} 
\ &= \ 
\tikzdiagh{0}{
	\draw (0,0) -- (0,1)  node[midway,tikzdot]{}   node[midway,xshift=1.75ex,yshift=.75ex]{\small $N_i$} ;
	\draw[stdhl] (1,0) node[below]{\small $N_i$} -- (1,1)  ;
} 
\label{eq:redR2}
\\
\tikzdiagh{0}{
	\draw  (0.5,0) .. controls (0.5,0.25) and (0, 0.25) ..  (0,0.5)
		 	  .. controls (0,0.75) and (0.5, 0.75) ..  (0.5,1);
	\draw (1,0)  .. controls (1,0.5) and (0, 0.75) ..  (0,1);
	\draw  [stdhl] (0,0) node[below]{\small $N_i$} .. controls (0,0.25) and (1, 0.5) ..  (1,1);
} 
\ &= \ 
\tikzdiagh[xscale=-1]{0}{
	\draw  (0,0) .. controls (0,0.25) and (1, 0.5) ..  (1,1);
	\draw  (0.5,0) .. controls (0.5,0.25) and (0, 0.25) ..  (0,0.5)
		 	  .. controls (0,0.75) and (0.5, 0.75) ..  (0.5,1);
	\draw [stdhl] (1,0) node[below]{\small $N_i$} .. controls (1,0.5) and (0, 0.75) ..  (0,1);
} 
&
\tikzdiagh{0}{
	\draw  (0,0) .. controls (0,0.25) and (1, 0.5) ..  (1,1);
	\draw  (0.5,0) .. controls (0.5,0.25) and (0, 0.25) ..  (0,0.5)
		 	  .. controls (0,0.75) and (0.5, 0.75) ..  (0.5,1);
	\draw [stdhl] (1,0) node[below]{\small $N_i$} .. controls (1,0.5) and (0, 0.75) ..  (0,1);
} 
\ &= \ 
\tikzdiagh[xscale=-1]{0}{
	\draw  (0.5,0) .. controls (0.5,0.25) and (0, 0.25) ..  (0,0.5)
		 	  .. controls (0,0.75) and (0.5, 0.75) ..  (0.5,1);
	\draw (1,0)  .. controls (1,0.5) and (0, 0.75) ..  (0,1);
	\draw  [stdhl] (0,0) node[below]{\small $N_i$} .. controls (0,0.25) and (1, 0.5) ..  (1,1);
} 
\label{eq:crossingslidered}
\\
\tikzdiagh{0}{
	\draw  (0,0) .. controls (0,0.25) and (1, 0.5) ..  (1,1);
	\draw  (1,0) .. controls (1,0.5) and (0, 0.75) ..  (0,1);
	\draw [stdhl] (0.5,0)node[below]{\small $N_i$}  .. controls (0.5,0.25) and (0, 0.25) ..  (0,0.5)
		 	  .. controls (0,0.75) and (0.5, 0.75) ..  (0.5,1);
} 
\ &= \ 
\tikzdiagh[xscale=-1]{0}{
	\draw  (0,0) .. controls (0,0.25) and (1, 0.5) ..  (1,1);
	\draw (1,0)  .. controls (1,0.5) and (0, 0.75) ..  (0,1);
	\draw [stdhl]  (0.5,0) node[below]{\small $N_i$} .. controls (0.5,0.25) and (0, 0.25) ..  (0,0.5)
		 	  .. controls (0,0.75) and (0.5, 0.75) ..  (0.5,1);
} 
\ + \sssum{k+\ell=\\N_i-1} \ 
\tikzdiagh{0}{
	\draw  (0,0) -- (0,1) node[midway,tikzdot]{} node[midway,xshift=-1.5ex,yshift=.75ex]{\small $k$};
	\draw  (1,0) --  (1,1) node[midway,tikzdot]{} node[midway,xshift=1.5ex,yshift=.75ex]{\small $\ell$};
	\draw [stdhl] (0.5,0)node[below]{\small $N_i$}  --  (0.5,1);
} \label{eq:redR3}
\end{align}
\end{itemize}

Multiplication is given by vertical concatenation of diagrams if the labels and colors of the strands agree, and is zero otherwise. 
As explained  in \cite[\S4]{webster}, the algebra $T_b^{\underline{N}}$ is finite-dimensional and $\bZ$-graded (we refer to this grading as $q$-grading), with
\begin{align}\label{eq:KLRWgrading}
\deg_q \left(
\ 
\tikzdiag{
	\draw (0,0)  ..controls (0,.5) and (1,.5) .. (1,1);
	\draw (1,0)  ..controls (1,.5) and (0,.5) .. (0,1);
}
\ 
\right)
&= -2,
&
\deg_q \left(
\ 
\tikzdiag{
	\draw (0,0) -- (0,1) node [midway,tikzdot]{};
}
\ 
\right)
& = 2,
&
\deg_q \left(
\ 
\tikzdiag{
	\draw (1,0)  ..controls (1,.5) and (0,.5) .. (0,1);
	\draw[stdhl] (0,0) node[below]{\small $N_i$}  ..controls (0,.5) and (1,.5) .. (1,1);
}
\ 
\right)
=
\deg_q \left(
\ 
\tikzdiag{
	\draw (0,0)  ..controls (0,.5) and (1,.5) .. (1,1);
	\draw[stdhl] (1,0) node[below]{\small $N_i$}  ..controls (1,.5) and (0,.5) .. (0,1);
}
\ 
\right)
&= N_i.
\end{align} 

In the case of $ \underline{N}=(N)$ the algebra $T_b^{(N)}$ contains a single red strand labeled $N$, and is isomorphic to the cyclotomic nilHecke algebra $\nh_b^N$. 

\begin{defn}\label{def:dgwebsteralg}
  The \emph{dg-enhanced KLRW algebra}  $T_b^{\lambda,\underline{N}}$ is the diagrammatic $\Bbbk$-algebra carrying an homological degree generated by braid-like diagrams on $b$ black strands, $r$ red strands and a blue strand on the left. Red strands are labeled from left to right by $N_1, \dots, N_r$ and the blue strand is labeled $\lambda$. Black strands can carry dots and intersect transversely with black and red strands. Moreover, the left-most black strand can be \emph{nailed} on the blue strand, giving a 4-valent vertex as follows:
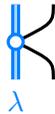
\[
	\tikzdiagh{0}{
		\draw (.5,-.5) .. controls (.5,-.25)  .. 
			(0,0) .. controls (.5,.25)  .. (.5,.5);
	           \draw[vstdhl] (0,-.5) node[below]{\small $\lambda$} -- (0,.5) node [midway,nail]{};
  	}
\]
We put the crossings and the dot in homological degree $0$, while the nail is in homological degree $1$. 
These diagrams are taken modulo graded braid-like planar isotopy, and subject to the local relations~\eqref{eq:nhR2andR3}-\eqref{eq:redR3} of $T_b^{\underline{N}}$, together with the local relations: 
\begin{align}
  \label{eq:relNail}
	\tikzdiagh{0}{
		\draw (.5,-.5) .. controls (.5,-.25)  ..  
			(0,0) node[midway, tikzdot]{}
			 .. controls (.5,.25)  .. (.5,.5);
	           \draw[vstdhl] (0,-.5) node[below]{\small $\lambda$} -- (0,.5) node [midway,nail]{};
  	}
  	\ &= \ 
	\tikzdiagh{0}{
		\draw (.5,-.5) .. controls (.5,-.25)  .. 
			(0,0) 
			.. controls (.5,.25)  .. (.5,.5) node[midway, tikzdot]{};
	           \draw[vstdhl] (0,-.5) node[below]{\small $\lambda$} -- (0,.5) node [midway,nail]{};
  	}
&
	\tikzdiagh{0}{
		\draw (.5,-1) .. controls (.5,-.75) .. (0,-.4)
				.. controls (.5,-.05) .. (.5,.2)
				-- (.5,1);
		\draw (1,-1) .. controls (1,0) .. (0, .4)
			.. controls (1,.75) .. (1,1); 
	     	\draw [vstdhl]  (0,-1) node[below]{\small $\lambda$} --  (0,1) node [pos=.3,nail] {} node [pos=.7,nail] {} ;
	 } \ &= \ -
	\tikzdiagh[yscale=-1]{0}{
		\draw (.5,-1) .. controls (.5,-.75) .. (0,-.4)
				.. controls (.5,-.05) .. (.5,.2)
				-- (.5,1);
		\draw (1,-1) .. controls (1,0) .. (0, .4)
			.. controls (1,.75) .. (1,1); 
	     	\draw [vstdhl]  (0,-1)  --  (0,1) node[below]{\small $\lambda$} node [pos=.3,nail] {} node [pos=.7,nail] {} ;
	 }
&
	\tikzdiagh{0}{
		\draw (.5,-1) .. controls (.5,-.75) .. (0,-.4)
			.. controls (.5,-.4) and (.5,.4) ..
			(0,.4) .. controls (.5,.75) .. (.5,1);
	     	\draw [vstdhl]  (0,-1) node[below]{\small $\lambda$} --  (0,1) node [pos=.3,nail] {} node [pos=.7,nail] {} ;
	 }
	\ &= \ 
	0. 
\end{align}
\end{defn}

\begin{rem}
Note that there can be no black or red strand at the left of the blue strand. 
\end{rem}

\begin{rem}
Note that since nails are stuck on the left, we can not exchange them using a graded braid-like planar isotopy. Thus, because nails are the only generators carrying a non-zero homological degree, we could  consider diagrams up to usual braid-like planar isotopy. 
However, the homological degree of the nail will play an important role in the categorification of the structure constant $[\beta + k]_q$ appearing in $M(\lambda) \otimes V(\underline{N})$, and graded braid-like planar isotopy will play an important role in \cref{sec:bimod}.
\end{rem}

Clearly, there is an injection of algebra $\widetilde T_b^{\underline{N}} \hookrightarrow T_b^{\lambda,\underline{N}}$ given by adding a vertical blue strand at the left of a diagram in $\widetilde T_b^{\underline{N}}$.

We endow $T_b^{\lambda,\underline{N}}$ with an extra $\bZ^2$-grading, the first one being inherited from $\widetilde T_b^{\underline{N}}$ and denoted $q$, the second is written $\lambda$. We declare that
\begin{align*}
\deg_{q,\lambda}  \left(
	\tikzdiagh{-1ex}{
		\draw (.5,-.5) .. controls (.5,-.25)  .. 
			(0,0) .. controls (.5,.25)  .. (.5,.5);
	           \draw[vstdhl] (0,-.5) node[below]{\small $\lambda$} -- (0,.5) node [midway,nail]{};
  	}
  	\right) &:= (0,2),
\end{align*}
and the elements without a nail are all in degree $\lambda$ zero and have the same $q$-degree as in \cref{eq:KLRWgrading}, so that the inclusion $\widetilde T_b^{\underline{N}} \hookrightarrow T_b^{\lambda,\underline{N}}$ preserves the $q$-grading. 
One easily checks that it is well-defined.

In the case of $ \underline{N}=\varnothing$, the algebra $T_b^{\lambda,\varnothing}$ contains only a blue strand labeled $\lambda$ and no red strands, and is isomorphic to the dg-enhanced nilHecke algebra introduced in~\cite[Definition 2.3]{naissevaz2}. To match with the notation from~\cite{naissevaz2}, we write $A_b := T_b^{\lambda,\varnothing}$. 

We will often endow $T_b^{\lambda,\underline{N}}$ with a trivial differential, turning it into a $\bZ^2$-graded dg-algebra $(T_b^{\lambda,\underline{N}}, 0)$. 

\subsection{Basis theorem}

For any $\rho=(b_0,b_1,\dots,b_r)\in \mathcal{P}_b^r$, define the idempotent 
\[
1_{\rho} := \tikzdiagh{0}{
	\draw[vstdhl] (0,0) node[below]{\small $\lambda$} --(0,1);
	\draw (.5,0) -- (.5,1);
	\node at(1,.5) {\tiny$\dots$};
	\draw (1.5,0) -- (1.5,1);
	\draw[decoration={brace,mirror,raise=-8pt},decorate]  (.4,-.35) -- node {$b_0$} (1.6,-.35);
	\draw[stdhl] (2,0)  node[below]{\small $N_1$} --(2,1);
	\draw (2.5,0) -- (2.5,1);
	\node at(3,.5) {\tiny$\dots$};
	\draw (3.5,0) -- (3.5,1);
	\draw[decoration={brace,mirror,raise=-8pt},decorate]  (2.4,-.35) -- node {$b_1$} (3.6,-.35);
	\draw[stdhl] (4,0)  node[below]{\small $N_2$} --(4,1);
	\node[red] at  (5,.5) {\dots};
	\draw[stdhl] (6,0)  node[below]{\small $N_{r}$} --(6,1);
	\draw (6.5,0) -- (6.5,1);
	\node at(7,.5) {\tiny$\dots$};
	\draw (7.5,0) -- (7.5,1);
	\draw[decoration={brace,mirror,raise=-8pt},decorate]  (6.4,-.35) -- node {$b_r$} (7.6,-.35);
}
\]
of $T_b^{\lambda,\underline{N}}$. Note that $T_b^{\lambda,\underline{N}} \cong \bigoplus_{\kappa, \rho \in \cP_b^r} 1_\kappa T_b^{\lambda,\underline{N}} 1_\rho$ as $\bZ\times\bZ^2$-graded $\Bbbk$-module.

\subsubsection{Polynomial action}
We now define an action of the dg-algebra $T_b^{\lambda,\underline{N}}$ on $\Pol_b^r := \bigoplus_{\rho \in \cP_b^r} \Pol_b  \varepsilon_\rho$, the free module over the ring $\Pol_b:=\mathbb{Z}[x_1,\dots,x_b]\otimes \bV^{\bullet}(\omega_1,\dots,\omega_b)$ generated by $\varepsilon_\rho$ for each $\rho\in\mathcal{P}_b^r$.

We recall the action of the symmetric group $S_b$ on $\Pol_b$ used in~\cite[\S2.2]{naissevaz2}. We view $S_b$ as a Coxeter group with generators $\sigma_i = (i\ i+1)$. The generator $\sigma_i$ acts on $\Pol_b$ as follows,
\begin{align*}
  \sigma_i(x_j) &:= x_{\sigma_i(j)},\\
  \sigma_i(\omega_j) &:= \omega_j + \delta_{i,j}(x_i-x_{i+1})\omega_{i+1}.
\end{align*}

For $\kappa,\rho \in \mathcal{P}_b^r$, an element of $1_{\kappa}T_b^{\lambda,\und N}1_{\rho}$ acts by zero on any $\Pol_b\varepsilon_{\rho'}$ for $\rho'\neq \rho$ and sends $\Pol_b\varepsilon_{\rho}$ to $\Pol_b\varepsilon_{\kappa}$. It remains to describe the action of the local generators of $T^{\lambda,\und N}_b$ on a polynomial $f\in \Pol_b$. First, similarly as in~\cite[Lemma 4.12]{webster}, we put
\begin{align*}
  \tikzdiagh{-1.5ex}{
  \node at(0,.5) {\tiny$\dots$};
  \draw (0.5,0) -- (0.5,1) node [midway,tikzdot]{};
  \node at(1,.5) {\tiny$\dots$};
  }\cdot f &:= x_if,
             &
  \tikzdiagh{-1.5ex}{
  \node at(0,.5) {\tiny$\dots$};
  \draw (0.5,0) ..controls (0.5,.5) and (1.5,.5) .. (1.5,1);
  \draw (1.5,0) ..controls (1.5,.5) and (0.5,.5) .. (0.5,1);
  \node at(2,.5) {\tiny$\dots$};
  }\cdot f &:= \frac{f-\sigma_i(f)}{x_i-x_{i+1}},\\
  \tikzdiagh{0}{
  \node at(0,.5) {\tiny$\dots$};
  \draw (0.5,0) ..controls (0.5,.5) and (1.5,.5) .. (1.5,1);
  \draw[stdhl] (1.5,0) node[below]{\small $N$} ..controls (1.5,.5) and (0.5,.5) .. (0.5,1);
  \node at(2,.5) {\tiny$\dots$};
  }\cdot f &:= f,
             &
  \tikzdiagh{0}{
  \node at(0,.5) {\tiny$\dots$};
  \draw (1.5,0) ..controls (1.5,.5) and (0.5,.5) .. (0.5,1);
  \draw[stdhl] (0.5,0) node[below]{\small $N$} ..controls (0.5,.5) and (1.5,.5) .. (1.5,1);
  \node at(2,.5) {\tiny$\dots$};
  }\cdot f &:= x_i^{N}f,
\end{align*}
where we identify $x_i \in \Pol_b\varepsilon_{\rho}$ with $x_i \in \Pol_b\varepsilon_{\kappa}$, and 
where we only have drawn the $i$-th or the $i$-th and $(i+1)$-th black strands, counting from left to right. Furthermore, we put
\[
  \\
  \tikzdiagh{0}{
  \draw (.5,0) .. controls (.5,.25) .. (0,0.5) .. controls (.5,.75)  .. (.5,1);
  \draw[vstdhl] (0,0) node[below]{\small $\lambda$} -- (0,1) node [midway,nail]{};
  \node at(1,.5) {\tiny$\dots$};
}\cdot f := \omega_1 f.
\]

\begin{lem}
  The rules above define an action of $T_b^{\lambda,\underline{N}}$ on $\Pol_b^r$.
\end{lem}

\begin{proof}
  We easily check that the relations \eqref{eq:nhR2andR3}-\eqref{eq:redR3} and \eqref{eq:relNail} are satisfied.
\end{proof}

Fix $\rho=(b_0,\dots,b_r) \in \mathcal{P}_b^r$.
Let $\nh_n$ be the nilHecke algebra on $n$-strands (it is described as a diagrammatic algebra with only black strands having dots and relations \cref{eq:nhR2andR3} and \cref{eq:nhdotslide}). 
There is a map $\eta_\rho \colon A_{b_0} \otimes \nh_{b_1} \otimes \cdots \otimes \nh_{b_r} \rightarrow T_b^{\lambda,\underline{N}}$, diagrammatically given by 
\begin{equation}\label{eq:defetainclusion}
\tikzdiagh[xscale=1.5]{0}{
	\draw [vstdhl] (-.25,-.5) node[below]{\small $\lambda$} -- (-.25,.5);
	\draw (0,-.5) -- (0,.5);
	\node at(.25,-.4) {\tiny $\dots$};
	\node at(.25,.4) {\tiny $\dots$};
	\draw (.5,-.5) -- (.5,.5);
	\filldraw [fill=white, draw=black] (-.35,-.25) rectangle (.6,.25) node[midway] {$A_{b_0}$};
}
\otimes
\tikzdiag[xscale=1.5]{
	\draw (0,-.5) -- (0,.5);
	\node at(.25,-.4) {\tiny $\dots$};
	\node at(.25,.4) {\tiny $\dots$};
	\draw (.5,-.5) -- (.5,.5);
	\filldraw [fill=white, draw=black] (-.1,-.25) rectangle (.6,.25) node[midway] { $\nh_{b_{1}}$};
}
\otimes
\cdots
\otimes
\tikzdiag[xscale=1.5]{
	\draw (0,-.5) -- (0,.5);
	\node at(.25,-.4) {\tiny $\dots$};
	\node at(.25,.4) {\tiny $\dots$};
	\draw (.5,-.5) -- (.5,.5);
	\filldraw [fill=white, draw=black] (-.1,-.25) rectangle (.6,.25) node[midway] { $\nh_{b_r}$};
}
\ \xmapsto{\eta_\rho} \ 
\tikzdiagh[xscale=1.5]{0}{
	\draw [vstdhl] (-.25,-.5) node[below]{\small $\lambda$} -- (-.25,.5);
	\draw (0,-.5) -- (0,.5);
	\node at(.25,-.4) {\tiny $\dots$};
	\node at(.25,.4) {\tiny $\dots$};
	\draw (.5,-.5) -- (.5,.5);
	%
	\draw [stdhl] (.75,-.5)  node[below]{\small $N_1$} -- (.75,.5);
	\draw (1,-.5) -- (1,.5);
	\node at(1.25,-.4) {\tiny $\dots$};
	\node at(1.25,.4) {\tiny $\dots$};
	\draw (1.5,-.5) -- (1.5,.5);
	%
	\draw [stdhl] (1.75,-.5) node[below]{\small $N_2$}   -- (1.75,.5);
	\node[red] at(2.125,0) { $\dots$};
	\draw [stdhl] (2.5,-.5)   node[below]{\small $N_r$} -- (2.5,.5);
	\draw (2.75,-.5) -- (2.75,.5);
	\node at(3,-.4) {\tiny $\dots$};
	\node at(3,.4) {\tiny $\dots$};
	\draw (3.25,-.5) -- (3.25,.5);
	%
	\filldraw [fill=white, draw=black] (-.35,-.25) rectangle (.6,.25) node[midway] {$A_{b_0}$};
	\filldraw [fill=white, draw=black] (1-.1,-.25) rectangle (1+.6,.25) node[midway] { $\nh_{b_1}$};
	\filldraw [fill=white, draw=black] (2.75-.1,-.25) rectangle (2.75+.6,.25) node[midway] { $\nh_{b_r}$};
}
\end{equation}
where we recall that $A_{b_0}$ is isomorphic to the dg-enhanced nilHecke algebra of \cite{naissevaz2}, identifying the nilHecke generators with each other and the the nail with the ``leftmost floating dot''. 
The tensor product $A_{b_0}\otimes \nh_{b_1} \otimes \cdots \nh_{b_r}$ acts on $\Pol_b^r$ through $\eta_\rho$. This action is only non-zero on $\Pol_b\varepsilon_{\rho}$ and it is readily checked that this action coincides with the tensor product of the polynomial actions of $A_{b_0}$ on $\mathbb{Z}[x_1,\dots,x_{b_0}]\otimes \bV^{\bullet}(\omega_1,\dots,\omega_{b_0}) \subset \Pol_b$ from \cite[\S2.2]{naissevaz2}, and of the usual action of the nilHecke algebra $\nh_{b_i}$ on $\mathbb{Z}[x_{b_0+\dots+b_{i-1}+1},\dots,x_{b_0+\dots+b_i}] \subset \Pol_b$ (see for example \cite[\S2.3]{KL1}).

\begin{lem}\label{lem:injectstoANH}
  The map $\eta_\rho$ is injective.
\end{lem}

\begin{proof}
  It follows immediately from the faithfulness of the polynomial actions of $A_{b_0}$~\cite[Corollary 3.9]{naissevaz2} and of $\nh_{b_i}$ \cite[Corollary 2.6]{KL1}.
\end{proof}

\subsubsection{Left-adjusted expressions}
We recall the notion of a left-adjusted expression as in~\cite[Section 2.2.1]{naissevaz2}: a reduced expression $\sigma_{i_1}\cdots\sigma_{i_k}$ of an element $w\in S_{r+b}$ is said to be \emph{left-adjusted} if $i_1+\cdots+i_k$ is minimal. One can obtain a left-adjusted expression of any element of $S_{r+b}$ by taking recursively its representative in the left coset decomposition
\[
  S_n = \bigsqcup_{t=1}^{n}S_{n-1}\sigma_{n-1}\cdots\sigma_{t}.
\]

As one easily confirms, if we think of permutations in terms of string diagrams, then a left-reduced expression is obtained by pulling every strand as far as possible to the left.

\subsubsection{A basis of $T^{\lambda,\protect\underline{N}}_b$}\label{ssec:basisTl}
We now turn to the diagrammatic description of a basis of $T^{\lambda,\underline{N}}_b$ similar to~\cite[Section 3.2.3]{naissevaz3}. For an element $\rho \in \mathcal{P}^r_b$ and $1 \leq k \leq b$, we define the tightened nail $\theta_k \in 1_\rho T^{\lambda,\underline{N}}_b 1_\rho$ as the following element:
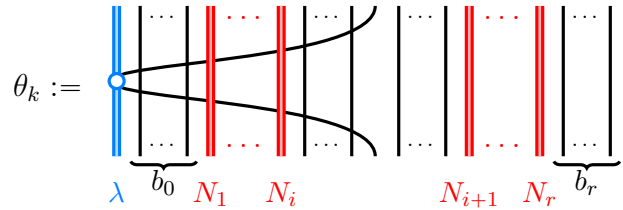
\[
\theta_k:=\tikzdiagh[xscale=1.25]{0}{
	\draw (0,-1) -- (0,1);
	\node at(.25,.85) {\tiny $\dots$};
	\node at(.25,-.85) {\tiny $\dots$};
	\draw (.5,-1) -- (.5,1);
	\draw[decoration={brace,mirror,raise=-8pt},decorate]  (-.1,-1.35) -- node {\small $b_0$} (.6,-1.35);
	\node[red] at(1.125,-.85) { $\dots$};
	\node[red] at(1.125,.85) { $\dots$};
	\draw (1.75,-1) -- (1.75,1);
	\node at(2,-.85) {\tiny $\dots$};
	\node at(2,.85) {\tiny $\dots$};
	\draw (2.25,-1) -- (2.25,1);
	\draw (2.5,-1) .. controls (2.5,-.25) and (-.5,-.25) ..  (-.25,0) .. controls (-.5,.25) and (2.5,.25) ..  (2.5,1);
	\draw (2.75,-1) -- (2.75,1);
	\node at(3,-.85) {\tiny $\dots$};
	\node at(3,.85) {\tiny $\dots$};
	\draw (3.25,-1) -- (3.25,1);
	\draw [stdhl]  (3.5,-1)  node[below,yshift={-1ex}]{\small $N_{i+1}$} -- (3.5,1);
	\node[red] at(3.875,-.85) { $\dots$};
	\node[red] at(3.875,.85) { $\dots$};
	\draw [stdhl]  (4.25,-1)  node[below,yshift={-1ex}]{\small $N_{r}$} -- (4.25,1);
	\draw (4.5,-1) -- (4.5,1);
	\node at(4.75,-.85) {\tiny $\dots$};
	\node at(4.75,.85) {\tiny $\dots$};
	\draw (5,-1) -- (5,1);
	\draw[decoration={brace,mirror,raise=-8pt},decorate]  (4.4,-1.35) -- node {\small $b_r$} (5.1,-1.35);
	\draw [stdhl] (1.5,-1)  node[below,yshift={-1ex}]{\small $N_i$} -- (1.5,1);
	\draw [stdhl] (.75,-1)  node[below,yshift={-1ex}]{\small $N_1$} -- (.75,1);
	\draw [vstdhl] (-.25,-1) node[below,yshift={-1ex}]{\small $\lambda$} -- (-.25,1)  node[midway,nail]{};
	}
\]
where the nailed strand is the $k$-th black strand 
counting from left to right. This element has degree $\deg_{h,q,\lambda}(\theta_k) = (1,-4(k-1)+2(N_1+\cdots+N_i),2)$.

\begin{lem}
  \label{lem:anticom_theta}
  Tightened nails anticommute with each other, up to terms with a smaller number of crossings:
  \begin{align*}
      \theta_k \theta_\ell &= -\theta_\ell \theta_k + R & \theta_k^2 &= 0 + R',
  \end{align*}
  where $R$ (resp. $R'$) is a sum of diagrams with strictly fewer crossings than $\theta_k\theta_\ell$ (resp. $\theta_k^2$), for all $1 \leq k, \ell \leq b$. 
\end{lem}

\begin{proof}
  Similar to \cite[Lemma 3.12]{naissevaz3}, and omitted.
\end{proof}

\begin{rem}
  If $k,\ell \leq b_0$, then we have $\theta_k\theta_\ell = -\theta_\ell \theta_k$. Moreover, if $k \not \in\{b_0+1,b_0+b_1+1,\ldots,b_0+\ldots+b_r+1\}$, then we have $\theta_k^2=0$.
\end{rem}

Now fix $\kappa,\rho\in\mathcal{P}^r_b$ and consider the subset of permutations ${}_\kappa S_\rho\subset S_{r+b}$, viewed as diagrams with a blue strand, $b$ black strands and $r$ red strands, such that:
\begin{itemize}
\item the blue strand is always on the left of the diagram,
\item the strands are ordered at the bottom by $1_\rho$ and at the top by $1_\kappa$,
\item for any reduced expression of $w\in {}_\kappa S_\rho$, there are no red/red crossings.
\end{itemize}

\begin{ex}
  If $\kappa=\rho=(0,1,1)$, then the set ${}_\kappa S_\rho$ has two elements, namely
  \[
    \tikzdiagh{-1.5ex}{
      \draw[vstdhl] (0,0) -- (0,1);
      \draw[stdhl] (.25,0) -- (.25,1);
      \draw (0.5,0) -- (.5,1);
      \draw[stdhl] (.75,0) -- (.75,1);
      \draw (1,0) -- (1,1);
    }
    \quad\text{and}\quad
    \tikzdiagh{-1.5ex}{
      \draw[vstdhl] (0,0) -- (0,1);
      \draw[stdhl] (.25,0) -- (.25,1);
      \draw (.5,0) ..controls (.5,.5) and (1,.5) .. (1,1);
      \draw (1,0) ..controls (1,.5) and (.5,.5) .. (.5,1);
      \draw[stdhl] (.75,0) ..controls (.75,.35) and (1,.15) .. (1,.5);
      \draw[stdhl] (1,.5) ..controls (1,.85) and (.75,.65) ..(.75,1);
    }
  \]
  Note that the second element is not left-adjusted.
\end{ex}

For each $w\in {}_\kappa S_\rho,\ \underline{l}=(l_1,\ldots,l_b)\in \{0,1\}^b$ and $\underline{a}=(a_1,\ldots,a_b)\in\mathbb{N}^b$ we define an element $b_{w,\underline{l},\underline{a}}\in 1_\kappa T^{\lambda,\underline{N}}_m 1_\rho$ as follows:
\begin{enumerate}
\item we choose a left-reduced expression of $w$ in terms of diagrams as above;
\item for each $1\leq i \leq b$, if $l_i=1$, then we nail the $i$-th black strand at the top, counting from the left, on the blue strand by pulling it from its leftmost position;
\item finally, for each $1\leq i \leq b$, we add $a_i$ dots on the $i$-th black strand at the top.
\end{enumerate}
Let ${}_\kappa B_\rho := \{ b_{w,\underline{l},\underline{a}} | w\in {}_\kappa S_\rho,\ \underline{l}\in \{0,1\}^b, \underline{a}\in\mathbb{N}^b\}$.

\begin{ex}
  We continue the example of $\kappa=\rho=(0,1,1)$. If we choose $\underline{l}=(1,0)$ and $\underline{a}=(0,1)$ for $w$ the permutation with a black/black crossing, after left-adjusting it, then we obtain
  \[
    b_{w,\underline{l},\underline{a}}  =
      \tikzdiagh{-1.5ex}{
      \draw (.5,0) ..controls (.5,.5) and (1,.5) .. (1,1) node [pos=.85,tikzdot]{};
      \draw (.5,1) ..controls (.5,.85) and (0,.90) .. (0,.75);
      \draw (0,.75) ..controls (0,.60) and (1,.75) .. (1,0);
      \draw[stdhl] (.25,0) -- (.25,1);
      \draw[stdhl] (.75,0) ..controls (.75,.35) and (.5,.15) .. (.5,.5);
      \draw[stdhl] (.5,.5) ..controls (.5,.85) and (.75,.65) ..(.75,1);
      \draw[vstdhl] (0,0) -- (0,1)  node [pos=.75,nail]{};
    }
  \] 
\end{ex}

\begin{thm}\label{thm:Tbasis}
The set ${}_\kappa B_\rho$ is a basis of $1_\kappa T^{\lambda,\underline{N}}_b 1_\rho$ as a $\bZ\times\bZ^2$-graded $\Bbbk$-module.
\end{thm}

\begin{proof}
  By \cref{lem:anticom_theta}, with arguments similar to \cite[Proposition 3.13]{naissevaz3}, one shows that this set generates $1_\kappa T^{\lambda,\underline{N}}_m 1_\rho$ as a $\Bbbk$-module. The proof consists in an induction on the number of crossings, allowing to apply braid-moves in order to reduce diagrams. 
  In order to show that this set is linearly independent over $\Bbbk$, we apply \cref{lem:injectstoANH}.
\end{proof}

In the following, we draw $T_b^{\lambda,\underline{N}} 1_{\rho}$ with $\rho = (b_0, \dots, b_r)$ as a box diagram
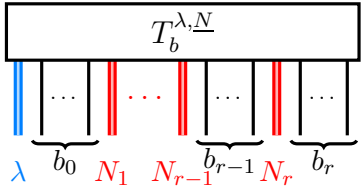
\[
\tikzdiag[xscale=1.25]{
	\draw [vstdhl] (-.25,-.5) node[below,yshift={-1ex}]{\small $\lambda$} -- (-.25,1);
	\draw (0,-.5) -- (0,1);
	\node at(.25,0) {\tiny $\dots$};
	\draw (.5,-.5) -- (.5,1);
	\draw[decoration={brace,mirror,raise=-8pt},decorate]  (-.1,-.85) -- node {\small $b_0$} (.6,-.85);
	\draw [stdhl] (.75,-.5)  node[below,yshift={-1ex}]{\small $N_1$} -- (.75,1);
	\node[red] at(1.125,0) { $\dots$};
	\draw [stdhl] (1.5,-.5)  node[below,yshift={-1ex}]{\small $N_{r-1}$} -- (1.5,1);
	\draw (1.75,-.5) -- (1.75,1);
	\node at(2,0) {\tiny $\dots$};
	\draw (2.25,-.5) -- (2.25,1);
	\draw[decoration={brace,mirror,raise=-8pt},decorate]  (1.65,-.85) -- node {\small $b_{r-1}$} (2.35,-.85);
	\draw [stdhl] (2.5,-.5)  node[below,yshift={-1ex}]{\small $N_{r}$} -- (2.5,1);
	\draw (2.75,-.5) -- (2.75,1);
	\node at(3,0) {\tiny $\dots$};
	\draw (3.25,-.5) -- (3.25,1);
	\draw[decoration={brace,mirror,raise=-8pt},decorate]  (2.65,-.85) -- node {\small $b_r$} (3.35,-.85);
	\filldraw [fill=white, draw=black] (-.375,.5) rectangle (3.375,1.25) node[midway] { $T_b^{\lambda,\underline{N}}$};
}
\]
Moreover, when we draw something like
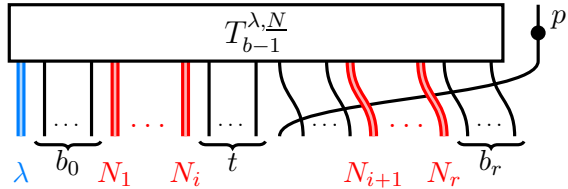
\[
\tikzdiag[xscale=1.25]{
	\draw[fill=white, color=white] (-.35,0) circle (.15cm);
	\draw [vstdhl] (-.25,-.5) node[below,yshift={-1ex}]{\small $\lambda$} -- (-.25,1);
	\draw (0,-.5) -- (0,1);
	\node at(.25,-.35) {\tiny $\dots$};
	\draw (.5,-.5) -- (.5,1);
	\draw[decoration={brace,mirror,raise=-8pt},decorate]  (-.1,-.85) -- node {\small $b_0$} (.6,-.85);
	\draw [stdhl] (.75,-.5)  node[below,yshift={-1ex}]{\small $N_1$} -- (.75,1);
	\node[red] at(1.125,-.35) { $\dots$};
	\draw [stdhl] (1.5,-.5)  node[below,yshift={-1ex}]{\small $N_i$} -- (1.5,1);
	\draw (1.75,-.5) -- (1.75,1);
	\node at(2,-.35) {\tiny $\dots$};
	\draw (2.25,-.5) -- (2.25,1);
	\draw[decoration={brace,mirror,raise=-8pt},decorate]  (1.65,-.85) -- node {\small $t$} (2.35,-.85);
	\draw (2.5,-.5) .. controls (2.5,0) and (5.25,0) ..  (5.25,.5) -- (5.25,1.25) node[midway,tikzdot]{} node[midway, xshift=1.5ex, yshift=1ex]{\small $p$};
	\draw (2.75,-.5) .. controls (2.75,0) and (2.5,0) .. (2.5,.5);
	\node at(3,-.35) {\tiny $\dots$};
	\draw (3.25,-.5) .. controls (3.25,0) and (3,0) .. (3,.5);
	\draw [stdhl]  (3.5,-.5)  node[below,yshift={-1ex}]{\small $N_{i+1}$} .. controls (3.5,0) and (3.25,0) .. (3.25,.5);
	\node[red] at(3.875,-.35) { $\dots$};
	\draw [stdhl]  (4.25,-.5)  node[below,yshift={-1ex}]{\small $N_{r}$} .. controls (4.25,0) and (4,0) .. (4,.5);
	\draw (4.5,-.5) .. controls (4.5,0) and (4.25,0) .. (4.25,.5);
	\node at(4.75,-.35) {\tiny $\dots$};
	\draw (5,-.5) .. controls (5,0) and (4.75,0) .. (4.75,.5);
	\draw[decoration={brace,mirror,raise=-8pt},decorate]  (4.4,-.85) -- node {\small $b_r$} (5.1,-.85);
	\filldraw [fill=white, draw=black] (-.375,.5) rectangle (4.875,1.25) node[midway] { $T_{b-1}^{\lambda,\underline{N}}$};
}
\]
with $p \geq 0$ and $0 \leq t < b_i$, it means we consider the subset of $T_b^{\lambda,\underline{N}} 1_{\rho}$ given replacing the box labeled $T_{b-1}^{\lambda,\underline{N}}$ with any diagram of $T_{b-1}^{\lambda,\underline{N}}$ in the diagram above, and consider it as a diagram of $T_b^{\lambda,\underline{N}} 1_{\rho}$.

\begin{cor}\label{prop:Tdecomp}
As a $\bZ\times\bZ^2$-graded $\Bbbk$-module, $T_b^{\lambda,\underline{N}} 1_{\rho}$ decomposes as a direct sum
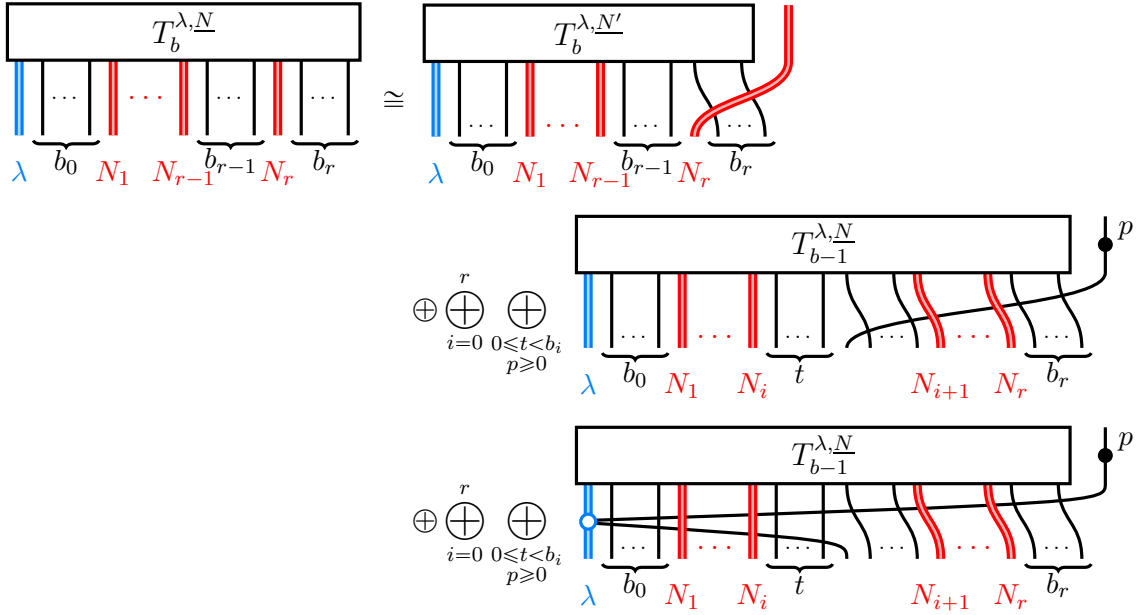
\begin{align*}
\tikzdiag[xscale=1.25]{
	\draw [vstdhl] (-.25,-.5) node[below,yshift={-1ex}]{\small $\lambda$} -- (-.25,1);
	\draw (0,-.5) -- (0,1);
	\node at(.25,0) {\tiny $\dots$};
	\draw (.5,-.5) -- (.5,1);
	\draw[decoration={brace,mirror,raise=-8pt},decorate]  (-.1,-.85) -- node {\small $b_0$} (.6,-.85);
	\draw [stdhl] (.75,-.5)  node[below,yshift={-1ex}]{\small $N_1$} -- (.75,1);
	\node[red] at(1.125,0) { $\dots$};
	\draw [stdhl] (1.5,-.5)  node[below,yshift={-1ex}]{\small $N_{r-1}$} -- (1.5,1);
	\draw (1.75,-.5) -- (1.75,1);
	\node at(2,0) {\tiny $\dots$};
	\draw (2.25,-.5) -- (2.25,1);
	\draw[decoration={brace,mirror,raise=-8pt},decorate]  (1.65,-.85) -- node {\small $b_{r-1}$} (2.35,-.85);
	\draw [stdhl] (2.5,-.5)  node[below,yshift={-1ex}]{\small $N_{r}$} -- (2.5,1);
	\draw (2.75,-.5) -- (2.75,1);
	\node at(3,0) {\tiny $\dots$};
	\draw (3.25,-.5) -- (3.25,1);
	\draw[decoration={brace,mirror,raise=-8pt},decorate]  (2.65,-.85) -- node {\small $b_r$} (3.35,-.85);
	\filldraw [fill=white, draw=black] (-.375,.5) rectangle (3.375,1.25) node[midway] { $T_b^{\lambda,\underline{N}}$};
}
\ \cong& \ 
\tikzdiag[xscale=1.25]{
	\draw (0,-.5) -- (0,1);
	\node at(.25,-.35) {\tiny $\dots$};
	\draw (.5,-.5) -- (.5,1);
	\draw[decoration={brace,mirror,raise=-8pt},decorate]  (-.1,-.85) -- node {\small $b_0$} (.6,-.85);
	\draw [stdhl] (.75,-.5)  node[below,yshift={-1ex}]{\small $N_1$} -- (.75,1);
	\node[red] at(1.125,-.35) { $\dots$};
	\draw [stdhl] (1.5,-.5)  node[below,yshift={-1ex}]{\small $N_{r-1}$} -- (1.5,1);
	\draw (1.75,-.5) -- (1.75,1);
	\node at(2,-.35) {\tiny $\dots$};
	\draw (2.25,-.5) -- (2.25,1);
	\draw[decoration={brace,mirror,raise=-8pt},decorate]  (1.65,-.85) -- node {\small $b_{r-1}$} (2.35,-.85);
	\draw (2.75,-.5) .. controls (2.75,0) and (2.5,0) .. (2.5,.5) -- (2.5,1);
	\node at(3,-.35) {\tiny $\dots$};
	\draw (3.25,-.5) .. controls (3.25,0) and (3,0) .. (3,.5) -- (3,1);
	\draw[decoration={brace,mirror,raise=-8pt},decorate]  (2.65,-.85) -- node {\small $b_r$} (3.35,-.85);
	\draw [stdhl] (2.5,-.5)  node[below,yshift={-1ex}]{\small $N_{r}$} .. controls (2.5,0) and (3.5,0) .. (3.5,.5) -- (3.5,1.25);
	\draw [vstdhl] (-.25,-.5) node[below,yshift={-1ex}]{\small $\lambda$} -- (-.25,1);
	\filldraw [fill=white, draw=black] (-.375,.5) rectangle (3.125,1.25) node[midway] { $T_b^{\lambda,\underline{N'}}$};
}
\\
&\oplus
\bigoplus_{i=0}^r
\ssbigoplus{0 \leq t < b_i \\ p \geq 0}
\tikzdiag[xscale=1.25]{
	\draw[fill=white, color=white] (-.35,0) circle (.15cm);
	\draw [vstdhl] (-.25,-.5) node[below,yshift={-1ex}]{\small $\lambda$} -- (-.25,1);
	\draw (0,-.5) -- (0,1);
	\node at(.25,-.35) {\tiny $\dots$};
	\draw (.5,-.5) -- (.5,1);
	\draw[decoration={brace,mirror,raise=-8pt},decorate]  (-.1,-.85) -- node {\small $b_0$} (.6,-.85);
	\draw [stdhl] (.75,-.5)  node[below,yshift={-1ex}]{\small $N_1$} -- (.75,1);
	\node[red] at(1.125,-.35) { $\dots$};
	\draw [stdhl] (1.5,-.5)  node[below,yshift={-1ex}]{\small $N_i$} -- (1.5,1);
	\draw (1.75,-.5) -- (1.75,1);
	\node at(2,-.35) {\tiny $\dots$};
	\draw (2.25,-.5) -- (2.25,1);
	\draw[decoration={brace,mirror,raise=-8pt},decorate]  (1.65,-.85) -- node {\small $t$} (2.35,-.85);
	\draw (2.5,-.5) .. controls (2.5,0) and (5.25,0) ..  (5.25,.5) -- (5.25,1.25) node[midway,tikzdot]{} node[midway, xshift=1.5ex, yshift=1ex]{\small $p$};
	\draw (2.75,-.5) .. controls (2.75,0) and (2.5,0) .. (2.5,.5);
	\node at(3,-.35) {\tiny $\dots$};
	\draw (3.25,-.5) .. controls (3.25,0) and (3,0) .. (3,.5);
	\draw [stdhl]  (3.5,-.5)  node[below,yshift={-1ex}]{\small $N_{i+1}$} .. controls (3.5,0) and (3.25,0) .. (3.25,.5);
	\node[red] at(3.875,-.35) { $\dots$};
	\draw [stdhl]  (4.25,-.5)  node[below,yshift={-1ex}]{\small $N_{r}$} .. controls (4.25,0) and (4,0) .. (4,.5);
	\draw (4.5,-.5) .. controls (4.5,0) and (4.25,0) .. (4.25,.5);
	\node at(4.75,-.35) {\tiny $\dots$};
	\draw (5,-.5) .. controls (5,0) and (4.75,0) .. (4.75,.5);
	\draw[decoration={brace,mirror,raise=-8pt},decorate]  (4.4,-.85) -- node {\small $b_r$} (5.1,-.85);
	\filldraw [fill=white, draw=black] (-.375,.5) rectangle (4.875,1.25) node[midway] { $T_{b-1}^{\lambda,\underline{N}}$};
}
\\
&\oplus
\bigoplus_{i=0}^r
\ssbigoplus{0 \leq t < b_i \\ p \geq 0}
\tikzdiag[xscale=1.25]{
	\draw (0,-.5) -- (0,1);
	\node at(.25,-.35) {\tiny $\dots$};
	\draw (.5,-.5) -- (.5,1);
	\draw[decoration={brace,mirror,raise=-8pt},decorate]  (-.1,-.85) -- node {\small $b_0$} (.6,-.85);
	\node[red] at(1.125,-.35) { $\dots$};
	\draw (1.75,-.5) -- (1.75,1);
	\node at(2,-.35) {\tiny $\dots$};
	\draw (2.25,-.5) -- (2.25,1);
	\draw[decoration={brace,mirror,raise=-8pt},decorate]  (1.65,-.85) -- node {\small $t$} (2.35,-.85);
	\draw (2.5,-.5) .. controls (2.5,-.25) ..  (-.5,0) .. controls (5.25,.25) ..  (5.25,.5) -- (5.25,1.25) node[midway,tikzdot]{} node[midway, xshift=1.5ex, yshift=1ex]{\small $p$};
	\draw (2.75,-.5) .. controls (2.75,0) and (2.5,0) .. (2.5,.5);
	\node at(3,-.35) {\tiny $\dots$};
	\draw (3.25,-.5) .. controls (3.25,0) and (3,0) .. (3,.5);
	\draw [stdhl]  (3.5,-.5)  node[below,yshift={-1ex}]{\small $N_{i+1}$} .. controls (3.5,0) and (3.25,0) .. (3.25,.5);
	\node[red] at(3.875,-.35) { $\dots$};
	\draw [stdhl]  (4.25,-.5)  node[below,yshift={-1ex}]{\small $N_{r}$} .. controls (4.25,0) and (4,0) .. (4,.5);
	\draw (4.5,-.5) .. controls (4.5,0) and (4.25,0) .. (4.25,.5);
	\node at(4.75,-.35) {\tiny $\dots$};
	\draw (5,-.5) .. controls (5,0) and (4.75,0) .. (4.75,.5);
	\draw[decoration={brace,mirror,raise=-8pt},decorate]  (4.4,-.85) -- node {\small $b_r$} (5.1,-.85);
	\draw [stdhl] (1.5,-.5)  node[below,yshift={-1ex}]{\small $N_i$} -- (1.5,1);
	\draw [stdhl] (.75,-.5)  node[below,yshift={-1ex}]{\small $N_1$} -- (.75,1);
	\draw[fill=white, color=white] (-.35,0) circle (.15cm);
	\draw [vstdhl] (-.25,-.5) node[below,yshift={-1ex}]{\small $\lambda$} -- (-.25,1)  node[nail,pos=.33]{};
	\filldraw [fill=white, draw=black] (-.375,.5) rectangle (4.875,1.25) node[midway] { $T_{b-1}^{\lambda,\underline{N}}$};
}
\end{align*}
where $\und N' = (N_1,\dots, N_{r-1})$, and the isomorphism is given by inclusion. 
\end{cor}

\begin{proof}
The claim follows immediately from \cref{thm:Tbasis}.
\end{proof}

\subsection{Dg-enhancement}\label{sec:dgenh}

For each $N \in \bN$, we want to define a non-trivial differential $d_N$ on $T_b^{\lambda, \underline{N}}$. 
First, we collapse the $\bZ^2$-grading into a single $\bZ$-grading, which we also call $q$-degree, through the map $\bZ^2 \rightarrow \bZ, (a,b) \mapsto a + bN$ (i.e. specializing $\lambda = q^N$). Then, we put
\[
d_N\left(
	\tikzdiagh{0}{
		\draw (.5,-.5) .. controls (.5,-.25)  .. 
			(0,0) .. controls (.5,.25)  .. (.5,.5);
	           \draw[vstdhl] (0,-.5) node[below]{\small $\lambda$} -- (0,.5) node [midway,nail]{};
  	}
  	\right) 
\ := \ 
	\tikzdiagh{0}{
		\draw (.5,-.5) -- (.5,.5) node[midway,tikzdot]{} node[midway,xshift=1.75ex,yshift=.75ex]{\small $N$};
	           \draw[vstdhl] (0,-.5) node[below]{\small $\lambda$} -- (0,.5);
  	}
\]
and $d_N(t) := 0$ for
all element $t$ of $\widetilde T_b^{\underline N} \subset T_b^{\lambda, \underline N}$,
and extending by the graded Leibniz rule w.r.t. the homological grading. 
A straightforward computation shows that $d_N$ respects all the defining relations of $T_b^{\lambda, \underline N}$, and therefore is well-defined.

\begin{thm}\label{thm:dNformal}
The $\bZ$-graded dg-algebra $(T^{\lambda,\underline{N}}_b,d_N)$ is formal with
\[
	H(T^{\lambda,\underline{N}}_b,d_N) \cong T^{(N,\underline{N})}_b , 
\]
were $(N,\underline{N}) := (N, N_1, \dots, N_r) \in \bN^{r+1}$.
\end{thm}

\begin{proof}
The proof follows by similar arguments as in~\cite[Theorem 4.4]{naissevaz3}, by using \cref{prop:Tdecomp}. We leave the details to the reader. 
\end{proof}




\section{A categorification of $M(\lambda)\otimes V(\protect\underline{N})$}\label{sec:catTensProd}

In this section we explain how derived categories of $(T^{\lambda,\underline{N}}_b,0)$-dg-modules categorify the $U_q(\slt)$-module $M(\lambda)\otimes V(\underline{N})$. 
Since the construction is very similar to the one in~\cite{naissevaz2} and~\cite{naissevaz3}, 
we will assume some familiarity with~\cite{naissevaz2} and~\cite{naissevaz3}, and we will refer to these papers for several details.

\smallskip

We introduce the notations 
\begin{align*}
\oplus_{[k]_q} (-)& := \bigoplus_{p = 0}^{k-1} q^{k-1-2p} (-), 
\\
\oplus_{[\beta+k]_q} (-) & := \bigoplus_{p \geq 0} \lambda q^{1+2p+k} (-) [1] \oplus \lambda^{-1} q^{1+2p-k}(-),
\end{align*}
where we recall that $q^a \lambda^b(-)$ is a shift up by $(a,b)$ in the $\bZ^2$-grading, and $(-)[1]$ is a shift up by $1$ in the homological grading. 
We write $\otimes$ for $\otimes_\Bbbk$, and $\otimes_b$ for $\otimes_{(T^{\lambda,\underline{N}}_b,0)}$. 
We also write $\cD_{dg}(T^{\lambda,\underline{N}}_b, 0)$ for the dg-enhanced derived category of $\bZ^2$-graded dg-modules over $(T^{\lambda,\underline{N}}_b, 0)$ (see \cref{sec:dgdercat} for a precise definition).

\subsection{Categorical action}

Let $1_{b,1} \in T^{\lambda,\underline{N}}_{b+1}$ be the idempotent given by
\[
1_{b,1} := \sum_{\rho \in  \mathcal{P}_b^r}  \ 
\tikzdiagh{0}{
	\draw[vstdhl] (0,0) node[below]{\small $\lambda$} --(0,1);
	\draw (.5,0) -- (.5,1);
	\node at(1,.5) {\tiny$\dots$};
	\draw (1.5,0) -- (1.5,1);
	\draw[decoration={brace,mirror,raise=-8pt},decorate]  (.4,-.35) -- node {$b_0$} (1.6,-.35);
	\draw[stdhl] (2,0)  node[below]{\small $N_1$} --(2,1);
	\draw (2.5,0) -- (2.5,1);
	\node at(3,.5) {\tiny$\dots$};
	\draw (3.5,0) -- (3.5,1);
	\draw[decoration={brace,mirror,raise=-8pt},decorate]  (2.4,-.35) -- node {$b_1$} (3.6,-.35);
	\draw[stdhl] (4,0)  node[below]{\small $N_2$} --(4,1);
	\node[red] at  (5,.5) {\dots};
	\draw[stdhl] (6,0)  node[below]{\small $N_{r}$} --(6,1);
	\draw (6.5,0) -- (6.5,1);
	\node at(7,.5) {\tiny$\dots$};
	\draw (7.5,0) -- (7.5,1);
	\draw[decoration={brace,mirror,raise=-8pt},decorate]  (6.4,-.35) -- node {$b_r$} (7.6,-.35);
	\draw (8,0) -- (8, 1);
}
\]
There is a (non-unital) map of algebras 
$
T^{\lambda,\underline{N}}_b \rightarrow T^{\lambda,\underline{N}}_{b+1}
$
 given by adding a vertical black strand to the right of a diagram from $T^{\lambda,\underline{N}}_b$:
\begin{equation}\label{eq:addblackstrand}
\tikzdiagh[xscale=1.25]{0}{
	\draw [vstdhl] (-.25,0) node[below]{\small $\lambda$} -- (-.25,1);
	\draw (0,0) -- (0,1);
	\node at(.25,.125) {\tiny $\dots$};
	\node at(.25,.875) {\tiny $\dots$};
	\draw (.5,0) -- (.5,1);
	\draw [stdhl] (.75,0)  node[below]{\small $N_1$} -- (.75,1);
	\node[red] at(1.125,.125) { $\dots$};
	\node[red] at(1.125,.875) { $\dots$};
	\draw [stdhl] (1.5,0)  node[below]{\small $N_{r-1}$} -- (1.5,1);
	\draw (1.75,0) -- (1.75,1);
	\node at(2,.125) {\tiny $\dots$};
	\node at(2,.875) {\tiny $\dots$};
	\draw (2.25,0) -- (2.25,1);
	\draw [stdhl] (2.5,0)  node[below]{\small $N_{r}$} -- (2.5,1);
	\draw (2.75,0) -- (2.75,1);
	\node at(3,.125) {\tiny $\dots$};
	\node at(3,.875) {\tiny $\dots$};
	\draw (3.25,0) -- (3.25,1);
	\filldraw [fill=white, draw=black] (-.375,.25) rectangle (3.375,.75) node[midway] { $D$};
}
\ \mapsto \ 
\tikzdiagh[xscale=1.25]{0}{
	\draw [vstdhl] (-.25,0) node[below]{\small $\lambda$} -- (-.25,1);
	\draw (0,0) -- (0,1);
	\node at(.25,.125) {\tiny $\dots$};
	\node at(.25,.875) {\tiny $\dots$};
	\draw (.5,0) -- (.5,1);
	\draw [stdhl] (.75,0)  node[below]{\small $N_1$} -- (.75,1);
	\node[red] at(1.125,.125) { $\dots$};
	\node[red] at(1.125,.875) { $\dots$};
	\draw [stdhl] (1.5,0)  node[below]{\small $N_{r-1}$} -- (1.5,1);
	\draw (1.75,0) -- (1.75,1);
	\node at(2,.125) {\tiny $\dots$};
	\node at(2,.875) {\tiny $\dots$};
	\draw (2.25,0) -- (2.25,1);
	\draw [stdhl] (2.5,0)  node[below]{\small $N_{r}$} -- (2.5,1);
	\draw (2.75,0) -- (2.75,1);
	\node at(3,.125) {\tiny $\dots$};
	\node at(3,.875) {\tiny $\dots$};
	\draw (3.25,0) -- (3.25,1);
	\filldraw [fill=white, draw=black] (-.375,.25) rectangle (3.375,.75) node[midway] { $D$};
	\draw (3.5,0) -- (3.5,1);
}
\end{equation}
sending the unit $1 \in T^{\lambda,\underline{N}}$ to the idempotent $1_{b,1}$.
This map gives rise to derived induction and restriction dg-functors
\begin{align*}
\Ind_b^{b+1} &: \cD_{dg}(T^{\lambda,\underline{N}}_{b},0) \rightarrow \cD_{dg}(T^{\lambda,\underline{N}}_{b+1},0), 
 &&\Ind_b^{b+1}(-) := (T^{\lambda,\underline{N}}_{b+1},0) 1_{b,1} \Lotimes_b (-),\\[1ex]
\Res_b^{b+1} &: \cD_{dg}(T^{\lambda,\underline{N}}_{b+1},0) \rightarrow \cD_{dg}(T^{\lambda,\underline{N}}_{b},0), 
 &&\Res_b^{b+1}(-) :=  \RHOM_{b}(-,1_{b,1}(T^{\lambda,\underline{N}}_{b+1},0)), 
\end{align*}
which are adjoint (see \cref{sec:deriveddghomtensor}). By \cref{prop:Tdecomp}, we know that $(T^{\lambda,\underline{N}}_{b+1},0)$ is a cofibrant dg-module over $(T^{\lambda,\underline{N}}_{b},0)$, so that we can replace derived tensor products (resp. derived homs) by usual tensor products
\begin{align*}
\Ind_b^{b+1}(-) &\cong (T^{\lambda,\underline{N}}_{b+1},0) 1_{b,1} \otimes_b (-),
&
\Res_b^{b+1}(-) &\cong  1_{b,1} (T^{\lambda,\underline{N}}_{b+1},0) \otimes_{b+1} (-).
\end{align*}
 Then, we define
\begin{align*}
\F_b &:= \Ind_b^{b+1},
&
\E_b &:= \lambda^{-1} q^{1+2b-|\underline{N}|} \Res_b^{b+1},
\end{align*}
and $\id_b$ is the identity dg-functor on $\cD_{dg}(T^{\lambda,\underline{N}}_{b},0)$.

\begin{thm}\label{thm:sl2comqi}
There is a quasi-isomorphism
\[
\cone(\F_{b-1}\E_{b-1} \rightarrow \E_b\F_b) \xrightarrow{\cong} \oplus_{[\beta+|\underline{N}|-2b]_q} \id_b,
\]
 of dg-functors.  
\end{thm}

\begin{proof}
Consider the map 
\[
\psi : q^{-2} (T^{\lambda,\underline{N}}_{b}  1_{b-1,1}\otimes_{b-1} 1_{b-1,1} T^{\lambda,\underline{N}}_b) \rightarrow1_{b,1}  T^{\lambda,\underline{N}}_{b+1} 1_{b,1},
\]
given by 
\[
x \otimes_{b-1} y \mapsto x \tau_b y,
\]
where $\tau_b$ is a crossing between the $b$-th and $(b+1)$-th black strands. Diagrammatically, one can picture it as
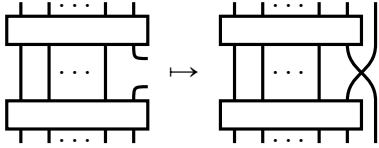
\[
\tikzdiag[xscale=.75,yscale=.75]{
	\draw (0,-1.25) -- (0,1.25);
	\draw (.5,-1.25) -- (.5,1.25);
	\draw (1.5,-1.25) -- (1.5,1.25);
	\draw (2,-1.25) -- (2,-.5) .. controls (2,-.25) .. (2.25,-.25);
	\draw (2,1.25) -- (2,.5) .. controls (2,.25) .. (2.25,.25);
	\node at(1,1.2) {\small $\dots$};
	\filldraw [fill=white, draw=black] (-.25,-1) rectangle (2.25,-.5);
	\node at(1,0) {\small $\dots$}; 
	\filldraw [fill=white, draw=black] (-.25,.5) rectangle (2.25,1);
	\node at(1,-1.2) {\small $\dots$};
}
\ \mapsto \  
\tikzdiag[xscale=.75,yscale=.75]{
	\draw (0,-1.25) -- (0,1.25);
	\draw (.5,-1.25) -- (.5,1.25);
	\draw (1.5,-1.25) -- (1.5,1.25);
	\draw (2,-1.25) -- (2,-.5) .. controls (2,0) and (2.5,0) .. (2.5,.5) -- (2.5,1.25);
	\draw (2,1.25) -- (2,.5) .. controls (2,0) and (2.5,0) .. (2.5,-.5) -- (2.5,-1.25);
	\node at(1,1.2) {\small $\dots$};
	\filldraw [fill=white, draw=black] (-.25,-1) rectangle (2.25,-.5);
	\node at(1,0) {\small $\dots$};
	\filldraw [fill=white, draw=black] (-.25,.5) rectangle (2.25,1);
	\node at(1,-1.2) {\small $\dots$};
}
\]
where the bent black strands informally depict the induction/restriction functors. 
Then, as in \cite[Theorem 5.1]{naissevaz3}, we obtain an exact sequence of $(T^{\lambda,\underline{N}}_{b},T^{\lambda,\underline{N}}_{b})$-bimodules
\begin{align*}
0 \rightarrow q^{-2} (T^{\lambda,\underline{N}}_{b}  1_{b-1,1}\otimes_{b-1} 1_{b-1,1} T^{\lambda,\underline{N}}_b) &\xrightarrow{\ \psi\ }1_{b,1}  T^{\lambda,\underline{N}}_{b+1} 1_{b,1}
\\
 &\xrightarrow{\ \phi\ } \bigoplus_{p\geq 0} q^{2p}  (T^{\lambda,\underline{N}}_{b}) \oplus \lambda^2 q^{2p+2|\underline{N}|-4b} (T^{\lambda,\underline{N}}_b) \rightarrow 0,
\end{align*}
where $\phi$ is the projection onto the following summands 
\begin{align*}
\bigoplus_{p \geq 0} \ 
\tikzdiag[xscale=1.25]{
	\draw (0,-.5) -- (0,1);
	\node at(.25,-.35) {\tiny $\dots$};
	\draw (.5,-.5) -- (.5,1);
	\draw[decoration={brace,mirror,raise=-8pt},decorate]  (-.1,-.85) -- node {\small $b_0$} (.6,-.85);
	\draw [stdhl] (.75,-.5)  node[below,yshift={-1ex}]{\small $N_1$} -- (.75,1);
	\node[red] at(1.125,-.35) { $\dots$};
	\draw [stdhl] (1.5,-.5)  node[below,yshift={-1ex}]{\small $N_{r}$} -- (1.5,1);
	\draw (1.75,-.5) -- (1.75,1);
	\node at(2,-.35) {\tiny $\dots$};
	\draw (2.25,-.5) -- (2.25,1);
	\draw[decoration={brace,mirror,raise=-8pt},decorate]  (1.65,-.85) -- node {\small $b_{r}-1$} (2.35,-.85);
	\draw (2.75, -.5) -- (2.75, 1.25) node[pos=.825,tikzdot]{} node[pos=.825, xshift=1.5ex, yshift=1ex]{\small $p$};
	%
	\draw [vstdhl] (-.25,-.5) node[below,yshift={-1ex}]{\small $\lambda$} -- (-.25,1);
	\filldraw [fill=white, draw=black] (-.375,.5) rectangle (2.375,1.25) node[midway] { $T_{b-1}^{\lambda,\underline{N}}$};
}
\oplus
 \ 
\tikzdiag[xscale=1.25]{
	\draw (0,-.5) -- (0,1);
	\node at(.25,-.35) {\tiny $\dots$};
	\draw (.5,-.5) -- (.5,1);
	\draw[decoration={brace,mirror,raise=-8pt},decorate]  (-.1,-.85) -- node {\small $b_0$} (.6,-.85);
	\draw (2.75, -.5) .. controls (2.75,-.25) .. (-.5,0) .. controls (2.75,.25) ..  (2.75, .75) ;
	\draw [stdhl] (.75,-.5)  node[below,yshift={-1ex}]{\small $N_1$} -- (.75,1);
	\node[red] at(1.125,-.35) { $\dots$};
	\draw [stdhl] (1.5,-.5)  node[below,yshift={-1ex}]{\small $N_{r}$} -- (1.5,1);
	\draw (1.75,-.5) -- (1.75,1);
	\node at(2,-.35) {\tiny $\dots$};
	\draw (2.25,-.5) -- (2.25,1);
	\draw[decoration={brace,mirror,raise=-8pt},decorate]  (1.65,-.85) -- node {\small $b_{r}-1$} (2.35,-.85);
	\draw (2.75, .75) -- (2.75, 1.25) node[pos=.5,tikzdot]{} node[midway, xshift=1.5ex, yshift=1ex]{\small $p$};
	%
	\draw[fill=white, color=white] (-.35,0) circle (.15cm);
	\draw [vstdhl] (-.25,-.5) node[below,yshift={-1ex}]{\small $\lambda$} -- (-.25,1) node[pos=.33, nail]{};
	\filldraw [fill=white, draw=black] (-.375,.5) rectangle (2.375,1.25) node[midway] { $T_{b-1}^{\lambda,\underline{N}}$};
}
\end{align*}
of \cref{prop:Tdecomp} (i.e. when $i=r$ and $t=b_r-1$). 
Note that, a priori, this only defines a map of left modules. Fortunately, by applying similar arguments as in~\cite[Lemma 5.4]{naissevaz3}, it is possible to show that it defines a map of bimodules. Exactness follows from a dimensional argument using \cref{prop:Tdecomp}. 
\end{proof}

\subsubsection{Recovering $V(N)\otimes V(\protect\underline{N})$}

Introducing the differential $d_N$ from \cref{sec:dgenh} in the picture, the map \cref{eq:addblackstrand} lifts to a map of dg-algebras $(T^{\lambda,\underline{N}}_b,d_N) \rightarrow (T^{\lambda,\underline{N}}_{b+1},d_N)$. 
Then we define dg-functors
\begin{align*}
\F_b^N(-) &:= (T^{\lambda,\underline{N}}_{b+1},d_N) 1_{b,1} \otimes_b (-),
&
\E_b^N(-) &:=  q^{2b-|\underline{N}|-N} 1_{b,1}(T^{\lambda,\underline{N}}_{b+1},d_N) \otimes_{b+1} (-).
\end{align*}
These corresponds with derived induction and (shifted) derived restriction dg-functors along \cref{eq:addblackstrand}, by \cref{prop:Tdecomp} again. 

Recall the notion of a strongly projective dg-module from  \cite{moore}  (or see \cref{sec:stronglyproj}). 

\begin{prop}\label{prop:stronglyproj}
As $(T^{\lambda,\underline{N}}_b,d_N)$-module, $(T^{\lambda,\underline{N}}_{b+1},d_N)$ is strongly projective.  
\end{prop}

\begin{proof}
As in \cite[Proposition 5.15]{naissevaz3}, and omitted.
\end{proof}

By \cref{prop:stronglyproj}, \cref{thm:sl2comqi} can be seen as a quasi-isomorphism of mapping cones
\[
\cone\bigl(\F_{b-1}^N\E_{b-1}^N \rightarrow \E_b^N\F_b^N\bigr) \xrightarrow{\simeq} \cone\bigl(\bigoplus_{p \geq 0} q^{1+2p+N+|\underline{N}|-2b} \id_b \xrightarrow{h_N} \bigoplus_{p \geq 0} q^{1+2p-N-|\underline{N}|+2b} \id_b \bigr),
\]
where $h_N$ is given by multiplication by the element
\[
\tikzdiag[xscale=1.25]{
	\draw (0,-1) -- (0,1);
	\node at(.25,-.85) {\tiny $\dots$};
	\node at(.25,.85) {\tiny $\dots$};
	\draw (.5,-1) -- (.5,1);
	\node[red] at(1.125,-.85) { $\dots$};
	\node[red] at(1.125,.85) { $\dots$};
	\draw (1.75,-1) -- (1.75,1);
	\node at(2,-.85) {\tiny $\dots$};
	\node at(2,.85) {\tiny $\dots$};
	\draw (2.25,-1) -- (2.25,1);
	\draw (2.5,-1) .. controls (2.5,-.25) and (-.5,-.25) ..  (-.25,0) node[pos=1,tikzdot]{} node[pos=1,xshift=-1.5ex,yshift=.75ex]{\small $N$} .. controls (-.5,.25) and (2.5,.25) ..  (2.5,1);
	\draw [stdhl] (1.5,-1)  node[below,yshift={-1ex}]{\small $N_r$} -- (1.5,1);
	\draw [stdhl] (.75,-1)  node[below,yshift={-1ex}]{\small $N_1$} -- (.75,1);
	\draw [vstdhl] (-.75,-1) node[below,yshift={-1ex}]{\small $\lambda$} -- (-.75,1) ;
	}
\]

\begin{prop}\label{prop:actionN}
There is a quasi-isomorphism
\[
\cone\bigl(\bigoplus_{p \geq 0} q^{1+2p+N+|\underline{N}|-2b} \id_b \xrightarrow{h_N} \bigoplus_{p \geq 0} q^{1+2p-N-|\underline{N}|+2b} \id_b \bigr)
\xrightarrow{\cong}
\oplus_{[N+|\underline{N}|-2b]_q} \id_b,
\]
where $\oplus_{[-k]_q} M := \oplus_{[k]_q} M[1]$.
\end{prop}

\begin{proof}
As in \cite[Proposition 5.9]{naissevaz3}, and omitted.
\end{proof}

\subsubsection{Induction along red strands}\label{sec:redind}

Take $\underline{N} = (N_1, \dots, N_r)$ and $\underline N ' = (\underline N, N_{r+1})$. 
Consider the (non-unital) map of algebras $T^{\lambda,\underline{N}}_b \rightarrow T^{\lambda,\underline{N'}}_{b}$ that consists in adding a vertical red strand labeled $N_{r+1}$ at the right a diagram:
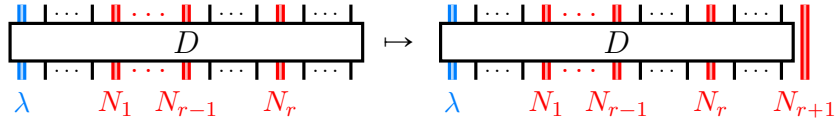
\begin{equation*}
\tikzdiagh[xscale=1.25]{0}{
	\draw [vstdhl] (-.25,0) node[below]{\small $\lambda$} -- (-.25,1);
	\draw (0,0) -- (0,1);
	\node at(.25,.125) {\tiny $\dots$};
	\node at(.25,.875) {\tiny $\dots$};
	\draw (.5,0) -- (.5,1);
	\draw [stdhl] (.75,0)  node[below]{\small $N_1$} -- (.75,1);
	\node[red] at(1.125,.125) { $\dots$};
	\node[red] at(1.125,.875) { $\dots$};
	\draw [stdhl] (1.5,0)  node[below]{\small $N_{r-1}$} -- (1.5,1);
	\draw (1.75,0) -- (1.75,1);
	\node at(2,.125) {\tiny $\dots$};
	\node at(2,.875) {\tiny $\dots$};
	\draw (2.25,0) -- (2.25,1);
	\draw [stdhl] (2.5,0)  node[below]{\small $N_{r}$} -- (2.5,1);
	\draw (2.75,0) -- (2.75,1);
	\node at(3,.125) {\tiny $\dots$};
	\node at(3,.875) {\tiny $\dots$};
	\draw (3.25,0) -- (3.25,1);
	\filldraw [fill=white, draw=black] (-.375,.25) rectangle (3.375,.75) node[midway] { $D$};
}
\ \mapsto \ 
\tikzdiagh[xscale=1.25]{0}{
	\draw [vstdhl] (-.25,0) node[below]{\small $\lambda$} -- (-.25,1);
	\draw (0,0) -- (0,1);
	\node at(.25,.125) {\tiny $\dots$};
	\node at(.25,.875) {\tiny $\dots$};
	\draw (.5,0) -- (.5,1);
	\draw [stdhl] (.75,0)  node[below]{\small $N_1$} -- (.75,1);
	\node[red] at(1.125,.125) { $\dots$};
	\node[red] at(1.125,.875) { $\dots$};
	\draw [stdhl] (1.5,0)  node[below]{\small $N_{r-1}$} -- (1.5,1);
	\draw (1.75,0) -- (1.75,1);
	\node at(2,.125) {\tiny $\dots$};
	\node at(2,.875) {\tiny $\dots$};
	\draw (2.25,0) -- (2.25,1);
	\draw [stdhl] (2.5,0)  node[below]{\small $N_{r}$} -- (2.5,1);
	\draw (2.75,0) -- (2.75,1);
	\node at(3,.125) {\tiny $\dots$};
	\node at(3,.875) {\tiny $\dots$};
	\draw (3.25,0) -- (3.25,1);
	\filldraw [fill=white, draw=black] (-.375,.25) rectangle (3.375,.75) node[midway] { $D$};
	\draw [stdhl]  (3.5,0)  node[below]{\small $N_{r+1}$}  -- (3.5,1);
}
\end{equation*}
Let $\mathfrak{I} : \cD_{dg}(T^{\lambda,\underline{N}}_{b},0) \rightarrow \cD_{dg}(T^{\lambda,\underline{N}'}_{b},0)$ be the corresponding induction dg-functor, and let $ \mathfrak{\bar I} : \cD_{dg}(T^{\lambda,\underline{N}'}, 0) \rightarrow \cD_{dg}(T^{\lambda,\underline{N}}_{b},0) $ be the restriction dg-functor. 

\begin{prop}\label{prop:indresredstrand}
There is an isomorphism $ \mathfrak{\bar I} \circ \mathfrak{I} \cong \id$. 
\end{prop}

\begin{proof}
The statement follows from \cref{prop:Tdecomp}. 
\end{proof}

\subsection{Categorification theorem}

In this section we suppose that $\Bbbk$ is a field. 
Recall the notion of an asymptotic Grothendieck group $\bKO^\Delta$ from \cite[\S 8]{asympK0} (or see \cref{sec:asympK0}). 
Since $(T^{\lambda,\underline{N}}_{b},0)$ is a positive c.b.l.f. dimensional $\bZ^2$-graded dg-algebra (see \cref{def:positivecblfdgalg}), we have by \cref{thm:triangtopK0genbyPi} that $\bKO^\Delta(T^{\lambda,\underline{N}}_{b},0)$ is a free $\bZ\pp{q,\lambda}$-module generated by the classes of projective $T^{\lambda,\underline{N}}_{b}$-modules with a trivial differential. 
Let ${}_\bQ \bKO^\Delta(-) := \bKO^\Delta(-) \otimes_{\bZ\pp{q,\lambda}} \bQ\pp{q,\lambda}$.

\smallskip

For each $\rho \in \mathcal{P}_{b}^{r}$, there is a projective $T^{\lambda,\underline{N}}_{b}$-module given by 
\[
\BP_{\rho} := 
T^{\lambda,\underline{N}}_{b}1_\rho.
\]
Recall the inclusion 
$\eta_\rho \colon A_{b_0}\otimes \nh_{b_1} \otimes \cdots \otimes \nh_{b_r} \hookrightarrow T_b^{\lambda,\underline{N}}$ 
defined in \cref{eq:defetainclusion}. 
It is well-known (see for example~\cite[\S~2.2.3]{KL1}) that $\nh_n$ admits a unique primitive idempotent up to equivalence given by
\[
e_{n} := \tau_{\vartheta_n} x_1^{n-1} x_2^{n-2} \cdots x_{n-1} \in \nh_n,
\]
where $\vartheta_n \in S_n$ is the longest element, $\tau_{w_1w_2\cdots w_k} := \tau_{w_1}\tau_{w_2}\cdots\tau_{w_k}$, with $\tau_i$ being a crossing between the $i$-th and $(i+1)$-th strands, and $x_i$ is a dot on the $i$-th strand. 
There is a similar result for $\nh_{b_0} \subset A_{b_0}$ (see \cite[\S 2.5.1]{naissevaz2}).
Moreover, for degree reasons, any primitive idempotent of $T^{\lambda,\underline{N}}_{b}$ is the image of a collection of idempotents under the inclusion $\eta_\rho $ for some $\rho$, and thus is of the form
\[
e_{\rho} := \eta_\rho \left( e_{b_1} \otimes \cdots \otimes e_{b_n} \right).
\]

It is also well-known (see for example~\cite[\S~2.2.3]{KL1}) that there is a decomposition
\[
\nh_n \cong q^{n(n-1)/2} \bigoplus_{[n]_q!} \nh_n e_n,
\]
as left $\nh_n$-modules. For the same reasons, we obtain
\[
\BP_\rho \cong q^{\sum_{i=0}^r b_i(b_i-1)/2}  \bigoplus_{\prod_{i=0}^r [b_i]_q!} T^{\lambda,\underline{N}}_{b}e_\rho.
\]

\subsubsection{Categorifed Shapovalov form}
Let $T^{\lambda,\underline{N}} := \bigoplus_{b \geq 0} T_b^{\lambda,\underline{N}}$. 
As in \cite[\S2.5]{KL1}, let $\psi : T^{\lambda,\underline{N}} \rightarrow \opalg{(T^{\lambda,\underline{N}})}$ be the map that takes the mirror image of diagrams along the horizontal axis.
Given a left $(T^{\lambda,\underline{N}},0)$-module $M$, we obtain a right $(T^{\lambda,\underline{N}},0)$-module $M^\psi$ with action given by $m^\psi \cdot r := (-1)^{\deg_h(r) \deg_h(m)} \psi(r) \cdot m$ for $m \in M$ and $r \in T^{\lambda,\underline{N}}$. 
Then we define the dg-bifunctor
\begin{align*}
(-,-) &:  \cD_{dg}(T^{\lambda,\underline{N}},0)  \times \cD_{dg}(T^{\lambda,\underline{N}},0)   \rightarrow \cD_{dg}(\Bbbk, 0), 
&
(W,W') := W^\psi \Lotimes_{(T^{\lambda,\underline{N}},0)} W'.
\end{align*}

\begin{prop}\label{prop:catshap}
The dg-bifunctor defined above satisfies:
\begin{itemize}
\item $((T^{\lambda,\underline{N}}_0,d_N),(T^{\lambda,\underline{N}}_0,d_N)) \cong (\Bbbk,0)$;
\item $(\Ind_b^{b+1} M,M') \cong (M, \Res_b^{b+1} M')$ for all $M,M' \in \cD_{dg}(T^{\lambda,\underline{N}},0)$; 
\item $(\oplus_f M,M') \cong (M, \oplus_f M') \cong \oplus_f (M,M')$ for all $f \in \bZ\pp{q,\lambda}$;
\item $(M,M') \cong (\mathfrak{I}(M), \mathfrak{I}(M'))$.
\end{itemize}
\end{prop}

\begin{proof}
Straightforward, except for the last point which follows from \cref{prop:indresredstrand}, together with the adjunction $\mathfrak{I} \vdash  \mathfrak{\bar I}$.
\end{proof}

Comparing \cref{prop:catshap} to \cref{sec:shepfortensor}, we deduce 
that $(-,-)$ has the same properties on the asymptotic Grothendieck group of $(T^{\lambda,r},0)$ as the Shapovalov form on $M\otimes V^r$.

\subsubsection{The categorification theorem}

Let $\E := \bigoplus_{b \geq 0} \E_b$ and $\F := \bigoplus_{b \geq 0} \F_b$. 
By \cref{thm:sl2comqi} and \cref{prop:cblfbiminduceKO}, we know that ${}_\bQ\bKO^\Delta(T^{\lambda,\underline{N}},0)$ is an $U_q(\slt)$-module, with action given by the pair $[\F], [\E]$.

\begin{lem} \label{lem:rankcompare}
We have $\dim_{\bQ\pp{q,\lambda}} \bigl( {}_\bQ\bKO^\Delta(T^{\lambda,\underline{N}}_b,0) \bigr) \leq \dim_{\bQ\pp{q,\lambda}} \bigl( M(\lambda) \otimes V(\underline{N})_{\lambda q^{r-2b}} \bigr)$. Moreover,  ${}_\bQ\bKO^\Delta(T^{\lambda,\underline{N}},0)$ is spanned by the classes $\{ [P_\rho] \}_{\rho \in \cP_b^{r,\und N}}$.
\end{lem}

\begin{proof}
It is well-known (see for example \cite[Lemma 7.2]{naissevaz1}) that whenever $k > n$, then the unit element in $\nh_k$ can be rewritten as a combination of elements having $n$ consecutive dots somewhere on the left-most strand. Thus, for any $\rho' \in \cP_b^r$, we obtain that $1_{\rho'}$ can be rewritten as a combination of elements factorizing through elements in $\{1_{\rho}\}_{\rho \in \cP_b^{r,\und N}}$.  
\end{proof}

We consider $ M(\lambda) \otimes V(\underline{N})$ over the ground ring $\bQ\pp{q,\lambda}$ instead of $\bQ(q,\lambda)$. 

\begin{thm}\label{thm:K0}
There are isomorphisms of $U_q(\slt)$-modules
\[
{}_\bQ\bKO^\Delta(T^{\lambda,\underline{N}},0) \cong M(\lambda) \otimes V(\underline{N}),
\]
and
\[
{}_\bQ\bKO^\Delta(T^{\lambda,\underline{N}},d_N)\cong V(N) \otimes V(\underline{N}),
\]
for all $N \in \bN$. 
\end{thm}

\begin{proof}
We have a $\bQ\pp{q,\lambda}$-linear map 
\[
M(\lambda) \otimes V(\underline{N}) \rightarrow {}_\bQ\bKO^\Delta(T^{\lambda,\underline{N}},0), \quad
v_\rho \mapsto [\BP_\rho].
\]
By \cref{lem:rankcompare}, this map is surjective. 
It commutes with the action of $K^{\pm 1}$ and $E$ because of \cref{prop:Tdecomp}. 
By \cref{prop:catshap}, the map intertwines the Shapovalov form with the bilinear form induced by the bifunctor $(-,-)$ on ${}_\bQ\bKO^\Delta(T^{\lambda,\underline{N}},0)$. Thus, it is a $\bQ\pp{q,\lambda}$-linear isomorphism.
Since the map intertwines the Shapovalov form with the bifunctor $(-,-)$, and commutes with the action of $E$ and $K^{\pm 1}$, we deduce by non-degeneracy of the Shapovalov form that it also commutes with the action of $F$. Thus, it is a map of $U_q(\slt)$-modules. 

The case ${}_\bQ\bKO^\Delta(T^{\lambda,\underline{N}},d_N) $ follows from \cref{thm:dNformal} together with \cite[Theorem 4.38]{webster}.
\end{proof}




\section{Cups, caps and double braiding functors}\label{sec:bimod}

Throughout this section, we fix $\underline{N} = (1, 1, \dots, 1) \in \bN^r$ and write $T^{\lambda,r} := T^{\lambda,\underline{N}}$, and $\otimes_T := \otimes_{(\oplus_{r\geq 0} T^{\lambda,r})}$. 
Also, when we will talk about (bi)modules, we will generally mean $\bZ^2$-graded dg-(bi)module, assuming it is clear from the context.

\subsection{Cup and cap functors}\label{sec:Tlaction}

Following \cite[\S 7]{webster} (see also \cite[\S4.3]{webstersl2}), we define the \emph{cup bimodule $B_i$} for $1 \leq i \leq r+1$ as the $(T^{\lambda,r+2},0)$-$(T^{\lambda,r},0)$-bimodule 
 generated by the diagrams 
 \begin{equation}\label{eq:cupbim}
	 \tikzdiagh{0}{
		\draw[vstdhl] (0,0) node[below]{\small $\lambda$} --(0,1);
		\draw (.5,0) -- (.5,1);
		\node at(1,.5) {\tiny$\dots$};
		\draw (1.5,0) -- (1.5,1);
		\draw[decoration={brace,mirror,raise=-8pt},decorate]  (.4,-.35) -- node {$b_0$} (1.6,-.35);
		\draw[stdhl] (2,0)  node[below]{\small $1$} --(2,1);
		\draw (2.5,0) -- (2.5,1);
		\node at(3,.5) {\tiny$\dots$};
		\draw (3.5,0) -- (3.5,1);
		\draw[decoration={brace,mirror,raise=-8pt},decorate]  (2.4,-.35) -- node {$b_1$} (3.6,-.35);
		\draw[stdhl] (4,0)  node[below]{\small $1$} --(4,1);
		\node[red] at  (5,.5) {\dots};
		\draw[stdhl] (6,0)  node[below]{\small $1$} --(6,1);
		\draw (6.5,0) -- (6.5,1);
		\node at(7,.5) {\tiny$\dots$};
		\draw (7.5,0) -- (7.5,1);
		\draw[decoration={brace,mirror,raise=-8pt},decorate]  (6.4,-.35) -- node {$b_{i-1}$} (7.6,-.35);
		\draw (8.5,.5) -- (8.5,1);
		\draw [stdhl] (8,1) .. controls (8,.25) and (9,.25) .. (9,1);
		\draw (9.5,0) -- (9.5,1);
		\node at(10,.5) {\tiny$\dots$};
		\draw (10.5,0) -- (10.5,1);
		\draw[decoration={brace,mirror,raise=-8pt},decorate]  (9.4,-.35) -- node {$b_{i-1}'$} (10.6,-.35);
		\draw[stdhl] (11,0)  node[below]{\small $1$} --(11,1);
		\node[red] at  (12,.5) {\dots};
		\draw[stdhl] (13,0)  node[below]{\small $1$} --(13,1);
		\draw (13.5,0) -- (13.5,1);
		\node at(14,.5) {\tiny$\dots$};
		\draw (14.5,0) -- (14.5,1);
		\draw[decoration={brace,mirror,raise=-8pt},decorate]  (13.4,-.35) -- node {$b_{r}$} (14.6,-.35);
	}
\end{equation}
for all $(b_0, \dots ,b_{i-2}, b_{i-1}, b_{i-1}', b_{i}, \dots, b_r) \in \bN^{r+2}$. 
Here, generated means that elements of $B_i$ are given by taking the diagram above and gluing any diagram of $T^{\lambda,r+2}$ on the top, and any diagram of $T^{\lambda,r}$ on the bottom.
The diagrams in $B_i$ are considered up to graded braid-like planar isotopy, with the cup being in homological degree $0$, 
and subject to the same local relations as the dg-enhanced KLRW algebra \eqref{eq:dotredstrand}-\eqref{eq:redR3} and \eqref{eq:relNail}, together with the following extra local relations:
\begin{align}
	\tikzdiag[xscale=-1]{		
		\draw (3.5,1.25) .. controls (3.5,.875) and (4.5,.875) .. (4.5,.5) ;
		\draw [stdhl]  (4,1.25) -- (4,1) .. controls (4,.25) and (5,.25) .. (5,1) -- (5,1.25);
	}
\ &= \  0,
&
	\tikzdiag{		
		\draw (3.5,1.25) .. controls (3.5,.875) and (4.5,.875) .. (4.5,.5) ;
		\draw [stdhl]  (4,1.25) -- (4,1) .. controls (4,.25) and (5,.25) .. (5,1) -- (5,1.25);
	}
\ &= \  0,
\label{eq:killcup}
\\
	\tikzdiag{
		\draw (4.5,.5) -- (4.5,1.25);
		\draw (3.5,0) .. controls (3.5,.75) and (5.5,1) .. (5.5,1.25);
		\draw [stdhl]  (4,1.25) -- (4,1) .. controls (4,.25) and (5,.25) .. (5,1) -- (5,1.25);
	}
\ &= \  
	\tikzdiag{
		\draw (4.5,.75) -- (4.5,1.25);
		\draw (3.5,0) .. controls (3.5,.25) and (5.5,.5) .. (5.5,1.25);
		\draw [stdhl]  (4,1.25)  .. controls (4,.5) and (5,.5) .. (5,1.25);
	}
&
	\tikzdiag[xscale=-1]{
		\draw (4.5,.5) -- (4.5,1.25);
		\draw (3.5,0) .. controls (3.5,.75) and (5.5,1) .. (5.5,1.25);
		\draw [stdhl]  (4,1.25) -- (4,1) .. controls (4,.25) and (5,.25) .. (5,1) -- (5,1.25);
	}
\ &= - \  
	\tikzdiag[xscale=-1]{
		\draw (4.5,.75) -- (4.5,1.25);
		\draw (3.5,0) .. controls (3.5,.25) and (5.5,.5) .. (5.5,1.25);
		\draw [stdhl]  (4,1.25)  .. controls (4,.5) and (5,.5) .. (5,1.25);
	}
	\label{eq:cupstrandslides}
\end{align}
We set the $\bZ^2$-degree of the generator in~\eqref{eq:cupbim} as
\[
\deg_{q,\lambda}\left(
	\tikzdiag{
		\draw (4.5,.5) -- (4.5,1);
		\draw [stdhl] (4,1) .. controls (4,.25) and (5,.25) .. (5,1);
	}
\right)
 \ := \ (0,0).
\]

\smallskip

Similarly, we define the \emph{cap bimodule $\overline{B}_i$} by taking the mirror along the horizontal axis of $B_i$. However, we declare that the cap is in homological degree $-1$, and with $\bZ^2$-degree given by
\[
\deg_{q,\lambda}\left(
	\tikzdiag[yscale=-1]{
		\draw (4.5,.5) -- (4.5,1);
		\draw [stdhl] (4,1) .. controls (4,.25) and (5,.25) .. (5,1);
	}
\right)
 \ := \ (-1,0).
\]
Note that since the red cap has a $-1$ homological degree, it anticommutes with the nails when applying a graded planar isotopy. 

\smallskip

From this, one defines the \emph{coevaluation} and \emph{evaluation dg-functors} as
\begin{align*}
\B_i := B_i \Lotimes_T - : \cD_{dg}(T^{\lambda,r},0) \rightarrow \cD_{dg}(T^{\lambda,r+2},0), \\
\overline{\B}_i := \overline{B}_i \Lotimes_T - : \cD_{dg}(T^{\lambda,r+2},0) \rightarrow \cD_{dg}(T^{\lambda,r},0). 
\end{align*}

\subsubsection{Biadjointness}
Note that
\[
\overline{\B}_i\cong  q \RHOM_T( B_i ,-) [1],
\]
by \cref{prop:cofBi} below. 
 Thus, $q^{-1}\overline{\B}_i [-1]$ is right adjoint to $\B_i$. Similarly, we obtain that $q \overline{\B}_i [1]$ is left-adjoint to $\B_i$. 

The unit and counit of $\B_i \dashv q^{-1}\overline{\B}_i [-1]$ gives a pair of maps of bimodules
\begin{align*}
\eta_i &: q(T^{\lambda,r})[1] \rightarrow \oB_i \Lotimes_T B_i,
&
\varepsilon_i &: B_i \Lotimes \oB_i \rightarrow q(T^{\lambda,r})[1],
\intertext{and similarly $q \overline{\B}_i [1] \dashv \B_i$ gives}
\overline{\eta}_i &: q^{-1} (T^{\lambda,r})[-1] \rightarrow B_i \Lotimes \oB_i, 
&
\overline{\varepsilon}_i &: \oB_i \Lotimes_T B_i \rightarrow q^{-1} (T^{\lambda,r})[-1].
\end{align*}

\subsubsection{Tightened basis}\label{sec:tightbasisBi}

Take $\kappa=(b_0,\ldots,b_{r+2}) \in \cP_{b}^{r+2}$ and $\rho \in \cP_{b}^{r}$. 
Let $\bar \kappa^i$ be given by $(b_0,  b_1, \dots, b_{i-2}, b_{i-1} + b_i - 1 + b_{i+1} , \widehat b_i, \widehat b_{i+1}, b_{i+2}, \dots b_r)\in\cP_{b}^{r}$.
For each $1 \leq \ell \leq b_i$, consider the map 
\[
g_\ell : q^{b_i+1-2\ell} (1_{\bar \kappa^i} T^{\lambda,r} 1_\rho)  \rightarrow 1_\kappa B_i  1_\rho,
\]
 given by gluing on the top the following element:
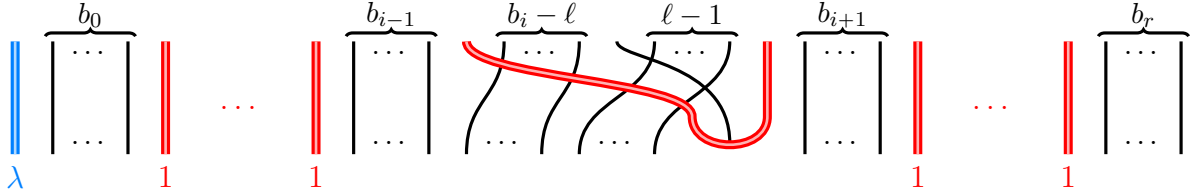
\[
\tikzdiag{
	\draw[vstdhl] (-6,-.5) node[below]{$\lambda$} -- (-6,1);
	\draw (-5.5,1) -- (-5.5,-.5);
	\node at(-5,.85){\small $\dots$};
	\node at(-5,-.35){\small $\dots$};
	\draw (-4.5,1) -- (-4.5,-.5);
	\tikzbraceop{-5.5}{-4.5}{1}{\small $b_{0}$};
	\draw[stdhl] (-4,-.5) node[below]{1} -- (-4,1);
	\node[red] at(-3,.125) {$\dots$};
	\draw[stdhl] (-2,-.5) node[below]{1} -- (-2,1);
	\draw (-1.5,1) -- (-1.5,-.5);
	\node at(-1,.85){\small $\dots$};
	\node at(-1,-.35){\small $\dots$};
	\draw (-.5,1) -- (-.5,-.5);
	\tikzbraceop{-1.5}{-.5}{1}{\small $b_{i-1}$};
	\draw (3.5,1) .. controls (3.5,.25) and (2.5,.25) .. (2.5,-.5);
	\node at(3,.85){\small $\dots$};
	\node at(2,-.35){\small $\dots$};
	\draw (2.5,1) .. controls (2.5,.25) and (1.5,.25) .. (1.5,-.5);
	\tikzbraceop{2.5}{3.5}{1}{\small $\ell - 1$};
	\draw (2,1) .. controls (2,.75) and (3.5,.75) .. (3.5,-.35);
	\draw (1.5,1) .. controls (1.5,.25) and (1,.25) .. (1,-.5);
	\node at(1,.85){\small $\dots$};
	\node at(.5,-.35){\small $\dots$};
	\draw (.5,1) .. controls (.5,.25) and (0,.25) .. (0,-.5);
	\tikzbraceop{.5}{1.5}{1}{\small $b_i - \ell$};
	\draw[stdhl] (4,1) -- (4,0) .. controls (4,-.5) and (3,-.5) .. (3,0) .. controls (3,.5) and (0,.5) .. (0,1);
	%
	%
	\draw (4.5,1) -- (4.5,-.5);
	\node at(5,.85){\small $\dots$};
	\node at(5,-.35){\small $\dots$};
	\draw (5.5,1) -- (5.5,-.5);
	\tikzbraceop{4.5}{5.5}{1}{\small $b_{i+1}$};
	\draw[stdhl] (6,-.5) node[below]{1} -- (6,1);
	\node[red] at(7,.125) {$\dots$};
	\draw[stdhl] (8,-.5) node[below]{1} -- (8,1);
	\draw (8.5,1) -- (8.5,-.5);
	\node at(9,.85){\small $\dots$};
	\node at(9,-.35){\small $\dots$};
	\draw (9.5,1) -- (9.5,-.5);
	\tikzbraceop{8.5}{9.5}{1}{\small $b_{r}$};
}
\]
Recall the basis ${}_\kappa B_\rho$ of \cref{thm:Tbasis}. 
We claim that
\begin{equation}\label{eq:basisBi}
	\bigsqcup_{\ell = 1}^{b_i} g_\ell({}_{\overline \kappa^i} B_\rho), 
\end{equation}
is a basis for $1_\kappa B_i 1_\rho$. We postpone the proof of this for later.

\subsection{Cofibrant replacement of $B_i$}\label{sec:cofBi}

As explained in~\cite[\S4.3]{webstersl2}, $B_i$ admits an easily describable cofibrant replacement as a left module. But before describing it, let us introduce some extra notations. Let 
$T_{i, \ \tikzRBR}$
be the left $(T^{\lambda,r},0)$-module generated by the elements
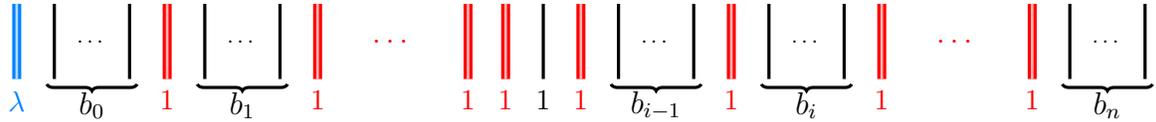
\[
\tikzdiagh{0}{
	\draw[vstdhl] (0,0)  node[below]{\small $\lambda$} --(0,1);
	\draw (.5,0) -- (.5,1);
	\node at(1,.5) {\tiny$\dots$};
	\draw (1.5,0) -- (1.5,1);
	\draw[decoration={brace,mirror,raise=-8pt},decorate]  (.4,-.35) -- node {$b_0$} (1.6,-.35);
	\draw[stdhl] (2,0)  node[below]{\small $1$} --(2,1);
	\draw (2.5,0) -- (2.5,1);
	\node at(3,.5) {\tiny$\dots$};
	\draw (3.5,0) -- (3.5,1);
	\draw[decoration={brace,mirror,raise=-8pt},decorate]  (2.4,-.35) -- node {$b_1$} (3.6,-.35);
	\draw[stdhl] (4,0)  node[below]{\small $1$} --(4,1);
	\node[red] at  (5,.5) {\dots};
	\draw[stdhl] (6,0)  node[below]{\small $1$} --(6,1);
	\draw[stdhl] (6.5,0)  node[below]{\small $1$} --(6.5,1);
	\draw (7,0)  node[below]{\small $1$} --(7,1);
	\draw[stdhl] (7.5,0)  node[below]{\small $1$} --(7.5,1);
	\draw (8,0) -- (8,1);
	\node at(8.5,.5) {\tiny$\dots$};
	\draw (9,0) -- (9,1);
	\draw[decoration={brace,mirror,raise=-8pt},decorate]  (7.9,-.35) -- node {$b_{i-1}$} (9.1,-.35);
	\draw[stdhl] (9.5,0)  node[below]{\small $1$} --(9.5,1);
	\draw (10,0) -- (10,1);
	\node at(10.5,.5) {\tiny$\dots$};
	\draw (11,0) -- (11,1);
	\draw[decoration={brace,mirror,raise=-8pt},decorate]  (9.9,-.35) -- node {$b_{i}$} (11.1,-.35);
	\draw[stdhl] (11.5,0)  node[below]{\small $1$} --(11.5,1);
	\node[red] at  (12.5,.5) {\dots};
	\draw[stdhl] (13.5,0)  node[below]{\small $1$} --(13.5,1);
	\draw (14,0) -- (14,1);
	\node at(14.5,.5) {\tiny$\dots$};
	\draw (15,0) -- (15,1);
	\draw[decoration={brace,mirror,raise=-8pt},decorate]  (13.9,-.35) -- node {$b_{n}$} (15.1,-.35);
}
\]
for all $(b_0, b_1, \dots, b_{r})$. 
We define similarly 
$T_{i,\  \tikzBRR}$ 
and
$T_{i,\  \tikzRRB}$.

\smallskip

Let $\br B_i$ be the left $(T^{\lambda,r+2},0)$-module given by the dg-module
\[
\br B_i := 
\begin{tikzcd}[row sep = 1ex]
 & q (T_{i,\  \tikzBRR}) [1]
 \ar{dr} \ar[no head]{dr}{
	 {\tikzdiag[scale=.5]{
	 	\draw (1,0) .. controls (1,.5) and (0,.5) .. (0,1);
	 	\draw[stdhl] (0,0) .. controls (0,.5) and (1,.5) .. (1,1);
	 	\draw[stdhl] (2,0) -- (2,1);
	 }}} 
 & \\
q^2 (T_{i,\  \tikzRBR}) [2]
\ar{ur} 
 \ar[no head]{ur}{
	 {\tikzdiag[scale=.5]{
	 	\draw (0,0) .. controls (0,.5) and (1,.5) .. (1,1);
	 	\draw[stdhl] (1,0) .. controls (1,.5) and (0,.5) .. (0,1);
	 	\draw[stdhl] (2,0) -- (2,1);
	 }}} 
\ar{dr}
\ar[no head,swap]{dr}{
	 -\ {\tikzdiag[scale=.5]{
	 	\draw (1,0) .. controls (1,.5) and (0,.5) .. (0,1);
	 	\draw[stdhl] (-1,0) -- (-1,1);
	 	\draw[stdhl] (0,0) .. controls (0,.5) and (1,.5) .. (1,1);
	 }}}
 & \oplus &  
 T_{i,\  \tikzRBR} 
 \\
 &  q (T_{i,\  \tikzRRB}) [1]
 \ar{ur}
\ar[no head,swap]{ur}{
	 {\tikzdiag[scale=.5]{
	 	\draw (0,0) .. controls (0,.5) and (1,.5) .. (1,1);
	 	\draw[stdhl] (1,0) .. controls (1,.5) and (0,.5) .. (0,1);
	 	\draw[stdhl] (-1,0) -- (-1,1);
	 }}}
	  & 
\end{tikzcd}
\]
where the differential is given by the arrows, which are 
 the maps given by adding the term in the label at the bottom of $\tikzRBR$ , $\tikzRRB$ or $\tikzBRR$. 
Similarly, we define a right cofibrant replacement $\oB_i \rb \overset{\simeq}{\twoheadrightarrow}\oB_i$ by taking the symmetric along the horizontal line and shifting everything by $q^{-1}(-) [-1]$. 

\begin{prop}\label{prop:cofBi}
There is a surjective quasi-isomorphism of left $\bZ^2$-graded $(T^{\lambda,r+2},0)$-modules
\[
\br B_i \overset{\simeq}{\twoheadrightarrow} B_i.
\]
\end{prop}

\begin{proof}
Consider the surjective map $T_{i,\  \tikzRBR} \twoheadrightarrow B_i$ that closes the elements $\tikzRBR$ at the bottom by a cup:
 \[
\tikzdiag{
\draw[stdhl](0,0) -- (0,1);\draw (.5,0) -- (.5,1);\draw[stdhl](1,0) -- (1,1);
}
\quad \mapsto \quad  
 	\tikzdiag{
		\draw (4.5,.5) -- (4.5,1);
		\draw [stdhl] (4,1) .. controls (4,.25) and (5,.25) .. (5,1);
	}
 \]
 This map is indeed surjective since any black strand going to the left of the cap factors through a black strand going to the right, using \cref{eq:cupstrandslides}. 
 Then the claim follows by observing that 
 \[
 \begin{tikzcd}[row sep = 1ex]
& & q T_{i,\  \tikzBRR} \ar[hookrightarrow]{dr} \ar[no head]{dr}{
	 {\tikzdiag[scale=.5]{
	 	\draw (1,0) .. controls (1,.5) and (0,.5) .. (0,1);
	 	\draw[stdhl] (0,0) .. controls (0,.5) and (1,.5) .. (1,1);
	 	\draw[stdhl] (2,0) -- (2,1);
	 }}} 
 & & &\\
0 \ar{r} & q^2 T_{i,\  \tikzRBR} \ar[hookrightarrow]{ur} 
 \ar[no head]{ur}{
	 {\tikzdiag[scale=.5]{
	 	\draw (0,0) .. controls (0,.5) and (1,.5) .. (1,1);
	 	\draw[stdhl] (1,0) .. controls (1,.5) and (0,.5) .. (0,1);
	 	\draw[stdhl] (2,0) -- (2,1);
	 }}} 
\ar[hookrightarrow]{dr}
\ar[no head,swap]{dr}{
	 -\ {\tikzdiag[scale=.5]{
	 	\draw (1,0) .. controls (1,.5) and (0,.5) .. (0,1);
	 	\draw[stdhl] (-1,0) -- (-1,1);
	 	\draw[stdhl] (0,0) .. controls (0,.5) and (1,.5) .. (1,1);
	 }}}
 & \oplus & T_{i,\  \tikzRBR} \ar[twoheadrightarrow]{r} 
 \ar[no head]{r}{
	 {\tikzdiag[scale=.5]{
		\draw (4.5,.5) -- (4.5,1);
		\draw [stdhl] (4,1) .. controls (4,.25) and (5,.25) .. (5,1);
	 }}}
	  & B_i \ar{r} & 0,\\
 && q T_{i,\  \tikzRRB} \ar[hookrightarrow]{ur}
\ar[no head,swap]{ur}{
	 {\tikzdiag[scale=.5]{
	 	\draw (0,0) .. controls (0,.5) and (1,.5) .. (1,1);
	 	\draw[stdhl] (1,0) .. controls (1,.5) and (0,.5) .. (0,1);
	 	\draw[stdhl] (-1,0) -- (-1,1);
	 }}}
	  & &  & 
\end{tikzcd}
\]
is an exact sequence. 

Indeed, by \cref{thm:Tbasis}, we know that adding a black/red crossing is an injective operation, and thus the sequence is exact on $q^2 T_{i,\  \tikzRBR}$. For the same reason we also have that 
\[
\ker\left( 
 \begin{tikzcd}[row sep = 1ex]
  q T_{i,\  \tikzBRR} \ar[hookrightarrow]{dr} \ar[no head]{dr}{
	 {\tikzdiag[scale=.5]{
	 	\draw (1,0) .. controls (1,.5) and (0,.5) .. (0,1);
	 	\draw[stdhl] (0,0) .. controls (0,.5) and (1,.5) .. (1,1);
	 	\draw[stdhl] (2,0) -- (2,1);
	 }}} 
 & 
 \\
 \oplus & T_{i,\  \tikzRBR}
	  \\
q T_{i,\  \tikzRRB} \ar[hookrightarrow]{ur}
\ar[no head,swap]{ur}{
	 {\tikzdiag[scale=.5]{
	 	\draw (0,0) .. controls (0,.5) and (1,.5) .. (1,1);
	 	\draw[stdhl] (1,0) .. controls (1,.5) and (0,.5) .. (0,1);
	 	\draw[stdhl] (-1,0) -- (-1,1);
	 }}}
	  & 
\end{tikzcd}
\right)
\cong 
 T_{i,\  \stikzdiag{
	 	\draw (1,0) .. controls (1,.5) and (0,.5) .. (0,1);
	 	\draw[stdhl] (0,0) .. controls (0,.5) and (1,.5) .. (1,1);
	 	\draw[stdhl] (2,0) -- (2,1);
 }}
 \cap 
 T_{i,\  \stikzdiag{
	 	\draw (0,0) .. controls (0,.5) and (1,.5) .. (1,1);
	 	\draw[stdhl] (1,0) .. controls (1,.5) and (0,.5) .. (0,1);
	 	\draw[stdhl] (-1,0) -- (-1,1);
 }}.
\]
By \cref{thm:Tbasis}, we know that if an element can be written as a diagram with a black strand crossing a red strand on the left, and as a different diagram with the same black strand crossing a red strand on the right, then it can be rewritten as a diagram with the same strand going straight, but carrying a dot.
These elements correspond exactly with the image of the preceding map in the complex, which is thus exact at the second position. 
Finally, we observe that 
\[
B_i
\cong 
T_{i,\  \tikzRBR}
\ /\ \bigl(
 T_{i,\  \stikzdiag{
	 	\draw (1,0) .. controls (1,.5) and (0,.5) .. (0,1);
	 	\draw[stdhl] (0,0) .. controls (0,.5) and (1,.5) .. (1,1);
	 	\draw[stdhl] (2,0) -- (2,1);
 }}
 + 
 T_{i,\  \stikzdiag{
	 	\draw (0,0) .. controls (0,.5) and (1,.5) .. (1,1);
	 	\draw[stdhl] (1,0) .. controls (1,.5) and (0,.5) .. (0,1);
	 	\draw[stdhl] (-1,0) -- (-1,1);
 }}
 \bigr),
\]
by constructing an inverse of the map that adds a cup on the bottom, by pulling the cup to the bottom. It is not hard, but a bit lengthy, to check that it respects the defining relations of $B_i$ in the quotient $T_{i,\  \tikzRBR}
\ /\ \bigl(
 T_{i,\  \stikzdiag{
	 	\draw (1,0) .. controls (1,.5) and (0,.5) .. (0,1);
	 	\draw[stdhl] (0,0) .. controls (0,.5) and (1,.5) .. (1,1);
	 	\draw[stdhl] (2,0) -- (2,1);
 }}
 + 
 T_{i,\  \stikzdiag{
	 	\draw (0,0) .. controls (0,.5) and (1,.5) .. (1,1);
	 	\draw[stdhl] (1,0) .. controls (1,.5) and (0,.5) .. (0,1);
	 	\draw[stdhl] (-1,0) -- (-1,1);
 }}
 \bigr)$. 
\end{proof}

\begin{cor}\label{thm:basisBi}
The elements in \cref{eq:basisBi} form a $\bZ\times\bZ^2$-graded $\Bbbk$-basis for $1_\kappa B_i 1_\rho$. 
\end{cor}

\begin{proof}
As in \cref{thm:Tbasis}, one can show that the elements in \cref{eq:basisBi} span the space $1_\kappa B_i 1_\rho$, mainly using \cref{eq:cupstrandslides} and \cref{eq:redR3}. 
Linear independence follows from a dimensional argument, using \cref{prop:cofBi} and \cref{thm:Tbasis}. The computation of the dimensions can be done at the non-categorified level, and thus is a consequence of \cref{eq:caponk} of \cref{lem:explicitaction}.
\end{proof}

Therefore, the map $\sum g_\ell : \oplus_{\ell = 1}^{b_i} q^{b_i+1-2\ell} (1_{\bar \rho^i} T^{\lambda,r}) \xrightarrow{\simeq} 1_\rho B_i$ of right modules is an isomorphism, where $\bar \rho^{i}$ and $g_\ell$  are as in \cref{sec:tightbasisBi}. 
In particular, $B_i $ is a cofibrant right dg-module. 

With \cref{thm:K0} in mind, this means that $\overline{\B}_i$ acts on ${}_\bQ\bKO^\Delta(T^{\lambda,\underline{N}},0)$ as the cap of $\cB$ on $M \otimes V^r$ (see \cref{eq:caponk}), and \cref{prop:cofBi} means that $\B_i$ acts as the cup (see \cref{eq:cuponk}).

\subsection{Double braiding functor}\label{sec:dbbraiding}

Inspired by the definition of the braiding functor in~\cite[\S6]{webster} (see also~\cite[\S4.1]{webstersl2}), we introduce a double braiding functor that will play the role of a categorification of the action of $\xi$ on $M \otimes V^r$.

\begin{defn}\label{def:dbbraiding}
The \emph{double braiding  bimodule $X$} 
(see \cref{rem:dbbraiding} for an explanation about the terminology) 
 is the $(T^{\lambda,r},0)$-$(T^{\lambda,r},0)$-bimodule generated by the diagrams
 \[
 \tikzdiagh{0}{
	\draw[stdhl] (0,0) node[below]{\small $1$} .. controls (0,.25) .. (-1,.5)
			.. controls (0,.75) .. (0,1);
	\draw[fill=white, color=white] (-1.1,.5) circle (.1cm);
	\draw[vstdhl] (-1,0)  node[below]{\small $\lambda$} -- (-1,1);
 	%
	\draw (.5,0) -- (.5,1);
	\node at(1,.5) {\tiny$\dots$};
	\draw (1.5,0) -- (1.5,1);
	\draw[decoration={brace,mirror,raise=-8pt},decorate]  (.4,-.35) -- node {$b_0$} (1.6,-.35);
	\draw[stdhl] (2,0)  node[below]{\small $1$} --(2,1);
	\draw (2.5,0) -- (2.5,1);
	\node at(3,.5) {\tiny$\dots$};
	\draw (3.5,0) -- (3.5,1);
	\draw[decoration={brace,mirror,raise=-8pt},decorate]  (2.4,-.35) -- node {$b_1$} (3.6,-.35);
	\draw[stdhl] (4,0)  node[below]{\small $1$} --(4,1);
	\node[red] at  (5,.5) {\dots};
	\draw[stdhl] (6,0)  node[below]{\small $1$} --(6,1);
	\draw (6.5,0) -- (6.5,1);
	\node at(7,.5) {\tiny$\dots$};
	\draw (7.5,0) -- (7.5,1);
	\draw[decoration={brace,mirror,raise=-8pt},decorate]  (6.4,-.35) -- node {$b_r$} (7.6,-.35);
}
 \]
for all $(b_0, \dots, b_r) \in \bN^{r+1}$. 
We consider diagrams in $X$ up to graded braid-like planar isotopy with the generators being in homological degree $0$, and subject to the relations\eqref{eq:dotredstrand}-\eqref{eq:redR3} and \eqref{eq:relNail}, and the extra local relations
\begin{align}\label{eq:nailslidedcross}
	\tikzdiagh{0}{
		\draw (.5,-.5) .. controls (.5,-.3) .. (0,-.1) .. controls (.5,.1) ..  (.5,.3) -- (.5,1.5);
		\draw[stdhl] (1,-.5) node[below]{\small $1$} -- (1,0) .. controls (1,.25) .. (0,.5)
				.. controls (1,.75) .. (1,1) -- (1,1.5);
		\draw[fill=white, color=white] (-.1,.5) circle (.1cm);
		\draw[vstdhl] (0,-.5)  node[below]{\small $\lambda$} -- (0,1.5) node[pos=.2,nail]{};
	}
\ &= \ 
	\tikzdiagh[yscale=-1]{0}{
		\draw (.5,-.5) .. controls (.5,-.3) .. (0,-.1) .. controls (.5,.1) ..  (.5,.3) -- (.5,1.5);
		\draw[stdhl] (1,-.5)-- (1,0) .. controls (1,.25) .. (0,.5)
				.. controls (1,.75) .. (1,1) -- (1,1.5) node[below]{\small $1$} ;
		\draw[fill=white, color=white] (-.1,.5) circle (.1cm);
		\draw[vstdhl] (0,-.5)   -- (0,1.5) node[pos=.2,nail]{} node[below]{\small $\lambda$};
	}
&
	\tikzdiagh{0}{
		\draw (1,-.5) .. controls (1,-.3) .. (0,-.1) .. controls (1,.1) ..  (1,.3) -- (1,1.5);
		\draw[stdhl] (.5,-.5) node[below]{\small $1$} -- (.5,0) .. controls (.5,.25) .. (0,.5)
				.. controls (.5,.75) .. (.5,1) -- (.5,1.5);
		\draw[fill=white, color=white] (-.1,.5) circle (.1cm);
		\draw[vstdhl] (0,-.5)  node[below]{\small $\lambda$} -- (0,1.5) node[pos=.2,nail]{};
	}
\ &= \ 
	\tikzdiagh[yscale=-1]{0}{
		\draw (1,-.5) .. controls (1,-.3) .. (0,-.1) .. controls (1,.1) ..  (1,.3) -- (1,1.5);
		\draw[stdhl] (.5,-.5) -- (.5,0) .. controls (.5,.25) .. (0,.5)
				.. controls (.5,.75) .. (.5,1) -- (.5,1.5)  node[below]{\small $1$};
		\draw[fill=white, color=white] (-.1,.5) circle (.1cm);
		\draw[vstdhl] (0,-.5)  -- (0,1.5) node[pos=.2,nail]{}  node[below]{\small $\lambda$};
	}
\end{align}
We set the $\bZ^2$-degree of the generator as 
\begin{align*}
\deg_{q,\lambda}\left(
	\tikzdiag{
		\draw[stdhl] (1,0) node[below]{\small $1$} .. controls (1,.25) .. (0,.5)
				.. controls (1,.75) .. (1,1);
		\draw[fill=white, color=white] (-.1,.5) circle (.1cm);
		\draw[vstdhl] (0,0)  node[below]{\small $\lambda$} -- (0,1);
	}
\right) &:= (0,-1).
\end{align*}
\end{defn}

We define the \emph{double braiding functor} as
\[
\Xi := X \Lotimes_T - : \cD_{dg}(T^{\lambda,r},0) \rightarrow \cD_{dg}(T^{\lambda,r},0). 
\]

\subsubsection{Tightened basis}

Let us now describe a basis of the bimodule $X$, similar as the basis of $T^{\lambda,r}_b$ given in \cref{thm:Tbasis}. We fix $\kappa$ and $\rho$ two elements of $\mathcal{P}^r_b$ and recall the set ${}_\kappa S_\rho$ defined in \cref{ssec:basisTl}. For each $w\in {}_\kappa S_\rho,\ \underline{l}=(l_1,\ldots,l_b)\in \{0,1\}^b$ and $\underline{a}=(a_1,\ldots,a_b)\in\mathbb{N}^b$ we define an element $x_{w,\underline{l},\underline{a}}\in 1_\kappa X 1_\rho$ as follows:
\begin{enumerate}
\item choose a left-reduced expression of $w$ in terms of diagrams as above,
\item for each $1\leq i \leq b$, if $l_i=1$, nail the $i$-th black strand (counting on the top from the left) on the blue strand by pulling it from its leftmost position, 
\item for each $1\leq i \leq b$, add $a_i$ dots on the $i$-th black strand at the top,
\item finally, attach the first red strand to the blue strand by pulling it from its leftmost position.
\end{enumerate}

\begin{defn}\label{def:unbraidingmap}
Define the \emph{unbraiding map}
  \[
u : \lambda  X \rightarrow T^{\lambda,r}, 
\]
as the map given by removing the double braiding
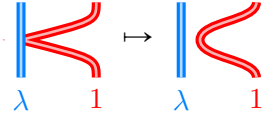
\[
	\tikzdiagh{0}{
		\draw[stdhl] (1,0) node[below]{\small $1$} .. controls (1,.25) .. (0,.5)
				.. controls (1,.75) .. (1,1);
		\draw[fill=white, color=white] (-.1,.5) circle (.1cm);
		\draw[vstdhl] (0,0)  node[below]{\small $\lambda$} -- (0,1);
	}
	\mapsto
	\tikzdiagh{0}{
		\draw[stdhl] (1,0) node[below]{\small $1$} .. controls (1,.25) and (.25,.25)  .. (.25,.5)
				.. controls (.25,.75) and (1,.75) .. (1,1);
		\draw[vstdhl] (0,0)  node[below]{\small $\lambda$} -- (0,1);
	}
        \]
        \end{defn}
 Note that the unbraiding map is a map of $(T_b^{\lambda,r},0)$-$(T_b^{\lambda,r},0)$-bimodules.

\begin{thm}\label{thm:X0basis}
  The set $\left\{x_{w,\underline{l},\underline{a}}\ \middle\vert w\in {}_\kappa S_\rho,\ \underline{l}\in \{0,1\}^b,\ \underline{a}\in\mathbb{N}^b\right\}$ is a $\bZ\times\bZ^2$-graded $\mathbb{\Bbbk}$-basis of $1_\kappa X 1_\rho$.
\end{thm} 

\begin{proof}
  Showing that this set generates $1_\kappa X 1_\rho$ is similar to \cite[Proposition 3.13]{naissevaz3} and we leave the details to the reader.

To show that the elements $(x_{w,\underline{l},\underline{a}})_{w,\underline{l},\underline{a}}$ are linearly independant  
we consider a linear combination $\sum_{w,\underline{l},\underline{a}}\alpha_{w,\underline{l},\underline{a}}x_{w,\underline{l},\underline{a}}=0$ and apply the unbraiding map $u$. We now pull the first red strand to its original position before the last step of the construction of $x_{w,\underline{l},\underline{a}}$. This has the effect of adding dots on some black strands because of \cref{eq:redR2}.

      We now rewrite $u\left(\sum_{w,\underline{l},\underline{a}}\alpha_{w,\underline{l},\underline{a}}x_{w,\underline{l}\underline{a}}\right)=0$ in terms of the tightened basis of $T^{\lambda,r}_b$. We carefully look at the terms with the highest number of crossings: by pulling the dots at the top, we obtain different elements of the tightened basis of $T^{\lambda,r}_b$ plus terms with a lower number of crossings. From the freeness of the tightened basis of $T^{\lambda,r}_b$, we deduce that the coefficient of the terms with the highest number of crossings must be zero and we can proceed by a descending induction on the number of crossings.
\end{proof}

\begin{cor}\label{cor:uinj}
  The unbraiding map $u : \lambda  X \rightarrow T^{\lambda,r}$ is injective.
\end{cor}

\begin{proof}
  The matrix of $u$ in terms of tightened bases can be made in column echelon form with pivots being $1$. 
\end{proof}

\subsection{Cofibrant replacement of $X$}\label{sec:cofX}

We now want to construct a left cofibrant replacement for $X$.  
Take $\rho = (b_2, \dots, b_{r}) \in \cP_b^{r-2}$ and consider the idempotent $1_{k,\ell,\rho} := 1_{k,\ell,b_2,\dots,b_{r}}$. 
We also write
\[
\bar 1_{\ell,\rho} :=  
\tikzdiagh{0}{
	\draw (.5,0) -- (.5,1);
	\node at(1,.5) {\tiny$\dots$};
	\draw (1.5,0) -- (1.5,1);
	\draw[decoration={brace,mirror,raise=-8pt},decorate]  (.4,-.35) -- node {$\ell$} (1.6,-.35);
	\draw[stdhl] (2,0)  node[below]{\small $1$} --(2,1);
	\draw (2.5,0) -- (2.5,1);
	\node at(3,.5) {\tiny$\dots$};
	\draw (3.5,0) -- (3.5,1);
	\draw[decoration={brace,mirror,raise=-8pt},decorate]  (2.4,-.35) -- node {$b_2$} (3.6,-.35);
	\draw[stdhl] (4,0)  node[below]{\small $1$} --(4,1);
	\node[red] at  (5,.5) {\dots};
	\draw[stdhl] (6,0)  node[below]{\small $1$} --(6,1);
	\draw (6.5,0) -- (6.5,1);
	\node at(7,.5) {\tiny$\dots$};
	\draw (7.5,0) -- (7.5,1);
	\draw[decoration={brace,mirror,raise=-8pt},decorate]  (6.4,-.35) -- node {$b_{r}$} (7.6,-.35);
}
\]
so that for example
\[
1_{0,k+\ell,\rho} =
\tikzdiagh{0}{
	\draw[vstdhl] (-.5,0) node[below]{\small $\lambda$} --(-.5,1);
	\draw[stdhl] (0,0) node[below]{\small $1$} --(0,1);
	\draw (.5,0) -- (.5,1);
	\node at(1,.5) {\tiny$\dots$};
	\draw (1.5,0) -- (1.5,1);
	\draw[decoration={brace,mirror,raise=-8pt},decorate]  (.4,-.35) -- node {$k$} (1.6,-.35);
}
\otimes \bar 1_{\ell,\rho}.
\]
For $k \geq 0, \ell \geq 0$ and $\rho \in \cP_b^{r-2}$, we define
\begin{align*}
Y^1_{k,\ell,\rho} &:= \bigoplus_{t=0}^{k-1} Y^{1,t}_k, & Y^{1,t}_{k,\ell,\rho} &:= \lambda q^{k-2t+1} (T_b^{\lambda,r} 1_{1,k+\ell-1,\rho})[1], 
\end{align*}
\begin{align*}
Y^0_{k,\ell,\rho} &:= {Y'}^0_{k,\ell,\rho} \oplus  \bigoplus_{t = 0}^{k-1} Y^{0,t}_{k,\ell,\rho}, &
{Y'}^0_{k,\ell,\rho} :=  \lambda^{-1} q^{k} (T_b^{\lambda,r}  1_{0,k+\ell,\rho}), \quad Y^{0,t}_{k,\ell,\rho} &:= \lambda q^{k-2t}  (T_b^{\lambda,r} 1_{0,k+\ell,\rho})[1].
\end{align*}
Note that $Y^1_0 = 0$ and $Y^0_0 =  \lambda^{-1} (T_b^{\lambda,r}  1_{0,\ell,\rho})$. 

We write
\begin{align*}
X_k &:= \bigoplus_{\ell \geq 0, \rho \in  \cP_b^{r-2}} X 1_{k,\ell,\rho},
&
Y^1_{k} &:= \bigoplus_{\ell \geq 0, \rho \in  \cP_b^{r-2}} Y^1_{k,\ell,\rho},
&
Y^0_{k}  &:= \bigoplus_{\ell \geq 0, \rho \in  \cP_b^{r-2}} Y^0_{k,\ell,\rho},
\end{align*}
and similarly for $Y^{1,t}_{k}$, ${Y'}^0_{k}$ and $Y^{0,t}_{k}$.

Define the cofibrant $(T^{\lambda,r}_b,0)$-module $\br X_k$ given by the mapping cone
\[
\br X_k :=
\cone\bigl(
Y^1_k
\xrightarrow{\ \imath_k\ }
Y^0_k
\bigr),
\]
where $\imath_k := \sum_{t = 0}^{k-1} \imath_k^t$ for
\begin{align*}
\imath_k^t :& Y^{1,t}_k \rightarrow {Y'}^0_k \oplus Y^{0,t}_k, \\
&
\tikzdiagh{0}{
	\draw[vstdhl] (-.5,0) node[below]{\small $\lambda$} --(-.5,1);
	\draw (0,0) -- (0,1);
	\draw[stdhl] (.5,0) node[below]{\small $1$} --(.5,1);
}
\otimes \bar 1_{k+\ell-1,\rho}
\ \mapsto \ 
\left(
-\ 
\tikzdiagh{0}{
	\draw (1.25,0) .. controls (1.25,.5) .. (-.5,.75) .. controls (0,.875) .. (0,1);
	\draw[stdhl] (0,0) node[below]{\small $1$} .. controls(0,.5) and (.5,.5) .. (.5,1);
	\draw[vstdhl] (-.5,0) node[below]{\small $\lambda$} --(-.5,1) node[pos=.75,nail]{};
	\draw (.5,0) .. controls (.5,.5) and (.75,.5) .. (.75,1);
	\node at (.75,.15){\tiny $\dots$};
	\draw (1,0) .. controls (1,.5) and (1.25,.5)..  (1.25,1);
	\draw[decoration={brace,mirror,raise=-8pt},decorate]  (.4,-.35) -- node { \small $t$} (1.1,-.35);
}
\otimes \bar  1_{\ell+k-1-t,\rho},
\ 
\tikzdiagh{0}{
	\draw[vstdhl] (-.5,0) node[below]{\small $\lambda$} --(-.5,1);
	\draw (.5,0) .. controls (.5,.5) and (0,.5) .. (0,1);
	\draw[stdhl] (0,0) node[below]{\small $1$} .. controls(0,.5) and (.5,.5) .. (.5,1);
}
\otimes \bar  1_{\ell+k-1,\rho}
\right)
\end{align*}
Note that each $\imath_k^t$ is injective, and therefore so is $\imath_k$. 
Then, consider the left module map
\[
\gamma_k : \br X_k \rightarrow X_k, 
\]
given by $\gamma_k := \gamma_k' + \sum_{t = 0}^{k-1} \gamma_k^t$ where
\begin{align*}
\gamma_k' : & {Y'}^0_k \rightarrow X_k, \\
&
\tikzdiagh{0}{
	\draw[vstdhl] (-.5,0) node[below]{\small $\lambda$} --(-.5,1);
	\draw[stdhl] (0,0) node[below]{\small $1$} --(0,1);
	\draw (.5,0) -- (.5,1);
	\node at(.75,.5) {\tiny$\dots$};
	\draw (1,0) -- (1,1);
	\draw[decoration={brace,mirror,raise=-8pt},decorate]  (.4,-.35) -- node {$k$} (1.1,-.35);
}
\otimes \bar  1_{\ell,\rho}
\ \mapsto \ 
\tikzdiagh{0}{
	\draw (0,0) .. controls (0,.5) and (.5,.5) .. (.5,1);
	\node at(.75,.75) {\tiny$\dots$};
	\draw (.5,0) .. controls (.5,.5) and (1,.5) ..  (1,1);
	\draw[decoration={brace,mirror,raise=-8pt},decorate]  (-.1,-.35) -- node {$k$} (.6,-.35);
	\draw[stdhl] (1,0) node[below]{\small $1$} .. controls (1,.25) .. (-.5,.5)
			.. controls (0,.75) .. (0,1);
	\draw[fill=white, color=white] (-.6,.5) circle (.1cm);
	\draw[vstdhl] (-.5,0)  node[below]{\small $\lambda$} -- (-.5,1);
}
\otimes \bar  1_{\ell,\rho},
\end{align*}
and
\begin{align*}
\gamma_k^t : & Y^{0,t}_k \rightarrow X_k, \\
&
\tikzdiagh{0}{
	\draw[vstdhl] (-.5,0) node[below]{\small $\lambda$} --(-.5,1);
	\draw[stdhl] (0,0) node[below]{\small $1$} --(0,1);
	\draw (.5,0) -- (.5,1);
	\node at(.75,.5) {\tiny$\dots$};
	\draw (1,0) -- (1,1);
	\draw[decoration={brace,mirror,raise=-8pt},decorate]  (.4,-.35) -- node {$k$} (1.1,-.35);
}
\otimes  \bar 1_{\ell,\rho}
\ \mapsto \ 
\tikzdiagh{0}{
	\draw (0,-.5) .. controls (0,0) and (.75,0) .. (.75,1);
	\node at(1,.75) {\tiny$\dots$};
	\draw (.5,-.5) .. controls (.5,0) and (1.25,0) .. (1.25,1);
	\draw (.75,-.5) .. controls (.75,-.25) .. (-.5,0) .. controls (.5,.5) ..  (.5,1);
	\draw (1,-.5) .. controls (1,0) and (1.5,0) .. (1.5,1);
	\node at(1.75,.75) {\tiny$\dots$};
	\draw (1.5,-.5) .. controls (1.5,0) and (2,0) .. (2,1);
	\draw[decoration={brace,mirror,raise=-8pt},decorate]  (-.1,-.85) -- node {$t$} (.6,-.85);
	\draw[stdhl] (2,-.5) node[below]{\small $1$}  -- (2,0) .. controls (2,.25) .. (-.5,.5)
			.. controls (0,.75) .. (0,1);
	\draw[fill=white, color=white] (-.6,.5) circle (.1cm);
	\draw[vstdhl] (-.5,-.5)  node[below]{\small $\lambda$} -- (-.5,1) node[pos=.33,nail]{};
}
\otimes \bar  1_{\ell,\rho},
\end{align*}
for all $0 \leq t \leq k-1$. 

\begin{lem}\label{lem:gammasurjective}
The map $\gamma_k : \br X_k \rightarrow X_k$ is surjective.
\end{lem}

\begin{proof}
The statement can be proved by observing that $X_k$ is generated as a left $(T_b^{\lambda,r},0)$-module by the elements
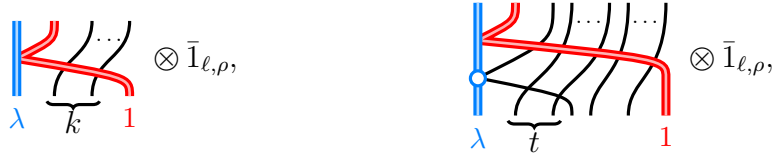
\begin{align*}
\tikzdiagh{0}{
	\draw (0,0) .. controls (0,.5) and (.5,.5) .. (.5,1);
	\node at(.75,.75) {\tiny$\dots$};
	\draw (.5,0) .. controls (.5,.5) and (1,.5) ..  (1,1);
	\draw[decoration={brace,mirror,raise=-8pt},decorate]  (-.1,-.35) -- node {$k$} (.6,-.35);
	\draw[stdhl] (1,0) node[below]{\small $1$} .. controls (1,.25) .. (-.5,.5)
			.. controls (0,.75) .. (0,1);
	\draw[fill=white, color=white] (-.6,.5) circle (.1cm);
	\draw[vstdhl] (-.5,0)  node[below]{\small $\lambda$} -- (-.5,1);
  }
&\otimes \bar  1_{\ell,\rho},
&
\tikzdiagh{0}{
	\draw (0,-.5) .. controls (0,0) and (.75,0) .. (.75,1);
	\node at(1,.75) {\tiny$\dots$};
	\draw (.5,-.5) .. controls (.5,0) and (1.25,0) .. (1.25,1);
	\draw (.75,-.5) .. controls (.75,-.25) .. (-.5,0) .. controls (.5,.5) ..  (.5,1);
	\draw (1,-.5) .. controls (1,0) and (1.5,0) .. (1.5,1);
	\node at(1.75,.75) {\tiny$\dots$};
	\draw (1.5,-.5) .. controls (1.5,0) and (2,0) .. (2,1);
	\draw[decoration={brace,mirror,raise=-8pt},decorate]  (-.1,-.85) -- node {$t$} (.6,-.85);
	\draw[stdhl] (2,-.5) node[below]{\small $1$}  -- (2,0) .. controls (2,.25) .. (-.5,.5)
			.. controls (0,.75) .. (0,1);
	\draw[fill=white, color=white] (-.6,.5) circle (.1cm);
	\draw[vstdhl] (-.5,-.5)  node[below]{\small $\lambda$} -- (-.5,1) node[pos=.33,nail]{};
}
&\otimes \bar  1_{\ell,\rho},
\end{align*}
for all $0 \leq t \leq k-1$. 
The details can be found in \cref{sec:proofsofsecbimod}.
\end{proof}

\begin{lem}\label{lem:sesX0}
The sequence
\[
0 \rightarrow Y^1_k \xrightarrow{\imath_k} Y^0_k \xrightarrow{\gamma_k} X_k \rightarrow 0,
\]
is a short exact sequence of left $\bZ^2$-graded $(T^{\lambda,r}, 0)$-modules. 
\end{lem}

\begin{proof}
Since we already have a complex with an injection and a surjection, it is enough to show that 
\[
 \gdim X_k = 
\gdim Y^0_k - \gdim Y^1_k,
\]
where $\gdim$ is the graded dimension in the form of a Laurent series in $\bN\llbracket h^{\pm 1}, \lambda^{\pm 1}, q^{\pm 1} \rrbracket$. 
This can be shown by induction on $k$, and the details are in~\cref{sec:proofsofsecbimod}.
\end{proof}

From that, we induce a right $(T^{\lambda,r},0)$-$A_\infty$-action on $\br X := \bigoplus_{k \geq 0} \br X_k$ (see \cref{sec:Ainftyaction}), turning it into a $\bZ^2$-graded  $(T^{\lambda,r},0)$-$(T^{\lambda,r},0)$-$A_\infty$-bimodule, and we obtain:

\begin{prop}\label{prop:gammaqi}
The map $\gamma := \sum_{k \geq 0} \gamma_k : \br X \twoheadrightarrow X$ is a quasi-isomorphism of $\bZ^2$-graded $(T^{\lambda,r},0)$-$(T^{\lambda,r},0)$-$A_\infty$-bimodules.  
\end{prop}

\begin{proof}
It is an immediate consequence of \cref{lem:sesX0}.
\end{proof}

Again, having \cref{thm:K0} in mind, it means $\Xi$ acts on ${}_\bQ\bKO^\Delta(T^{\lambda,\underline{N}},0)$ as the element $\xi$ of $\cB$ on $M \otimes V^r$ (see \cref{eq:xionk}). 




\section{A categorification of the blob algebra}\label{sec:catTLB}

As in~\cite[\S7]{webster}, the cup and cap functors respect a categorical instance of the Temperley--Lieb algebra relations \eqref{eq:TLrels}--\eqref{eq:TLloopremov}. 
We additionally show that the double braiding functor respects a categorical version of the blob relations \eqref{eq:TLBloopremov} and \eqref{eq:TLBdoublebraid}. 
Note that Webster also proves that the cup and cap functors intertwine the categorical $U_q(\slt)$-action, which
categorifies the fact that the Temperley--Lieb algebra describes morphisms of $U_q(\slt)$-
modules. We start by proving the same for these functors in the dg-setting as well as for
the double braiding functors:

\begin{prop}\label{prop:catactioncommutes}
We have natural isomorphisms $\E \circ \Xi \cong \Xi \circ \E$ and $\F \circ \Xi \cong \Xi \circ \F$, and also $\E \circ\B_i  \cong \B_i \circ\E$, $\F \circ\B_i \cong \B_i \circ\F$, and similarly for $\overline\B_i$.
\end{prop}

\begin{proof}
Since $\E$ and $\F$ are given by derived tensor product with a dg-bimodule that is cofibrant both as left and as right module, all compositions are given by usual tensor product of dg-bimodules. Then, the first isomorphism is equivalent to
\[
1_{b,1}(T^{\lambda,r}_{b+1}) \otimes_{b+1} X_{b+1} \cong   X_{b}  \otimes_{b}1_{b,1}(T^{\lambda,r}_{b+1}), 
\]
which in turn follows from \cref{thm:X0basis} and \cref{prop:Tdecomp}. The case with $\F$ is identical, and so is the proof for $\B_i$ using \cref{thm:basisBi}.
\end{proof}

Then, we use all this to show that compositions of the functors $\B_i,\bar \B_i$ and $\Xi$ realize a categorification of $\cB$.

\subsection{Temperley--Lieb relations}\label{sec:catTLaction}
This section is an extension of Webster's results~\cite[\S7]{webster} for the dg-enhanced KLRW algebra $T^{\lambda,r}$. 

\begin{prop}
There is an isomorphism 
\[
 \bar B_{i \pm 1} \Lotimes_T B_i \cong T^{\lambda,r},
\]
of $\bZ^2$-graded $(T^{\lambda,r}, 0)$-$(T^{\lambda,r},0)$-$A_\infty$-bimodules.
\end{prop}

\begin{proof} 
We prove $\bar B_{i - 1} \Lotimes_T B_i \cong T^{\lambda,r}$, the other case follows similarly. 
Using \cref{prop:catactioncommutes} and the fact that $\B_i \circ  \mathfrak{I} \cong \mathfrak{I} \circ \B_i$ for $ i < r-1$ (where we recall $\mathfrak{I}$ is the induction along a red strand defined in \cref{sec:redind}), we can work locally, supposing that $i = r-1$ and $b_i = 0$. Then, we have that $\bar B_{i - 1} \otimes_T \br B_i $ looks like
\[
\begin{tikzcd}[row sep = 1ex]
 & 
 q\left(
\tikzdiag[scale=.75]{
	\draw (1,0) -- (1,1.5);
	\filldraw [fill=white, draw=black] (-.25,.3) rectangle (2.25,1);
	\filldraw [fill=white, draw=white, dotted] (-.25,.3) rectangle (.15,1);
	\draw [stdhl] (.5,0) -- (.5,1)  .. controls (.5,1.75) and (1.5,1.75) .. (1.5,1)--(1.5,0);
	\draw [stdhl] (2,0) -- (2,1.5);
	%
}
\right)[1]
 \ar{dr} \ar[no head]{dr}{
	 {\tikzdiag[scale=.5]{
	 	\draw (1,0) .. controls (1,.5) and (0,.5) .. (0,1);
	 	\draw[stdhl] (0,0) .. controls (0,.5) and (1,.5) .. (1,1);
	 	\draw[stdhl] (2,0) -- (2,1);
	 }}} 
 & \\
 q^2\left(
\tikzdiag[scale=.75]{
	\draw (1.5,0) -- (1.5,.5) (1,1) -- (1,1.5);
	\filldraw [fill=white, draw=black] (-.25,.3) rectangle (2.25,1);
	\filldraw [fill=white, draw=white, dotted] (-.25,.3) rectangle (.15,1);
	\draw [stdhl] (.5,0) -- (.5,1)  .. controls (.5,1.75) and (1.5,1.75) .. (1.5,1) .. controls (1.5,.5) and (1,.5) .. (1,0);
	\draw [stdhl] (2,0) -- (2,1.5);
	%
}
\right)[2]
\ar{ur} 
 \ar[no head]{ur}{
	 {\tikzdiag[scale=.5]{
	 	\draw (0,0) .. controls (0,.5) and (1,.5) .. (1,1);
	 	\draw[stdhl] (1,0) .. controls (1,.5) and (0,.5) .. (0,1);
	 	\draw[stdhl] (2,0) -- (2,1);
	 }}} 
\ar{dr}
\ar[no head,swap]{dr}{
	 -\ {\tikzdiag[scale=.5]{
	 	\draw (1,0) .. controls (1,.5) and (0,.5) .. (0,1);
	 	\draw[stdhl] (-1,0) -- (-1,1);
	 	\draw[stdhl] (0,0) .. controls (0,.5) and (1,.5) .. (1,1);
	 }}}
 & \oplus &  
\tikzdiag[scale=.75]{
	\draw (1.5,0) -- (1.5,.5) (1,1) -- (1,1.5);
	\filldraw [fill=white, draw=black] (-.25,.3) rectangle (2.25,1);
	\filldraw [fill=white, draw=white, dotted] (-.25,.3) rectangle (.15,1);
	\draw [stdhl] (.5,0) -- (.5,1)  .. controls (.5,1.75) and (1.5,1.75) .. (1.5,1) .. controls (1.5,.5) and (1,.5) .. (1,0);
	\draw [stdhl] (2,0) -- (2,1.5);
	%
}
 \\
 &  
 q\left(
\tikzdiag[scale=.75]{
	\draw (2,0) -- (2,.5) (1,1) -- (1,1.5);
	\filldraw [fill=white, draw=black] (-.25,.3) rectangle (2.25,1);
	\filldraw [fill=white, draw=white, dotted] (-.25,.3) rectangle (.15,1);
	\draw [stdhl] (.5,0) -- (.5,1)  .. controls (.5,1.75) and (1.5,1.75) .. (1.5,1) .. controls (1.5,.5) and (1,.5) ..(1,0);
	\draw [stdhl] (1.5,0) .. controls (1.5,.5) and (2,.5) .. (2,1) -- (2,1.5);
	%
}
\right)[1]
 \ar{ur}
\ar[no head,swap]{ur}{
	 {\tikzdiag[scale=.5]{
	 	\draw (0,0) .. controls (0,.5) and (1,.5) .. (1,1);
	 	\draw[stdhl] (1,0) .. controls (1,.5) and (0,.5) .. (0,1);
	 	\draw[stdhl] (-1,0) -- (-1,1);
	 }}}
	  & 
\end{tikzcd}
\]
which is isomorphic to
\[
\begin{tikzcd}[row sep = 1ex]
 & 
T^{\lambda,r}
 \ar{dr}
 & \\
0
\ar{ur} 
\ar{dr}
 & \oplus &  
0
 \\
 &  
0
 \ar{ur}
	  & 
\end{tikzcd}
\]
because of \cref{eq:killcup}. 
Note that it is an isomorphism of dg-bimodules, since all the higher composition maps of the $A_\infty$-structure must be zero by degree reasons, concluding the proof. 
\end{proof}

\begin{cor}\label{cor:TLbiadj}
There is a natural isomorphism $\bar \B_{i \pm 1} \circ \B_i \cong \id$. 
\end{cor}

\begin{prop}\label{prop:distinguishedBitriangle}
There is a distinguished triangle
\[
q (T^{\lambda,r})[1] \xrightarrow{\eta_i} \overline{B}_i \Lotimes_T B_i \xrightarrow{\overline \varepsilon_i} q^{-1} (T^{\lambda,r})[-1]  \xrightarrow{0} 
\]
of $\bZ^2$-graded  $(T^{\lambda,r}, 0)$-$(T^{\lambda,r},0)$-$A_\infty$-bimodules.
\end{prop}

\begin{proof}
We have
\[
\bar B_{i} \Lotimes_T B_i \cong \bar B_{i} \rb \otimes_T  B_i,
\]
which looks like
\[
\begin{tikzcd}[row sep = 1ex]
 & 0
 \ar{dr}
 & \\
q ( B_{i}^{\tikzRBR})[1] \ar{ur} 
\ar{ur}
\ar{dr}
 & \oplus &  q^{-1}( B_{i}^{\tikzRBR})[-1]. \\
 &  0
\ar{ur}
	  & \\
\end{tikzcd}
\]
Thus, since $B_{i}^{\tikzRBR} \cong  T^{\lambda,r}$, we have that
\[
H(\bar B_{i} \Lotimes_T B_i) \cong q(T^{\lambda,r})[1] \oplus q^{-1}(T^{\lambda,r})[-1].
\]
In order to compute $\eta_i$, recall (or see \cref{sec:unitandcounit}) that the unit of the adjunction $(B_i \Lotimes_T -) \vdash (\RHOM_T(B_i, -))$ is given by 
\[
\eta_i' : T^{\lambda,r} \rightarrow \RHOM_T( B_i, B_i \Lotimes_T T^{\lambda,r}) \cong \HOM_T( \br B_i, B_i), 
\quad
t \mapsto \left[ x \mapsto \overline{x} \cdot t \right],
\]
where $\overline{x}$ is the image of $x$ under the map $\br B_i \twoheadrightarrow B_i$. Moreover, $\HOM_T( \br B_i, B_i)$ is given by
\[
\begin{tikzcd}[row sep = 1ex]
 & 0
 \ar[leftarrow]{dr}
 & \\
q^{-2} \HOM_T(T_{i,\  \tikzRBR}, B_i)[-2] \ar[leftarrow]{ur} 
\ar[leftarrow]{ur}
\ar[leftarrow]{dr}
 & \oplus &  \HOM(T_{i,\  \tikzRBR}, B_i), \\
 &  0
\ar[leftarrow]{ur}
	  & \\
\end{tikzcd}
\]
and then, $\eta_i'$ is the map $T^{\lambda,r} \xrightarrow{\simeq} \HOM(T_{i,\  \tikzRBR}, B_i) \cong B_i^{\tikzRBR}$ that adds a cup on the top. 
Thus, $\eta_i$ identifies $q(T^{\lambda,r})[1]$ with 
$q(B_i^{\tikzRBR})[1] \subset H(\bar B_{i} \Lotimes_T B_i)$ in homology. 
Similarly, the counit of the adjunction $(\oB_i \Lotimes -) \vdash (\RHOM_T(\oB_i, -))$ is
\[
\overline \varepsilon' : \oB_i \Lotimes_T \RHOM_T (\oB_i, T^{\lambda,r})  \cong \oB_i \rb \otimes_T \HOM_T(\oB_i, T^{\lambda,r})
\rightarrow T^{\lambda,r}, 
\quad 
t \otimes f \mapsto 
f(\overline t).
\]
Then, we obtain that $\oB_i \rb \otimes_T \HOM_T(\oB_i, T^{\lambda,r})$ is isomorphic to
\[
\begin{tikzcd}[row sep = 1ex]
 & 0
 \ar{dr}
 & \\
q^2 ( B_{i}^{\tikzRBR})[2] \ar{ur} 
\ar{ur}
\ar{dr}
 & \oplus &  B_{i}^{\tikzRBR}, \\
 &  0
\ar{ur}
	  & \\
\end{tikzcd}
\]
and thus, $\overline \varepsilon'$ is the isomorphism $B_{i}^{\tikzRBR} \xrightarrow{\simeq} T^{\lambda,r}$. Therefore, $\overline \varepsilon$ identifies $ q^{-1} (T^{\lambda,r})[-1] $ with $q^{-1}(B_i^{\tikzRBR})[-1] \subset H(\bar B_{i} \Lotimes_T B_i)$ in homology.   
\end{proof}

Because the connecting morphism in \cref{prop:distinguishedBitriangle} is zero, the triangle splits and we have
\[
 \overline{B}_i \Lotimes_T B_i \cong q (T^{\lambda,r})[1] \oplus q^{-1} (T^{\lambda,r})[-1].
\]

\begin{cor}\label{cor:TLloop}
There is a natural isomorphism
\[
 \overline \B_i \circ \B_i \cong q \id [1]  \oplus q^{-1} \id [-1]. 
\]
\end{cor}

\subsection{Blob relations}

Proving the blob relations requires some preparation. 

\subsubsection{Quadratic relation}
We define recursively the following element by setting $z_0 := 0$, 
\begin{align}\label{eq:defzn}
z_1 &:= \ 
\tikzdiag{
	\draw (0,-.5) -- (0,1.5);
	\filldraw [fill=white, draw=black,rounded corners] (-.25,.3) rectangle (0.25,.8) node[midway] { $z_1$};
}
\ :=\ 
\tikzdiag{
	\draw (0,-.5) -- (0,1.5);
}\ ,
&
z_{t+2} &:=\ 
\tikzdiagh{0}{
	\draw (0,-.5) -- (0,1.5);
	\node at (.5,-.25) {\tiny $\dots$};
	\node at (.5,1.25) {\tiny $\dots$};
	\draw (1,-.5) -- (1,1.5);
	\draw (1.5,-.5) -- (1.5,1.5);
	\draw (2,-.5) -- (2,1.5);
	\filldraw [fill=white, draw=black,rounded corners] (-.25,.3) rectangle (2.25,.8) node[midway] { $z_{t+2}$};
	\draw[decoration={brace,mirror,raise=-8pt},decorate]  (-0.15,-.85) -- node {$t$} (1.15,-.85)
}
\ := \ 
\tikzdiagh{0}{
	\draw (0,-.5) -- (0,1.5);
	\node at (.5,-.25) {\tiny $\dots$};
	\node at (.5,1.25) {\tiny $\dots$};
	\draw (1,-.5) -- (1,1.5);
	\draw (1.5,-.5) -- (1.5,1) .. controls (1.5,1.25) and (2,1.25) .. (2,1.5);
	\draw (2,-.5) -- (2,1) .. controls (2,1.25) and (1.5,1.25) .. (1.5,1.5);
	\filldraw [fill=white, draw=black,rounded corners] (-.25,.3) rectangle (1.75,.8) node[midway] { $z_{t+1}$};
	\draw[decoration={brace,mirror,raise=-8pt},decorate]  (-0.15,-.85) -- node {$t$} (1.15,-.85)
}
\ + \ 
\tikzdiagh{0}{
	\draw (0,-.5) -- (0,1.5);
	\node at (.5,-.25) {\tiny $\dots$};
	\node at (.5,1.25) {\tiny $\dots$};
	\draw (1,-.5) -- (1,1.5);
	\draw (2,-.5) .. controls (2,-.25) and (1.5,-.25) .. (1.5,0) -- (1.5,1) .. controls (1.5,1.25) and (2,1.25) .. (2,1.5);
	\draw (1.5,-.5) .. controls (1.5,-.25) and (2,-.25) .. (2,0) -- (2,1) .. controls (2,1.25) and (1.5,1.25) .. (1.5,1.5) node [near end, tikzdot]{};
	\filldraw [fill=white, draw=black,rounded corners] (-.25,.3) rectangle (1.75,.8) node[midway] { $z_{t+1}$};
	\draw[decoration={brace,mirror,raise=-8pt},decorate]  (-0.15,-.85) -- node {$t$} (1.15,-.85)
}
\end{align}
 for all $t \geq 0$. Note that $z_2$ is given by a single crossing
\[
\tikzdiag{
	\draw (0,0) -- (0,1);
	\draw (.5,0) -- (.5,1);
	\filldraw [fill=white, draw=black,rounded corners] (-.25,.3) rectangle (0.75,.8) node[midway] { $z_2$};
}
\ =\ 
\tikzdiag{
	\draw (0,0) .. controls (0,.5) and (.5,.5) .. (.5,1);
	\draw (.5,0) .. controls (.5,.5) and (0,.5) .. (0,1);
}
\]
since the second term is zero in this case. 
One easily sees that $\deg_q(z_t) = 2-2t$.

Define a map of left modules
\[
\varphi_k^1 : \lambda q^2 (X_k)[1] \rightarrow X \otimes_T Y^1_k ,
\]
as $\varphi_k^1 :=   \sum_{t = 0}^{k-1} \varphi_k^{1,t}$, where each
\begin{align*}
\varphi_k^{1,t}: &\lambda q^2  (X_k) [1] \rightarrow  X \otimes_T Y^{1,t}_k \qquad \bigl(\cong \bigoplus_{\ell,\rho} \lambda q^{k-2t+1} (X 1_{1,k+\ell-1,\rho})[1] \bigr) ,
\end{align*}
is given by multiplication on the bottom by
\[
	\tikzdiag{
		\draw (.5,0)  .. controls (.5,.5) and (1.25,.5) .. (1.25,1) -- (1.25,1.5);
		\draw (1.5,0)  .. controls (1.5,.5) and (.5,.5) .. (.5,1) -- (.5,1.5) node[midway,tikzdot]{};
		\node at(1.75,.15) {\tiny $\dots$};
		\draw (2,0)  .. controls (2,.5) and (1,.5) .. (1,1) -- (1,1.5) node[midway,tikzdot]{}; 
		\draw (2.25,0)  .. controls (2.25,.5) and (1.5,.5) .. (1.5,1) -- (1.5,1.5);
		\node at(2.5,.15) {\tiny $\dots$};
		\draw (2.75,0)  .. controls (2.75,.5) and (2,.5) .. (2,1) -- (2,1.5);
		\draw[decoration={brace,mirror,raise=-8pt},decorate]  (1.35,-.35) -- node {$t$} (2.15,-.35);
		\draw[decoration={brace,raise=-8pt},decorate]  (.4,1.85) -- node {$k$} (2.125,1.85);
		\draw[stdhl] (1,0) node[below]{\small $1$}  .. controls (1,.5) and (2.75,.5) .. (2.75,1) -- (2.75,1.5);
		\draw[vstdhl] (0,0)  node[below]{\small $\lambda$} -- (0,1.5);
		\filldraw [fill=white, draw=black,rounded corners] (1.125,.9) rectangle (2.125,1.4) node[midway] { $z_{k-t}$};
	}
	\otimes \bar 1_{\ell,\rho}.
\]

Also define a map of left modules
\[
\varphi_k^0 : \bigoplus_{\ell,\rho}q^2 (T_b^{\lambda,r}1_{k,\ell,\rho} )[1]  \rightarrow X \otimes_T Y^0_k,
\]
as $\varphi_k^0 :=  {\varphi_k^0}' + \sum_{t = 0}^{k-1} \varphi_k^{0,t}$, where each
\begin{align*}
	{\varphi_k^0}' &: \bigoplus_{\ell,\rho}q^2 (T_b^{\lambda,r}1_{k,\ell,\rho} )[1] \rightarrow X \otimes_T {Y'}^0_k \qquad \bigl( \cong  \bigoplus_{\ell, \rho }\lambda^{-1} q^k  (X 1_{0,k+\ell,\rho}) \bigr), \\
	&\tikzdiagh{0}{
		\draw[stdhl] (1.5,0) node[below]{\small $1$}  -- (1.5,1);
		\draw[vstdhl] (0,0)  node[below]{\small $\lambda$} -- (0,1);
		\draw (.5,0) -- (.5,1);
		\node at (.75,.5) {\tiny $\dots$};
		\draw (1,0)  --(1,1);
		\draw[decoration={brace,mirror,raise=-8pt},decorate]  (.4,-.35) -- node {$k$} (1.1,-.35);
	}
	\otimes \bar 1_{\ell,\rho} \mapsto 
	- \sum_{t=0}^{k-1} 
	\tikzdiag{
		\draw (.5,-.5) .. controls (.5,.25) and (0,.25) .. (0,1.5) node[pos=.9,tikzdot]{}; 
		\node at(.75,-.25) {\tiny $\dots$};
		\draw (1,-.5) .. controls (1,.25) and (.5,.25) .. (.5,1.5) node[pos=.9,tikzdot]{}; 
		\draw (1.25,-.5) .. controls (1.25,.25) .. (-.5,.375) .. controls (.75,.5) .. (.75,1.5);
		\draw (1.5,-.5) .. controls (1.5,.5) and (1,.5) .. (1,1.5); 
		\node at(1.75,-.25) {\tiny $\dots$};
		\draw (2,-.5) .. controls (2,.5) and (1.5,.5) .. (1.5,1.5); 
		\filldraw [fill=white, draw=black,rounded corners] (.625,.9) rectangle (1.625,1.4) node[midway] { $z_{k-t}$};
		\draw[stdhl] (0,-.5) node[below]{\small $1$}   .. controls (0,-.25) .. (-.5,0)
			.. controls (2,.0) ..  (2,1) -- (2,1.5) ;
		\draw[fill=white, color=white] (-.6,0) circle (.1cm);
		\draw[vstdhl] (-.5,-.5) node[below]{\small $\lambda$} -- (-.5,1.5)  node[pos=.45,nail]{};
		\draw[decoration={brace,mirror,raise=-8pt},decorate]  (.4,-.85) -- node {$t$} (1.1,-.85);
		\draw[decoration={brace,raise=-8pt},decorate]  (-.1,1.85) -- node {$k$} (1.6,1.85);
	}
	\otimes \bar 1_{\ell,\rho},
\end{align*}
and where
\begin{align*}
	\varphi_k^{0,t} &: \bigoplus_{\ell,\rho}q^2 (T_b^{\lambda,r}1_{k,\ell,\rho} )[1]  \rightarrow X \otimes_T Y^{0,t}_k \qquad  \bigl(\cong \bigoplus_{\ell,\rho} \lambda q^{k-2t}  (X 1_{0,k+\ell,\rho})[1] \bigr), \\
	&\tikzdiagh{0}{
		\draw[stdhl] (1.5,0) node[below]{\small $1$}  -- (1.5,1);
		\draw[vstdhl] (0,0)  node[below]{\small $\lambda$} -- (0,1);
		\draw (.5,0) -- (.5,1);
		\node at (.75,.5) {\tiny $\dots$};
		\draw (1,0)  --(1,1);
		\draw[decoration={brace,mirror,raise=-8pt},decorate]  (.4,-.35) -- node {$k$} (1.1,-.35);
	}
	\otimes \bar 1_{\ell,\rho} \mapsto 
	\tikzdiag{
		\draw (.5,-.5) .. controls (.5,0) and (0,0) .. (0,.5) .. controls (0,.75) and (.75,.75) .. (.75,1) -- (.75, 1.5);
		\draw (.75,-.5) .. controls (.75,.25) and (0,.75) .. (0,1) -- (0,1.5) node[midway, tikzdot]{};
		\node at(1,-.35) {\tiny $\dots$};
		\draw (1.25,-.5) .. controls (1.25,.25) and (.5,.75) .. (.5,1) -- (.5,1.5) node[midway, tikzdot]{};
		\draw (1.5,-.5) .. controls (1.5,.25) and (1,.75) .. (1,1) -- (1,1.5);
		\node at(1.75,-.35) {\tiny $\dots$};
		\draw (2,-.5) .. controls (2,.25) and (1.5,.75) .. (1.5,1) -- (1.5,1.5);
		\filldraw [fill=white, draw=black,rounded corners] (.625,.9) rectangle (1.625,1.4) node[midway] { $z_{k-t}$};
		\draw[stdhl] (0,-.5) node[below]{\small $1$}   .. controls (0,-.25) .. (-.5,0)
				.. controls (2,.25) ..  (2,.5) -- (2,1.5) ;
		\draw[fill=white, color=white] (-.6,0) circle (.1cm);
		\draw[vstdhl] (-.5,-.5) node[below]{\small $\lambda$} -- (-.5,1.5)  ;
		\draw[decoration={brace,mirror,raise=-8pt},decorate]  (.65,-.85) -- node {$t$} (1.35,-.85);
		\draw[decoration={brace,raise=-8pt},decorate]  (-.1,1.85) -- node {$k$} (1.6,1.85); 
	}
	\otimes \bar 1_{\ell,\rho}.
\end{align*}

Recall that the unbraiding map (\cref{def:unbraidingmap})
\[
u : \lambda  X \hookrightarrow T_b^{\lambda,r}, 
\]
is given by 
\[
	\tikzdiagh{0}{
		\draw[stdhl] (1,0) node[below]{\small $1$} .. controls (1,.25) .. (0,.5)
				.. controls (1,.75) .. (1,1);
		\draw[fill=white, color=white] (-.1,.5) circle (.1cm);
		\draw[vstdhl] (0,0)  node[below]{\small $\lambda$} -- (0,1);
	}
	\mapsto
	\tikzdiagh{0}{
		\draw[stdhl] (1,0) node[below]{\small $1$} .. controls (1,.25) and (.25,.25)  .. (.25,.5)
				.. controls (.25,.75) and (1,.75) .. (1,1);
		\draw[vstdhl] (0,0)  node[below]{\small $\lambda$} -- (0,1);
	}
\]

\begin{lem}The diagram
\[
\begin{tikzcd}
X \otimes_T Y^1_k \ar[hookrightarrow]{r}{1 \otimes \imath_k} & X \otimes_T Y^0_k \ar[twoheadrightarrow]{r}{u \otimes \gamma_k} & \lambda^{-1} X \\
 \lambda q^2 (X_k)[1]  \ar{u}{\varphi_k^1} \ar[hookrightarrow]{r}{u} &  \bigoplus_{\ell,\rho} q^2 (T_b^{\lambda,r}1_{k,\ell,\rho})[1] \ar{u}{\varphi_k^0} \ar{r} & 0 \ar{u}
\end{tikzcd}
\]
commutes.
\end{lem}

\begin{proof}
The proof is a straightforward computation using \cref{eq:dotredstrand} and \cref{eq:crossingslidered} together with \cref{eq:nailslidedcross}. We leave the details to the reader. 
\end{proof}

Thus, there is an induced map 
\[
\varphi_k :  \cone(\lambda q^2 (X_k) [1] \xrightarrow{u} \bigoplus_{\ell,\rho} q^2 (T_b^{\lambda,r} 1_{k,\ell,\rho})[1] )[1] \rightarrow \cone(X \Lotimes_T X_k \xrightarrow{1\otimes u} \lambda^{-1} X_k),
\]
as left modules. 

\begin{thm}\label{thm:catdoublebraidphiiso}
The map 
\[
\varphi := \sum_{k = 0}^{m} (-1)^k \varphi_k  : \cone(\lambda q^2 X [1] \xrightarrow{u} q^2 T_b^{\lambda,r} [1] )[1] \rightarrow \cone(X \Lotimes_T X \xrightarrow{1\otimes u} \lambda^{-1} X),
\]
is a quasi-isomorphism. 
\end{thm}

\begin{proof}
The statement can be proven by showing that $\cone(\varphi)$ has a trivial homology, and thus is acyclic. This is done in details in~\cref{sec:proofofacyclicity}.
\end{proof}
 The next step is to prove that $\varphi$ defines a map of $A_\infty$-bimodules. Luckily, by the following proposition, we do not need to use any $A_\infty$-structure here.

\begin{prop}\label{prop:XoLXisXoX}
The map
\[
X \otimes_T \br X \xrightarrow{1 \otimes \gamma} X \otimes_T X,
\]
is a quasi-isomorphism of $A_\infty$-bimodules. 
\end{prop}

\begin{proof}
Tensoring to the left is a right-exact functor, thus \cref{lem:sesX0} gives us an exact sequence
\[
X \otimes_T Y^1_k \xrightarrow{1 \otimes \imath_k} X \otimes_T Y^0_k \xrightarrow{1 \otimes \gamma_k} X_k \rightarrow 0.
\]
It is not hard to see that $1 \otimes \imath_k$ is injective, and thus we have a short exact sequence
\[
0 \rightarrow X \otimes_T Y^1_k \xrightarrow{1 \otimes \imath} X \otimes_T Y^0_k \xrightarrow{1 \otimes \gamma_k} X_k \rightarrow 0,
\]
so that $1 \otimes \gamma_k$ is a quasi-isomorphism.
\end{proof}

Taking a mapping cone preserves quasi-isomorphisms. Thus, we have a quasi-isomorphism
\begin{equation}\label{eq:qimappingconegamma}
\cone(X \Lotimes_T X \xrightarrow{1\otimes u} \lambda^{-1} X) \xrightarrow{\simeq} \cone(X \otimes_T X \xrightarrow{1\otimes u} \lambda^{-1} X). 
\end{equation}
Let 
\[
\tilde \varphi : \cone\bigl(\lambda q^2 X [1] \xrightarrow{u} q^2 T_b^{\lambda,r}[1]\bigr)[1] \rightarrow \cone(X \otimes_T X \xrightarrow{1\otimes u} \lambda^{-1} X)
\]
be the map given by composing $\varphi$ with the quasi-isomorphism in~\cref{eq:qimappingconegamma}. We also write $\tilde \varphi^0 := (1\otimes \gamma) \circ \varphi^0$. Therefore, by \cref{lem:indAinftymap}, proving that $\varphi$ is a map of $A_\infty$-bimodules ends up being the same as proving that $\tilde \varphi^0$ is a map of dg-bimodules. 

\begin{thm}\label{thm:phiisAinfty}
The map $\varphi$ is a map of $\bZ^2$-graded $(T^{\lambda,r},0)$-$(T^{\lambda,r},0)$-$A_\infty$-bimodules. 
\end{thm}

\begin{proof}
The statement follows by proving that  $\tilde \varphi^0$ is a map of dg-bimodules, which is done in details in \cref{sec:proofofbimodulemap}.
\end{proof}

\begin{cor}
There is an exact sequence
\[
0 \rightarrow \lambda q^2 (X) [1] \xrightarrow{u} q^2 (T_b^{\lambda,r})[1] \xrightarrow{\tilde \varphi_0} X \otimes_T X  \xrightarrow{1\otimes u} \lambda^{-1} X \rightarrow 0,
\]
 of dg-bimodules.
\end{cor}

\begin{cor}\label{cor:qi-Xquadratic}
There is a quasi-isomorphism 
\[
\cone\bigl(\lambda q^2 \Xi [1] \rightarrow q^2 \id [1]\bigr)[1] \xrightarrow{\simeq} \cone( \Xi \circ \Xi \rightarrow \lambda^{-1} \Xi),
\]
of dg-functors. 
\end{cor}

\subsubsection{Inverse of $\Xi$}

Recall the notations from \cref{sec:cofX}. 

\begin{lem}\label{lem:Xklgenerated}
As a right $(T^{\lambda,r},0)$-module, $1_{1,k+\ell-1,\rho}X$ is generated by the elements
\begin{align}\label{eq:Xklgenerator}
\tikzdiagh{0}{
	\draw (1.25,0) .. controls (1.25,.5) .. (-.5,.75) .. controls (0,.875) .. (0,1);
	\draw[stdhl] (0,0) node[below]{\small $1$} .. controls (0,.1) .. (-.5,.25) .. controls (.5,.7) .. (.5,1);
	\draw[fill=white, color=white] (-.6,.25) circle (.1cm);
	\draw[vstdhl] (-.5,0) node[below]{\small $\lambda$} --(-.5,1) node[pos=.75,nail]{};
	\draw (.5,0) .. controls (.5,.5) and (.75,.5) .. (.75,1);
	\node at (.75,.15){\tiny $\dots$};
	\draw (1,0) .. controls (1,.5) and (1.25,.5)..  (1.25,1);
	\draw[decoration={brace,mirror,raise=-8pt},decorate]  (.4,-.35) -- node { \small $k-1$} (1.1,-.35);
}
\otimes \bar  1_{\ell,\rho},
&&\text{and} &&
\tikzdiagh{0}{
	\draw (.5,0) .. controls (.5,.5) and (0,.5) .. (0,1);
	\draw[stdhl] (0,0) node[below]{\small $1$} .. controls (0,.1) .. (-.5,.25) .. controls (.5,.7) .. (.5,1);
	\draw[fill=white, color=white] (-.6,.25) circle (.1cm);
	\draw[vstdhl] (-.5,0) node[below]{\small $\lambda$} --(-.5,1);
}
\otimes \bar  1_{\ell+k-1,\rho}.
\end{align}
\end{lem}

\begin{proof}
The statement can be proven using an induction on $k$, as done in details in \cref{sec:proofsofcatTLB}.
\end{proof}

\begin{lem}\label{lem:surjcircimath}
The map
\[
(- \circ \imath_k) : \HOM_T(Y^0_k, X) \twoheadrightarrow \HOM_T(Y^1_k, X),
\]
is surjective.
\end{lem}

\begin{proof}
We have
\begin{align*}
 \HOM_T(Y^0_k, X) &\cong \bigoplus_{\ell,\rho}\left( \lambda q^{-k} (1_{0,k+\ell,\rho} X) \oplus \bigoplus_{t=0}^{k-1} \lambda^{-1}q^{-(k-2t)}(1_{0,k+\ell,\rho}X)[-1]\right),
 \\
 \HOM_T(Y^1_k, X) &\cong \bigoplus_{\ell,\rho}\left( \bigoplus_{t=0}^{k-1}  \lambda^{-1} q^{-(k-2t+1)} (1_{1,k+\ell-1,\rho}X)[-1]\right).
\end{align*}
Then, the map
\[
(- \circ \imath_k) : \lambda q^{-k} (1_{0,k+\ell,\rho} X) \oplus \lambda^{-1}q^{-(k-2t)}(1_{0,k+\ell,\rho}X)[-1] \rightarrow \lambda^{-1} q^{-(k-2t+1)} (1_{1,k+\ell-1,\rho}X)[-1]
\]
is given by gluing 
\[
\left(
-\ 
\tikzdiagh{0}{
	\draw (1.25,0) .. controls (1.25,.5) .. (-.5,.75) .. controls (0,.875) .. (0,1);
	\draw[stdhl] (0,0) node[below]{\small $1$} .. controls(0,.5) and (.5,.5) .. (.5,1);
	\draw[vstdhl] (-.5,0) node[below]{\small $\lambda$} --(-.5,1) node[pos=.75,nail]{};
	\draw (.5,0) .. controls (.5,.5) and (.75,.5) .. (.75,1);
	\node at (.75,.15){\tiny $\dots$};
	\draw (1,0) .. controls (1,.5) and (1.25,.5)..  (1.25,1);
	\draw[decoration={brace,mirror,raise=-8pt},decorate]  (.4,-.35) -- node { \small $t$} (1.1,-.35);
}
\otimes \bar  1_{\ell+k-1-t,\rho},
\ 
\tikzdiagh{0}{
	\draw[vstdhl] (-.5,0) node[below]{\small $\lambda$} --(-.5,1);
	\draw (.5,0) .. controls (.5,.5) and (0,.5) .. (0,1);
	\draw[stdhl] (0,0) node[below]{\small $1$} .. controls(0,.5) and (.5,.5) .. (.5,1);
}
\otimes \bar  1_{\ell+k-1,\rho}
\right)
\]
on the top of diagrams, 
for all $0 \leq t \leq k-1$. 

Then we observe that the map $(- \circ \imath_k) :  \lambda q^{-k} (1_{0,k+\ell,\rho} X) \rightarrow  \lambda^{-1} q^{-(k-2t+1)} (1_{1,k+\ell-1,\rho}X)[-1]$ sends
\[
	\tikzdiag[xscale=.5]{
		\draw (1,0) -- (1,1);
		\node at(1.5,.5) {\tiny $\dots$};
		\draw (2,0) -- (2,1);
		\draw (2.5,0) .. controls (2.5,.5) and (3,.5) .. (3,1);
		\node at(3,.15) {\tiny $\dots$};
		\node at(3.5,.85) {\tiny $\dots$};
		\draw (3.5,0) .. controls (3.5,.5) and (4,.5) .. (4,1);
		\draw (4,0) .. controls (4,.5) and (2.5,.5) .. (2.5,1);
		\draw[stdhl] (.5,0) node[below]{\small $1$} .. controls (.5,.25) .. (-.5,.5)
				.. controls (.5,.75) .. (.5,1);
		\draw[fill=white, color=white] (-.6,.5) circle (.1cm);
		\draw[vstdhl] (-.5,0)  node[below]{\small $\lambda$} -- (-.5,1);
		\tikzbrace{1}{2}{0}{\small $s$};
		\tikzbraceop{1}{4}{1}{\small $k$};
	}
	\otimes \bar 1_{\ell,\rho},
	\ \mapsto \ 
	\begin{cases}
		0, & \text{if $ s< t$}, \\
		\tikzdiagh{0}{
			\draw (1.25,0) .. controls (1.25,.5) .. (-.5,.75) .. controls (0,.875) .. (0,1);
			\draw[stdhl] (0,0) node[below]{\small $1$} .. controls (0,.1) .. (-.5,.25) .. controls (.5,.7) .. (.5,1);
			\draw[fill=white, color=white] (-.6,.25) circle (.1cm);
			\draw[vstdhl] (-.5,0) node[below]{\small $\lambda$} --(-.5,1) node[pos=.75,nail]{};
			\draw (.5,0) .. controls (.5,.5) and (.75,.5) .. (.75,1);
			\node at (.75,.15){\tiny $\dots$};
			\draw (1,0) .. controls (1,.5) and (1.25,.5)..  (1.25,1);
			\draw[decoration={brace,mirror,raise=-8pt},decorate]  (.4,-.35) -- node { \small $k-1$} (1.1,-.35);
		}
		\otimes \bar  1_{\ell,\rho}, & \text{if $s = t$}, \\
		\tikzdiag{
			\draw (1.25,0) .. controls (1.25,.5) .. (-.5,.75) .. controls (0,.875) .. (0,1);
			\draw[stdhl] (0,0) node[below]{\small $1$} .. controls (0,.1) .. (-.5,.25) .. controls (.5,.7) .. (.5,1);
			\draw[fill=white, color=white] (-.6,.25) circle (.1cm);
			\draw[vstdhl] (-.5,0) node[below]{\small $\lambda$} --(-.5,1) node[pos=.75,nail]{};
			\draw (.5,0) .. controls (.5,.5) and (.75,.5) .. (.75,1);
			\node at (.75,.15){\tiny $\dots$};
			\draw (1,0) .. controls (1,.5) and (1.25,.5)..  (1.25,1);
			\draw (1.5,0) -- (1.5,1);
			\node at (1.75,.5){\tiny $\dots$};
			\draw (2,0) -- (2,1);
			\draw (2.25,0) .. controls (2.25,.5) and (2.5,.5) .. (2.5,1);
			\node at (2.5,.15){\tiny $\dots$};
			\draw (2.75,0) .. controls (2.75,.5) and (3,.5) .. (3,1);
			\draw (3,0) .. controls (3,.5) and (2.25,.5) .. (2.25,1);
			\draw[decoration={brace,mirror,raise=-8pt},decorate]  (.4,-.35) -- node { \small $t$} (1.1,-.35);
			\tikzbraceop{.75}{2}{1}{\small $s-1$};
		}
		\otimes \bar  1_{\ell,\rho}, &\text{if $s> t$},
	\end{cases}
\]
for all $0 \leq s \leq k-1$. 
Thus, $(- \circ \imath_k) : \HOM_T(Y^0_k, X) \twoheadrightarrow \HOM_T(Y^1_k, X)$ has a triangular form when applied to the elements above, and is surjective by \cref{lem:surjcircimath}. 
\end{proof}

\begin{prop}\label{prop:Xi-autoequiv}
The functor $\Xi :  \cD_{dg}(T^{\lambda,r},0) \rightarrow \cD_{dg}(T^{\lambda,r},0)$ is an autoequivalence, with inverse given by $\Xi^{-1} := \RHOM_T(X,-):  \cD_{dg}(T^{\lambda,r},0) \rightarrow \cD_{dg}(T^{\lambda,r},0).$
\end{prop}

\begin{proof}
By \cref{lem:surjcircimath} and \cref{prop:gammaqi}, we have
\[
\RHOM_T(X 1_\rho, X  1_{\rho'})  \cong \HOM_T(X 1_\rho, X  1_{\rho'}). 
\]
Then, we compute
\[
\gdim \HOM_T(X 1_\rho, X  1_{\rho'}) = \gdim \HOM_T(\BP_\rho, \BP_{\rho'}),
\]
using the fact that $\Xi$ decategorifies to the action of $\xi$. 
More precisely, as in \cite[\S4.7]{webster}, the bifunctor $\RHOM_T(-,-)$ decategorifies to a sesquilinear version of the Shapovalov form when restricted to a particular subcategory of $\cD_{dg}(T^{\lambda,r},0)$, and this sesquilinear form respects $(\xi w, \xi w') = (w,w')$. 
Finally, we observe that the map
\[
\HOM_T(\BP_\rho, \BP_{\rho'}) \xhookrightarrow{\id_X \otimes (-)} \HOM_T(X 1_\rho, X 1_{\rho'}),
\]
is injective, since the map $\BP_\rho \rightarrow X 1_\rho$ 
given by gluing
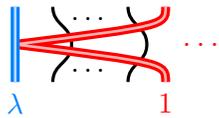
\[
	\tikzdiag{
		\draw (.5,0) .. controls (.5,.25) and (.75,.25) .. (.75,.5) .. controls (.75,.75) and (.5,.75) .. (.5,1);
		\node at (1,.15) {\small $\dots$};
		\node at (1,.85) {\small $\dots$};
		\draw (1.5,0) .. controls (1.5,.25) and (1.75,.25) .. (1.75,.5) .. controls (1.75,.75) and (1.5,.75) .. (1.5,1);
		\draw[stdhl] (2,0) node[below]{\small $1$} .. controls (2,.25) .. (0,.5)
				.. controls (2,.75) .. (2,1);
		\draw[fill=white, color=white] (-.25,.5) circle (.2cm);
		\draw[vstdhl] (0,0)  node[below]{\small $\lambda$} -- (0,1);
		\node[red] at(2.5,.5) {$\dots$}
	}
\]
on the top of diagrams 
is injective. This can be seen by composing the above map $\BP_\rho \rightarrow X 1_\rho$ with the injection $u :  X 1_\rho \rightarrow \BP_\rho$, and observing it yields an injective map. 
Therefore,  $\RHOM_T(X 1_\rho, X  1_{\rho'})  \cong 1_\rho T^{\lambda,r} 1_{\rho'}$, and $\Xi$ is an autoequivalence.  
\end{proof}

\subsubsection{Categorification of relation \cref{eq:TLBloopremov}.}

\begin{lem}\label{lem:XB1qi}
There is a quasi-isomorphism 
\[
X \Lotimes_T B_1 \cong X \otimes_T \br B_1 \xrightarrow{\simeq} X \otimes_T B_1,
\]
of $A_\infty$-bimodules.
\end{lem}

\begin{proof}
Let us write $X_{\tikzRRB} := X \otimes_T T_{1,\tikzRRB}$. Then we have
\[
X \otimes_T \br B_1 \cong 
\begin{tikzcd}[row sep = 1ex]
 & q (X_{\tikzBRR}) [1]
 \ar{dr} \ar[no head]{dr}{
	 {\tikzdiag[scale=.5]{
	 	\draw (1,0) .. controls (1,.5) and (0,.5) .. (0,1);
	 	\draw[stdhl] (0,0) .. controls (0,.5) and (1,.5) .. (1,1);
	 	\draw[stdhl] (2,0) -- (2,1);
	 }}} 
 & \\
q^2 (X_{\tikzRBR}) [2]
\ar{ur} 
 \ar[no head]{ur}{
	 {\tikzdiag[scale=.5]{
	 	\draw (0,0) .. controls (0,.5) and (1,.5) .. (1,1);
	 	\draw[stdhl] (1,0) .. controls (1,.5) and (0,.5) .. (0,1);
	 	\draw[stdhl] (2,0) -- (2,1);
	 }}} 
\ar{dr}
\ar[no head,swap]{dr}{
	 -\ {\tikzdiag[scale=.5]{
	 	\draw (1,0) .. controls (1,.5) and (0,.5) .. (0,1);
	 	\draw[stdhl] (-1,0) -- (-1,1);
	 	\draw[stdhl] (0,0) .. controls (0,.5) and (1,.5) .. (1,1);
	 }}}
 & \oplus &  
 X_{\tikzRBR} 
 \\
 &  q (X_{\tikzRRB}) [1]
 \ar{ur}
\ar[no head,swap]{ur}{
	 {\tikzdiag[scale=.5]{
	 	\draw (0,0) .. controls (0,.5) and (1,.5) .. (1,1);
	 	\draw[stdhl] (1,0) .. controls (1,.5) and (0,.5) .. (0,1);
	 	\draw[stdhl] (-1,0) -- (-1,1);
	 }}}
	  & \\
\end{tikzcd}
\]
The statement follows by observing that the first map is injective, and its image coincides with the kernel of the second one. 
\end{proof}

Our goal will be to show the following:
\begin{prop}\label{prop:blobbubble}
There is a quasi-isomorphism
\[
\lambda q (T^{\lambda,r})[1] \oplus \lambda^{-1} q^{-1} (T^{\lambda,r}) [-1] \xrightarrow{\simeq} \bar B_1 \Lotimes_T X \Lotimes_T B_1,
\]
of $A_\infty$-bimodules. 
\end{prop}

For this, we will need to understand the left $A_\infty$-action on $B_1 \rb$:
\[
B_1 \rb := 
\begin{tikzcd}[row sep = 1ex]
 & T^{\ \tikzBRR}
 \ar{dr} \ar[no head]{dr}{
	 {\tikzdiag[scale=.5,yscale=-1]{
	 	\draw (1,0) .. controls (1,.5) and (0,.5) .. (0,1);
	 	\draw[stdhl] (0,0) .. controls (0,.5) and (1,.5) .. (1,1);
	 	\draw[stdhl] (2,0) -- (2,1);
	 }}} 
 & \\
q (T^{\ \tikzRBR}) [1]
\ar{ur} 
 \ar[no head]{ur}{
	 {\tikzdiag[scale=.5,yscale=-1]{
	 	\draw (0,0) .. controls (0,.5) and (1,.5) .. (1,1);
	 	\draw[stdhl] (1,0) .. controls (1,.5) and (0,.5) .. (0,1);
	 	\draw[stdhl] (2,0) -- (2,1);
	 }}} 
\ar{dr}
\ar[no head,swap]{dr}{
	 -\ {\tikzdiag[scale=.5,yscale=-1]{
	 	\draw (1,0) .. controls (1,.5) and (0,.5) .. (0,1);
	 	\draw[stdhl] (-1,0) -- (-1,1);
	 	\draw[stdhl] (0,0) .. controls (0,.5) and (1,.5) .. (1,1);
	 }}}
 & \oplus &  
q^{-1}(T^{\ \tikzRBR})[-1].
 \\
 &  T^{\ \tikzRRB}
 \ar{ur}
\ar[no head,swap]{ur}{
	 {\tikzdiag[scale=.5,yscale=-1]{
	 	\draw (0,0) .. controls (0,.5) and (1,.5) .. (1,1);
	 	\draw[stdhl] (1,0) .. controls (1,.5) and (0,.5) .. (0,1);
	 	\draw[stdhl] (-1,0) -- (-1,1);
	 }}}
	  & \\
\end{tikzcd}
\]
We start by constructing a composition map $T \otimes B_1 \rb \rightarrow B_1 \rb$, by defining it on each generator of $T$. We extend it by first rewriting elements in $T$ as basis elements and then applying recursively the definition in terms of generating elements (so that it is well-defined). Dots and crossings act on each of the summand by simply adding the three missing vertical strands between the $\lambda$-strand and the remaining of the diagram, and gluing on top. For example in $q^{-1} (T^{\ \tikzRBR}) [-1]$, we have
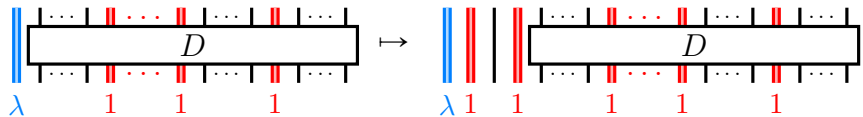
\[
\tikzdiagh[xscale=1.25]{0}{
	\draw [vstdhl] (-.25,0) node[below]{\small $\lambda$} -- (-.25,1);
	\draw (0,0) -- (0,1);
	\node at(.25,.125) {\tiny $\dots$};
	\node at(.25,.875) {\tiny $\dots$};
	\draw (.5,0) -- (.5,1);
	\draw [stdhl] (.75,0)  node[below]{\small $1$} -- (.75,1);
	\node[red] at(1.125,.125) { $\dots$};
	\node[red] at(1.125,.875) { $\dots$};
	\draw [stdhl] (1.5,0)  node[below]{\small $1$} -- (1.5,1);
	\draw (1.75,0) -- (1.75,1);
	\node at(2,.125) {\tiny $\dots$};
	\node at(2,.875) {\tiny $\dots$};
	\draw (2.25,0) -- (2.25,1);
	\draw [stdhl] (2.5,0)  node[below]{\small $1$} -- (2.5,1);
	\draw (2.75,0) -- (2.75,1);
	\node at(3,.125) {\tiny $\dots$};
	\node at(3,.875) {\tiny $\dots$};
	\draw (3.25,0) -- (3.25,1);
	\filldraw [fill=white, draw=black] (-.125,.25) rectangle (3.375,.75) node[midway] { $D$};
}
\ \mapsto \ 
\tikzdiagh[xscale=1.25]{0}{
	\draw [vstdhl] (-1,0) node[below]{\small $\lambda$} -- (-1,1);
	\draw[stdhl] (-.75,0)  node[below]{\small $1$}  -- (-.75,1);
	\draw (-.5,0)  -- (-.5,1);
	\draw[stdhl] (-.25,0)  node[below]{\small $1$}  -- (-.25,1);
	\draw (0,0) -- (0,1);
	\node at(.25,.125) {\tiny $\dots$};
	\node at(.25,.875) {\tiny $\dots$};
	\draw (.5,0) -- (.5,1);
	\draw [stdhl] (.75,0)  node[below]{\small $1$} -- (.75,1);
	\node[red] at(1.125,.125) { $\dots$};
	\node[red] at(1.125,.875) { $\dots$};
	\draw [stdhl] (1.5,0)  node[below]{\small $1$} -- (1.5,1);
	\draw (1.75,0) -- (1.75,1);
	\node at(2,.125) {\tiny $\dots$};
	\node at(2,.875) {\tiny $\dots$};
	\draw (2.25,0) -- (2.25,1);
	\draw [stdhl] (2.5,0)  node[below]{\small $1$} -- (2.5,1);
	\draw (2.75,0) -- (2.75,1);
	\node at(3,.125) {\tiny $\dots$};
	\node at(3,.875) {\tiny $\dots$};
	\draw (3.25,0) -- (3.25,1);
	\filldraw [fill=white, draw=black] (-.125,.25) rectangle (3.375,.75) node[midway] { $D$};
}
\]
The action of the nail is a bit trickier. On $q (T^{\ \tikzRBR}) [1]$ and on $q^{-1} (T^{\ \tikzRBR}) [-1]$ it acts by gluing 
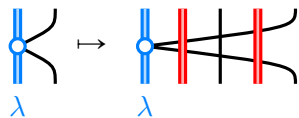
\[
\tikzdiagh{0}{
	\draw (.5,-.5) .. controls (.5,-.25)  .. 
		(0,0) .. controls (.5,.25)  .. (.5,.5);
           \draw[vstdhl] (0,-.5) node[below]{\small $\lambda$} -- (0,.5) node [midway,nail]{};
  }
  \ \mapsto \ 
\tikzdiagh{0}{
	\draw (2,-.5) .. controls (2,-.25)  .. 
		(0,0) .. controls (2,.25)  .. (2,.5);
	\draw[stdhl] (.5,-.5) -- (.5,.5);
	\draw (1,-.5) -- (1,.5);
	\draw[stdhl] (1.5,-.5) -- (1.5,.5);
	\draw[fill=white, color=white] (-.1,0) circle (.1cm);
           \draw[vstdhl] (0,-.5) node[below]{\small $\lambda$} -- (0,.5) node [midway,nail]{};
  }
\]
on the top of the diagrams. 
On $T^{\ \tikzBRR}$ it acts by
\[
\tikzdiagh{0}{
	\draw (.5,-.5) .. controls (.5,-.25)  .. 
		(0,0) .. controls (.5,.25)  .. (.5,.5);
           \draw[vstdhl] (0,-.5) node[below]{\small $\lambda$} -- (0,.5) node [midway,nail]{};
  	}
  	\ \mapsto \ 
  	\left(
	\tikzdiagh{0}{
		\draw (2,-.5) .. controls (2,-.25)  .. 
			(0,0) .. controls (2,.25)  .. (2,.5);
		\draw (.5,-.5) -- (.5,.5);
		\draw[stdhl] (1,-.5) -- (1,.5);
		\draw[stdhl] (1.5,-.5) -- (1.5,.5);
		\draw[fill=white, color=white] (-.1,0) circle (.1cm);
	           \draw[vstdhl] (0,-.5) node[below]{\small $\lambda$} -- (0,.5) node [midway,nail]{};
	  }
	  \ - \ 
	  \tikzdiagh{0}{
		\draw (.5,-.5) .. controls (.5,-.25)  .. 
			(0,0) .. controls (.5,.25)  .. (.5,.5);
		\draw (2,-.5) -- (2,.5);
		\draw[stdhl] (1,-.5) -- (1,.5);
		\draw[stdhl] (1.5,-.5) -- (1.5,.5);
		\draw[fill=white, color=white] (-.1,0) circle (.1cm);
	           \draw[vstdhl] (0,-.5) node[below]{\small $\lambda$} -- (0,.5) node [midway,nail]{};
	  }
  	\ ,\ 
  	\tikzdiagh{0}{
		\draw (.5,-.5) .. controls (.5,-.25)  .. 
			(0,0) .. controls (2,.25)  .. (2,.5);
		\draw[stdhl] (1,-.5) .. controls (1,0) and (.5,0) .. (.5,.5);
		\draw[stdhl] (1.5,-.5) .. controls (1.5,0) and (1,0) .. (1,.5);
		\draw (2,-.5) .. controls (2,0) and (1.5,0) .. (1.5,.5);
		\draw[fill=white, color=white] (-.1,0) circle (.1cm);
	           \draw[vstdhl] (0,-.5) node[below]{\small $\lambda$} -- (0,.5) node [midway,nail]{};
	  }
  	\right)
  	\ \in T^{\ \tikzBRR} \oplus T^{\ \tikzRRB},
\]
and on $T^{\ \tikzRRB}$ by
\[
\tikzdiagh{0}{
	\draw (.5,-.5) .. controls (.5,-.25)  .. 
		(0,0) .. controls (.5,.25)  .. (.5,.5);
           \draw[vstdhl] (0,-.5) node[below]{\small $\lambda$} -- (0,.5) node [midway,nail]{};
  	}
  	\ \mapsto \ 
  	\left(
  	\tikzdiagh[yscale=-1]{0}{
		\draw (.5,-.5) .. controls (.5,-.25)  .. 
			(0,0) .. controls (2,.25)  .. (2,.5);
		\draw[stdhl] (1,-.5) .. controls (1,0) and (.5,0) .. (.5,.5);
		\draw[stdhl] (1.5,-.5) .. controls (1.5,0) and (1,0) .. (1,.5);
		\draw (2,-.5) .. controls (2,0) and (1.5,0) .. (1.5,.5);
		\draw[fill=white, color=white] (-.1,0) circle (.1cm);
	           \draw[vstdhl] (0,-.5)  -- (0,.5) node [midway,nail]{} node[below]{\small $\lambda$};
	  }
  	\ ,\ 
	\tikzdiagh{0}{
		\draw (2,-.5) .. controls (2,-.25)  .. 
			(0,0) .. controls (2,.25)  .. (2,.5);
		\draw[stdhl] (.5,-.5) -- (.5,.5);
		\draw[stdhl] (1,-.5) -- (1,.5);
		\draw (1.5,-.5) -- (1.5,.5);
		\draw[fill=white, color=white] (-.1,0) circle (.1cm);
	           \draw[vstdhl] (0,-.5) node[below]{\small $\lambda$} -- (0,.5) node [midway,nail]{};
	  }
  	\right)
  	\ \in T^{\ \tikzBRR} \oplus T^{\ \tikzRRB}.
\]
One can easily verify that this respects the differential in $B_1 \rb$. The higher multiplication maps $T \otimes B_1 \rb \otimes T \rightarrow B_1 \rb$ and $T \otimes T \otimes B_1 \rb$ compute the defect of the map $T \otimes B_1 \rb \rightarrow B_1 \rb$ for being a left $T$-action. 
Concretely, it means that we can compute these higher multiplication maps by looking how both side of each defining relation of $T$ act on $B_1 \rb$. 

For example, the relation
\[
	\tikzdiagh{0}{
		\draw (.5,-.5) .. controls (.5,-.25)  ..  
			(0,0) node[midway, tikzdot]{}
			 .. controls (.5,.25)  .. (.5,.5);
	           \draw[vstdhl] (0,-.5) node[below]{\small $\lambda$} -- (0,.5) node [midway,nail]{};
  	}
  	\ = \ 
	\tikzdiagh{0}{
		\draw (.5,-.5) .. controls (.5,-.25)  .. 
			(0,0) 
			.. controls (.5,.25)  .. (.5,.5) node[midway, tikzdot]{};
	           \draw[vstdhl] (0,-.5) node[below]{\small $\lambda$} -- (0,.5) node [midway,nail]{};
  	}
\]
is respected on $q^{-1} (T^{\ \tikzRBR}) [-1]$ up to adding the elements appearing in the right of the following equation:
\[
\tikzdiagh{0}{
	\draw (2,-.5) .. controls (2,-.25)  .. 
		(0,0)  node[pos=.3,tikzdot]{} .. controls (2,.25)  .. (2,.5); 
	\draw[stdhl] (.5,-.5) -- (.5,.5);
	\draw (1,-.5) -- (1,.5);
	\draw[stdhl] (1.5,-.5) -- (1.5,.5);
	\draw[fill=white, color=white] (-.1,0) circle (.1cm);
           \draw[vstdhl] (0,-.5) node[below]{\small $\lambda$} -- (0,.5) node [midway,nail]{};
  }
\ = \ 
\tikzdiagh{0}{
	\draw (2,-.5) .. controls (2,-.25)  .. 
		(0,0) .. controls (2,.25)  .. (2,.5) node[pos=.7,tikzdot]{}; 
	\draw[stdhl] (.5,-.5) -- (.5,.5);
	\draw (1,-.5) -- (1,.5);
	\draw[stdhl] (1.5,-.5) -- (1.5,.5);
	\draw[fill=white, color=white] (-.1,0) circle (.1cm);
           \draw[vstdhl] (0,-.5) node[below]{\small $\lambda$} -- (0,.5) node [midway,nail]{};
  }
  \ + \ 
	  \tikzdiagh{0}{
		\draw (2,-.5) .. controls (2,-.125)  .. 
			(0,.25) .. controls (1,.375)  .. (1,.5);
		\draw (1,-.5) .. controls (1,0) and (2,0) .. (2,.5);
		\draw[stdhl] (.5,-.5) -- (.5,.5);
		\draw[stdhl] (1.5,-.5) .. controls (1.5,-.25) and (1,-.25) .. (1,0) .. controls (1,.25) and (1.5,.25) .. (1.5,.5);
		\draw[fill=white, color=white] (-.1,.25) circle (.1cm);
	           \draw[vstdhl] (0,-.5) node[below]{\small $\lambda$} -- (0,.5) node [pos=.75,nail]{};
	  }
 \ - \ 
	  \tikzdiagh[yscale=-1]{0}{
		\draw (2,-.5) .. controls (2,-.125)  .. 
			(0,.25) .. controls (1,.375)  .. (1,.5);
		\draw (1,-.5) .. controls (1,0) and (2,0) .. (2,.5);
		\draw[stdhl] (.5,-.5) -- (.5,.5);
		\draw[stdhl] (1.5,-.5) .. controls (1.5,-.25) and (1,-.25) .. (1,0) .. controls (1,.25) and (1.5,.25) .. (1.5,.5);
		\draw[fill=white, color=white] (-.1,.25) circle (.1cm);
	           \draw[vstdhl] (0,-.5)  -- (0,.5) node [pos=.75,nail]{} node[below]{\small $\lambda$};
	  }
\]
so that the higher multiplication map $T \otimes T \otimes q^{-1} (T^{\ \tikzRBR}) [-1] \rightarrow B_1 \rb$ gives
\[
	\tikzdiagh{0}{
		\draw (.5,-.5) .. controls (.5,-.25)  ..  
			(0,0) 
			 .. controls (.5,.25)  .. (.5,.5);
	           \draw[vstdhl] (0,-.5) node[below]{\small $\lambda$} -- (0,.5) node [midway,nail]{};
  	}
\ \otimes \ 
	\tikzdiagh{0}{
		\draw (.5,-.5) --  
			(.5,.5) node[midway, tikzdot]{};
	           \draw[vstdhl] (0,-.5) node[below]{\small $\lambda$} -- (0,.5);
  	}
\  \otimes \  1  \ \mapsto \ 
\left(
	  \tikzdiagh{0}{
		\draw (2,-.5) .. controls (2,-.125)  .. 
			(0,.25) .. controls (.5,.375)  .. (.5,.5);
		\draw (1,-.5) .. controls (1,0) and (2,0) .. (2,.5);
		\draw[stdhl] (.5,-.5) .. controls (.5,0) and (1,0) .. (1,.5);
		\draw[stdhl] (1.5,-.5) .. controls (1.5,-.25) and (1.125,-.25) .. (1.125,0) .. controls (1.125,.25) and (1.5,.25) .. (1.5,.5);
		\draw[fill=white, color=white] (-.1,.25) circle (.1cm);
	           \draw[vstdhl] (0,-.5) node[below]{\small $\lambda$} -- (0,.5) node [pos=.75,nail]{};
	  }
\ , - \ 
	  \tikzdiagh[yscale=-1]{0}{
		\draw (2,-.5) .. controls (2,-.125)  .. 
			(0,.25) .. controls (1,.375)  .. (1,.5);
		\draw (1.5,-.5) .. controls (1.5,0) and (2,0) .. (2,.5);
		\draw[stdhl] (.5,-.5) -- (.5,.5);
		\draw[stdhl] (1,-.5) .. controls (1,0) and (1.5,0) .. (1.5,.5);
		\draw[fill=white, color=white] (-.1,.25) circle (.1cm);
	           \draw[vstdhl] (0,-.5)  -- (0,.5) node [pos=.75,nail]{} node[below]{\small $\lambda$};
	  }
\right)
  \ \in T^{\ \tikzBRR} \oplus T^{\ \tikzRRB}.
\]
Note that it means the higher maps only involve elements coming from \cref{eq:relNail}. Also, one can easily verify that the other two relations in \cref{eq:relNail} are already respected for the multiplication map $T \otimes  q^{-1} (T^{\ \tikzRBR}) [-1] \rightarrow B_1 \rb$, so that our computation above completely determine $T \otimes T \otimes q^{-1} (T^{\ \tikzRBR}) [-1] \rightarrow B_1 \rb$. There is a similar higher multiplication map $T \otimes q^{-1} (T^{\ \tikzRBR}) [-1] \otimes T \rightarrow B_1 \rb$, which is non-trivial in the case 
\[
\tikzdiagh{0}{
		\draw (.5,-.5) .. controls (.5,-.25)  ..  
			(0,0) 
			 .. controls (.5,.25)  .. (.5,.5);
	           \draw[vstdhl] (0,-.5) node[below]{\small $\lambda$} -- (0,.5) node [midway,nail]{};
  	}
\ \otimes \  1  \ \otimes  \ 
	\tikzdiagh{0}{
		\draw (.5,-.5) --  
			(.5,.5) node[midway, tikzdot]{};
	           \draw[vstdhl] (0,-.5) node[below]{\small $\lambda$} -- (0,.5);
  	}
\]
for similar reasons. 
We will not need to compute the other higher composition maps. 

\begin{proof}[Proof of \cref{prop:blobbubble}]
Tensoring $B_1 \rb$ with $X \otimes_T B_1$ gives a complex 
where the elements are locally of the form
\[
\begin{tikzcd}
&0 \ar{dr}&
\\
q\left(
\tikzdiag[scale=.75]{
	\draw (1,.5) -- (1,2);
	\draw[stdhl] (.5,1) .. controls (.5,1.25) .. (0,1.5)
			.. controls (.5,1.75) .. (.5,2);
	\draw [stdhl] (.5,1)  .. controls (.5,.25) and (1.5,.25) .. (1.5,1)--(1.5,2);
	\draw[fill=white, color=white] (-.1,1.5) circle (.1cm);
	\draw[vstdhl] (0,.5)  node[below]{\small $\lambda$} -- (0,2);
}
\right)[1]
\ar{ur}
\ar{dr}
\ar[no head,swap]{dr}{
	 -\ {\tikzdiag[scale=.5,yscale=-1]{
	 	\draw (1,0) .. controls (1,.5) and (0,.5) .. (0,1);
	 	\draw[stdhl] (-1,0) -- (-1,1);
	 	\draw[stdhl] (0,0) .. controls (0,.5) and (1,.5) .. (1,1);
	 }}}
&&
q^{-1}\left(
\tikzdiag[scale=.75]{
	\draw (1,.5) -- (1,2);
	\draw[stdhl] (.5,1) .. controls (.5,1.25) .. (0,1.5)
			.. controls (.5,1.75) .. (.5,2);
	\draw [stdhl] (.5,1)  .. controls (.5,.25) and (1.5,.25) .. (1.5,1)--(1.5,2);
	\draw[fill=white, color=white] (-.1,1.5) circle (.1cm);
	\draw[vstdhl] (0,.5)  node[below]{\small $\lambda$} -- (0,2);
}
\right)[-1]
\\
&
{
\tikzdiagh[yscale=-1,scale=.75]{0}{
	\draw[stdhl] (1.5,0) -- (1.5,1)  ;
	\draw (.5,0)	.. controls (.5,.25) and (1,.25) .. (1,1) -- (1,1.5);
	\draw[stdhl] (1,0).. controls (1,.25) .. (0,.5)
			.. controls (.5,.75) .. (.5,1)   ;
	\draw [stdhl] (.5,1)  .. controls (.5,1.75) and (1.5,1.75) .. (1.5,1);
	\draw[fill=white, color=white] (-.1,.5) circle (.1cm);
	\draw[vstdhl] (0,0) -- (0,1.5)   node[below]{\small $\lambda$};
}\ ,
\tikzdiagh[yscale=-1,scale=.75]{0}{
	\draw[stdhl] (1.5,0) -- (1.5,1)  ;
	\draw (.5,0)	.. controls (.5,.125) .. (0,.25)
		.. controls (1,.5) .. (1,1) -- (1,1.5);
	\draw[stdhl] (1,0).. controls (1,.25) .. (0,.5)
			.. controls (.5,.75) .. (.5,1)  ;
	\draw [stdhl] (.5,1)  .. controls (.5,1.75) and (1.5,1.75) .. (1.5,1);
	\draw[fill=white, color=white] (-.1,.5) circle (.1cm);
	\draw[vstdhl] (0,0) -- (0,1)    node[nail, pos=.25]{} -- (0,1.5)node[below]{\small $\lambda$} ;
}
}
\ar[swap]{ur}{0}
&
\end{tikzcd}
\]
which, after eliminating the acyclic subcomplex, yields
\[
\tikzdiagh[yscale=-1,scale=.75]{0}{
	\draw[stdhl] (1.5,0) -- (1.5,1)  ;
	\draw (.5,0)	.. controls (.5,.125) .. (0,.25)
		.. controls (1,.5) .. (1,1) -- (1,1.5);
	\draw[stdhl] (1,0).. controls (1,.25) .. (0,.5)
			.. controls (.5,.75) .. (.5,1)  ;
	\draw [stdhl] (.5,1)  .. controls (.5,1.75) and (1.5,1.75) .. (1.5,1);
	\draw[fill=white, color=white] (-.1,.5) circle (.1cm);
	\draw[vstdhl] (0,0) -- (0,1)    node[nail, pos=.25]{} -- (0,1.5)node[below]{\small $\lambda$} ;
}
\ \oplus \ 
q^{-1}\left(
\tikzdiag[scale=.75]{
	\draw (1,.5) -- (1,2);
	\draw[stdhl] (.5,1) .. controls (.5,1.25) .. (0,1.5)
			.. controls (.5,1.75) .. (.5,2);
	\draw [stdhl] (.5,1)  .. controls (.5,.25) and (1.5,.25) .. (1.5,1)--(1.5,2);
	\draw[fill=white, color=white] (-.1,1.5) circle (.1cm);
	\draw[vstdhl] (0,.5)  node[below]{\small $\lambda$} -- (0,2);
}
\right)[-1]
\]
All higher multiplications maps vanish: except for  $T \otimes T \otimes (B_1\rb \otimes_T X \otimes B_1) \rightarrow (B_1\rb \otimes_T X \otimes B_1)$ and $T \otimes ( B_1\rb \otimes_T X \otimes B_1) \otimes T \rightarrow (B_1\rb \otimes_T X \otimes B_1)$, all of these are zero for degree reasons, and the remaining two are zero by the calculations above. 
Therefore, what remains is isomorphic to $\lambda q (T^{\lambda,r})[1] \oplus \lambda^{-1} q^{-1} (T^{\lambda,r}) [-1]$, as dg-bimodules. 
We conclude by applying \cref{lem:XB1qi}. 
\end{proof}

\begin{cor}\label{cor:qi-bubbleremv}
There is a quasi-isomorphism 
\[
\lambda q (\id)[1] \oplus \lambda^{-1} q^{-1} (\id) [-1] \xrightarrow{\simeq} \bar \B_1 \circ \Xi \circ \B_1,
\]
of dg-functors.
\end{cor}

\subsection{The blob 2-category}\label{sec:blob2cat}
In this section, we suppose $\Bbbk$ is a field. 

Let $\mathfrak B(r,r')$ be the subcategory of dg-functors $\cD_{dg}(T^{\lambda,r}, 0) \rightarrow \cD_{dg}(T^{\lambda,r'}, 0)$ c.b.l.f. generated by all compositions of $\Xi, \B_i$ and $\bar \B_i$, and identity functor whenever $r = r'$, where c.b.l.f. generated means it is given by certain (potentially infinite) iterated extensions of these objects (see \cref{def:cblfgenerated} for a precise definition). 
As explained in \cref{sec:cblfdgfunctors}, there is an induced morphism
\[
 {}_\bQ \bKO^\Delta(\mathfrak B(r,r'))  \xrightarrow{\eqref{eq:K0homK0}}  \Hom_{\bQ\pp{q,\lambda}}(  {}_\bQ\bKO^\Delta(T^{\lambda,r}, 0),   {}_\bQ\bKO^\Delta(T^{\lambda,r'}, 0)),
\]
sending the equivalence class of an exact dg-functor to its induced map on the asymptotic Grothendieck groups of its source and target (this is similar to the fact that an exact functor between triangulated categories induces a map on their triangulated Grothendieck groups).

Recall the blob category $\cB$, but consider it as defined over $\bQ\pp{q,\lambda}$ instead of $\bQ(q,\lambda)$.  

\begin{thm}
There is an isomorphism
\[
\Hom_\cB(r,r') \cong {}_\bQ \bKO^\Delta(\mathfrak B(r,r')).
\]
\end{thm}

\begin{proof}
Comparing the action of $\cB$ on $M \otimes V^r$ from \cref{sec:blobalgebra} with the cofibrant replacement $\br X$ from \cref{sec:cofX}, and $\br B_i$ and $\br \bar B_i$ from  \cref{sec:cofBi}, we deduce there is a commutative diagram
\[
\begin{tikzcd}
\Hom_\cB(r,r') 
\ar[hookrightarrow]{r}{\eqref{eq:BcongMV}}
\ar[twoheadrightarrow,swap]{d}{f}
&
 \Hom_{\bQ\pp{q,\lambda}}(M \otimes V^r, M \otimes V^{r'})
\\
 {}_\bQ \bKO^\Delta(\mathfrak B(r,r'))  \ar[swap]{r}{\eqref{eq:K0homK0}}& \Hom_{\bQ\pp{q,\lambda}}(  {}_\bQ\bKO^\Delta(T^{\lambda,r}, 0),   {}_\bQ\bKO^\Delta(T^{\lambda,r'}, 0) ) \ar[sloped]{u}{\simeq}
\end{tikzcd}
\]
where the arrow $f$ is the obvious surjective one, sending $\xi$ to $[\Xi]$, and cup/caps to $[\B_i]$/$[\bar \B_i]$. 
Because the diagram commutes and using \cref{thm:BcongMV}, we deduce that $f$ is injective, and thus it is an isomorphism. 
\end{proof}

In particular, if we write $\mathfrak B_r := \mathfrak B(r,r)$, then we have:

\begin{cor}\label{cor:twoblobalg}
There is an isomorphism of $\bQ\pp{q,\lambda}$-algebras
\[
{}_\bQ\bKO^\Delta(\mathfrak{B}_r) \cong \cB_r.
\]
\end{cor}

By Faonte~\cite{faonte}, we know that $A_\infty$-categories form an $(\infty,2)$-category, where the hom-spaces are given by  Lurie's dg-nerves \cite{lurie} of the dg-categories of $A_\infty$-functors (or equivalently quasi-functors, see \cref{sec:dgfunctors}). Thus, we can define the following:

\begin{defn}
Let  $\mathfrak{B}$  be the $(\infty,2)$-category defined by
\begin{itemize}
\item objects are non-negative integers $r \in \bN$ (corresponding to $\cD_{dg}(T^{\lambda,r},0)$);
\item $\Hom_{\mathfrak{B}}(r,r')$ is Lurie's dg-nerve of the dg-category $\mathfrak B(r,r')$.
\end{itemize}
We refer to $\mathfrak{B}$  as the \emph{blob 2-category}.
\end{defn}

We define ${}_\bQ\bKO^\Delta(\mathfrak{B})$ to be the category with objects being non-negative integers $r \in \bN$ and homs are given by asymptotic Grothendieck groups of the homotopy categories of $\Hom_{\mathfrak{B}}(r,r')$. These homs are equivalent to  ${}_\bQ \bKO^\Delta(\mathfrak B(r,r'))$. 

\begin{cor}\label{cor:twoblob}
There is an equivalence of categories
\[
{}_\bQ\bKO^\Delta(\mathfrak{B}) \cong \cB.
\]
\end{cor}



\section{Variants and generalizations}\label{sec:losends}


\subsection{Zigzag algebras}\label{sec:zigazag}

In~\cite[\S4]{qi-sussan} it was proven that for $\g=\slt$ the KLRW algebra $T_1^{1,\dotsc,1}$ with $r$ red strands and only one black strand is isomorphic to 
a preprojective algebra $A^!_r$  of type $A$. 
It is a Koszul algebra, whose quadratic dual was used by Khovanov--Seidel in~\cite{khovanov-seidel} to construct a categorical braid group action. 

Let $\Bbbk$ be a field of any characteristic and let $\cQ_r$ be the following quiver
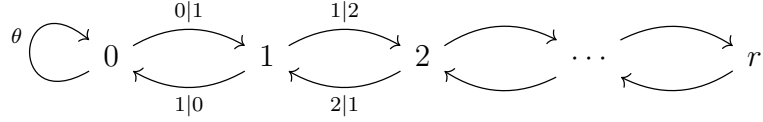
\[
\begin{tikzcd}[column sep = 1.5cm]
0 \arrow[r,bend left,"0\vert 1"]   \arrow[out=210,in=140,loop,"\theta"]
  &		 
1 \arrow[l,bend left,"1\vert0"] \arrow[r,bend left, "1\vert 2"] 
&
2 \arrow[l,bend left,"2\vert 1"]  \arrow[r,bend left] 
&
\dotsm \arrow[r,bend left] \arrow[l,bend left] 
&
r \arrow[l,bend left]
\end{tikzcd}
\]
and $\Bbbk\cQ_r$ its path algebra. 
We endow $\Bbbk\cQ_r$ with a $\bZ \times \bZ^2$-grading by declaring that
\[
\deg( i\vert i\pm 1) := (0,1,0) , \mspace{40mu} \deg(\theta) := (1,0,2) .
\]
We consider the first grading as homological, and the second and third gradings are called the $q$-grading and the $\lambda$-grading respectively. 
We denote the straight path that starts on $i_1$ and ends at $i_n$ by $(i_1\vert i_2\vert\dotsc \vert i_{n-1}\vert i_n)$ and
the constant path on $i$ by $(i)$. The set $\{(0),\dotsm,(r)\}$ forms a complete set of primitive orthogonal idempotents in $\Bbbk\cQ_r$.
\begin{defn}
Let $\cA_r^!$ be algebra given by 
   the quotient of  the path algebra $\Bbbk\cQ_r$ by the relations 
  \begin{align*}
    (i\vert i-1\vert i) &= (i\vert i+1\vert i), \rlap{\hspace{6ex}for $i>0$,}\\
   \theta(0\vert 1\vert 0) &= (0\vert 1\vert 0)\theta ,   \\
    \theta ^2 &= 0.
   \end{align*}
\end{defn}

We usually consider $\cA_r^!$ as a dg-algebra $(\cA_r^!, 0)$ with zero differential. 
We can also consider a version of  $\cA_r^!$ with a non-trivial differential $d$ given by 
\[
d(X) := \begin{cases}
(0\vert 1\vert 0), & \text{if }X=\theta,
\\
  0 , & \text{otherwise},
\end{cases}
\]
of which one easily checks that it is well-defined.

\begin{prop} \label{prop:isozigzag}
The $\bZ\times\bZ^2$ algebra $\cA_r^!$ is isomorphic to the $\bZ \times \bZ^2$ algebra $T_1^{\lambda, r}$ in $r$ red strands and $1$ black strand
by the map sending
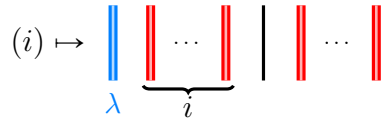
\begin{align*}\allowdisplaybreaks
  (i) &\mapsto \tikzdiagh{0}{
	\draw[vstdhl] (0,0) node[below]{\small $\lambda$} --(0,1);
	\draw[stdhl] (.5,0)  --(.5,1);
	\node at(1,.5) {\tiny$\dots$};
	\draw[stdhl] (1.5,0)  --(1.5,1);
	\draw (2,0) -- (2,1);
	\draw[stdhl] (2.5,0)  --(2.5,1);
	\node at(3,.5) {\tiny$\dots$};
	\draw[stdhl] (3.5,0)  --(3.5,1);
	\tikzbrace{.5}{1.5}{0}{$i$}
  }
\intertext{where the black strand  comes right after the $i$th red, and}
  (i-1\vert i) &\mapsto \tikzdiagh{0}{
	\draw[vstdhl] (0,0) node[below]{\small $\lambda$} --(0,1);
	\draw[stdhl] (.5,0)  --(.5,1);
	\node at(1,.5) {\tiny$\dots$};
	\draw (2,0)  ..controls (2,.3) and (1.5,.7) ..  (1.5,1);
	\draw[stdhl] (1.5,0)  ..controls (1.5,.3) and (2,.7) .. (2,1);
	\draw[stdhl] (2.5,0)  --(2.5,1);
	\node at(3,.5) {\tiny$\dots$};
	\draw[stdhl] (3.5,0)  --(3.5,1); 
	\tikzbrace{.5}{1.5}{0}{$i$}
}\\
(i+1\vert i) &\mapsto \tikzdiagh{0}{
	\draw[vstdhl] (0,0) node[below]{\small $\lambda$} --(0,1);
	\draw[stdhl] (.5,0)  --(.5,1);
	\node at(1,.5) {\tiny$\dots$};
	\draw[stdhl] (1.5,0)  --(1.5,1);
	\draw (2,0)  ..controls (2,.3) and (2.5,.7) .. (2.5,1);
	\draw[stdhl] (2.5,0)  ..controls (2.5,.3) and (2,.7) .. (2,1);
	\node at(3,.5) {\tiny$\dots$};
	\draw[stdhl] (3.5,0)  --(3.5,1); 
	\tikzbrace{.5}{1.5}{0}{$i$}
}\\
  \theta &\mapsto  \tikzdiagh{0}{
	\draw (.5,0) .. controls (.5,.25) .. (0,.5) .. controls (.5,.75) .. (.5,1);
        \draw[vstdhl] (0,0) node[below]{\small $\lambda$} -- (0,1) node [midway,nail]{};
	\draw[stdhl] (1,0)  --(1,1);
	\node at(1.75,.5) {\tiny$\dots$};
	\draw[stdhl] (2.5,0)  --(2.5,1);
	\node at(3,.5) {\tiny$\dots$};
	\draw[stdhl] (3.5,0)  --(3.5,1); 
} 
\end{align*}
Furthermore, the isomorphism upgrades to isomorphisms of dg-algebras $(\cA_r^!, 0) \cong (T_1^{\lambda, r},0)$ and $(\cA_r^!, d) \cong (T_1^{\lambda, r},d_{1})$.
\end{prop}

\begin{proof}
First, one can show by a straightforward computation that the map defined above respects all defining relations of $\cA_r^!$. 
Moreover, by turning any dot in $T_1^{\lambda, r}$ to a double crossings using~\cref{eq:redR2}, it is not hard to construct an inverse of the map defined above. 
We leave the details to the reader.
\end{proof}

\begin{cor}
The homology of $(\cA_r^!,d)$ is concentrated in homological degree $0$ and is isomorphic to the preprojective algebra $A_r^!$. 
\end{cor}

Moreover, by \cref{prop:isozigzag}, the results in \cref{sec:catTLB} can be pulled to the derived category of $\bZ^2$-graded $(\cA_r^!,0)$-modules, endowing $\cD_{dg}(\cA_r^!,0) \cong \cD_{dg}(T_1^{\lambda, r}, 0)$ with a categorical action of $\cB_r$.


\subsection{Dg-enhanced KLRW algebras: the general case}\label{sec:dgWebstergeneral}

Fix a symmetrizable Kac--Moody algebra $\g$ with set of simple roots $I$ and dominant integral weights
$\underline{\mu}:=(\mu_1,\dotsc ,\mu_d)$.

\subsubsection{Dg-enhanced KLRW algebras: $\g$ symmetrizable}

Recall that the KLRW algebra~\cite[\S4]{webster} $T_b^{\underline{\mu}}(\g)$ on $b$ strands is the diagrammatic $\Bbbk$-algebra generated by braid-like diagrams on $b$ black strands and $r$ red strands. 
Red strands are labeled from left to right by $\mu_1, \dots, \mu_r$ and cannot intersect each other,
while black strands are labeled by simple roots and can intersect red strands transversally, they can intersect transversally among themselves and can carry dots. 
Diagrams are taken up to braid-like planar isotopy and satisfy the following local relations:
\begin{itemize}
\item the KLR local relations (2.5a)-(2.5g) in~\cite[Definition~2.4]{webster};
\item the local black/red relations~\eqref{eq:gdotredstrand}-\eqref{eq:gredR3} for all $\nu\in\underline{\mu}$ and for all $\alpha_j$, $\alpha_k\in I$, given below;
\item a black strand in the leftmost region is $0$.
\end{itemize}
\begin{align}
\tikzdiagh{0}{
	\draw (1,0)node[below]{\small $\alpha_j$} ..controls (1,.5) and (0,.5) .. (0,1) node [near end,tikzdot]{};
	\draw[stdhl] (0,0) node[below]{\small $\nu$} ..controls (0,.5) and (1,.5) .. (1,1);
}
\ &= \ 
\tikzdiagh{0}{
	\draw (1,0)node[below]{\small $\alpha_j$} ..controls (1,.5) and (0,.5) .. (0,1) node [near start,tikzdot]{};
	\draw[stdhl] (0,0) node[below]{\small $\nu$} ..controls (0,.5) and (1,.5) .. (1,1);
}
&
\tikzdiagh{0}{
	\draw (0,0)node[below]{\small $\alpha_j$} ..controls (0,.5) and (1,.5) .. (1,1) node [near start,tikzdot]{};
	\draw[stdhl] (1,0) node[below]{\small $\nu$} ..controls (1,.5) and (0,.5) .. (0,1);
}
\ &= \ 
\tikzdiagh{0}{
	\draw (0,0)node[below]{\small $\alpha_j$} ..controls (0,.5) and (1,.5) .. (1,1) node [near end,tikzdot]{};
	\draw[stdhl] (1,0) node[below]{\small $\nu$} ..controls (1,.5) and (0,.5) .. (0,1);
} 
\label{eq:gdotredstrand}
\\
\tikzdiagh{0}{
	\draw (1,0)node[below]{\small $\alpha_j$} ..controls (1,.25) and (0,.25) .. (0,.5)..controls (0,.75) and (1,.75) .. (1,1)  ;
	\draw[stdhl] (0,0) node[below]{\small $\nu$} ..controls (0,.25) and (1,.25) .. (1,.5) ..controls (1,.75) and (0,.75) .. (0,1)  ;
} 
\ &= \ 
\tikzdiagh{0}{
	\draw[stdhl] (0,0) node[below]{\small $\nu$} -- (0,1)  ;
	\draw (1,0)node[below]{\small $\alpha_j$} -- (1,1)  node[midway,tikzdot]{}  node[midway,xshift=1.75ex,yshift=.75ex]{\small $\nu_j$} ;
} 
&
\tikzdiagh{0}{
	\draw (0,0)node[below]{\small $\alpha_j$} ..controls (0,.25) and (1,.25) .. (1,.5) ..controls (1,.75) and (0,.75) .. (0,1)  ;
	\draw[stdhl] (1,0) node[below]{\small $\nu$} ..controls (1,.25) and (0,.25) .. (0,.5)..controls (0,.75) and (1,.75) .. (1,1)  ;
} 
\ &= \ 
\tikzdiagh{0}{
	\draw (0,0)node[below]{\small $\alpha_j$} -- (0,1)  node[midway,tikzdot]{}   node[midway,xshift=1.75ex,yshift=.75ex]{\small $\nu_j$} ;
	\draw[stdhl] (1,0) node[below]{\small $\nu$} -- (1,1)  ;
} 
\label{eq:gredR2}
\\
\tikzdiagh{0}{
	\draw  (0.5,0)node[below]{\small $\alpha_j$} .. controls (0.5,0.25) and (0, 0.25) ..  (0,0.5)
		 	  .. controls (0,0.75) and (0.5, 0.75) ..  (0.5,1);
	\draw (1,0)node[below]{\small $\alpha_k$}  .. controls (1,0.5) and (0, 0.75) ..  (0,1);
	\draw  [stdhl] (0,0) node[below]{\small $\nu$} .. controls (0,0.25) and (1, 0.5) ..  (1,1);
} 
\ &= \ 
\tikzdiagh[xscale=-1]{0}{
	\draw  (0,0)node[below]{\small $\alpha_k$} .. controls (0,0.25) and (1, 0.5) ..  (1,1);
	\draw  (0.5,0)node[below]{\small $\alpha_j$} .. controls (0.5,0.25) and (0, 0.25) ..  (0,0.5)
		 	  .. controls (0,0.75) and (0.5, 0.75) ..  (0.5,1);
	\draw [stdhl] (1,0) node[below]{\small $\nu$} .. controls (1,0.5) and (0, 0.75) ..  (0,1);
} 
&
\tikzdiagh{0}{
	\draw  (0,0)node[below]{\small $\alpha_j$} .. controls (0,0.25) and (1, 0.5) ..  (1,1);
	\draw  (0.5,0)node[below]{\small $\alpha_k$} .. controls (0.5,0.25) and (0, 0.25) ..  (0,0.5)
		 	  .. controls (0,0.75) and (0.5, 0.75) ..  (0.5,1);
	\draw [stdhl] (1,0) node[below]{\small $\nu$} .. controls (1,0.5) and (0, 0.75) ..  (0,1);
} 
\ &= \ 
\tikzdiagh[xscale=-1]{0}{
	\draw  (0.5,0)node[below]{\small $\alpha_k$} .. controls (0.5,0.25) and (0, 0.25) ..  (0,0.5)
		 	  .. controls (0,0.75) and (0.5, 0.75) ..  (0.5,1);
	\draw (1,0)node[below]{\small $\alpha_j$}  .. controls (1,0.5) and (0, 0.75) ..  (0,1);
	\draw  [stdhl] (0,0) node[below]{\small $\nu$} .. controls (0,0.25) and (1, 0.5) ..  (1,1);
} 
\label{eq:gcrossingslidered}
\\
\tikzdiagh{0}{
	\draw  (0,0)node[below]{\small $\alpha_j$} .. controls (0,0.25) and (1, 0.5) ..  (1,1);
	\draw  (1,0)node[below]{\small $\alpha_k$} .. controls (1,0.5) and (0, 0.75) ..  (0,1);
	\draw [stdhl] (0.5,0)node[below]{\small $\nu$}  .. controls (0.5,0.25) and (0, 0.25) ..  (0,0.5)
		 	  .. controls (0,0.75) and (0.5, 0.75) ..  (0.5,1);
} 
\ &= \ 
\tikzdiagh[xscale=-1]{0}{
	\draw  (0,0)node[below]{\small $\alpha_k$} .. controls (0,0.25) and (1, 0.5) ..  (1,1);
	\draw (1,0)node[below]{\small $\alpha_j$}  .. controls (1,0.5) and (0, 0.75) ..  (0,1);
	\draw [stdhl]  (0.5,0) node[below]{\small $\nu$} .. controls (0.5,0.25) and (0, 0.25) ..  (0,0.5)
		 	  .. controls (0,0.75) and (0.5, 0.75) ..  (0.5,1);
} 
\ + \delta_{j,k}\sssum{a+b= \nu_j-1} \ 
\tikzdiagh{0}{
	\draw  (0,0)node[below]{\small $\alpha_j$} -- (0,1) node[midway,tikzdot]{} node[midway,xshift=-1.5ex,yshift=.75ex]{\small $a$};
	\draw  (1,0)node[below]{\small $\alpha_k$} --  (1,1) node[midway,tikzdot]{} node[midway,xshift=1.5ex,yshift=.75ex]{\small $b$};
	\draw [stdhl] (0.5,0)node[below]{\small $\nu$}  --  (0.5,1);
} \label{eq:gredR3}
\end{align}
Multiplication is given by concatenation of diagrams that are read from bottom to top, and it is zero if the labels do not match.  
The algebra $T_b^{\underline{\mu}}(\g)$ is finite-dimensional and can be endowed with a $\bZ$-grading (we refer to~\cite[Definition~4.4]{webster} for the definition of the grading).  

In the case of $ \underline{\mu}=\nu$ the algebra $T_b^{\nu}(\g)$ contains a single red strand labeled $\nu$ and is isomorphic to the cyclotomic KLR algebra $R^{\nu}(b)$ for $\g$ in $b$ strands.


\begin{defn}
Fix a $\g$-weight $\lambda=(\lambda_1,\dotsc ,\lambda_{\vert I\vert})$ with each $\lambda_i$ being a formal parameter. 
  The \emph{dg-enhanced KLRW algebra}  $T_b^{\lambda,\underline{\mu}}(\g)$ 
  is defined as in \cref{def:dgwebsteralg}, with a blue strand labeled by $\lambda$ and with the $r$ red strands labeled by $\mu_1, \dots, \mu_r$ and the black strands labeled by simple roots. 
  The black strands can carry dots and be nailed on the blue strand:
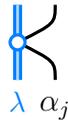
\[
	\tikzdiagh{0}{
		\draw (.5,-.5) .. controls (.5,-.25)  .. 
			(0,0) .. controls (.5,.25)  .. (.5,.5);\node at (.5,-.88) {$\alpha_j$};
	           \draw[vstdhl] (0,-.5) node[below]{\small $\lambda$} -- (0,.5) node [midway,nail]{};
  	}
\]
with everything in homological degree $0$, except that a nail is in homological degree $1$. 
The diagrams are taken up to graded braid-like planar isotopy, and are required to satisfy the same local relations as $T_b^{\underline{\mu}}(\g)$, together with the following extra local relations:
\begin{align*}
	\tikzdiagh{0}{
		\draw (.5,-.5) .. controls (.5,-.25)  ..  
			(0,0) node[midway, tikzdot]{}
			 .. controls (.5,.25)  .. (.5,.5); \node at (.5,-.88) {$\alpha_j$};
	           \draw[vstdhl] (0,-.5) node[below]{\small $\lambda$} -- (0,.5) node [midway,nail]{};
  	}
  	 &= \ 
	\tikzdiagh{0}{
		\draw (.5,-.5) .. controls (.5,-.25)  .. 
			(0,0) 
			.. controls (.5,.25)  .. (.5,.5) node[midway, tikzdot]{}; \node at (.5,-.88) {$\alpha_j$};
	           \draw[vstdhl] (0,-.5) node[below]{\small $\lambda$} -- (0,.5) node [midway,nail]{};
  	}
&
	\tikzdiagh{0}{
		\draw (.5,-1) .. controls (.5,-.75) .. (0,-.4)
				.. controls (.5,-.05) .. (.5,.2)
				-- (.5,1);\node at (.5,-1.38) {$\alpha_j$};
		\draw (1,-1) .. controls (1,0) .. (0, .4)
			.. controls (1,.75) .. (1,1); \node at (1,-1.38) {$\alpha_k$};
	     	\draw [vstdhl]  (0,-1) node[below]{\small $\lambda$} --  (0,1) node [pos=.3,nail] {} node [pos=.7,nail] {} ;
	 }  &= \ -
	\tikzdiagh[yscale=-1]{0}{
		\draw (.5,-1) .. controls (.5,-.75) .. (0,-.4)
				.. controls (.5,-.05) .. (.5,.2)
				-- (.5,1);\node at (.5,1.38) {$\alpha_j$};
		\draw (1,-1) .. controls (1,0) .. (0, .4)
			.. controls (1,.75) .. (1,1);  \node at (1,1.38) {$\alpha_k$};
	     	\draw [vstdhl]  (0,-1)  --  (0,1) node[below]{\small $\lambda$} node [pos=.3,nail] {} node [pos=.7,nail] {} ;
	 }
&
	\tikzdiagh{0}{
		\draw (.5,-1) .. controls (.5,-.75) .. (0,-.4)
			.. controls (.5,-.4) and (.5,.4) ..
			(0,.4) .. controls (.5,.75) .. (.5,1);\node at (.5,-1.38) {$\alpha_j$};
	     	\draw [vstdhl]  (0,-1) node[below]{\small $\lambda$} --  (0,1) node [pos=.3,nail] {} node [pos=.7,nail] {} ;
	 }
	 &= \ 
	0 , 
\end{align*}
for all $\alpha_j$, $\alpha_k\in I$.
\end{defn}

Note that we have an inclusion $T_b^{\underline{\mu}}(\g)\subset T_b^{\lambda,\underline{\mu}}(\g)$ by adding a vertical blue strand at the left of a diagram. 
The algebra $T_b^{\lambda,\underline{\mu}}(\g)$ can be endowed with the $q$-grading inherited from $T_b^{\underline{\mu}}(\g)$. It can be additionally endowed with
a $\lambda_k$-grading for each $\alpha_k\in I$ such that
$T_b^{\underline{\mu}}(\g)\subset T_b^{\lambda,\underline{\mu}}(\g)$ sits in $\lambda_k$-degree zero for all $k$, and
\begin{align*}
\deg_{q,\lambda_k}  \left(
	\tikzdiagh{-1ex}{
		\draw (.5,-.5) .. controls (.5,-.25)  .. 
			(0,0) .. controls (.5,.25)  .. (.5,.5);\node at (.5,-.88) {$\alpha_j$};
	           \draw[vstdhl] (0,-.5) node[below]{\small $\lambda$} -- (0,.5) node [midway,nail]{};
  	}
  	\right) &:= (0,2\delta_{k,j}). 
\end{align*}

We usually consider $T_b^{\lambda,\underline{\mu}}(\g)$ as a $\bZ^{1+|I|}$-graded dg-algebra $(T_b^{\lambda,\underline{\mu}}(\g), 0)$ with trivial differential. 
In the case of $ \underline{\mu}=\varnothing$ the algebra $T_b^{\lambda,\varnothing}(\g)$ contains a blue strand labeled $\lambda$ and is isomorphic to the $\mathfrak{b}$-KLR algebra $R_{\mathfrak{b}}(b)$ introduced in~\cite[\S3.1]{naissevaz3}.

The results of~\cref{sec:dgWebster} can be generalized to $T_b^{\lambda,\underline{\mu}}(\g)$.
In particular, one can prove it is free over $\Bbbk$ and that it admits a basis similar to the one in Theorem~\ref{thm:Tbasis}. 
Moreover, by using induction and restriction functors that add a black strand, we obtain a categorical action of $\g$ on $\cD_{dg}(T_b^{\lambda,\underline{\mu}}(\g),0)$ (in the sense of \cite{naissevaz3}), which categorifies the $U_q(\g)$-action on the tensor product of a universal Verma module and several integrable modules.

\smallskip

Fix an integrable dominant weight $\kappa$ of $\g$ and define a differential $d_\kappa$ on $T_b^{\lambda,\underline{\mu}}(\g)$
(after specialization of the $\lambda_j$-grading to $q^{\kappa_j}$) by setting
\[
d_{\kappa}\Biggl(
	\tikzdiagh{-1ex}{
		\draw (.5,-.5) .. controls (.5,-.25)  .. 
			(0,0) .. controls (.5,.25)  .. (.5,.5);\node at (.5,-.88) {$\alpha_j$};
	           \draw[vstdhl] (0,-.5) node[below]{\small $\lambda$} -- (0,.5) node [midway,nail]{};
  	}
  	\Biggr) 
\ = \ 
	\tikzdiagh{-1ex}{
		\draw (.5,-.5) -- (.5,.5) node[midway,tikzdot]{} node[midway,xshift=1.75ex,yshift=.75ex]{\small $\kappa_j$};\node at (.5,-.88) {$\alpha_j$};
	           \draw[vstdhl] (0,-.5) node[below]{\small $\lambda$} -- (0,.5);
  	}
\]
and $d_\kappa(t) = 0$ for all $t \in T_b^{\underline{\mu}}(\g) \subset T_b^{\lambda,\underline{\mu}}(\g)$, and extending by the graded Leibniz rule w.r.t. the homological grading. 
A straightforward computation shows that $d_\kappa$ is well-defined. 

\begin{prop}
The dg-algebra $(T^{\lambda,\underline{\mu}}_b(\g),d_\kappa)$ is formal with
\[
H(T^{\lambda,\underline{\mu}}_b(\g),d_\kappa) \cong T^{(\kappa,\underline{\mu})}_b(\g) . 
\]
\end{prop}

\begin{proof}
The proof follows by similar arguments as in~\cite[Theorem 4.4]{naissevaz3}. 
\end{proof}

\subsubsection{Dg-enhanced KLRW algebras for parabolic subalgebras}
Let $\p\subseteq\g$ be a parabolic subalgebra with partition $I=I_f\sqcup I_r$ of the set of simple roots,
and $(\lambda,n)=(\lambda_i)_{i\in I}$, with $\lambda_i$ a formal parameter if $i\in I_r$, and $\lambda_i=q^{n_i}$ with $n_i\in\bN$
if $i\in I_f$.

Introduce a differential $d_{\lambda,n}$ on $T_b^{\lambda,\underline{\mu}}(\g)$
(after specialization of the $\lambda_j$-grading to $q^{n_j}$ for each $\alpha_j\in I_r$) by setting
\[
d_{\lambda,n}\Biggl(
	\tikzdiagh{-1ex}{
		\draw (.5,-.5) .. controls (.5,-.25)  .. 
			(0,0) .. controls (.5,.25)  .. (.5,.5);\node at (.5,-.88) {$\alpha_j$};
	           \draw[vstdhl] (0,-.5) node[below]{\small $\lambda$} -- (0,.5) node [midway,nail]{};
  	}
  	\Biggr) 
        \ = \
        \begin{cases}
        \ \ \  0, & \text{if $\alpha_j\in I_r$}, \\[1ex]
	\tikzdiagh{-1ex}{
		\draw (.5,-.5) -- (.5,.5) node[midway,tikzdot]{} node[midway,xshift=1.75ex,yshift=.75ex]{\small $n_j$};\node at (.5,-.88) {$\alpha_j$};
	           \draw[vstdhl] (0,-.5) node[below]{\small $\lambda$} -- (0,.5);
  	} &\text{if $\alpha_j\in I_f$},
        \end{cases}
        \]
        and $d_{\lambda,n}(t) = 0$ for all $t \in T_b^{\underline{\mu}}(\g) \subset T_b^{\lambda,\underline{\mu}}(\g)$, and extending by the graded Leibniz rule w.r.t. the homological grading. As before, a straightforward computation shows that it is well-defined. 
        
\begin{prop} \label{prop:pKLRWformal}
The dg-algebra $(T^{\lambda,\underline{\mu}}_b(\g),d_{\lambda,n})$ is formal. 
\end{prop}

\begin{proof}
The proof follows by similar arguments as in~\cite[Theorem 4.4]{naissevaz3}. 
\end{proof}

  \begin{defn}
        We define the \emph{dg-enhanced $\p$-KLRW algebra} as 
\[
T_b^{\lambda,\underline{\mu}}(\g,\p) :=  H(T^{\lambda,\underline{\mu}}_b(\g),d_{\lambda,n}).
\]
\end{defn}

Note that by \cref{prop:pKLRWformal} we have a quasi-isomorphism  $(T_b^{\lambda,\underline{\mu}}(\g,\p), 0) \cong (T^{\lambda,\underline{\mu}}_b(\g),d_{\lambda,n})$.

Similarly as above,  $\cD_{dg} (T^{\lambda,\underline{\mu}}_b(\g),d_{\lambda,n})$  categorifies the tensor product of a parabolic Verma module and several integrable modules, and comes with a categorical action of $\g$.


\subsection{Dg-enhanced quiver Schur algebras}\label{sec:dgqSchur}

In order to define a quiver Schur algebra of type $A_1$,  we follow the approach of~\cite{KSY}, which best suits our goals.
 We actually use a slightly different definition because theirs corresponds to a thick version of KLRW algebra (see \cite[\S9.2]{KSY}), and we want to relate it to the version we use.

\subsubsection{Cyclic modules and quiver Schur algebras}

Recall that $\nh_b^N \cong T_b^{(N)}$ is the $N$-cyclotomic nilHecke algebra on $b$ strands. Fix $ r \geq 0$ and $\und N = (N_0, N_1, \dots, N_r) \in \bN^r$ such that $\sum_i N_i = N$. 
For $\rho = (b_0, b_1, \dots, b_r)$ such that $\sum_i b_i  = b$, we define the element
\[
x_{\rho}^{\und N} :=  
\prod_{i=1}^{r} (x_{b_r+ \cdots + b_{i+1}+1}^{N_{i}+ \cdots + N_1} \cdots x_{b_r + \cdots + b_{i+1} + b_i}^{N_{i}+ \cdots + N_1})
\in \nh_b^N,
\]
where we recall that $x_b$ is a dot on the $b$th black strand. 
Then, we consider the cyclic right $\nh_b^N$-module defined as
\[
\cyclicMod^{\und N}_{\rho} := 
x_{\rho}^{\und N}
\nh_b^N.
\]

The \emph{quiver Schur algebra} (of type $A_1$) is defined as the $\bZ$-graded algebra:
\[
\quiverSchur^{\und N}_b := \END_{\nh_b^N}\left( \bigoplus_{\rho \in \mathcal{P}_b^r} q^{-\deg_q(x_\rho^{\und N})/2} \cyclicMod^{\und N}_{\rho} \right),
\]
where $\END$ means the ($\bZ$-)graded endomorphism ring.  
The $\bZ$-graded algebra $\quiverSchur^{\und N}_b$ is isomorphic to $T^{\und N}_b$~\cite[Proposition 5.33]{webster}.
The \emph{reduced quiver Schur algebra} (of type $A_1$) is defined as
\[
\minquiverSchur :=  \END_{\nh_b^N}\left( \bigoplus_{\rho \in \mathcal{P}_b^{r, \und N}} q^{-\deg_q(x_\rho^{\und N})/2} \cyclicMod^{\und N}_{\rho} \right),
\]
where $
\mathcal{P}_{b}^{r, \und N} := \left\{ (b_0, b_1, \dots, b_r) | b_i \leq N_i  \text{ for $0 \leq i \leq r$}\right\} \subset \mathcal{P}_{b}^{r}
$. 
It is Morita equivalent to $\quiverSchur^{\und N}_b$ (this can be shown by observing that if $b_i > N_i$ for some $i$, then $\cyclicMod^{\und N}_{\rho}$ is isomorphic to a direct sum of elements in $ \{ q^{-\deg_q(x_{\rho'}^{\und N})/2} \cyclicMod^{\und N}_{\rho'} | {\rho' \in \mathcal{P}_b^{r, \und N}} \}$ ), and thus to $T^{\und N}_b$.

\subsubsection{Dg-enhanced cyclic modules}

Our goal is to construct a dg-enhancement of $\cyclicMod^{\und N}_{\rho}$ over $(T_b^{\lambda, \emptyset}, d_N)$, the dg-enhanced KLRW algebra without red strands. We will simply write $T_b^\lambda$ for $T_b^{\lambda, \emptyset}$. Recall from \cref{thm:dNformal} that $(T_b^\lambda, d_N)$ is quasi-isomorphic to $T_b^{(N)} \cong \nh_b^N$ . 

\smallskip

Let $T_b^{q^{\ell}\lambda}$ for $\ell \in \bZ$ be the algebra defined similarly as $T_b^\lambda$ (see \cref{def:dgwebsteralg}) except that the blue strand is labeled by $q^{\ell}\lambda$, and the nail is  in $\bZ^2$-degree:
\begin{align*}
\deg_{q,\lambda}  \left(
	\tikzdiagh{-1ex}{
		\draw (.5,-.5) .. controls (.5,-.25)  .. 
			(0,0) .. controls (.5,.25)  .. (.5,.5);
	           \draw[vstdhl] (0,-.5) node[below]{\small $q^{\ell}\lambda$} -- (0,.5) node [midway,nail]{};
  	}
  	\right) &= (2\ell,2). 
\end{align*}
Whenever $\ell \geq \ell'$ and $b \leq b'$, there is an inclusion of algebras
\begin{equation}\label{eq:inclusionTlshift}
T^{q^{\ell} \lambda}_{b} \hookrightarrow T^{q^{\ell'}\lambda}_{b'},
\end{equation}
given by first turning any $q^{\ell}\lambda$-nail into a $q^{\ell'}\lambda$-nail by adding dots:
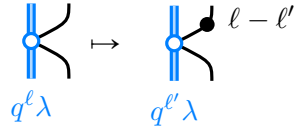
\[
\tikzdiagh{0}{
		\draw (.5,-.5) .. controls (.5,-.25)  .. 
			(0,0) .. controls (.5,.25)  .. (.5,.5);
	           \draw[vstdhl] (0,-.5) node[below]{\small $q^{\ell}\lambda$} -- (0,.5) node [midway,nail]{};
  	}
\ \mapsto  \ 
\tikzdiagh{0}{
		\draw (.5,-.5) .. controls (.5,-.25)  .. 
			(0,0) .. controls (.5,.25)  .. (.5,.5) node[midway,tikzdot]{} node[midway, xshift=4ex,yshift=.5ex]{\small $\ell - \ell'$};
	           \draw[vstdhl] (0,-.5) node[below]{\small $q^{\ell'}\lambda$} -- (0,.5) node [midway,nail]{};
  	}
\]
so that the blue strand labeled $q^{\ell} \lambda$ becomes labeled $q^{\ell'} \lambda$, 
and then adding $b'- b$ vertical black strands at the right:
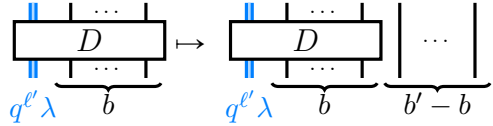
\[
\tikzdiagh[xscale=2]{0}{
	\draw [vstdhl] (-.25,0) node[below]{\small $q^{\ell'}\lambda$} -- (-.25,1);
	\draw (0,0) -- (0,1);
	\node at(.25,.125) {\tiny $\dots$};
	\node at(.25,.875) {\tiny $\dots$};
	\draw (.5,0) -- (.5,1);
	\tikzbrace{0}{.5}{0}{\small $b$};
	\filldraw [fill=white, draw=black] (-.375,.25) rectangle (.625,.75) node[midway] { $D$};
}
\mapsto 
\tikzdiagh[xscale=2]{0}{
	\draw [vstdhl] (-.25,0) node[below]{\small $q^{\ell'}\lambda$} -- (-.25,1);
	\draw (0,0) -- (0,1);
	\node at(.25,.125) {\tiny $\dots$};
	\node at(.25,.875) {\tiny $\dots$};
	\draw (.5,0) -- (.5,1);
	\tikzbrace{0}{.5}{0}{\small $b$};
	\draw (.75,0) -- (.75,1);
	\node at(1,.5) {\tiny $\dots$};
	\draw (1.25,0) -- (1.25,1);
	\tikzbrace{.75}{1.25}{0}{\small $b'-b$};
	\filldraw [fill=white, draw=black] (-.375,.25) rectangle (.625,.75) node[midway] { $D$};
}
\]
A straightforward computation shows that the map in \cref{eq:inclusionTlshift} is well-defined, and \cref{thm:Tbasis} shows that the map is injective.

By restriction, the inclusion $T^{q^{\ell} \lambda}_{b} \hookrightarrow T^{q^{\ell'}\lambda}_{b'}$ defines a left action of $T^{q^{\ell} \lambda}_{b}$ on any $T^{q^{\ell'} \lambda}_{b'}$-module. 

\begin{defn}
We define the right $T_b^\lambda$-modules 
\begin{align*}
\predgCyclicMod_{\rho}^{\und N} &:= 
T_{b_r}^{q^{-N_r-\cdots-N_1} \lambda}
\otimes_{T_{b_r}^{q^{-N_{r-1}-\cdots-N_1} \lambda}}
\cdots 
\otimes_{T_{b_r+\cdots+b_2}^{ q^{-N_1} \lambda}}
T_{b_r+\cdots+b_1}^{q^{-N_1} \lambda} \otimes_{T_{b_r+\cdots+b_1}^{\lambda}}
T_b^{\lambda},
\intertext{and}
\dgCyclicMod_{\rho}^{\und N}  &:=
x_{\rho}^{\und N}
\predgCyclicMod_{\rho}^{\und N}.
\end{align*}
\end{defn}

Note that we can endow $\dgCyclicMod_{\rho}^{\und N} $ with either a differential of the form $d_N$ (as in \cref{sec:dgenh}) or a trivial one, making it a right dg-module over $(T_b^\lambda, d_N)$ or $(T_b^\lambda,0)$ respectively. 

\begin{exe}
Take for example $r=2$. Then, we picture  $\dgCyclicMod_{\rho}^{\und N}$ in terms of diagrams as:
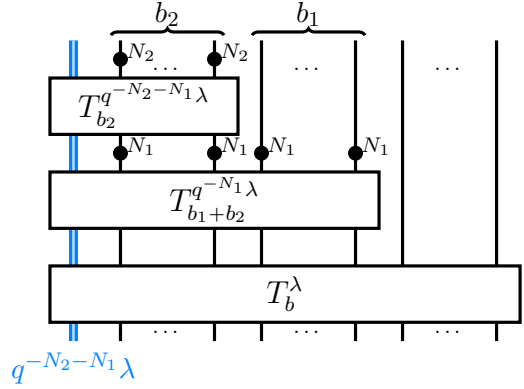
\[
\tikzdiag[xscale=1.25]{
 	\draw [vstdhl] (-.5,0) node[below]{\small $q^{-N_2-N_1}\lambda$} -- (-.5,4);
 	\draw (0,0) -- (0,2.25) -- (0,2.75) node[midway, tikzdot]{} node[midway, xshift=1.5ex, yshift=.5ex]{\tiny $N_1$}
 		-- (0,3.5) -- (0,4) node[midway, tikzdot]{} node[midway, xshift=1.5ex, yshift=.5ex]{\tiny $N_2$};
 	\node 	at(.5,3.625) {\tiny $\dots$};
 	\node 	at(.5,.125) {\tiny $\dots$};
 	\draw (1,0) -- (1,2.25) -- (1,2.75) node[midway, tikzdot]{} node[midway, xshift=1.5ex, yshift=.5ex]{\tiny $N_1$}
 		-- (1,3.5) -- (1,4) node[midway, tikzdot]{}  node[midway, xshift=1.5ex, yshift=.5ex]{\tiny $N_2$};
 	\draw (1.5,0) -- (1.5,2.25) -- (1.5,2.75) node[midway, tikzdot]{} node[midway, xshift=1.5ex, yshift=.5ex]{\tiny $N_1$}
 		 -- (1.5,4);
 	\node 	at(2,3.625) {\tiny $\dots$};
 	\node 	at(2,.125) {\tiny $\dots$};
 	\draw (2.5,0)  -- (2.5,2.25) -- (2.5,2.75) node[midway, tikzdot]{} node[midway, xshift=1.5ex, yshift=.5ex]{\tiny $N_1$}
 		-- (2.5,4);
 	\draw (3,0) -- (3,4);
 	\node 	at(3.5,3.625) {\tiny $\dots$};
 	\node 	at(3.5,.125) {\tiny $\dots$};
 	\draw (4,0) -- (4,4);
	\filldraw [fill=white, draw=black] (-.75,2.75) rectangle (1.25,3.5) node[midway] { $T_{b_2}^{q^{-N_2-N_1}\lambda}$};
	\filldraw [fill=white, draw=black] (-.75,1.5) rectangle (2.75,2.25) node[midway] { $T_{b_1+b_2}^{q^{-N_1}\lambda}$};
	\filldraw [fill=white, draw=black] (-.75,.25) rectangle (4.25,1) node[midway] { $T_{b}^{\lambda}$};
	\tikzbraceop{0}{1}{4}{\small $b_2$};
	\tikzbraceop{1.5}{2.5}{4}{\small $b_1$};
}
\]
\end{exe}

Note that whenever $N + \ell \geq 0$ we can equip $T^{q^\ell \lambda}_b$ with a differential $d_N$ given by 
\[
d_N\left(
	\tikzdiagh{-1ex}{
		\draw (.5,-.5) .. controls (.5,-.25)  .. 
			(0,0) .. controls (.5,.25)  .. (.5,.5);
	           \draw[vstdhl] (0,-.5) node[below]{\small $q^{\ell} \lambda$} -- (0,.5) node [midway,nail]{};
  	}
  	\right) 
\ := \ 
	\tikzdiagh{-1ex}{
		\draw (.5,-.5) -- (.5,.5) node[midway,tikzdot]{} node[midway,xshift=4ex,yshift=.75ex]{\small $N+\ell$};
	           \draw[vstdhl] (0,-.5) node[below]{\small $q^{\ell} \lambda$} -- (0,.5);
  	}
\]
and it is compatible with the inclusion in \cref{eq:inclusionTlshift}. 

We conjecture the following:

\begin{conj}
There is a quasi-isomorphism
\[
(\dgCyclicMod_{\rho}^{\und N}, d_{N}) \xrightarrow{\simeq} (\cyclicMod_{\rho}^{\und N}, 0).
\]
\end{conj}

\begin{lem}\label{lem:Tkndecomp}
There is a decomposition as graded vector spaces
\[
\tikzdiag[xscale=1]{
 	\draw [vstdhl] (-.5,0) node[below]{\small $q^{-n}\lambda$} -- (-.5,2.75);
 	\draw (0,0) -- (0,2.25) -- (0,2.75) node[midway, tikzdot]{} node[midway, xshift=1.5ex, yshift=.5ex]{\tiny $n$};
 	\node 	at(.5,2.375) {\tiny $\dots$};
 	\node 	at(.5,.125) {\tiny $\dots$};
 	\draw (1,0) -- (1,2.25) -- (1,2.75) node[midway, tikzdot]{} node[midway, xshift=1.5ex, yshift=.5ex]{\tiny $n$};
 	\draw (1.5,0) -- (1.5,2.25) -- (1.5,2.75) node[midway, tikzdot]{} node[midway, xshift=1.5ex, yshift=.5ex]{\tiny $n$};
	\filldraw [fill=white, draw=black] (-.75,1.5) rectangle (1.75,2.25) node[midway] { $T_{k+1}^{q^{-n}\lambda}$};
}
 \ \cong \ 
\tikzdiag[xscale=1]{
 	\draw [vstdhl] (-.5,0) node[below]{\small $q^{-n}\lambda$} -- (-.5,2.75);
 	\draw (0,0) -- (0,2.25) -- (0,2.75) node[midway, tikzdot]{} node[midway, xshift=1.5ex, yshift=.5ex]{\tiny $n$};
 	\node 	at(.5,2.375) {\tiny $\dots$};
 	\node 	at(.5,1.25) {\tiny $\dots$};
 	\node 	at(.5,.125) {\tiny $\dots$};
 	\draw (1,0) -- (1,2.25) -- (1,2.75) node[midway, tikzdot]{} node[midway, xshift=1.5ex, yshift=.5ex]{\tiny $n$};
 	\draw (1.5,0) -- (1.5,2.25) -- (1.5,2.75) node[midway, tikzdot]{} node[midway, xshift=1.5ex, yshift=.5ex]{\tiny $n$};
	\filldraw [fill=white, draw=black] (-.75,1.5) rectangle (1.25,2.25) node[midway] { $T_{k}^{q^{-n}\lambda}$};
	\filldraw [fill=white, draw=black] (-.25,.25) rectangle (1.75,1) node[midway] { $\nh_{k+1}$};
}
\ \oplus \ 
\tikzdiag[xscale=1]{
 	\draw (.5,0) -- (.5,1) .. controls (.5,1.25) and (0,1.25) .. (0,1.5)  -- (0,2.25) -- (0,2.75) node[midway, tikzdot]{} node[midway, xshift=1.5ex, yshift=.5ex]{\tiny $n$};
 	\node 	at(.5,2.375) {\tiny $\dots$};
 	\node 	at(.75,.125) {\tiny $\dots$};
 	\draw (1.5,0) -- (1.5,1) .. controls (1.5,1.25) and (1,1.25) .. (1,1.5)  -- (1,2.25) -- (1,2.75) node[midway, tikzdot]{} node[midway, xshift=1.5ex, yshift=.5ex]{\tiny $n$};
 	\draw (0,0) -- (0,1).. controls(0,1.125) .. (-.5,1.25) .. controls (1.5,1.375) ..(1.5,1.5)-- (1.5,2.25) -- (1.5,2.75) node[midway, tikzdot]{} node[midway, xshift=1.5ex, yshift=.5ex]{\tiny $n$};
 	\draw [vstdhl] (-.5,0) node[below]{\small $q^{-n}\lambda$} -- (-.5,2.75) node[pos=.45,nail]{};
	\filldraw [fill=white, draw=black] (-.75,1.5) rectangle (1.25,2.25) node[midway] { $T_{k}^{q^{-n}\lambda}$};
	\filldraw [fill=white, draw=black] (-.25,.25) rectangle (1.75,1) node[midway] { $\nh_{k+1}$};
}
\]
\end{lem}

\begin{proof}
The claim follows from \cref{thm:Tbasis}. 
\end{proof}

\begin{prop}\label{prop:incdgcyclmod1}
Suppose $\rho$ and $\rho'$ are such that $b_i = b_i'$ for all $0 \leq i \leq m$ except $i = j$ and $i = j+1$ where they respect $b_j = b_j' - 1$ and $b_{j+1} = b_{j+1}' + 1$. 
Then there is an inclusion of right dg-modules
\[
\dgCyclicMod_{\rho}^{\und N}  \hookrightarrow \dgCyclicMod_{\rho'}^{\und N}.
\]
\end{prop}

\begin{proof}
We can work locally, and thus we want to prove that
\[
G_2 :=
\tikzdiag[xscale=1.25]{
 	\draw [vstdhl] (-.5,1.25) node[below]{\small $q^{-n} \lambda $} -- (-.5,4);
 	\draw (0,1.25) -- (0,2.25) -- (0,2.75)
 		-- (0,3.5) -- (0,4) node[midway, tikzdot]{} node[midway, xshift=1.5ex, yshift=.5ex]{\tiny $n$};
 	\node 	at(.5,3.625) {\tiny $\dots$};
 	\node 	at(.5,1.375) {\tiny $\dots$};
 	\draw (1,1.25) -- (1,2.25) -- (1,2.75) 
 		-- (1,3.5) -- (1,4) node[midway, tikzdot]{}  node[midway, xshift=1.5ex, yshift=.5ex]{\tiny $n$};
 	\draw (1.5,1.25) -- (1.5,2.25) -- (1.5,2.75)
 		-- (1.5,3.5) -- (1.5,4) node[midway, tikzdot]{}  node[midway, xshift=1.5ex, yshift=.5ex]{\tiny $n$};
 	\draw (2,1.25) -- (2,2.25) -- (2,2.75)
 		 -- (2,4);
 	\node 	at(2.5,3.625) {\tiny $\dots$};
 	\node 	at(2.5,1.375) {\tiny $\dots$};
 	\draw (3,1.25)  -- (3,2.25) -- (3,2.75) 
 		-- (3,4);
	\filldraw [fill=white, draw=black] (-.75,2.75) rectangle (1.75,3.5) node[midway] { $T_{k+1}^{q^{-n}\lambda}$};
	\filldraw [fill=white, draw=black] (-.75,1.5) rectangle (3.25,2.25) node[midway] { $T_{b}^{\lambda}$};
	\tikzbraceop{0}{1}{4}{\small $k$};
}
 \ \subset  \ 
\tikzdiag[xscale=1.25]{
 	\draw [vstdhl] (-.5,1.25) node[below]{\small $q^{-n}\lambda$} -- (-.5,4);
 	\draw (0,1.25) -- (0,2.25) -- (0,2.75) 
 		-- (0,3.5) -- (0,4) node[midway, tikzdot]{} node[midway, xshift=1.5ex, yshift=.5ex]{\tiny $n$};
 	\node 	at(.5,3.625) {\tiny $\dots$};
 	\node 	at(.5,1.375) {\tiny $\dots$};
 	\draw (1,1.25) -- (1,2.25) -- (1,2.75)
 		-- (1,3.5) -- (1,4) node[midway, tikzdot]{}  node[midway, xshift=1.5ex, yshift=.5ex]{\tiny $n$};
 	\draw (1.5,1.25) -- (1.5,2.25) -- (1.5,2.75)
 		-- (1.5,3.5) -- (1.5,4);
 	\draw (2,1.25) -- (2,2.25) -- (2,2.75)
 		 -- (2,4);
 	\node 	at(2.5,3.625) {\tiny $\dots$};
 	\node 	at(2.5,1.375) {\tiny $\dots$};
 	\draw (3,1.25)  -- (3,2.25) -- (3,2.75) 
 		-- (3,4);
	\filldraw [fill=white, draw=black] (-.75,2.75) rectangle (1.25,3.5) node[midway] { $T_{k}^{q^{-n}\lambda}$};
	\filldraw [fill=white, draw=black] (-.75,1.5) rectangle (3.25,2.25) node[midway] { $T_{b}^{\lambda}$};
	\tikzbraceop{0}{1}{4}{\small $k$};
}
=: G_1
\]
We apply \cref{lem:Tkndecomp} on $T^{q^{-n}\lambda}_{k+1}$ inside $G_2$. The left summand is clearly in $G_1$. For the right summand, it is less clear since the nails in $T^{\lambda}_{b}$ all acts by adding a nail and $n$ dots on the blue strand labeled $q^{-n}\lambda$. Thus, we want to show that 
\[
\tikzdiagh{0}{
	\draw (0,-.5) .. controls (0,-.375) .. (-.5,-.25) .. controls (0,-.125) .. (0,0) .. controls (0,.5) and (1.5,.5) .. (1.5,1) -- (1.5,1.25) node[pos=.2,tikzdot]{} node[pos=.2, xshift=1ex,yshift=.5ex]{\tiny $n$};
	\draw (.5,-.5)--(.5,0) .. controls (.5,.5) and (0,.5) .. (0,1) -- (0,1.25)  node[pos=.2,tikzdot]{} node[pos=.2, xshift=1ex,yshift=.5ex]{\tiny $n$};
	\node at (1,-.25) {\small $\cdots$};
	\node at (.65,.75) {\small $\cdots$};
	\draw (1.5,-.5)--(1.5,0) .. controls (1.5,.5) and (1,.5) .. (1,1) -- (1,1.25)  node[pos=.2,tikzdot]{} node[pos=.2, xshift=1ex,yshift=.5ex]{\tiny $n$};
	\tikzbrace{.5}{1.5}{-.5}{\small $k$};
	\draw (2,-.5) -- (2,1.25);
	\node at (2.5,.375) {\small $\cdots$};
	\draw (3,-.5) -- (3,1.25);
	\draw[vstdhl] (-.5,-.5) node[below]{\small $q^{-n}\lambda$} -- (-.5,0) node[midway,nail]{} -- (-.5,1.25);
}
\ \in G_1.  
\]
By \cref{eq:nhdotslide}, we have
\begin{align*}
\tikzdiagh{0}{
	\draw (0,-.5) .. controls (0,-.375) .. (-.5,-.25) .. controls (0,-.125) .. (0,0) .. controls (0,.5) and (1.5,.5) .. (1.5,1) -- (1.5,1.25) node[pos=.2,tikzdot]{} node[pos=.2, xshift=1ex,yshift=.5ex]{\tiny $n$};
	\draw (.5,-.5)--(.5,0) .. controls (.5,.5) and (0,.5) .. (0,1) -- (0,1.25)  node[pos=.2,tikzdot]{} node[pos=.2, xshift=1ex,yshift=.5ex]{\tiny $n$};
	\node at (1,-.25) {\small $\cdots$};
	\node at (.65,.75) {\small $\cdots$};
	\draw (1.5,-.5)--(1.5,0) .. controls (1.5,.5) and (1,.5) .. (1,1) -- (1,1.25)  node[pos=.2,tikzdot]{} node[pos=.2, xshift=1ex,yshift=.5ex]{\tiny $n$};
	\tikzbrace{.5}{1.5}{-.5}{\small $k$};
	\draw[vstdhl] (-.5,-.5) node[below]{\small $q^{-n}\lambda$} -- (-.5,0) node[midway,nail]{} -- (-.5,1.25);
}
& \ =  \ 
\tikzdiagh{0}{
	\draw (0,-.5) .. controls (0,-.375) .. (-.5,-.25) .. controls (0,-.125) .. (0,0) .. controls (0,.5) and (1.5,.5) .. (1.5,1) node[pos=.1,tikzdot]{} node[pos=.1, xshift=-1ex,yshift=.5ex]{\tiny $n$}  -- (1.5,1.25) ;
	\draw (.5,-.5)--(.5,0) .. controls (.5,.5) and (0,.5) .. (0,1) -- (0,1.25)  node[pos=.2,tikzdot]{} node[pos=.2, xshift=1ex,yshift=.5ex]{\tiny $n$};
	\node at (1,-.25) {\small $\cdots$};
	\node at (.65,.75) {\small $\cdots$};
	\draw (1.5,-.5)--(1.5,0) .. controls (1.5,.5) and (1,.5) .. (1,1) -- (1,1.25)  node[pos=.2,tikzdot]{} node[pos=.2, xshift=1ex,yshift=.5ex]{\tiny $n$};
	\tikzbrace{.5}{1.5}{-.5}{\small $k$};
	\draw[vstdhl] (-.5,-.5) node[below]{\small $q^{-n}\lambda$} -- (-.5,0) node[midway,nail]{} -- (-.5,1.25);
}
\ - \sum_{\ell = 0}^k \sssum{r+s \\= n-1} \ 
\tikzdiagh{0}{
	\draw (0,-.5) .. controls (0,-.375) .. (-.5,-.25) .. controls (0,-.125) .. (0,0) .. controls (0,.5) and (1.5,.5) .. (1.5,1) node[pos=.8,tikzdot]{} node[pos=.8, xshift=1ex,yshift=-.5ex]{\tiny $r$}
	 -- (1.5,1.25) node[pos=.2,tikzdot]{} node[pos=.2, xshift=1ex,yshift=.5ex]{\tiny $n$};
	\draw (.5,-.5)--(.5,0) .. controls (.5,.5) and (0,.5) .. (0,1) -- (0,1.25)  node[pos=.2,tikzdot]{} node[pos=.2, xshift=1ex,yshift=.5ex]{\tiny $n$};
	\node at (1,-.25) {\small $\cdots$};
	\node at (.65,.75) {\small $\cdots$};
	\draw (1.5,-.5)--(1.5,0) .. controls (1.5,.5) and (1,.5) .. (1,1) -- (1,1.25)  node[pos=.2,tikzdot]{} node[pos=.2, xshift=1ex,yshift=.5ex]{\tiny $n$};
	\tikzbrace{.5}{1.5}{-.5}{\small $\ell$};
	\draw (2,-.5)  -- (2,0) .. controls (2,.5) and (3.5,.5) .. (3.5,1)  node[pos=.1,tikzdot]{} node[pos=.1, xshift=-1ex,yshift=.5ex]{\tiny $s$} -- (3.5,1.25);
	\draw (2.5,-.5)--(2.5,0) .. controls (2.5,.5) and (2,.5) .. (2,1) -- (2,1.25)  node[pos=.2,tikzdot]{} node[pos=.2, xshift=1ex,yshift=.5ex]{\tiny $n$};
	\node at (3,-.25) {\small $\cdots$};
	\node at (2.65,.75) {\small $\cdots$};
	\draw (3.5,-.5)--(3.5,0) .. controls (3.5,.5) and (3,.5) .. (3,1) -- (3,1.25)  node[pos=.2,tikzdot]{} node[pos=.2, xshift=1ex,yshift=.5ex]{\tiny $n$};
	\draw[vstdhl] (-.5,-.5) node[below]{\small $q^{-n}\lambda$} -- (-.5,0) node[midway,nail]{} -- (-.5,1.25);
}
\end{align*}
The term of the left is clearly in $G_1$ since there are $n$ dots next to the nail, so that it can be obtained from a nail in $T^{\lambda}_{b}$. The terms on the right are also in $G_1$ since we can slide the nail and crossings on the left to the top, into $T_k^{q^{-n}\lambda}$. 
\end{proof}

Consider $(x_1^n \cdots x_{k}^n) T^{q^{-n}\lambda}_k \otimes_{T^{\lambda}_k} T^{\lambda}_b$. We obtain an inclusion 
\[
(x_1^n \cdots x_{k}^n) T^{q^{-n}\lambda}_k \hookrightarrow (x_1^n \cdots x_k^n x_{k+1}^n) T^{q^{-n}\lambda}_{k+1},
\]
 of q-degree $2n$ by adding a vertical strand on the right on which we put $n$ dots (again, the fact it is an inclusion follows immediatly from \cref{thm:Tbasis}). In turns, it gives rise to a map of right (dg-)modules $(x_1^n \cdots x_{k}^n) T^{q^{-n}\lambda}_k \otimes_{T^{\lambda}_k} T^{\lambda}_b \rightarrow (x_1^n \cdots x_{k}^n x_{k+1}^n) T^{q^{-n}\lambda}_{k+1} \otimes_{T^{\lambda}_{k+1}} T^{\lambda}_b$. 
In terms of diagrams, we can picture the inclusion above as:
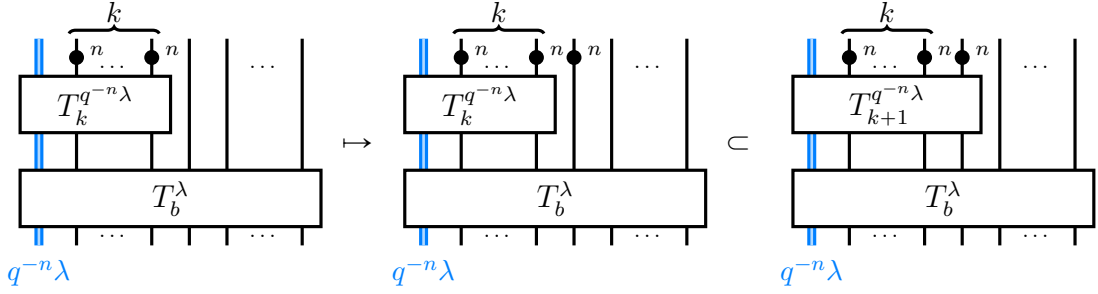
\[
\tikzdiag[xscale=1]{
 	\draw [vstdhl] (-.5,0) node[below]{\small $q^{-n}\lambda$} -- (-.5,2.75);
 	\draw (0,0) -- (0,2.25) -- (0,2.75) node[midway, tikzdot]{} node[midway, xshift=1.5ex, yshift=.5ex]{\tiny $n$};
 	\node 	at(.5,2.375) {\tiny $\dots$};
 	\node 	at(.5,.125) {\tiny $\dots$};
 	\draw (1,0) -- (1,2.25) -- (1,2.75) node[midway, tikzdot]{} node[midway, xshift=1.5ex, yshift=.5ex]{\tiny $n$};
 	\draw (1.5,0) -- (1.5,2.25) -- (1.5,2.75);
 	\draw (2,0) -- (2,2.25) -- (2,2.75);
 	\node 	at(2.5,2.375) {\tiny $\dots$};
 	\node 	at(2.5,.125) {\tiny $\dots$};
 	\draw (3,0)  -- (3,2.25) -- (3,2.75);
	\filldraw [fill=white, draw=black] (-.75,1.5) rectangle (1.25,2.25) node[midway] { $T_{k}^{q^{-n}\lambda}$};
	\filldraw [fill=white, draw=black] (-.75,.25) rectangle (3.25,1) node[midway] { $T_{b}^{\lambda}$};
	\tikzbraceop{0}{1}{2.75}{\small $k$};
}
\ \mapsto 
\tikzdiag[xscale=1]{
 	\draw [vstdhl] (-.5,0) node[below]{\small $q^{-n}\lambda$} -- (-.5,2.75);
 	\draw (0,0) -- (0,2.25) -- (0,2.75) node[midway, tikzdot]{} node[midway, xshift=1.5ex, yshift=.5ex]{\tiny $n$};
 	\node 	at(.5,2.375) {\tiny $\dots$};
 	\node 	at(.5,.125) {\tiny $\dots$};
 	\draw (1,0) -- (1,2.25) -- (1,2.75) node[midway, tikzdot]{} node[midway, xshift=1.5ex, yshift=.5ex]{\tiny $n$};
 	\draw (1.5,0) -- (1.5,2.25) -- (1.5,2.75)  node[midway, tikzdot]{} node[midway, xshift=1.5ex, yshift=.5ex]{\tiny $n$};
 	\draw (2,0) -- (2,2.25) -- (2,2.75);
 	\node 	at(2.5,2.375) {\tiny $\dots$};
 	\node 	at(2.5,.125) {\tiny $\dots$};
 	\draw (3,0)  -- (3,2.25) -- (3,2.75);
	\filldraw [fill=white, draw=black] (-.75,1.5) rectangle (1.25,2.25) node[midway] { $T_{k}^{q^{-n}\lambda}$};
	\filldraw [fill=white, draw=black] (-.75,.25) rectangle (3.25,1) node[midway] { $T_{b}^{\lambda}$};
	\tikzbraceop{0}{1}{2.75}{\small $k$};
}
\ \subset \ 
\tikzdiag[xscale=1]{
 	\draw [vstdhl] (-.5,0) node[below]{\small $q^{-n}\lambda$} -- (-.5,2.75);
 	\draw (0,0) -- (0,2.25) -- (0,2.75) node[midway, tikzdot]{} node[midway, xshift=1.5ex, yshift=.5ex]{\tiny $n$};
 	\node 	at(.5,2.375) {\tiny $\dots$};
 	\node 	at(.5,.125) {\tiny $\dots$};
 	\draw (1,0) -- (1,2.25) -- (1,2.75) node[midway, tikzdot]{} node[midway, xshift=1.5ex, yshift=.5ex]{\tiny $n$};
 	\draw (1.5,0) -- (1.5,2.25) -- (1.5,2.75)  node[midway, tikzdot]{} node[midway, xshift=1.5ex, yshift=.5ex]{\tiny $n$};
 	\draw (2,0) -- (2,2.25) -- (2,2.75);
 	\node 	at(2.5,2.375) {\tiny $\dots$};
 	\node 	at(2.5,.125) {\tiny $\dots$};
 	\draw (3,0)  -- (3,2.25) -- (3,2.75);
	\filldraw [fill=white, draw=black] (-.75,1.5) rectangle (1.75,2.25) node[midway] { $T_{k+1}^{q^{-n}\lambda}$};
	\filldraw [fill=white, draw=black] (-.75,.25) rectangle (3.25,1) node[midway] { $T_{b}^{\lambda}$};
	\tikzbraceop{0}{1}{2.75}{\small $k$};
}
\]
This generalizes into the following proposition:

\begin{prop}\label{prop:incdgcyclmod2}
Under the same hypothesis as in \cref{prop:incdgcyclmod1}, we obtain a map of right dg-modules
\[
\dgCyclicMod_{\rho'}^{\und N} \rightarrow \dgCyclicMod_{\rho}^{\und N},
\]
of $q$-degree $2N_{j+1}$, diagrammatically given by gluing on top the dots $x_{b_r' + \cdots + b_{j+1}'+1} ^{N_{j+1}}$.
\end{prop}

\subsubsection{Dg-quiver Schur algebra}

\begin{defn}
We define the \emph{dg-quiver Schur algebras} as
\[
(\dgQuiverSchur^{\und N}_b, d_N) 
:= 
\END^{dg}_{(T^{\lambda}_b, d_N)} \left(
\bigoplus_{\rho \in \mathcal{P}_b^r} q^{-\deg_q(x_\rho^{\und N})/2}(\dgCyclicMod_{\rho}^{\und N} , d_N)
\right),
\]
and
\[
(\dgQuiverSchur^{\und N}_b,0) := \END^{dg}_{(T^{\lambda}_b, 0)} \left(
\bigoplus_{\rho \in \mathcal{P}_b^r} q^{-\deg_q(x_\rho^{\und N})/2}(\dgCyclicMod_{\rho}^{\und N}, 0)
\right),
\]
where $\END^{dg}$ is the $\bZ^2$-graded ($\bZ$-graded in the first case) dg-endomorphism ring (see \cref{sec:classicalhomandtensor}).  
We also define a reduced version as
\[
(\mindgQuiverSchur^{\und N}_b,0) := \END^{dg}_{(T^{\lambda}_b, 0)} \left(
\bigoplus_{\rho \in \mathcal{P}_b^{r, \und N}} q^{-\deg_q(x_\rho^{\und N})/2}(\dgCyclicMod_{\rho}^{\und N}, 0) 
\right).
\]
\end{defn}

\begin{conj}\label{conj:dgSchurCycSchur}
There is a quasi-isomorphism
\[
(\dgQuiverSchur^{\und N}_b, d_N) 
\xrightarrow{\simeq}
(\quiverSchur^{\und N}_b,0). 
\]
\end{conj}

Our goal is to construct a graded map of algebras 
\[
T_b^{\lambda, \und N} \rightarrow \dgQuiverSchur^{\und N}_b.
\]
For $\rho=(b_0,b_1,\dots,b_r)\in \mathcal{P}_b^r$, we send
\[
1_\rho \mapsto \id \in \END_{T_b^\lambda}(G_{\rho}^{\und N}) \subset \dgQuiverSchur^{\und N}_b.
\]
Dots on the $i$th black strand (resp. black/black crossings on the $i$th and $(i+1)$th black strands) on $1_\rho$ is sent to multiplication on the left (i.e. gluing on top) by a dot on the $i$th black strand (resp. crossing) on $G_{\rho}^{\und N}$. These are indeed maps of right $T_b^\lambda$-modules since the dots and crossing commutes with $x_i^{n} x_{i+1}^n$ for all $n \geq 0$.  Similarly, a nail on the blue strand labeled $\lambda$ in $T_b^{\lambda, \und N}$ is sent to multiplication on the left by a nail on the blue strand labeled $q^{-N_r-\cdots-N_1}\lambda$ in $G_{\rho}^{\und N}$. 

For black/red crossing $\tau_i$, if the red strand goes from bottom left to top right, then we have $1_{\rho'} \tau_i 1_{\rho}$ where $\rho$ and $\rho'$ are as in \cref{prop:incdgcyclmod1}. Then, we associate to it the map $G_{\rho}^{\und N} \rightarrow G_{\rho'}^{\und N}$ of \cref{prop:incdgcyclmod1}. If the red strand goes from bottom right to top left, then we have $1_{\rho} \tau_i 1_{\rho'}$, and we associate to it the map $G_{\rho'}^{\und N} \rightarrow G_{\rho}^{\und N}$ of \cref{prop:incdgcyclmod2}.

\begin{prop}\label{prop:mapdgquiverschur}
The map defined above gives rise to maps of $\bZ$-graded dg-algebras
\[
(T_b^{\lambda, \und N}, d_{N_0}) \rightarrow ( \dgQuiverSchur^{\und N}_b, d_N),
\]
and of $\bZ^2$-graded dg-algebras
\[
(T_b^{\lambda, \und N}, 0) \rightarrow ( \dgQuiverSchur^{\und N}_b, 0 ).
\]
\end{prop}

\begin{proof}
We show the assignment given above is a map of algebras, the commutation with the differentials being obvious since the image by $d_{N_0}$ of a nail on a blue strand labeled $\lambda$ consists of $N_0$ dots on the first black strand; and the image by $d_N$ of a nail on a blue strand labeled $q^{-N_r-\cdots-N_1} \lambda$ consists of $N-N_r-\cdots-N_1 = N_0$ dots. 

Thus, we need to prove the map respects all the defining relations in \cref{def:dgwebsteralg}. Relations in \cref{eq:nhR2andR3} and \cref{eq:nhdotslide} are immediate by construction. The relations in \cref{eq:dotredstrand} follow from commutations of dots. Since the map in \cref{prop:incdgcyclmod2} is multiplication by $n_{j+1}$ dots and the map in \cref{prop:incdgcyclmod1} is an inclusion, we have the relations in \cref{eq:redR2}. 
For the left side of \cref{eq:crossingslidered} both black/red crossings are given by an inclusion, and thus commutes with the multiplication on the left by the black/black crossing. For the right side, the black/red crossings give a multiplication by $x_i^{N_{j+1}} x_{i+1}^{N_{j+1}}$, which commutes with the black/black crossing. 
For \cref{eq:redR3}, one the black/red crossing is an inclusion and the other one is multiplication by $x_i^{N_{j+1}}$ on both side of the equality, so that the relation follows from \cref{eq:nhdotslide}. 
Finally, the relation in \cref{eq:relNail} is immediate by construction. 
\end{proof}

\begin{conj}\label{conj:WebQSchur}
The maps in \cref{prop:mapdgquiverschur} are isomorphisms. 
\end{conj}

We also conjecture that the reduced dg-quiver Schur algebra $(\mindgQuiverSchur^{\und N}_b,0)$ is dg-Morita equivalent to the non-reduced one $(\dgQuiverSchur^{\und N}_b,0)$.



\appendix


\section{Detailed proofs and computations}\label{sec:computations}

We give the detailed computations used to prove various results of the paper. 

\subsection{Proofs of \cref{sec:qsltTLB}}

\begin{citelem}{lem:explicitaction}
  The action of $\cB_r$ translates in terms of $v_\rho$-vectors of $M \otimes V^r$ as
  \begin{align}
    \tag{\ref*{eq:caponk}}
    \tikzdiagh[scale=0.75]{2}{
	\draw[ultra thick,myblue] (-.5,0) -- (-.5,1);
	\draw[red] (0,0) -- (0,1);
	\node at (.5,.5){\small $\dots$};
	\draw[red] (1,0)  -- (1,1);
	\draw[red] (1.5, 0)  .. controls (1.5,.5) and (2,.5) .. (2,0)  ;
	\draw[red] (2.5,0) -- (2.5,1);
	\node at (3,.5){\small $\dots$};
	\draw[red] (3.5,0)  -- (3.5,1);
	\tikzbrace{-.5}{1}{-0.2}{$i$};
}
    :& v_{(\dots, b_{i-1}, b_i, b_{i+1}, b_{i+2}, \dots)} \mapsto -q^{-1} [b_i]_q  v_{(\dots, b_{i-1} + b_i + b_{i+1} - 1, b_{i+2}, \dots)},
    \\      
    \tag{\ref*{eq:cuponk}}
    \tikzdiagh[scale=0.75,yscale=-1]{2}{
	\draw[ultra thick,myblue] (-.5,0) -- (-.5,1);
	\draw[red] (0,0) -- (0,1);
	\node at (.5,.5){\small $\dots$};
	\draw[red] (1,0)  -- (1,1);
	\draw[red] (1.5, 0)  .. controls (1.5,.5) and (2,.5) .. (2,0)  ;
	\draw[red] (2.5,0) -- (2.5,1);
	\node at (3,.5){\small $\dots$};
	\draw[red] (3.5,0) -- (3.5,1);
	\tikzbrace{-.5}{1}{1.9}{$i$};
}
    :&v_{\rho} \mapsto q[2]_q v_{(\dots, b_{i-1},1 ,0 ,b_{i}, \dots)} - q v_{(\dots, b_{i-1}+1, 0, 0,b_{i}, \dots)} -q v_{(\dots, b_{i-1},0 , 1,b_{i}, \dots)},
    \\
    \tag{\ref*{eq:xionk}}
    \tikzdiag[scale=0.75]{
    \draw[red] (0,0) .. controls (0,.25) .. (-.5,.5) .. controls (0,.75) ..  (0,1);
    \draw[fill=white, color=white] (-.52,.5) circle (.02cm);
    \draw[red] (0.5,0) -- (0.5,1);
    \node at (1,0.5){\small $\dots$};
    \draw[red] (1.5,0) -- (1.5,1); 
    \draw[ultra thick,myblue] (-.5,0) -- (-.5,1);
    }
    :&v_{(b_0 , b_1 ,\dots)} \mapsto (\lambda^{-1}q^{b_0} - \lambda q [b_0]_q) v_{(0,b_0+b_1,\dots)} + \lambda q^2 [b_0]_q v_{(1,b_0+b_1-1,\dots)}.
  \end{align}
\end{citelem}
\begin{proof}
  We start with the cap. We have
  \[
    v_{(\ldots,b_{i-1},b_i,b_{i+1},b_{i+2},\ldots)} = [b_i]_q v_{(\ldots,b_{i-1}+b_i-1,1,b_{i+1},b_{i+2},\ldots)} - [b_i-1]_q v_{(\ldots,b_{i-1}+b_i,0,b_{i+1},b_{i+2},\ldots)}, 
  \]
  and we easily check that $v_{(\ldots,b_{i-1}+b_i-1,1,b_{i+1},b_{i+2},\ldots)}$ is sent to $-q^{-1}v_{(\ldots,b_{i-1}+b_i+b_{i+1}-1,b_{i+2},\ldots)}$ and $v_{(\ldots,b_{i-1}+b_i,0,b_{i+1},b_{i+2},\ldots)}$ is sent to $0$.

  We now turn to the cup. It suffices to do the computation for $i=r+1$ because of the recursive definition of $v_\rho$. By definition, $v_{(b_0,\ldots,b_n)}$ is sent to $-qv_{(b_0,\ldots,b_n)} \otimes v_{1,0}\otimes F v_{1,0}+v_{(b_0,\ldots,b_n)}\otimes Fv_{1,0}\otimes v_{1,0}$. Since
  \[
    v_{(b_0,\ldots,b_n)} \otimes v_{1,0}\otimes F v_{1,0} = v_{(b_0,\ldots,b_n,0,1)} - q^2v_{(b_0,\ldots,b_n+1,0,0)} - q v_{(b_0,\ldots,b_n)} \otimes F v_{1,0}\otimes  v_{1,0}, 
  \]
  and
  \[
    v_{(b_0,\ldots,b_n)} \otimes F v_{1,0}\otimes  v_{1,0} = v_{(b_0,\ldots,b_n,1,0)} - q v_{(b_0,\ldots,b_n+1,0,0)}, 
  \]
  one finds the expected formula.

  Finally, we finish with $\xi$. Using the fact that $\xi$ is a morphism of $U_q(\slt)$-modules, it suffices to consider the case of the vector $v_{(b_0,b_1)}$. One may check that $v_{(b_0,b_1)}=[b_0]_q v_{(1,b_0+b_1-1)}-[b_0-1]_q v_{(0,b_0+b_1)}$ and therefore
  \[
    \xi(v_{(b_0,b_1)})=[b_0]_q F^{b_0+b_1-1}\xi(v_{(1,0)})-[b_0-1]_q F^{b_0+b_1}\xi(v_{(0,0)}).
  \]
  Using the definition of $\xi$ we have
  \[
    \xi(v_{(0,0)}) = \lambda^{-1}v_{(0,0)}, 
  \]
  and
  \[
    \xi(v_{(1,0)}) = \lambda q^2v_{(1,0)}-q(\lambda-\lambda^{-1})v_{(0,1)}.
  \]
  Hence we deduce that
  \[
    \xi(v_{(b_0,b_1)}) = \lambda q^2[b_0]_q v_{(1,b_0+b_1-1)} -(q(\lambda-\lambda^{-1})[b_0]_q+\lambda^{-1}[b_0-1]_q)v_{(0,b_0+b_1)}.
  \]
  We conclude by checking that $\lambda^{-1}q[b_0]_q-\lambda^{-1}[b_0-1]_q =\lambda^{-1}q^{b_0}$.
\end{proof}

\subsection{Proofs of \cref{sec:bimod}}\label{sec:proofsofsecbimod}

\begin{citelem}{lem:gammasurjective}
The map $\gamma_k : \br X_k \rightarrow X_k$ is surjective.
\end{citelem}

\begin{proof}
First, we recall the following well-known relation
\begin{equation}\label{eq:nhdoublecrossingid}
\tikzdiag{
	\draw (0,0) -- (0,1);
	\draw (.5,0) -- (.5,1);
}
\ = \ 
\tikzdiag{
	\draw (0,0) ..controls (0,.25) and (1,.25) .. (1,.5) node[tikzdot, near start]{} 
			 ..controls (1,.75) and (0,.75) .. (0,1)  ;
	\draw (1,0) ..controls (1,.25) and (0,.25) .. (0,.5) node[tikzdot, pos=1]{}
	..controls (0,.75) and (1,.75) .. (1,1)  ;
}
\ -  \ 
\tikzdiag{
	\draw (0,0) ..controls (0,.25) and (1,.25) .. (1,.5) 
			 ..controls (1,.75) and (0,.75) .. (0,1)  ;
	\draw (1,0) ..controls (1,.25) and (0,.25) .. (0,.5) node[tikzdot, pos=1]{}
	..controls (0,.75) and (1,.75) .. (1,1)  node[tikzdot, near end]{}  ;
}
\end{equation}
which follows easily from \cref{eq:nhR2andR3} and \cref{eq:nhdotslide}. 
We also observe 
\begin{equation}\label{eq:dottednailslide}
	\tikzdiagh{0}{
		\draw (.5,-.5) .. controls (.5,-.3) .. (0,-.1) .. controls (1,.5) ..  (1,1) -- (1,1.5) node[tikzdot,midway]{};
		\draw[stdhl] (1,-.5) node[below]{\small $1$} -- (1,0) .. controls (1,.25) .. (0,.5)
				.. controls (.5,.75) .. (.5,1) -- (.5,1.5);
		\draw[fill=white, color=white] (-.1,.5) circle (.1cm);
		\draw[vstdhl] (0,-.5)  node[below]{\small $\lambda$} -- (0,1.5) node[pos=.2,nail]{};
	}
	\ \overset{\eqref{eq:redR2}}{=} \ 
	\tikzdiagh{0}{
		\draw (.5,-.5) .. controls (.5,-.3) .. (0,-.1) .. controls (.75,.25) ..  (.75,.5) 
			.. controls (.75,.75) and (.25,.75) .. (.25,1) .. controls (.25,1.25) and (1,1.25) ..
			 (1,1.5);
		\draw[stdhl] (1,-.5) node[below]{\small $1$} -- (1,0) .. controls (1,.25) .. (0,.5)
				.. controls (.5,.75) .. (.5,1) -- (.5,1.5);
		\draw[fill=white, color=white] (-.1,.5) circle (.1cm);
		\draw[vstdhl] (0,-.5)  node[below]{\small $\lambda$} -- (0,1.5) node[pos=.2,nail]{};
	}
	\ \overset{\eqref{eq:nailslidedcross}}{=} \ 
	\tikzdiagh{0}{
		\draw (.5,-.5) .. controls (.5,0) and (.75,0) ..  (.75,.5) 
			.. controls (.75,.75)  .. (0,1) .. controls  (1,1.25) ..
			 (1,1.5);
		\draw[stdhl] (1,-.5) node[below]{\small $1$} -- (1,0) .. controls (1,.25) .. (0,.5)
				.. controls (.5,.75) .. (.5,1) -- (.5,1.5);
		\draw[fill=white, color=white] (-.1,.5) circle (.1cm);
		\draw[vstdhl] (0,-.5)  node[below]{\small $\lambda$} -- (0,1.5) node[pos=.75,nail]{};
	}
\end{equation}
Then, we compute 
\begin{align*}
\tikzdiagh{0}{
	\draw (0,-1) .. controls (0,-.875) .. (-.5,-.75) .. controls (.75,-.25) .. (.75,0) -- (.75,1);
	\draw (.25,-1) .. controls (.25,-.625) .. (-.5,-.375) .. controls (.5,-.25) .. (.5,0) -- (.5,1);
	\draw [stdhl] (.75, -1)  node[below]{\small $1$}  .. controls (.75,-.5) .. (-.5,0) .. controls (0,.5) .. (0,1);
	\draw[fill=white, color=white] (-.6,0) circle (.1cm);
	\draw[vstdhl] (-.5,-1)  node[below]{\small $\lambda$} -- (-.5,1) node[pos=.125,nail]{} node[pos=.  3125,nail]{};
}
\ \overset{\eqref{eq:nhdoublecrossingid}}{=} \ 
\tikzdiagh{0}{
	\draw (0,-1) .. controls (0,-.875) .. (-.5,-.75) .. controls (.75,-.25) .. (.75,0)
		 .. controls (.75,.25) and (.5,.25) .. (.5,.5) node[tikzdot, pos=1]{}
		 .. controls (.5,.75) and (.75,.75) .. (.75,1);
	\draw (.25,-1) .. controls (.25,-.625) .. (-.5,-.375) .. controls (.5,-.25) .. (.5,0) 
		 .. controls (.5,.25) and (.75,.25) .. (.75,.5) node[tikzdot, near start]{} 
		 .. controls (.75,.75) and (.5,.75) .. (.5,1);
	\draw [stdhl] (.75, -1)   node[below]{\small $1$}  .. controls (.75,-.5) .. (-.5,0) .. controls (0,.5) .. (0,1);
	\draw[fill=white, color=white] (-.6,0) circle (.1cm);
	\draw[vstdhl] (-.5,-1)  node[below]{\small $\lambda$} -- (-.5,1) node[pos=.125,nail]{} node[pos=.  3125,nail]{};
}
\ - \ 
\tikzdiagh{0}{
	\draw (0,-1) .. controls (0,-.875) .. (-.5,-.75) .. controls (.75,-.25) .. (.75,0)
		 .. controls (.75,.25) and (.5,.25) .. (.5,.5) node[tikzdot, pos=1]{}
		 .. controls (.5,.75) and (.75,.75) .. (.75,1) node[tikzdot, near end]{} ;
	\draw (.25,-1) .. controls (.25,-.625) .. (-.5,-.375) .. controls (.5,-.25) .. (.5,0) 
		 .. controls (.5,.25) and (.75,.25) .. (.75,.5) 
		 .. controls (.75,.75) and (.5,.75) .. (.5,1);
	\draw [stdhl] (.75, -1)   node[below]{\small $1$}  .. controls (.75,-.5) .. (-.5,0) .. controls (0,.5) .. (0,1);
	\draw[fill=white, color=white] (-.6,0) circle (.1cm);
	\draw[vstdhl] (-.5,-1)  node[below]{\small $\lambda$} -- (-.5,1) node[pos=.125,nail]{} node[pos=.  3125,nail]{};
}
\end{align*}
and 
\begin{align*}
\tikzdiagh[yscale=1.25]{0}{
	\draw (0,-1) .. controls (0,-.875) .. (-.5,-.75) .. controls (.75,-.25) .. (.75,0)
		 .. controls (.75,.25) and (.5,.25) .. (.5,.5) node[tikzdot, pos=1]{}
		 .. controls (.5,.75) and (.75,.75) .. (.75,1) node[tikzdot, near end]{} ;
	\draw (.25,-1) .. controls (.25,-.625) .. (-.5,-.375) .. controls (.5,-.25) .. (.5,0) 
		 .. controls (.5,.25) and (.75,.25) .. (.75,.5) 
		 .. controls (.75,.75) and (.5,.75) .. (.5,1);
	\draw [stdhl] (.75, -1)   node[below]{\small $1$}  .. controls (.75,-.5) .. (-.5,0) .. controls (0,.5) .. (0,1);
	\draw[fill=white, color=white] (-.6,0) circle (.1cm);
	\draw[vstdhl] (-.5,-1)  node[below]{\small $\lambda$} -- (-.5,1) node[pos=.125,nail]{} node[pos=.  3125,nail]{};
}
\ \overset{(\ref{eq:crossingslidered})}{=} \ 
\tikzdiagh[yscale=1.25]{0}{
	\draw (0,-1) .. controls (0,-.875) .. (-.5,-.75) .. controls (.25,-.5) .. (.25,-.375) 
		.. controls (.25,0) and (.5,0) .. (.5,.5)  node[tikzdot, pos=1]{}
		.. controls (.5,.75) and (.75,.75) .. (.75,1) node[tikzdot, near end]{} ;
	\draw (.25,-1) .. controls (.25,-.625) .. (-.5,-.375) .. controls (.75,-.25) .. (.75,0) -- (.75,.5)
		 .. controls (.75,.75) and (.5,.75) .. (.5,1);
	\draw [stdhl] (.75, -1)   node[below]{\small $1$}  .. controls (.75,0) .. (-.5,.25) .. controls (0,.5) .. (0,1);
	\draw[fill=white, color=white] (-.6,.25) circle (.1cm);
	\draw[vstdhl] (-.5,-1)  node[below]{\small $\lambda$} -- (-.5,1) node[pos=.125,nail]{} node[pos=.  3125,nail]{};
}
\ \overset{(\ref{eq:relNail})}{=} - \ 
\tikzdiagh[yscale=1.25]{0}{
	\draw (.25,-1) .. controls (.25,-.875) .. (-.5,-.75)
		.. controls (.75,-.25) .. (.75,.5)
		 .. controls (.75,.75) and (.5,.75) .. (.5,1);
	\draw (0,-1) -- (0,-.75) .. controls (0,-.625) .. (-.5,-.375)
		.. controls (.5,0) .. (.5,.5) node[tikzdot, pos=1]{}
		.. controls (.5,.75) and (.75,.75) .. (.75,1) node[tikzdot, near end]{} ;
	\draw [stdhl] (.75, -1)   node[below]{\small $1$}  .. controls (.75,0) .. (-.5,.25) .. controls (0,.5) .. (0,1);
	\draw[fill=white, color=white] (-.6,.25) circle (.1cm);
	\draw[vstdhl] (-.5,-1)  node[below]{\small $\lambda$} -- (-.5,1) node[pos=.125,nail]{} node[pos=.  3125,nail]{};
}
\end{align*}
Thus, using \cref{eq:dottednailslide} we obtain
\begin{align*}
\tikzdiagh{0}{
	\draw (0,-1) .. controls (0,-.875) .. (-.5,-.75) .. controls (.75,-.25) .. (.75,0) -- (.75,1.5);
	\draw (.25,-1) .. controls (.25,-.625) .. (-.5,-.375) .. controls (.5,-.25) .. (.5,0) -- (.5,1.5);
	\draw [stdhl] (.75, -1)  node[below]{\small $1$}  .. controls (.75,-.5) .. (-.5,0) .. controls (0,.5) .. (0,1.5);
	\draw[fill=white, color=white] (-.6,0) circle (.1cm);
	\draw[vstdhl] (-.5,-1)  node[below]{\small $\lambda$} -- (-.5,1) node[pos=.125,nail]{} node[pos=.  3125,nail]{} -- (-.5,1.5);
}
\ = \ 
\tikzdiagh{0}{
	\draw (0,-1) .. controls (0,-.75) .. (-.5,-.5) 
		.. controls (.75,.25) .. (.75,.5)
		.. controls (.75,.75) and (.5,.75) .. (.5,1) node[tikzdot, pos=1]{} 
		.. controls (.5,1.25) and (.75,1.25) .. (.75,1.5) ;
	\draw (.25,-1) .. controls (.25,-.5) and (.75,-.5) .. (.75,0)
		.. controls (.75,.25)  .. (-.5,.5)
		.. controls (.75,.75) .. (.75,1)
		.. controls (.75,1.25) and (.5,1.25) .. (.5,1.5);
	\draw [stdhl] (.75, -1)  node[below]{\small $1$}  .. controls (.75,-.5) .. (-.5,0) .. controls (0,.5) .. (0,1) -- (0,1.5);
	\draw[fill=white, color=white] (-.6,0) circle (.1cm);
	\draw[vstdhl] (-.5,-1)  node[below]{\small $\lambda$} -- (-.5,1) node[pos=.25,nail]{} node[pos=.75,nail]{} -- (-.5,1.5);
}
\ + \ 
\tikzdiagh{0}{
	\draw (.25,-1) .. controls (.25,-.75) .. (-.5,-.5)
		.. controls (.75,.25) .. (.75,.5)
		.. controls (.75,.75) and (.5,.75) .. (.5,1) -- (.5,1.5);
	\draw (0,-1) .. controls (0,-.5) and (.75,-.5) .. (.75,0)
		.. controls (.75,.25) .. (-.5,.5)
		.. controls (.75,.75) .. (.75,1) -- (.75,1.5) node[tikzdot, midway]{} ;
	\draw [stdhl] (.75, -1)  node[below]{\small $1$}  .. controls (.75,-.5) .. (-.5,0) .. controls (0,.5) .. (0,1) --  (0,1.5);
	\draw[fill=white, color=white] (-.6,0) circle (.1cm);
	\draw[vstdhl] (-.5,-1)  node[below]{\small $\lambda$} -- (-.5,1) node[pos=.25,nail]{} node[pos=.75,nail]{} -- (-.5,1.5);
}
\end{align*}
Consequently, using \cref{thm:X0basis}, we deduce that $X_k$ is generated as left $(T_b^{\lambda,r},0)$-module by the elements
\begin{align*}
\tikzdiagh{0}{
	\draw (0,0) .. controls (0,.5) and (.5,.5) .. (.5,1);
	\node at(.75,.75) {\tiny$\dots$};
	\draw (.5,0) .. controls (.5,.5) and (1,.5) ..  (1,1);
	\draw[decoration={brace,mirror,raise=-8pt},decorate]  (-.1,-.35) -- node {$k$} (.6,-.35);
	\draw[stdhl] (1,0) node[below]{\small $1$} .. controls (1,.25) .. (-.5,.5)
			.. controls (0,.75) .. (0,1);
	\draw[fill=white, color=white] (-.6,.5) circle (.1cm);
	\draw[vstdhl] (-.5,0)  node[below]{\small $\lambda$} -- (-.5,1);
  }
&\otimes \bar  1_{\ell,\rho},
&
\tikzdiagh{0}{
	\draw (0,-.5) .. controls (0,0) and (.75,0) .. (.75,1);
	\node at(1,.75) {\tiny$\dots$};
	\draw (.5,-.5) .. controls (.5,0) and (1.25,0) .. (1.25,1);
	\draw (.75,-.5) .. controls (.75,-.25) .. (-.5,0) .. controls (.5,.5) ..  (.5,1);
	\draw (1,-.5) .. controls (1,0) and (1.5,0) .. (1.5,1);
	\node at(1.75,.75) {\tiny$\dots$};
	\draw (1.5,-.5) .. controls (1.5,0) and (2,0) .. (2,1);
	\draw[decoration={brace,mirror,raise=-8pt},decorate]  (-.1,-.85) -- node {$t$} (.6,-.85);
	\draw[stdhl] (2,-.5) node[below]{\small $1$}  -- (2,0) .. controls (2,.25) .. (-.5,.5)
			.. controls (0,.75) .. (0,1);
	\draw[fill=white, color=white] (-.6,.5) circle (.1cm);
	\draw[vstdhl] (-.5,-.5)  node[below]{\small $\lambda$} -- (-.5,1) node[pos=.33,nail]{};
}
&\otimes \bar  1_{\ell,\rho},
\end{align*}
for all $0 \leq t \leq k-1$. In particular, $\gamma_k$ is surjective. 
\end{proof}

\begin{lem}\label{prop:X0decomp}
Suppose $r=1$ and $\ell = 0$. 
As a $\bZ\times\bZ^2$-graded $\Bbbk$-module, $X 1_{k,0}$ admits a decomposition
\[
X 1_{k,0} \cong
\lambda^{-1} q^{2k}(T_{k}^{\lambda,0}) \oplus
 \ssbigoplus{0 \leq t < k \\ p \geq 0} \bigl( 
q^{2p+1-2(k-t)} (X 1_{k-1,0}) \oplus \lambda^2 q^{2p+1-2(k+t)}(X 1_{k-1,0})[1]
\bigr).
\]
\end{lem}

\begin{proof}
It follows from \cref{thm:X0basis} that we have a decomposition
\[
\tikzdiag[xscale=.75]{
 	\draw [vstdhl] (-.5,0) node[below]{\small $\lambda$} -- (-.5,1);
 	\draw (0,0) -- (0,1);
 	\node 	at(.5,.25) {\tiny $\dots$};
 	\draw (1,0) -- (1,1);
 	\draw [stdhl] (1.5,0) node[below]{\small $1$} -- (1.5,1);
	\filldraw [fill=white, draw=black] (-.75,.5) rectangle (1.75,1) node[midway] { $X_{k}$};
}
\ \cong \ 
\tikzdiag[xscale=.75]{
 	\draw (0,-.5) .. controls (0,-.25) and (.5,-.25) .. (.5,0) .. controls (.5,.25) and (0,.25).. (0,.5) --  (0,1);
 	\node 	at(1,0) {\tiny $\dots$};
 	\draw (1,-.5) .. controls (1,-.25) and (1.5,-.25) .. (1.5,0) .. controls (1.5,.25) and (1,.25) .. (1,.5) -- (1,1);
 	\draw [stdhl] (1.5,-.5) node[below]{\small $1$} .. controls (1.5,-.25) .. (-.5,0) .. controls (1.5,.25) .. (1.5,.5) -- (1.5,1);
	\draw[fill=white, color=white] (-.8,0) circle (.2cm);
 	\draw [vstdhl] (-.5,-.5) node[below]{\small $\lambda$} -- (-.5,1);
	\filldraw [fill=white, draw=black] (-.75,.5) rectangle (1.25,1) node[midway] { $T_{k}$};
}
\oplus 
 \ssbigoplus{0 \leq t < k \\ p \geq 0} \left( 
 \tikzdiag[xscale=.75]{
 	\draw [vstdhl] (-.5,-.5) node[below]{\small $\lambda$} -- (-.5,1);
 	\draw (0,-.5) -- (0,1);
 	\node 	at(.5,.25) {\tiny $\dots$};
 	\draw (1,-.5) -- (1,1);
	\draw[decoration={brace,mirror,raise=-8pt},decorate]  (-.1,-.85) -- node {$t$} (1.1,-.85);
 	\draw (2,-.5)  .. controls (2,0) and (1.5,0) .. (1.5,.5) -- (1.5,1);
 	\node 	at(2,.25) {\tiny $\dots$};
 	\draw (3,-.5)  .. controls (3,0) and (2.5,0) ..  (2.5,.5) -- (2.5,1);
 	\draw (1.5,-.5) .. controls (1.5,0) and (3.5,0) ..(3.5,.5) -- (3.5,1) node[midway,tikzdot]{} node[midway, xshift=1.5ex, yshift=1ex]{\small $p$};
 	\draw [stdhl] (3.5,-.5) node[below]{\small $1$} .. controls (3.5,0) and (3,0) .. (3,.5) -- (3,1);
	\filldraw [fill=white, draw=black] (-.75,.5) rectangle (3.25,1) node[midway] { $X_{k-1}$};
}
\oplus
 \tikzdiag[xscale=.75]{
 	\draw (0,-.5) -- (0,1);
 	\node 	at(.5,.25) {\tiny $\dots$};
 	\draw (1,-.5) -- (1,1);
	\draw[decoration={brace,mirror,raise=-8pt},decorate]  (-.1,-.85) -- node {$t$} (1.1,-.85);
 	\draw (2,-.5)  .. controls (2,0) and (1.5,0) .. (1.5,.5) -- (1.5,1);
 	\node 	at(2,.25) {\tiny $\dots$};
 	\draw (3,-.5)  .. controls (3,0) and (2.5,0) ..  (2.5,.5) -- (2.5,1);
 	\draw (1.5,-.5) .. controls (1.5,-.25) ..(-.5,0) .. controls (3.5,.25)  ..
 		(3.5,.5) -- (3.5,1) node[midway,tikzdot]{} node[midway, xshift=1.5ex, yshift=1ex]{\small $p$};
 	\draw [stdhl] (3.5,-.5) node[below]{\small $1$} .. controls (3.5,0) and (3,0) .. (3,.5) -- (3,1);
	\draw[fill=white, color=white] (-.8,0) circle (.2cm);
 	\draw [vstdhl] (-.5,-.5) node[below]{\small $\lambda$} -- (-.5,1) node[pos=.33,nail]{};
	\filldraw [fill=white, draw=black] (-.75,.5) rectangle (3.25,1) node[midway] { $X_{k-1}$};
}
 \right)
\]
concluding the proof. 
\end{proof}

\begin{citelem}{lem:sesX0}
The sequence
\[
0 \rightarrow Y^1_k \xrightarrow{\imath_k} Y^0_k \xrightarrow{\gamma_k} X_k \rightarrow 0,
\]
is a short exact sequence of left $(T^{\lambda,r}, 0)$-modules. 
\end{citelem}

\begin{proof}
Since we already have a complex with an injection and a surjection, it is enough to show that 
\[
 \gdim X_k = 
\gdim Y^0_k - \gdim Y^1_k,
\]
where $\gdim$ is the graded dimension in the form of a Laurent series in $\bN\llbracket h^{\pm 1}, \lambda^{\pm 1}, q^{\pm 1} \rrbracket$. 
We will show this by induction on $k$. When $k=0$, this is immediate. Suppose it is true for $k$, and we will show it for $k+1$.  

Let
\[
[\beta+t]_q^h := \frac{\lambda^{-1}q^{-t} + h \lambda q^t }{q^{-1}-q} =q \frac{\lambda^{-1}q^{-t} + h \lambda q^t }{1-q^2}.
\]
Note that
\begin{align}
[k+1]_q [\beta+t+k]_q^h &= \sum_{r=0}^{k} [\beta+t-2r]_q^h, \label{eq:defqnbrprod}
\\
[k+1]_q &= q [k]_q + q^{-k}, \label{eq:qnbrplus}
\\
[\beta-k+1]_q^h &=  q^{-1} [\beta-k]_q^h - h \lambda q^{-k}, \label{eq:defqnbrplus}
\end{align}
and 
\begin{align}
[k+1]_q [\beta-k+1]_q^h  &= [k]_q [\beta-k]_q^h + q^{-1-k}[\beta-k]_q^h - h \lambda q^{-k} [k+1]_q. \label{eq:qnbrtimesdefqnbr}
\end{align}

We first restrict to the case $\ell = 0$ and $r=1$.  
By \cref{prop:X0decomp} using \eqref{eq:defqnbrprod}, followed by the induction hypothesis, we have
\begin{align*}
\gdim X 1_{k+1,0} =&  \lambda^{-1} q^{2(k+1)} \gdim T_{k+1}^{\lambda,0} +  \lambda q^{-2k} [k+1]_q[\beta-k]^h_q  \gdim X 1_{k,0} \\
=& \lambda^{-1} q^{2(k+1)} \gdim T_{k+1}^{\lambda,0} +  \lambda q^{-2k} [k+1]_q[\beta-k]^h_q 
\\
&\times \bigl( (\lambda^{-1}q^k + h \lambda q [k]_q ) \gdim T^{\lambda,1}_{k} 1_{0,k} - h \lambda q^2 [k]_q  \gdim T_k^{\lambda,1} 1_{1,k-1} \bigr).
\end{align*}
By definition, we have
\begin{align*}
\gdim Y_{k+1,0}^0  &= \bigl( \lambda^{-1} q^{k+1} + h \lambda q  [k+1]_q \bigr)  \gdim T^{\lambda,1}_{k+1} 1_{0,k+1}, 
\\
\gdim Y_{k+1,0}^1 &= h \lambda q^2 [k+1]_q \gdim T^{\lambda,1}_{k+1} 1_{1,k}.
\end{align*}
By \cref{prop:Tdecomp}, we have
\begin{align*}
\gdim T^{\lambda,1}_{k+1} 1_{0,k+1} = & 
q^{k+1} \gdim T^{\lambda,0}_{k+1} + \lambda q^{-2k} [k+1]_q[\beta+1-k]_q^h \gdim T^{\lambda,1}_k 1_{0,k},
\\
\gdim T^{\lambda,1}_{k+1} 1_{1,k} = & 
q^{k} \gdim T^{\lambda,0}_{k+1} 
+ \lambda q^{-2k} [\beta]_q^h    \gdim T^{\lambda,1}_{k} 1_{0,k} \\
&+ \lambda q^{-2k} [k]_q [\beta-k]_q^h \gdim T^{\lambda,1}_k 1_{1,k-1}.
\end{align*}
We now gather by $ \gdim T_{k+1}^{\lambda,0}$, $ \gdim T^{\lambda,1}_{k} 1_{0,k}$ and $\gdim T^{\lambda,1}_k 1_{1,k-1}$. 
For $\gdim T_{k+1}^{\lambda,0}$, we verify that
\[
	 \lambda^{-1} q^{2(k+1)} = \bigl( \lambda^{-1} q^{k+1} + \lambda q  [k+1]_q \bigr) q^{k+1} - \lambda q^2 [k+1]_q  q^{k}.
\]
Gathering by $ \gdim T^{\lambda,1}_{k} 1_{0,k}$, we obtain on one hand
\begin{align}
\begin{split}
&\lambda q^{-2k} [k+1]_q[\beta-k]^h_q  (\lambda^{-1}q^k + h \lambda q [k]_q ) 
 \\
&= q^{-k}[k+1]_q [\beta-k]_q^h + h \lambda^2 q^{1-2k}[k]_q[k+1]_q[\beta-k]_q^h,
\end{split} \label{eq:dim0kLHS}
\end{align}
and on the other hand
\begin{align*}
&\bigl( \lambda^{-1} q^{k+1} + h\lambda q  [k+1]_q \bigr) \lambda q^{-2k} [k+1]_q[\beta+1-k]_q^h
-  h\lambda q^2 [k+1]_q   \lambda q^{-2k} [\beta]_q^h 
\\
&= q^{1-k}[k+1]_q [\beta-k+1]_q^h + h \lambda^2 q^{1-2k} [k+1]_q [k+1_q] [\beta-k+1]_q^h
\\
 &\quad - h \lambda^2 q^{2-2k}[k+1]_q [\beta]_q^h 
\\
&=
q^{-k} [\beta-k]_q^h [k+1]_q - h \lambda q^{1-2k} [k+1]_q 
\\
&\quad + h \lambda^2 q^{1-2k} [k+1]_q \bigl( [k]_q[\beta-k]_q^h + q^{-1-k}[\beta-k]_q^h - h\lambda q^{-k}[k+1]_q \bigr)
\\
&\quad - h \lambda^2 q^{2-2k}[k+1]_q [\beta]_q^h,
\end{align*}
using \eqref{eq:defqnbrplus} and \eqref{eq:qnbrtimesdefqnbr}. We remark that the first and third terms coincide with \eqref{eq:dim0kLHS}. We gather the remaining terms, putting $h\lambda q^{-2k}[k+1]_q$ in evidence, so that we obtain
\begin{align*}
 & - q  + \lambda q^{-k}[\beta-k]_q^h - h\lambda^2 q^{1-k}[k+1]_q - \lambda q^{2} [\beta]_q^h
\\
&= \frac{1}{q^{-1}-q}\bigl( -q(q^{-1}-q) + \lambda q^{-k}(\lambda^{-1}q^k + h \lambda q^{-k}) - h \lambda^2q^{1-k}(q^{-k-1}-q^{k+1}) - \lambda q^2 (\lambda^{-1}+h\lambda) \bigr)
\\&= 0.
\end{align*}

Finally, for $\gdim T^{\lambda,1}_k 1_{1,k-1}$, we verify that
\begin{align*}
 \lambda q^{-2k} [k+1]_q[\beta-k]^h_q   (- h \lambda q^2 [k]_q  )
 &= 
-  h \lambda q^2 [k+1]_q   \lambda q^{-2k} [k]_q [\beta-k]_q^h, 
\end{align*}
concluding the proof in the case $\ell = 0$ and $r=1$. 

The case $\ell > 0$ comes from an induction on $\ell$ and using the case $\ell = 0$. Using a similar decomposition as in \cref{prop:X0decomp}, we obtain
\begin{align*}
\gdim X_{k,\ell} =& \lambda^{-1} q^{2k+\ell} \gdim T_{k+\ell}^{\lambda,0} + \lambda q^{-2\ell-2k-2}  [k]_q[\beta-k+1]_q^h  \gdim X_{k-1, \ell}
\\
&+ \lambda q^{-2k-2\ell} [\ell]_q [\beta-2k-\ell]_q^h \gdim X_{k, \ell-1},
\end{align*}
where $X_{k,\ell} := X_k 1_{k,\ell}$. Similarly, one can compute
\begin{align*}
\gdim T^{\lambda,1}_{k+\ell} 1_{0,k+\ell} 
 =&
q^{k+\ell} \gdim T_{k+\ell}^{\lambda,0} + \lambda q^{-2k-2\ell} [k+\ell]_q[\beta-k-\ell]_q^h \gdim T^{\lambda,1}_{k+\ell-1} 1_{0,k+\ell-1},
\\
\gdim T^{\lambda,1}_{k+\ell} 1_{1,k+\ell-1} 
=& 
q^{k+\ell-1} \gdim T_{k+\ell}^{\lambda,0} + \lambda q^{2-2(k+\ell)} [\beta]_q^h \gdim T_{k+\ell-1}^{\lambda,1} 1_{0,k+\ell-1}
\\
& + \lambda q^{-2k-2\ell} [k+\ell-1]_q [\beta-k-\ell-1]_q^h \gdim T_{k+\ell-1}^{\lambda,1} 1_{1,k+\ell-2}.
\end{align*}
By the same reasons as above, the part in $\gdim T_{k+\ell}^{\lambda,0}$ annihilates each others.
By induction hypothesis we know that
\[
\gdim X_{k, \ell-1} = 
\bigl( \lambda^{-1} q^{k} + h \lambda q  [k]_q \bigr)  \gdim T^{\lambda,1}_{k+\ell-1} 1_{0,k+\ell-1} 
-
 h \lambda q^2 [k]_q \gdim T^{\lambda,1}_{k+\ell-1} 1_{1,k+\ell-2}.
\]
Using the fact that
\begin{align*}
[k+\ell]_q[\beta-k-\ell]_q^h &= [k]_q[\beta-k]_q^h + [\ell]_q[\beta-2k-\ell]_q^h, \\
[k+\ell-1]_q[\beta-k-\ell-1]_q^h &= [k-1]_q[\beta-k+1]_q^h + [\ell]_q[\beta-2k-\ell]_q^h,
\end{align*}
together with the induction hypothesis, 
we cancel the part given by $\gdim (X_{k, \ell-1})$ in $X_{k,\ell}$ with the part given by
\[
  \lambda q^{-2k-2\ell}  [\ell]_q[\beta-2k-\ell]_q^h  \gdim T^{\lambda,1}_{k+\ell-1} 1_{0,k+\ell-1}
\]
in $Y_{k, \ell}^0$ minus the part given by
\[
 \lambda q^{-2k-2\ell}   [\ell]_q[\beta-2k-\ell]_q^h  \gdim T_{k+\ell-1}^{\lambda,1} 1_{1,k+\ell-2}.
\]
in $Y_{k,\ell}^1$. 

The remaining terms yields the same computations as for the case $\ell = 0$ (replacing $k+1$ by $k$), but shifting everything by $q^{-2\ell}$. Thus, it concludes the case $\ell > 0$.

The general case follows from a similar argument, using the fact that $X$ decomposes similarly to $T^{\lambda,r}_b$ whenever $r > 1$, that is as in \cref{prop:Tdecomp}, replacing all $T$ by $X$. We leave the details to the reader. 
\end{proof}

\subsection{Proofs of \cref{sec:catTLB}}\label{sec:proofsofcatTLB}

\begin{citelem}{lem:Xklgenerated}
As a right $(T^{\lambda,r},0)$-module, $1_{1,k+\ell-1,\rho}X$ is generated by the elements
\begin{align}\label{eq:Xklgenerator}
\tikzdiagh{0}{
	\draw (1.25,0) .. controls (1.25,.5) .. (-.5,.75) .. controls (0,.875) .. (0,1);
	\draw[stdhl] (0,0) node[below]{\small $1$} .. controls (0,.1) .. (-.5,.25) .. controls (.5,.7) .. (.5,1);
	\draw[fill=white, color=white] (-.6,.25) circle (.1cm);
	\draw[vstdhl] (-.5,0) node[below]{\small $\lambda$} --(-.5,1) node[pos=.75,nail]{};
	\draw (.5,0) .. controls (.5,.5) and (.75,.5) .. (.75,1);
	\node at (.75,.15){\tiny $\dots$};
	\draw (1,0) .. controls (1,.5) and (1.25,.5)..  (1.25,1);
	\draw[decoration={brace,mirror,raise=-8pt},decorate]  (.4,-.35) -- node { \small $k-1$} (1.1,-.35);
}
\otimes \bar  1_{\ell,\rho},
&&\text{and} &&
\tikzdiagh{0}{
	\draw (.5,0) .. controls (.5,.5) and (0,.5) .. (0,1);
	\draw[stdhl] (0,0) node[below]{\small $1$} .. controls (0,.1) .. (-.5,.25) .. controls (.5,.7) .. (.5,1);
	\draw[fill=white, color=white] (-.6,.25) circle (.1cm);
	\draw[vstdhl] (-.5,0) node[below]{\small $\lambda$} --(-.5,1);
}
\otimes \bar  1_{\ell+k-1,\rho}.
\end{align}
\end{citelem}

\begin{proof}
We prove this claim using an induction on $k$. The case $k=1$ is obvious. We suppose it is true for $k-1$, and thus it is enough to show that we can generate the element:
\[
\tikzdiagh{0}{
	\draw (1.25,0) .. controls (1.25,.5) .. (-.5,.75) .. controls (0,.875) .. (0,1);
	\draw[stdhl] (0,0) node[below]{\small $1$} .. controls (0,.1) .. (-.5,.25) .. controls (.5,.7) .. (.5,1);
	\draw[fill=white, color=white] (-.6,.25) circle (.1cm);
	\draw[vstdhl] (-.5,0) node[below]{\small $\lambda$} --(-.5,1) node[pos=.75,nail]{};
	\draw (.5,0) .. controls (.5,.5) and (.75,.5) .. (.75,1);
	\node at (.75,.15){\tiny $\dots$};
	\draw (1,0) .. controls (1,.5) and (1.25,.5)..  (1.25,1);
	\draw (1.5,0) -- (1.5,1);
	\draw[decoration={brace,mirror,raise=-8pt},decorate]  (.4,-.35) -- node { \small $k-2$} (1.1,-.35);
}
\otimes \bar  1_{\ell,\rho}.
\]
Using \cref{eq:nhdotslide}, we have
\[
  \tikzdiagh{0}{
    \draw (1.25,0) .. controls (1.25,.5) .. (-.5,.75) .. controls (0,.875) .. (0,1);
    \draw[stdhl] (0,0) node[below]{\small $1$} .. controls (0,.1) .. (-.5,.25) .. controls (.5,.7) .. (.5,1);
    \draw[fill=white, color=white] (-.6,.25) circle (.1cm);
    \draw[vstdhl] (-.5,0) node[below]{\small $\lambda$} --(-.5,1) node[pos=.75,nail]{};
    \draw (.5,0) .. controls (.5,.5) and (.75,.5) .. (.75,1);
    \node at (.75,.15){\tiny $\dots$};
    \draw (1,0) .. controls (1,.5) and (1.25,.5)..  (1.25,1);
    \draw (1.5,0) -- (1.5,1);
    \draw[decoration={brace,mirror,raise=-8pt},decorate]  (.4,-.35) -- node { \small $k-2$} (1.1,-.35);
  }
  =
  \tikzdiagh{0}{
    \draw (1.6,0) .. controls (1.6,.5) .. (-.5,.75) node[pos=.52,tikzdot]{} .. controls (0,.875) .. (0,1);
    \draw[stdhl] (0,0) node[below]{\small $1$} .. controls (0,.1) .. (-.5,.25) .. controls (.5,.7) .. (.5,1);
    \draw[fill=white, color=white] (-.6,.25) circle (.1cm);
    \draw[vstdhl] (-.5,0) node[below]{\small $\lambda$} --(-.5,1) node[pos=.75,nail]{};
    \draw (.5,0) .. controls (.5,.5) and (.75,.5) .. (.75,1);
    \node at (.75,.15){\tiny $\dots$};
    \draw (1,0) .. controls (1,.5) and (1.25,.5)..  (1.25,1);
    \draw (1.35,0) .. controls (1.35,.5) and (1.6,.5)..  (1.6,1);
    \draw[decoration={brace,mirror,raise=-8pt},decorate]  (.4,-.35) -- node { \small $k-2$} (1.1,-.35);
  }
  -
  \tikzdiagh{0}{
    \draw (1.25,0) .. controls (1.25,.5) .. (-.5,.75) node[pos=.2,tikzdot]{} .. controls (0,.875) .. (0,1);
    \draw[stdhl] (0,0) node[below]{\small $1$} .. controls (0,.1) .. (-.5,.25) .. controls (.5,.7) .. (.5,1);
    \draw[fill=white, color=white] (-.6,.25) circle (.1cm);
    \draw[vstdhl] (-.5,0) node[below]{\small $\lambda$} --(-.5,1) node[pos=.75,nail]{};
    \draw (.5,0) .. controls (.5,.5) and (.75,.5) .. (.75,1);
    \node at (.75,.15){\tiny $\dots$};
    \draw (1,0) .. controls (1,.5) and (1.25,.5)..  (1.25,1);
    \draw[decoration={brace,mirror,raise=-8pt},decorate]  (.4,-.35) -- node { \small $k-1$} (1.1,-.35);
  }
\]
The second term on the right-hand side is generated by the second element in \cref{eq:Xklgenerator}. For the first term of the right-hand side, we slide the dot to the left using repeatedly \cref{eq:nhdotslide}:
\[
  \tikzdiagh{0}{
    \draw (1.6,0) .. controls (1.6,.5) .. (-.5,.75) node[pos=.52,tikzdot]{} .. controls (0,.875) .. (0,1);
    \draw[stdhl] (0,0) node[below]{\small $1$} .. controls (0,.1) .. (-.5,.25) .. controls (.5,.7) .. (.5,1);
    \draw[fill=white, color=white] (-.6,.25) circle (.1cm);
    \draw[vstdhl] (-.5,0) node[below]{\small $\lambda$} --(-.5,1) node[pos=.75,nail]{};
    \draw (.5,0) .. controls (.5,.5) and (.75,.5) .. (.75,1);
    \node at (.75,.15){\tiny $\dots$};
    \draw (1,0) .. controls (1,.5) and (1.25,.5)..  (1.25,1);
    \draw (1.35,0) .. controls (1.35,.5) and (1.6,.5)..  (1.6,1);
    \draw[decoration={brace,mirror,raise=-8pt},decorate]  (.4,-.35) -- node { \small $k-2$} (1.1,-.35);
  }
  =
  \tikzdiagh{0}{
    \draw (1.3,0) .. controls (1.3,.5) .. (-.5,.75) node[pos=.75,tikzdot]{} .. controls (0,.875) .. (0,1);
    \draw[stdhl] (0,0) node[below]{\small $1$} .. controls (0,.1) .. (-.5,.25) .. controls (.5,.7) .. (.5,1);
    \draw[fill=white, color=white] (-.6,.25) circle (.1cm);
    \draw[vstdhl] (-.5,0) node[below]{\small $\lambda$} --(-.5,1) node[pos=.75,nail]{};
    \draw (.55,0) .. controls (.55,.5) and (.8,.5) .. (.8,1);
    \node at (.8,.15){\tiny $\dots$};
    \draw (1.05,0) .. controls (1.05,.5) and (1.3,.5)..  (1.3,1);
    \draw[decoration={brace,mirror,raise=-8pt},decorate]  (.4,-.35) -- node { \small $k-1$} (1.1,-.35);
  }
  -
  \sum_{j=1}^{k-2}
  \tikzdiagh{0}{
    \draw (1.25,0) .. controls (1.25,.5) .. (-.5,.75) .. controls (0,.875) .. (0,1);
    \draw[stdhl] (0,0) node[below]{\small $1$} .. controls (0,.1) .. (-.5,.25) .. controls (.5,.7) .. (.5,1);
    \draw[fill=white, color=white] (-.6,.25) circle (.1cm);
    \draw[vstdhl] (-.5,0) node[below]{\small $\lambda$} --(-.5,1) node[pos=.75,nail]{};
    \draw (.5,0) .. controls (.5,.5) and (.75,.5) .. (.75,1);
    \node at (.75,.15){\tiny $\dots$};
    \draw (1,0) .. controls (1,.5) and (1.25,.5)..  (1.25,1);
    \draw (2.25,0) .. controls (2.25,.5) and (1.5,.5) .. (1.5,1);
    \draw (1.5,0) .. controls (1.5,.5) and (1.75,.5) .. (1.75,1);
    \node at (1.75,.15){\tiny $\dots$};
    \draw (2,0) .. controls (2,.5) and (2.25,.5) .. (2.25,1);
    \draw[decoration={brace,mirror,raise=-8pt},decorate]  (1.4,-.35) -- node { \small $j$} (2.1,-.35);
  }  
\]
Because of the symmetric of \cref{eq:dottednailslide}, the first term on the right-hand side is generated by the second element in \cref{eq:Xklgenerator}. We now prove that every element of the sum on the right-hand side is generated by elements in \cref{eq:Xklgenerator}.

By applying the induction hypothesis, it suffices to show that for every $1 \leq j \leq k-2$, the elements
\begin{align*}
  \tikzdiagh{0}{
  \draw (.5,0) .. controls (.5,.5) and (0,.5) .. (0,1);
  \draw[stdhl] (0,0) node[below]{\small $1$} .. controls (0,.1) .. (-.5,.25) .. controls (.5,.7) .. (.5,1);
  \draw[fill=white, color=white] (-.6,.25) circle (.1cm);
  \draw[vstdhl] (-.5,0) node[below]{\small $\lambda$} --(-.5,1);
  \draw (.75,0) -- (.75,1);
  \draw (1.25,0) -- (1.25,1);
  \node at (1,.5){\tiny $\dots$};
  \draw (2.25,0) .. controls (2.25,.5) and (1.5,.5) .. (1.5,1);
  \draw (1.5,0) .. controls (1.5,.5) and (1.75,.5) .. (1.75,1);
  \node at (1.75,.15){\tiny $\dots$};
  \draw (2,0) .. controls (2,.5) and (2.25,.5) .. (2.25,1);
  \draw[decoration={brace,mirror,raise=-8pt},decorate]  (1.4,-.35) -- node { \small $j$} (2.1,-.35);
  }
  &&
     \text{and}
  &&
     \tikzdiagh{0}{
     \draw (2,0) .. controls (2,.5) .. (-.5,.75) .. controls (0,.875) .. (0,1);
     \draw[stdhl] (0,0) node[below]{\small $1$} .. controls (0,.1) .. (-.5,.25) .. controls (.5,.7) .. (.5,1);
     \draw[fill=white, color=white] (-.6,.25) circle (.1cm);
     \draw[vstdhl] (-.5,0) node[below]{\small $\lambda$} --(-.5,1) node[pos=.75,nail]{};
     \draw (.5,0) .. controls (.5,.5) and (.75,.5) .. (.75,1);
     \node at (.75,.15){\tiny $\dots$};
     \draw (1,0) .. controls (1,.5) and (1.25,.5) .. (1.25,1);
     \draw (1.25,0) .. controls (1.25,.5) and (1.75,.5) .. (1.75,1);
     \node at (1.5,.15){\tiny $\dots$};
     \draw (1.75,0) .. controls (1.75,.5) and (2.25,.5)..  (2.25,1);
     \draw (2.25,0) -- (2.25,.4);
     \draw (2.25,.4) .. controls (2.25,.75) and (1.5,.75) .. (1.5,1);
     \draw[decoration={brace,mirror,raise=-8pt},decorate]  (1.15,-.35) -- node { \small $j$} (1.85,-.35);
     }
\end{align*}
are in the right module generated be the elements in \cref{eq:Xklgenerator}, which is clear for the first diagram. Concerning the second one, we have by \cref{eq:nhdotslide}
\[
  \tikzdiagh{0}{
    \draw (2,0) .. controls (2,.5) .. (-.5,.75) .. controls (0,.875) .. (0,1);
    \draw[stdhl] (0,0) node[below]{\small $1$} .. controls (0,.1) .. (-.5,.25) .. controls (.5,.7) .. (.5,1);
    \draw[fill=white, color=white] (-.6,.25) circle (.1cm);
    \draw[vstdhl] (-.5,0) node[below]{\small $\lambda$} --(-.5,1) node[pos=.75,nail]{};
    \draw (.5,0) .. controls (.5,.5) and (.75,.5) .. (.75,1);
    \node at (.75,.15){\tiny $\dots$};
    \draw (1,0) .. controls (1,.5) and (1.25,.5) .. (1.25,1);
    \draw (1.25,0) .. controls (1.25,.5) and (1.75,.5) .. (1.75,1);
    \node at (1.5,.15){\tiny $\dots$};
    \draw (1.75,0) .. controls (1.75,.5) and (2.25,.5)..  (2.25,1);
    \draw (2.25,0) -- (2.25,.4);
    \draw (2.25,.4) .. controls (2.25,.75) and (1.5,.75) .. (1.5,1);
    \draw[decoration={brace,mirror,raise=-8pt},decorate]  (1.15,-.35) -- node { \small $j$} (1.85,-.35);
  }
  =
  \tikzdiagh{0}{
    \draw (2.25,0) .. controls (2.25,.5) .. (-.5,.75) node[pos=.39,tikzdot]{} .. controls (0,.875) .. (0,1);
    \draw[stdhl] (0,0) node[below]{\small $1$} .. controls (0,.1) .. (-.5,.25) .. controls (.5,.7) .. (.5,1);
    \draw[fill=white, color=white] (-.6,.25) circle (.1cm);
    \draw[vstdhl] (-.5,0) node[below]{\small $\lambda$} --(-.5,1) node[pos=.75,nail]{};
    \draw (.5,0) .. controls (.5,.5) and (.75,.5) .. (.75,1);
    \node at (.75,.15){\tiny $\dots$};
    \draw (1,0) .. controls (1,.5) and (1.25,.5) .. (1.25,1);
    \draw (1.25,0) .. controls (1.25,.5) and (1.75,.5) .. (1.75,1);
    \node at (1.5,.15){\tiny $\dots$};
    \draw (1.75,0) .. controls (1.75,.5) and (2.25,.5)..  (2.25,1);
    \draw (2,0) .. controls (2,.25) and (2.25,.25) .. (2.25,.5);
    \draw (2.25,.5) .. controls (2.25,.75) and (1.5,.75) .. (1.5,1);
    \draw[decoration={brace,mirror,raise=-8pt},decorate]  (1.15,-.35) -- node { \small $j$} (1.85,-.35);
  }
  -
  \tikzdiagh{0}{
    \draw (2.25,0) .. controls (2.25,.5) .. (-.5,.75) node[pos=.1,tikzdot]{} .. controls (0,.875) .. (0,1);
    \draw[stdhl] (0,0) node[below]{\small $1$} .. controls (0,.1) .. (-.5,.25) .. controls (.5,.7) .. (.5,1);
    \draw[fill=white, color=white] (-.6,.25) circle (.1cm);
    \draw[vstdhl] (-.5,0) node[below]{\small $\lambda$} --(-.5,1) node[pos=.75,nail]{};
    \draw (.5,0) .. controls (.5,.5) and (.75,.5) .. (.75,1);
    \node at (.75,.15){\tiny $\dots$};
    \draw (1,0) .. controls (1,.5) and (1.25,.5) .. (1.25,1);
    \draw (1.25,0) .. controls (1.25,.5) and (1.75,.5) .. (1.75,1);
    \node at (1.5,.15){\tiny $\dots$};
    \draw (1.75,0) .. controls (1.75,.5) and (2.25,.5)..  (2.25,1);
    \draw (2,0) .. controls (2,.25) and (2.25,.25) .. (2.25,.5);
    \draw (2.25,.5) .. controls (2.25,.75) and (1.5,.75) .. (1.5,1);
    \draw[decoration={brace,mirror,raise=-8pt},decorate]  (1.15,-.35) -- node { \small $j$} (1.85,-.35);
  }
\]
For the first term of the right-hand side, we again slide the dot to the left using \cref{eq:nhdotslide} and obtain
\begin{align*}
  \tikzdiagh{0}{
    \draw (2.25,0) .. controls (2.25,.5) .. (-.5,.75) node[pos=.39,tikzdot]{} .. controls (0,.875) .. (0,1);
    \draw[stdhl] (0,0) node[below]{\small $1$} .. controls (0,.1) .. (-.5,.25) .. controls (.5,.7) .. (.5,1);
    \draw[fill=white, color=white] (-.6,.25) circle (.1cm);
    \draw[vstdhl] (-.5,0) node[below]{\small $\lambda$} --(-.5,1) node[pos=.75,nail]{};
    \draw (.5,0) .. controls (.5,.5) and (.75,.5) .. (.75,1);
    \node at (.75,.15){\tiny $\dots$};
    \draw (1,0) .. controls (1,.5) and (1.25,.5) .. (1.25,1);
    \draw (1.25,0) .. controls (1.25,.5) and (1.75,.5) .. (1.75,1);
    \node at (1.5,.15){\tiny $\dots$};
    \draw (1.75,0) .. controls (1.75,.5) and (2.25,.5)..  (2.25,1);
    \draw (2,0) .. controls (2,.25) and (2.25,.25) .. (2.25,.5);
    \draw (2.25,.5) .. controls (2.25,.75) and (1.5,.75) .. (1.5,1);
    \draw[decoration={brace,mirror,raise=-8pt},decorate]  (1.15,-.35) -- node { \small $j$} (1.85,-.35);
  }
  &\overset{\phantom{\eqref{eq:nhR2andR3}}}{=}
    \tikzdiagh{0}{
    \draw (2.35,0) .. controls (2.35,.5) .. (-.5,.75) node[pos=.845,tikzdot]{} .. controls (0,.875) .. (0,1);
    \draw[stdhl] (0,0) node[below]{\small $1$} .. controls (0,.1) .. (-.5,.25) .. controls (.5,.7) .. (.5,1);
    \draw[fill=white, color=white] (-.6,.25) circle (.1cm);
    \draw[vstdhl] (-.5,0) node[below]{\small $\lambda$} --(-.5,1) node[pos=.75,nail]{};
    \draw (.6,0) .. controls (.6,.5) and (.85,.5) .. (.85,1);
    \node at (.85,.15){\tiny $\dots$};
    \draw (1.1,0) .. controls (1.1,.5) and (1.35,.5) .. (1.35,1);
    \draw (1.35,0) .. controls (1.35,.5) and (1.85,.5) .. (1.85,1);
    \node at (1.6,.15){\tiny $\dots$};
    \draw (1.85,0) .. controls (1.85,.5) and (2.35,.5)..  (2.35,1);
    \draw (2.1,0) .. controls (2.1,.25) and (2.35,.25) .. (2.35,.5);
    \draw (2.35,.5) .. controls (2.35,.75) and (1.6,.75) .. (1.6,1);
    \draw[decoration={brace,mirror,raise=-8pt},decorate]  (1.25,-.35) -- node { \small $j$} (1.95,-.35);
  }
    -
    \sum_{l=0}^{k-j-3}
    \tikzdiagh{0}{
    \draw (1.25,0) .. controls (1.25,.5) .. (-.5,.75) .. controls (0,.875) .. (0,1);
    \draw[stdhl] (0,0) node[below]{\small $1$} .. controls (0,.1) .. (-.5,.25) .. controls (.5,.7) .. (.5,1);
    \draw[fill=white, color=white] (-.6,.25) circle (.1cm);
    \draw[vstdhl] (-.5,0) node[below]{\small $\lambda$} --(-.5,1) node[pos=.75,nail]{};
    \draw (.5,0) .. controls (.5,.5) and (.75,.5) .. (.75,1);
    \node at (.75,.15){\tiny $\dots$};
    \draw (1,0) .. controls (1,.5) and (1.25,.5) .. (1.25,1);
    \draw (1.5,0) .. controls (1.5,.5) and (1.75,.5) .. (1.75,1);
    \node at (1.75,.15){\tiny $\dots$};
    \draw (2,0) .. controls (2,.5) and (2.25,.5) .. (2.25,1);
    \draw (2.25,0) .. controls (2.25,.5) and (2.75,.5) .. (2.75,1);
    \node at (2.5,.15){\tiny $\dots$};
    \draw (2.75,0) .. controls (2.75,.5) and (3.25,.5) .. (3.25,1);
    \draw (3,0) .. controls (3,.25) and (3.25,.25) .. (3.25,.5);
    \draw (3.25,.5) .. controls (3.25,.75) and (2.5,.75) .. (2.5,1);
    \draw (3.25,0) .. controls (3.25,.5) and (1.5,.5) .. (1.5,1);
    \draw[decoration={brace,mirror,raise=-8pt},decorate]  (1.4,-.35) -- node { \small $l$} (2.1,-.35);
    \draw[decoration={brace,mirror,raise=-8pt},decorate]  (2.15,-.35) -- node { \small $j$} (2.85,-.35);
    }\\
  &\overset{\eqref{eq:nhR2andR3}}{=}
    \tikzdiagh{0}{
    \draw (2.35,0) .. controls (2.35,.5) .. (-.5,.75) node[pos=.845,tikzdot]{} .. controls (0,.875) .. (0,1);
    \draw[stdhl] (0,0) node[below]{\small $1$} .. controls (0,.1) .. (-.5,.25) .. controls (.5,.7) .. (.5,1);
    \draw[fill=white, color=white] (-.6,.25) circle (.1cm);
    \draw[vstdhl] (-.5,0) node[below]{\small $\lambda$} --(-.5,1) node[pos=.75,nail]{};
    \draw (.6,0) .. controls (.6,.5) and (.85,.5) .. (.85,1);
    \node at (.85,.15){\tiny $\dots$};
    \draw (1.1,0) .. controls (1.1,.5) and (1.35,.5) .. (1.35,1);
    \draw (1.35,0) .. controls (1.35,.5) and (1.85,.5) .. (1.85,1);
    \node at (1.6,.15){\tiny $\dots$};
    \draw (1.85,0) .. controls (1.85,.5) and (2.35,.5)..  (2.35,1);
    \draw (2.1,0) .. controls (2.1,.25) and (2.35,.25) .. (2.35,.5);
    \draw (2.35,.5) .. controls (2.35,.75) and (1.6,.75) .. (1.6,1);
    \draw[decoration={brace,mirror,raise=-8pt},decorate]  (1.25,-.35) -- node { \small $j$} (1.95,-.35);
    }
    -
    \sum_{l=0}^{k-j-3}
    \tikzdiagh{0}{
    \draw (1.25,0) .. controls (1.25,.5) .. (-.5,.75) .. controls (0,.875) .. (0,1);
    \draw[stdhl] (0,0) node[below]{\small $1$} .. controls (0,.1) .. (-.5,.25) .. controls (.5,.7) .. (.5,1);
    \draw[fill=white, color=white] (-.6,.25) circle (.1cm);
    \draw[vstdhl] (-.5,0) node[below]{\small $\lambda$} --(-.5,1) node[pos=.75,nail]{};
    \draw (.5,0) .. controls (.5,.5) and (.75,.5) .. (.75,1);
    \node at (.75,.15){\tiny $\dots$};
    \draw (1,0) .. controls (1,.5) and (1.25,.5) .. (1.25,1);
    \draw (1.5,0) .. controls (1.5,.5) and (1.75,.5) .. (1.75,1);
    \node at (1.75,.15){\tiny $\dots$};
    \draw (2,0) .. controls (2,.5) and (2.25,.5) .. (2.25,1);
    \draw (2.25,0) .. controls (2.25,.5) and (2.75,.5) .. (2.75,1);
    \node at (2.5,.15){\tiny $\dots$};
    \draw (2.75,0) .. controls (2.75,.5) and (3.25,.5) .. (3.25,1);
    \draw (3,0) .. controls (3,.3) and (2.3,.25) .. (2.3,.5);
    \draw (2.3,.5) .. controls (2.3,.75) and (2.5,.75) .. (2.5,1);
    \draw (3.25,0) .. controls (3.25,.6) and (1.5,.5) .. (1.5,1);
    \draw[decoration={brace,mirror,raise=-8pt},decorate]  (1.4,-.35) -- node { \small $l$} (2.1,-.35);
    \draw[decoration={brace,mirror,raise=-8pt},decorate]  (2.15,-.35) -- node { \small $j$} (2.85,-.35);
    }
\end{align*}
Another application of the symmetry of \cref{eq:dottednailslide} deals with the first term, and every term of the sum is handled trough a descending induction on $j$, noting that the sum is zero if $j=k-2$.

For the second term, we apply once again \cref{eq:nhR2andR3} and obtain
\[
  \tikzdiagh{0}{
    \draw (2.25,0) .. controls (2.25,.5) .. (-.5,.75) node[pos=.1,tikzdot]{} .. controls (0,.875) .. (0,1);
    \draw[stdhl] (0,0) node[below]{\small $1$} .. controls (0,.1) .. (-.5,.25) .. controls (.5,.7) .. (.5,1);
    \draw[fill=white, color=white] (-.6,.25) circle (.1cm);
    \draw[vstdhl] (-.5,0) node[below]{\small $\lambda$} --(-.5,1) node[pos=.75,nail]{};
    \draw (.5,0) .. controls (.5,.5) and (.75,.5) .. (.75,1);
    \node at (.75,.15){\tiny $\dots$};
    \draw (1,0) .. controls (1,.5) and (1.25,.5) .. (1.25,1);
    \draw (1.25,0) .. controls (1.25,.5) and (1.75,.5) .. (1.75,1);
    \node at (1.5,.15){\tiny $\dots$};
    \draw (1.75,0) .. controls (1.75,.5) and (2.25,.5)..  (2.25,1);
    \draw (2,0) .. controls (2,.25) and (2.25,.25) .. (2.25,.5);
    \draw (2.25,.5) .. controls (2.25,.75) and (1.5,.75) .. (1.5,1);
    \draw[decoration={brace,mirror,raise=-8pt},decorate]  (1.15,-.35) -- node { \small $j$} (1.85,-.35);
  }
  =
  \tikzdiagh{0}{
    \draw (2.25,0) .. controls (2.25,.5) .. (-.5,.75) node[pos=.1,tikzdot]{} .. controls (0,.875) .. (0,1);
    \draw[stdhl] (0,0) node[below]{\small $1$} .. controls (0,.1) .. (-.5,.25) .. controls (.5,.7) .. (.5,1);
    \draw[fill=white, color=white] (-.6,.25) circle (.1cm);
    \draw[vstdhl] (-.5,0) node[below]{\small $\lambda$} --(-.5,1) node[pos=.75,nail]{};
    \draw (.5,0) .. controls (.5,.5) and (.75,.5) .. (.75,1);
    \node at (.75,.15){\tiny $\dots$};
    \draw (1,0) .. controls (1,.5) and (1.25,.5) .. (1.25,1);
    \draw (1.25,0) .. controls (1.25,.5) and (1.75,.5) .. (1.75,1);
    \node at (1.5,.15){\tiny $\dots$};
    \draw (1.75,0) .. controls (1.75,.5) and (2.25,.5)..  (2.25,1);
    \draw (2,0) .. controls (2,.5) and (1.35,.25) .. (1.35,.5);
    \draw (1.35,.5) .. controls (1.35,.75) and (1.5,.75) .. (1.5,1);
    \draw[decoration={brace,mirror,raise=-8pt},decorate]  (1.15,-.35) -- node { \small $j$} (1.85,-.35);
  }
\]
which has the desired form.
\end{proof}

\subsubsection{Acyclicity of $\cone(\varphi)$}\label{sec:proofofacyclicity}

\begin{citethm}{thm:catdoublebraidphiiso}
The map 
\[
\varphi := \sum_{k = 0}^{m} (-1)^k \varphi_k  : \cone(\lambda q^2 X [1] \xrightarrow{u} q^2 T_b^{\lambda,r} [1] )[1] \rightarrow \cone(X \Lotimes_T X \xrightarrow{1\otimes u} \lambda^{-1} X),
\]
is a quasi-isomorphism. 
\end{citethm}

The goal of this section is to prove \cref{thm:catdoublebraidphiiso}, which we will achieve by showing that $\cone(\varphi_k)$ is acyclic. 
We have that $\cone(\varphi_k)$ is given by the complex
\[
\begin{tikzcd}[row sep = 1ex]
& X\otimes_T Y^1_k \ar[hookrightarrow]{dr}{1 \otimes \imath_k} &&
\\
 \lambda q^2 (X_k)[1] \ar{ur}{\varphi_k^1} \ar[hookrightarrow,swap]{dr}{-u}&  & X\otimes_T Y^0_k \ar[twoheadrightarrow]{r}{u \otimes \gamma_k}  & \lambda^{-1} X_k.
 \\
&  \bigoplus_{\ell,\rho} q^2 (T_b^{\lambda,r} 1_{k,\ell,\rho})[1]  \ar[swap]{ur}{\varphi_k^0} &&
\end{tikzcd}
\]
The map $\varphi_k^1 - u$ is injective since $u$ is injective by \cref{cor:uinj}, and the map $u \otimes \gamma_k$ is surjective. We want to first show that $\varphi_k^0 + 1\otimes \imath_k$ is surjective on the kernel of $u \otimes \gamma_k$. This requires some preparation.

\begin{lem}\label{lem:dotanddoubledots}
For $k \geq 2$, the local relation 
\begin{equation}\label{eq:dotanddoubledots}
\tikzdiagh[xscale=1.25]{0}{
	\draw (0,0) .. controls (0,.5) and (1.5,.5) .. (1.5,2);
	\draw (.5,0) .. controls (.5,.5) and (0,.5) .. (0,1) node[tikzdot,pos=1]{}
			 .. controls (0,1.5) and (.5,1.5) .. (.5,2);
	\draw (.75,0) .. controls (.75,.5) and (.25,.5) .. (.25,1) node[tikzdot,pos=1]{}
			 .. controls (.25,1.5) and (.75,1.5) .. (.75,2);
	\node at (.5,1) {\tiny $\dots$};
	\draw (1.25,0) .. controls (1.25,.5) and (.75,.5) .. (.75,1) node[tikzdot,pos=1]{}
			 .. controls (.75,1.5) and (1.25,1.5) .. (1.25,2);
	\draw (1.5,0) .. controls (1.5,1.5) and (0,1.5) .. (0,2);
	\draw[stdhl] (.25,0) node[below]{\small $1$} .. controls (.25,.25) and (1.5,.25) .. (1.5,1)
				 .. controls (1.5,1.75) and (.25,1.75) .. (.25,2);
	\draw[decoration={brace,mirror,raise=-8pt},decorate]  (.4,-.35) -- node {\small $k-2$} (1.35,-.35);
}
\ - \ 
\sum_{s=0}^{k-2}
(-1)^{s}\ 
\tikzdiagh[xscale=1.25]{0}{
	\draw (0,0) .. controls (0,.5) and (1.25,.5) .. (1.25,2);
	\draw (.5,0) .. controls (.5,.5) and (.25,.5) .. (.25,1) node[near start, tikzdot]{}
			 .. controls (.25,1.5) and (.5,1.5) .. (.5,2)  node[near end, tikzdot]{};
	\node at (.5,1) {\tiny $\dots$};
	\draw (1,0) .. controls (1,.5) and (.75,.5) .. (.75,1)node[near start, tikzdot]{}
			 .. controls (.75,1.5) and (1,1.5) .. (1,2)  node[near end, tikzdot]{};
	\draw (1.25,0) .. controls (1.25,1.5) and (0,1.5) .. (0,2);
	\draw[stdhl] (.25,0) node[below]{\small $1$} .. controls (.25,.25) and (0,.25) .. (0,1)
				 .. controls (0,1.75) and (.25,1.75) .. (.25,2);
	\draw (1.5,0) -- (1.5,2);
	\node at (1.75,1) {\tiny $\dots$};
	\draw (2,0) -- (2,2);
	\draw[decoration={brace,mirror,raise=-8pt},decorate]  (.4,-.35) -- node {\small $s$} (1.1,-.35);
}
\ = (-1)^{k-1} \
\tikzdiagh[xscale=1.25]{0}{
	\draw (0,0) -- (0,2);
	\draw[stdhl] (.25,0) node[below]{\small $1$} -- (.25,2);
	\draw (.5,0) -- (.5,2);
	\draw (.75,0) -- (.75,2);
	\node at (1,1) {\tiny $\dots$};
	\draw (1.25,0) -- (1.25,2);
	\draw (1.5,0) -- (1.5,2);
	\draw[decoration={brace,mirror,raise=-8pt},decorate]  (.4,-.35) -- node {\small $k-2$} (1.35,-.35);
}
\end{equation}
holds in $T^{\lambda,r}$.
\end{lem}

\begin{proof}
We prove the statement by induction on $k-2$. If $k-2 = 0$, then the claim follows from \cref{eq:redR3}. Suppose by induction that \eqref{eq:dotanddoubledots} holds for $k-3$. We compute
\begin{equation}\label{eq:dotanddoubledots1}
\tikzdiagh[xscale=1.25]{0}{
	\draw (0,0) .. controls (0,.5) and (1.5,.5) .. (1.5,2);
	\draw (.5,0) .. controls (.5,.5) and (0,.5) .. (0,1) node[tikzdot,pos=1]{}
			 .. controls (0,1.5) and (.5,1.5) .. (.5,2);
	\draw (.75,0) .. controls (.75,.5) and (.25,.5) .. (.25,1) node[tikzdot,pos=1]{}
			 .. controls (.25,1.5) and (.75,1.5) .. (.75,2);
	\node at (.5,1) {\tiny $\dots$};
	\draw (1.25,0) .. controls (1.25,.5) and (.75,.5) .. (.75,1) node[tikzdot,pos=1]{}
			 .. controls (.75,1.5) and (1.25,1.5) .. (1.25,2);
	\draw (1.5,0) .. controls (1.5,1.5) and (0,1.5) .. (0,2);
	\draw[stdhl] (.25,0) node[below]{\small $1$} .. controls (.25,.25) and (1.5,.25) .. (1.5,1)
				 .. controls (1.5,1.75) and (.25,1.75) .. (.25,2);
	\draw[decoration={brace,mirror,raise=-8pt},decorate]  (.4,-.35) -- node {\small $k-2$} (1.35,-.35);
}
\ \overset{\eqref{eq:redR3}}{=} \ 
\tikzdiagh[xscale=1.25]{0}{
	\draw (0,0) .. controls (0,.5) and (1.5,.25) .. (1.5,2);
	\draw (.5,0) .. controls (.5,.5) and (0,.5) .. (0,1) node[tikzdot,pos=1]{}
			 .. controls (0,1.5) and (.5,1.5) .. (.5,2);
	\draw (.75,0) .. controls (.75,.5) and (.25,.5) .. (.25,1) node[tikzdot,pos=1]{}
			 .. controls (.25,1.5) and (.75,1.5) .. (.75,2);
	\node at (.5,1) {\tiny $\dots$};
	\draw (1.25,0) .. controls (1.25,.5) and (.75,.5) .. (.75,1) node[tikzdot,pos=1]{}
			 .. controls (.75,1.5) and (1.25,1.5) .. (1.25,2);
	\draw (1.5,0) .. controls (1.5,1.75) and (0,1.5) .. (0,2);
	\draw[stdhl] (.25,0) node[below]{\small $1$} .. controls (.25,.25) and (1.25,.25) .. (1.25,.6)
			.. controls (1.25,.8) and (1,.8) .. (1,1)
			.. controls (1,1.2) and (1.25,1.2) .. (1.25,1.4)
			.. controls (1.25,1.75) and (.25,1.75) .. (.25,2);
	\draw[decoration={brace,mirror,raise=-8pt},decorate]  (.4,-.35) -- node {\small $k-2$} (1.35,-.35);
}
\ - \ 
\tikzdiagh[xscale=1.25]{0}{
	\draw (0,0) .. controls (0,.5) and (1,.25) .. (1,1)
			.. controls (1,1.75) and (0,1.5) .. (0,2);
	\draw (.5,0) .. controls (.5,.5) and (0,.5) .. (0,1) node[tikzdot,pos=1]{}
			 .. controls (0,1.5) and (.5,1.5) .. (.5,2);
	\draw (.75,0) .. controls (.75,.5) and (.25,.5) .. (.25,1) node[tikzdot,pos=1]{}
			 .. controls (.25,1.5) and (.75,1.5) .. (.75,2);
	\node at (.5,1) {\tiny $\dots$};
	\draw (1.25,0) .. controls (1.25,.5) and (.75,.5) .. (.75,1) node[tikzdot,pos=1]{}
			 .. controls (.75,1.5) and (1.25,1.5) .. (1.25,2);
	\draw (1.5,0)  -- (1.5,2);
	\draw[stdhl] (.25,0) node[below]{\small $1$} .. controls (.25,.25) and (1.25,.25) .. (1.25,.6)
			-- (1.25,1.4)
			.. controls (1.25,1.75) and (.25,1.75) .. (.25,2);
	\draw[decoration={brace,mirror,raise=-8pt},decorate]  (.4,-.35) -- node {\small $k-2$} (1.35,-.35);
}
\end{equation}
and
\begin{align}
\label{eq:dotanddoubledots2}
\tikzdiagh[xscale=1.25]{0}{
	\draw (0,0) .. controls (0,.5) and (1,.25) .. (1,1)
			.. controls (1,1.75) and (0,1.5) .. (0,2);
	\draw (.5,0) .. controls (.5,.5) and (0,.5) .. (0,1) node[tikzdot,pos=1]{}
			 .. controls (0,1.5) and (.5,1.5) .. (.5,2);
	\draw (.75,0) .. controls (.75,.5) and (.25,.5) .. (.25,1) node[tikzdot,pos=1]{}
			 .. controls (.25,1.5) and (.75,1.5) .. (.75,2);
	\node at (.5,1) {\tiny $\dots$};
	\draw (1.25,0) .. controls (1.25,.5) and (.75,.5) .. (.75,1) node[tikzdot,pos=1]{}
			 .. controls (.75,1.5) and (1.25,1.5) .. (1.25,2);
	\draw (1.5,0)  -- (1.5,2);
	\draw[stdhl] (.25,0) node[below]{\small $1$} .. controls (.25,.25) and (1.25,.25) .. (1.25,.6)
			-- (1.25,1.4)
			.. controls (1.25,1.75) and (.25,1.75) .. (.25,2);
	\draw[decoration={brace,mirror,raise=-8pt},decorate]  (.4,-.35) -- node {\small $k-2$} (1.35,-.35);
}
\ &\overset{(\ref{eq:nhdotslide}, \ref{eq:nhR2andR3})}{=} \ 
\tikzdiagh[xscale=1.25]{0}{
	\draw (0,0) .. controls (0,.5) and (1.25,.5) .. (1.25,2);
	\draw (.5,0) .. controls (.5,.5) and (0,.5) .. (0,1) node[tikzdot,pos=1]{}
			 .. controls (0,1.5) and (.5,1.5) .. (.5,2);
	\node at (.25,1) {\tiny $\dots$};
	\draw (1,0) .. controls (1,.5) and (.5,.5) .. (.5,1) node[tikzdot,pos=1]{}
			 .. controls (.5,1.5) and (1,1.5) .. (1,2);
	\draw (1.25,0) .. controls (1.25,1.5) and (0,1.5) .. (0,2);
	\draw[stdhl] (.25,0) node[below]{\small $1$} .. controls (.25,.25) and (1.25,.25) .. (1.25,1)
				 .. controls (1.25,1.75) and (.25,1.75) .. (.25,2);
	\draw[decoration={brace,mirror,raise=-8pt},decorate]  (.4,-.35) -- node {\small $k-3$} (1.1,-.35);
	\draw (1.5,0) -- (1.5,2);
}
\\
\label{eq:dotanddoubledots3}
\tikzdiagh[xscale=1.25,yscale=1.25]{0}{
	\draw (0,0) .. controls (0,.5) and (1.5,.25) .. (1.5,2);
	\draw (.5,0) .. controls (.5,.5) and (0,.5) .. (0,1) node[tikzdot,pos=1]{}
			 .. controls (0,1.5) and (.5,1.5) .. (.5,2);
	\draw (.75,0) .. controls (.75,.5) and (.25,.5) .. (.25,1) node[tikzdot,pos=1]{}
			 .. controls (.25,1.5) and (.75,1.5) .. (.75,2);
	\node at (.5,1) {\tiny $\dots$};
	\draw (1.25,0) .. controls (1.25,.5) and (.75,.5) .. (.75,1) node[tikzdot,pos=1]{}
			 .. controls (.75,1.5) and (1.25,1.5) .. (1.25,2);
	\draw (1.5,0) .. controls (1.5,1.75) and (0,1.5) .. (0,2);
	\draw[stdhl] (.25,0) node[below]{\small $1$} .. controls (.25,.25) and (1.25,.25) .. (1.25,.6)
			.. controls (1.25,.8) and (1,.8) .. (1,1)
			.. controls (1,1.2) and (1.25,1.2) .. (1.25,1.4)
			.. controls (1.25,1.75) and (.25,1.75) .. (.25,2);
	\draw[decoration={brace,mirror,raise=-8pt},decorate]  (.4,-.35) -- node {\small $k-2$} (1.35,-.35);
}
\ &\overset{\eqref{eq:redR2}}{=} \ 
\tikzdiagh[xscale=1.25,yscale=1.25]{0}{
	\draw (0,0) .. controls (0,.5) and (1.5,0) .. (1.5,2);
	\draw (.5,0) .. controls (.5,.5) and (.25,.5) .. (.25,1)
			 .. controls (.25,1.5) and (.5,1.5) .. (.5,2);
	\draw (.75,0) .. controls (.75,.5) and (.5,.5) .. (.5,1) 
			 .. controls (.5,1.5) and (.75,1.5) .. (.75,2);
	\node at (.75,1) {\tiny $\dots$};
	\draw (1.25,0) .. controls (1.25,.5) and (1,.5) .. (1,1)
			 .. controls (1,1.5) and (1.25,1.5) .. (1.25,2);
	\draw (1.5,0) .. controls (1.5,2) and (0,1.5) .. (0,2);
	\draw[stdhl] (.25,0) node[below]{\small $1$} .. controls (.25,.25) and (1.25,.25) .. (1.25,.6)
			.. controls (1.25,1) and (0,.8) .. (0,1)
			.. controls (0,1.2) and (1.25,1) .. (1.25,1.4)
			.. controls (1.25,1.75) and (.25,1.75) .. (.25,2);
	\draw[decoration={brace,mirror,raise=-8pt},decorate]  (.4,-.35) -- node {\small $k-2$} (1.35,-.35);
}
\ \overset{\eqref{eq:redR3}}{=} \ 
\tikzdiagh[xscale=1.25,yscale=1.25]{0}{
	\draw (0,0) .. controls (0,1.5) and (1.5,0) .. (1.5,2);
	\draw (.5,0) .. controls (.5,.5) and (.25,.5) .. (.25,1)
			 .. controls (.25,1.5) and (.5,1.5) .. (.5,2);
	\draw (.75,0) .. controls (.75,.5) and (.5,.5) .. (.5,1) 
			 .. controls (.5,1.5) and (.75,1.5) .. (.75,2);
	\node at (.75,1) {\tiny $\dots$};
	\draw (1.25,0) .. controls (1.25,.5) and (1,.5) .. (1,1)
			 .. controls (1,1.5) and (1.25,1.5) .. (1.25,2);
	\draw (1.5,0) .. controls (1.5,2) and (0,.5) .. (0,2);
	\draw[stdhl] (.25,0) node[below]{\small $1$} .. controls (.25,.1) and (1.375,.1) .. (1.375,.2)
			.. controls (1.375,.5) and (0,0) .. (0,1)
			.. controls (0,2) and (1.375,1.5) .. (1.375,1.8)
			.. controls (1.375,1.9) and (.25,1.9) .. (.25,2);
	\draw[decoration={brace,mirror,raise=-8pt},decorate]  (.4,-.35) -- node {\small $k-2$} (1.35,-.35);
}
\ \overset{\eqref{eq:redR2}}{=} \ 
\tikzdiagh[xscale=1.25,yscale=1.25]{0}{
	\draw (0,0) .. controls (0,1.5) and (1.5,0) .. (1.5,2);
	\draw (.5,0) .. controls (.5,.5) and (.25,.5) .. (.25,1) node[near start, tikzdot]{}
			 .. controls (.25,1.5) and (.5,1.5) .. (.5,2) node[near end, tikzdot]{};
	\draw (.75,0) .. controls (.75,.5) and (.5,.5) .. (.5,1) node[near start, tikzdot]{}
			 .. controls (.5,1.5) and (.75,1.5) .. (.75,2)node[near end, tikzdot]{};
	\node at (.75,1) {\tiny $\dots$};
	\draw (1.25,0) .. controls (1.25,.5) and (1,.5) .. (1,1)node[near start, tikzdot]{}
			 .. controls (1,1.5) and (1.25,1.5) .. (1.25,2)node[near end, tikzdot]{};
	\draw (1.5,0) .. controls (1.5,2) and (0,.5) .. (0,2);
	\draw[stdhl] (.25,0) node[below]{\small $1$} .. controls (.25,.5) and (0,.5) .. (0,1)
			.. controls (0,1.5) and (.25,1.5) .. (.25,2);
	\draw[decoration={brace,mirror,raise=-8pt},decorate]  (.4,-.35) -- node {\small $k-2$} (1.35,-.35);
}
\end{align}
Applying the induction hypothesis on \cref{eq:dotanddoubledots2}, and inserting the result together with \cref{eq:dotanddoubledots3} in \cref{eq:dotanddoubledots1} gives \cref{eq:dotanddoubledots}.
\end{proof}

\begin{lem}\label{lem:computezn}
We have
\begin{equation}\label{eq:computezn}
\tikzdiag{
	\draw (0,0) -- (0,1);
	\node at (.5,.1) {\tiny $\dots$};
	\node at (.5,.9) {\tiny $\dots$};
	\draw (1,0) -- (1,1);
	\draw (1.5,0) -- (1.5,1);
	\draw (2,0) -- (2,1);
	\filldraw [fill=white, draw=black,rounded corners] (-.25,.25) rectangle (2.25,.75) node[midway] { $z_{t+2}$};
}
\ = \ 
\tikzdiag{
	\draw (0,0) .. controls (0,.25) and (2,.25) .. (2,1); 
	\draw (2,0) .. controls (2,.75) and (0,.75) .. (0,1); 
	\draw (.5,0) .. controls (.5,.25) and (0,.25) .. (0,.5) .. controls (0,.75) and (.5,.75) .. (.5,1) node[pos=0,tikzdot]{};
	\draw (1.5,0) .. controls (1.5,.25) and (1,.25) .. (1,.5) .. controls (1,.75) and (1.5,.75) .. (1.5,1) node[pos=0,tikzdot]{};
	\node at(.5,.5){\tiny $\dots$};
}
\end{equation}
for all $t \geq 0$. 
\end{lem}

\begin{proof}
We prove the statement by induction on $t$. The claim is clearly true for $t=0$. Suppose it is true for $t$. We compute
\begin{align*}
\tikzdiag{
	\draw (0,-.5) -- (0,1.5);
	\node at (.5,-.25) {\tiny $\dots$};
	\node at (.5,1.25) {\tiny $\dots$};
	\draw (1,-.5) -- (1,1.5);
	\draw (1.5,-.5) -- (1.5,1.5);
	\draw (2,-.5) -- (2,1.5);
	\draw (2.5,-.5) -- (2.5,1.5);
	\filldraw [fill=white, draw=black,rounded corners] (-.25,.3) rectangle (2.75,.8) node[midway] { $z_{t+3}$};
}
\ &\overset{\eqref{eq:defzn}}{=}  \
\tikzdiag{
	\draw (0,-.5) -- (0,1.5);
	\node at (.5,-.25) {\tiny $\dots$};
	\node at (.5,1.25) {\tiny $\dots$};
	\draw (1,-.5) -- (1,1.5);
	\draw (1.5,-.5) -- (1.5,1.5);
	\draw (2,-.5) -- (2,1) .. controls (2,1.25) and (2.5,1.25) .. (2.5,1.5);
	\draw (2.5,-.5) -- (2.5,1) .. controls (2.5,1.25) and (2,1.25) .. (2,1.5);
	\filldraw [fill=white, draw=black,rounded corners] (-.25,.3) rectangle (2.25,.8) node[midway] { $z_{t+2}$};
}
\ + \ 
\tikzdiag{
	\draw (0,-.5) -- (0,1.5);
	\node at (.5,-.25) {\tiny $\dots$};
	\node at (.5,1.25) {\tiny $\dots$};
	\draw (1,-.5) -- (1,1.5);
	\draw (1.5,-.5) -- (1.5,1.5);
	\draw (2.5,-.5) .. controls (2.5,-.25) and (2,-.25) .. (2,0) -- (2,1) .. controls (2,1.25) and (2.5,1.25) .. (2.5,1.5);
	\draw (2,-.5) .. controls (2,-.25) and (2.5,-.25) .. (2.5,0) -- (2.5,1) .. controls (2.5,1.25) and (2,1.25) .. (2,1.5) node [near end, tikzdot]{};
	\filldraw [fill=white, draw=black,rounded corners] (-.25,.3) rectangle (2.25,.8) node[midway] { $z_{t+2}$};
}
\\
\ &\overset{\eqref{eq:computezn}}{=}  \
\tikzdiag{
	\draw (0,-.5) -- (0,0) .. controls (0,.25) and (2,.25) .. (2,1) .. controls (2,1.25) and (2.5,1.25) .. (2.5,1.5);
	\draw (2,-.5) -- (2,0) .. controls (2,.75) and (0,.75) .. (0,1) -- (0,1.5);
	\draw (.5,-.5) -- (.5,0) .. controls (.5,.25) and (0,.25) .. (0,.5) .. controls (0,.75) and (.5,.75) .. (.5,1) node[pos=0,tikzdot]{}  -- (.5,1.5);
	\draw (1.5,-.5) -- (1.5,0) .. controls (1.5,.25) and (1,.25) .. (1,.5) .. controls (1,.75) and (1.5,.75) .. (1.5,1) node[pos=0,tikzdot]{}  -- (1.5,1.5);
	\node at(.5,.5){\tiny $\dots$};
	\draw (2.5,-.5) -- (2.5,1) .. controls (2.5,1.25) and (2,1.25) .. (2,1.5);
}
\ + \ 
\tikzdiag{
	\draw (0,-.5) -- (0,0) .. controls (0,.25) and (2,.25) .. (2,1) .. controls (2,1.25) and (2.5,1.25) .. (2.5,1.5);
	\draw (2.5,-.5) .. controls (2.5,-.25) and (2,-.25) .. (2,0) .. controls (2,.75) and (0,.75) .. (0,1) -- (0,1.5);
	\draw (.5,-.5) -- (.5,0) .. controls (.5,.25) and (0,.25) .. (0,.5) .. controls (0,.75) and (.5,.75) .. (.5,1) node[pos=0,tikzdot]{}  -- (.5,1.5);
	\draw (1.5,-.5) -- (1.5,0) .. controls (1.5,.25) and (1,.25) .. (1,.5) .. controls (1,.75) and (1.5,.75) .. (1.5,1) node[pos=0,tikzdot]{}  -- (1.5,1.5);
	\node at(.5,.5){\tiny $\dots$};
	\draw (2,-.5) .. controls (2,-.25) and (2.5,-.25) .. (2.5,0) -- (2.5,1) .. controls (2.5,1.25) and (2,1.25) .. (2,1.5)  node [near end, tikzdot]{};
}
\ \overset{(\ref{eq:nhR2andR3},\ref{eq:nhdotslide})}{=} \ 
\tikzdiag{
	\draw (0,-.5) .. controls (0,0) and (2.5,0) .. (2.5,1.5); 
	\draw (2.5,-.5) .. controls (2.5,1) and (0,1) .. (0,1.5); 
	\draw (.5,-.5) .. controls (.5,0) and (0,0) .. (0,.5) .. controls (0,1) and (.5,1) .. (.5,1.5) node[pos=0,tikzdot]{};
	\draw (1.5,-.5) .. controls (1.5,0) and (1,0) .. (1,.5) .. controls (1,1) and (1.5,1) .. (1.5,1.5) node[pos=0,tikzdot]{};
	\draw (2,-.5) .. controls (2,0) and (1.5,0) .. (1.5,.5) .. controls (1.5,1) and (2,1) .. (2,1.5) node[pos=0,tikzdot]{};
	\node at(.5,.5){\tiny $\dots$};
}
\end{align*}
concluding the proof.
\end{proof}

\begin{lem}\label{lem:allcrossingzniszero}
We have
\[
\tikzdiag{
	\draw (0,0) -- (0,1) .. controls (0,1.25) and (.5,1.25) .. (.5,1.5);
	\node at (.5,.1) {\tiny $\dots$};
	\node at (.5,1) {\tiny $\dots$};
	\draw (1,0) -- (1,1) .. controls (1,1.25) and (1.5,1.25) .. (1.5,1.5);
	\draw (1.5,0) -- (1.5,1) .. controls (1.5,1.25) and (2,1.25) .. (2,1.5);
	\draw (2,0) -- (2,1) .. controls (2,1.25) and (0,1.25) .. (0,1.5);
	\draw (2.5,0) -- (2.5,1.5);
	\node at (3,.1) {\tiny $\dots$};
	\node at (3,1) {\tiny $\dots$};
	\draw (3.5,0) -- (3.5,1.5);
	\filldraw [fill=white, draw=black,rounded corners] (-.25,.3) rectangle (3.75,.8) node[midway] { $z_{t+2}$};
}
\ = \ 
0,
\]
for all $t \geq 0$. 
\end{lem}

\begin{proof}
It is a direct consequence of \cref{lem:computezn} together with \cref{eq:nhR2andR3}.
\end{proof}

\begin{lem}\label{lem:phizeroimathgenker}
We have
\begin{equation}\label{eq:phizeroimathgenker}
\begin{aligned}
\varphi_k^0&\left(
\tikzdiagh[xscale=1.25]{0}{
	\draw[vstdhl] (-.25,.75) node[below]{\small $\lambda$} -- (-.25,2);
	\draw (0,.75) .. controls (0,1.25) and (.5,1.25) .. (.5,2);
	\node at (.25,.85) {\tiny $\dots$};
	\draw (.5,.75) .. controls (.5,1.25) and (1,1.25) .. (1,2);
	\draw (.75,.75) .. controls (.75,1.25) and (0,1.25) .. (0,2); 
	\draw[stdhl] (1,.75) node[below]{\small $1$}
				 .. controls (1,1.5) and (.25,1.5) .. (.25,2);
	\draw[decoration={brace,mirror,raise=-8pt},decorate]  (-.1,.4) -- node {\small $k-1$} (.6,.4);
}
\otimes 
\bar 1_{\ell,\rho}
\right)
-
\sum_{s = 0}^{k-2}
(-1)^s
(1\otimes\imath_k^{k-1}) \left(
\tikzdiagh[xscale=1.25]{0}{
	\draw (0,0) .. controls (0,.5) and (1.25,.5) .. (1.25,2);
	\draw (.5,0) .. controls (.5,.5) and (.25,.5) .. (.25,1) node[near start, tikzdot]{}
			 .. controls (.25,1.5) and (.5,1.5) .. (.5,2)  node[near end, tikzdot]{};
	\node at (.5,1) {\tiny $\dots$};
	\draw (1,0) .. controls (1,.5) and (.75,.5) .. (.75,1)node[near start, tikzdot]{}
			 .. controls (.75,1.5) and (1,1.5) .. (1,2)  node[near end, tikzdot]{};
	\draw (1.25,0) .. controls (1.25,1.5) and (0,1.5) .. (0,2);
	\draw[stdhl] (.25,0) node[below]{\small $1$} .. controls (.25,.5).. (-.25,1)
				 .. controls (.25,1.5)  .. (.25,2);
	\draw[fill=white, color=white] (-.35,1) circle (.1cm);
	\draw[vstdhl] (-.25,0) node[below]{\small $\lambda$} -- (-.25,2);
	\draw (1.5,0) -- (1.5,2);
	\node at (1.75,1) {\tiny $\dots$};
	\draw (2,0) -- (2,2);
	\draw[decoration={brace,mirror,raise=-8pt},decorate]  (.4,-.35) -- node {\small $s$} (1.1,-.35);
}
\otimes
\bar 1_{\ell,\rho}
\right)
\\
&= (-1)^{k-1}
\left(
-\ 
\tikzdiagh{0}{
	\draw (1.25,0) .. controls (1.25,.5) .. (-.5,.75) .. controls (0,.875) .. (0,1);
	\draw[stdhl] (0,0) node[below]{\small $1$} .. controls (0,.1) .. (-.5,.25) .. controls (.5,.7) .. (.5,1);
	\draw[fill=white, color=white] (-.6,.25) circle (.1cm);
	\draw[vstdhl] (-.5,0) node[below]{\small $\lambda$} --(-.5,1) node[pos=.75,nail]{};
	\draw (.5,0) .. controls (.5,.5) and (.75,.5) .. (.75,1);
	\node at (.75,.15){\tiny $\dots$};
	\draw (1,0) .. controls (1,.5) and (1.25,.5)..  (1.25,1);
	\draw[decoration={brace,mirror,raise=-8pt},decorate]  (.4,-.35) -- node { \small $k-1$} (1.1,-.35);
}
\otimes \bar  1_{\ell,\rho}
,
\tikzdiagh{0}{
	\draw (.5,0) .. controls (.5,.5) and (0,.5) .. (0,1);
	\draw[stdhl] (0,0) node[below]{\small $1$} .. controls (0,.1) .. (-.5,.25) .. controls (.5,.7) .. (.5,1);
	\draw[fill=white, color=white] (-.6,.25) circle (.1cm);
	\draw[vstdhl] (-.5,0) node[below]{\small $\lambda$} --(-.5,1);
}
\otimes \bar  1_{\ell+k-1,\rho}
\right)
\in
 (X\otimes_T {Y'}^0_k) 
\oplus
(X\otimes_T Y^{0,k-1}_k).
\end{aligned}
\end{equation}
\end{lem}

\begin{proof}
The case $k=1$ is clear, thus we assume $k > 1$. 
First, let us write $\star$ and $\star_s$ for the inputs of $\varphi_k^0$ and of $1\otimes \imath_k^{k-1}$ in \cref{eq:phizeroimathgenker}, respectively. 
Then, on one hand, we note that $\varphi_k^{0,t'}(\star) = 0$ by \cref{lem:allcrossingzniszero} whenever $t' \neq k-1$, because of \eqref{eq:nhR2andR3}. For $t' = k-1$, we obtain
\begin{equation}
\label{eq:phizeroimathgenker1}
\varphi_k^{0,k-1}(\star) 
= \ 
\tikzdiagh[xscale=1.25]{0}{
	\draw (.25,-.5) .. controls (.25,-.1) and (0,-.1) .. (0,.15) .. controls (0,.5) and (1.5,.5) .. (1.5,2);
	\draw (.5,-.5) -- (.5,0) .. controls (.5,.5) and (0,.5) .. (0,1) node[tikzdot,pos=1]{}
			 .. controls (0,1.5) and (.5,1.5) .. (.5,2);
	\draw (.75,-.5) -- (.75,0) .. controls (.75,.5) and (.25,.5) .. (.25,1) node[tikzdot,pos=1]{}
			 .. controls (.25,1.5) and (.75,1.5) .. (.75,2);
	\node at (.5,1) {\tiny $\dots$};
	\draw (1.25,-.5) -- (1.25,0) .. controls (1.25,.5) and (.75,.5) .. (.75,1) node[tikzdot,pos=1]{}
			 .. controls (.75,1.5) and (1.25,1.5) .. (1.25,2);
	\draw (1.5,-.5) -- (1.5,0) .. controls (1.5,1.5) and (0,1.5) .. (0,2);
	\draw[stdhl] (0,-.5) node[below]{\small $1$}   .. controls (0,-.25) ..  (-.25,0)
				.. controls (1.5,.5)  .. (1.5,1)
				 .. controls (1.5,1.75) and (.25,1.75) .. (.25,2);
	\draw[fill=white, color=white] (-.35,0) circle (.1cm);
	\draw[decoration={brace,mirror,raise=-8pt},decorate]  (.4,-.85) -- node {\small $k-1$} (1.6,-.85);
	\draw[vstdhl] (-.25,-.5) node[below]{\small $\lambda$} -- (-.25,2);
}
\otimes 
\bar 1_{\ell,\rho}
\end{equation}
using \cref{eq:nhR2andR3} and \cref{eq:nhdotslide}. 
Similarly, using \cref{lem:allcrossingzniszero}, we get
\begin{equation}
\label{eq:phizeroimathgenker2}
{\varphi_k^0}'(\star) = \ 
\tikzdiagh[xscale=1.25]{0}{
	\draw (1.5,-.5) .. controls (1.5,-.4) .. (-.25,.35) .. controls (1.5,1) .. (1.5,2);
	\draw (.25,-.5) .. controls (.25,-.25) and (.5,-.25) .. (.5,0) .. controls (.5,.5) and (0,.5) .. (0,1) node[tikzdot,pos=1]{}
			 .. controls (0,1.5) and (.5,1.5) .. (.5,2);
	\draw (.5,-.5) .. controls (.5,-.25) and (.75,-.25) .. (.75,0) .. controls (.75,.5) and (.25,.5) .. (.25,1) node[tikzdot,pos=1]{}
			 .. controls (.25,1.5) and (.75,1.5) .. (.75,2);
	\node at (.5,1) {\tiny $\dots$};
	\draw (1,-.5) .. controls (1,-.25) and (1.25,-.25) .. (1.25,0) .. controls (1.25,.5) and (.75,.5) .. (.75,1) node[tikzdot,pos=1]{}
			 .. controls (.75,1.5) and (1.25,1.5) .. (1.25,2);
	\draw (1.25,-.5) .. controls (1.25,-.25) and (1.5,-.25) .. (1.5,0) .. controls (1.5,1.5) and (0,1.5) .. (0,2);
	\draw[stdhl] (0,-.5) node[below]{\small $1$}   .. controls (0,-.25) ..  (-.25,0)
				.. controls (1.5,.5)  .. (1.5,1)
				 .. controls (1.5,1.75) and (.25,1.75) .. (.25,2);
	\draw[fill=white, color=white] (-.35,0) circle (.1cm);
	\draw[decoration={brace,mirror,raise=-8pt},decorate]  (.15,-.85) -- node {\small $k-1$} (1.35,-.85);
	\draw[vstdhl] (-.25,-.5) node[below]{\small $\lambda$} -- (-.25,2) node[pos=.335,nail]{};
}
\otimes 
\bar 1_{\ell,\rho}
\end{equation}
On the other hand, we compute
\begin{equation}
\label{eq:phizeroimathgenker3}
\begin{aligned}
(1\otimes\imath_k^{k-1})(\star_s) &=
(-1)^s
 \left(
 \tikzdiagh[xscale=1.25]{0}{
	\draw (.25,-.5) .. controls (.25,-.25) and (0,-.25) .. (0,0) .. controls (0,.5) and (1.25,.5) .. (1.25,2);
	\draw (.5,-.5)  -- (.5,0) .. controls (.5,.5) and (.25,.5) .. (.25,1) node[near start, tikzdot]{}
			 .. controls (.25,1.5) and (.5,1.5) .. (.5,2)  node[near end, tikzdot]{};
	\node at (.5,1) {\tiny $\dots$};
	\draw (1,-.5)  -- (1,0) .. controls (1,.5) and (.75,.5) .. (.75,1)node[near start, tikzdot]{}
			 .. controls (.75,1.5) and (1,1.5) .. (1,2)  node[near end, tikzdot]{};
	\draw (1.25,-.5)  -- (1.25,0) .. controls (1.25,1.5) and (0,1.5) .. (0,2);
	\draw[stdhl] (0,-.5)node[below]{\small $1$}  .. controls (0,-.25) and (.25,-.25) ..(.25,0) 
				.. controls (.25,.5).. (-.25,1)
				 .. controls (.25,1.5)  .. (.25,2);
	\draw[fill=white, color=white] (-.35,1) circle (.1cm);
	\draw[vstdhl] (-.25,-.5) node[below]{\small $\lambda$} -- (-.25,2);
	\draw (1.5,-.5) -- (1.5,2);
	\node at (1.75,1) {\tiny $\dots$};
	\draw (2,-.5) -- (2,2);
	\draw[decoration={brace,mirror,raise=-8pt},decorate]  (.4,-.85) -- node {\small $s$} (1.1,-.85);
}
\otimes
\bar 1_{\ell,\rho}
,
\tikzdiagh[xscale=1.25]{0}{
	\draw (2,-.5) .. controls (2,-.25) .. (-.25,0) 
		.. controls (1.25,1) .. (1.25,2);
	\draw (.25,-.5) .. controls (.25,-.25) and (.5,-.25) .. (.5,0) 
			.. controls (.5,.5) and (.25,.5) .. (.25,1) node[near start, tikzdot]{}
			.. controls (.25,1.5) and (.5,1.5) .. (.5,2)  node[near end, tikzdot]{};
	\node at (.5,1) {\tiny $\dots$};
	\draw (.75,-.5) .. controls (.75,-.25) and (1,-.25) .. (1,0) 
			.. controls (1,.5) and (.75,.5) .. (.75,1)node[near start, tikzdot]{}
			.. controls (.75,1.5) and (1,1.5) .. (1,2)  node[near end, tikzdot]{};
	\draw (1,-.5) .. controls (1,-.25) and (1.25,-.25) .. (1.25,0) 
			.. controls (1.25,1.5) and (0,1.5) .. (0,2);
	\draw[stdhl] (0,-.5)node[below]{\small $1$}  .. controls (0,-.25) and (.25,-.25) ..(.25,0) 
				 .. controls (.25,.5).. (-.25,1)
				 .. controls (.25,1.5)  .. (.25,2);
	\draw[fill=white, color=white] (-.35,1) circle (.1cm);
	\draw[vstdhl] (-.25,-.5) node[below]{\small $\lambda$} -- (-.25,2) node[nail, pos = .2]{};
	\draw (1.25,-.5) .. controls (1.25,-.25) and (1.5,-.25) .. (1.5,0) -- (1.5,2);
	\node at (1.75,1) {\tiny $\dots$};
	\draw (1.75,-.5) .. controls (1.75,-.25) and (2,-.25) ..(2,0) -- (2,2);
	\draw[decoration={brace,mirror,raise=-8pt},decorate]  (.15,-.85) -- node {\small $s$} (.85,-.85);
}
\otimes
\bar 1_{\ell,\rho}
\right)
\\
&= 
(-1)^s
 \left(
  \tikzdiagh[xscale=1.25]{0}{
	\draw (.25,-.5) .. controls (.25,-.25) and (0,-.25) .. (0,0) .. controls (0,.5) and (1.25,.5) .. (1.25,2);
	\draw (.5,-.5)  -- (.5,0) .. controls (.5,.5) and (.25,.5) .. (.25,1) node[near start, tikzdot]{}
			 .. controls (.25,1.5) and (.5,1.5) .. (.5,2)  node[near end, tikzdot]{};
	\node at (.5,1) {\tiny $\dots$};
	\draw (1,-.5)  -- (1,0) .. controls (1,.5) and (.75,.5) .. (.75,1)node[near start, tikzdot]{}
			 .. controls (.75,1.5) and (1,1.5) .. (1,2)  node[near end, tikzdot]{};
	\draw (1.25,-.5)  -- (1.25,0) .. controls (1.25,1.5) and (0,1.5) .. (0,2);
	\draw[stdhl] (0,-.5) node[below]{\small $1$}  .. controls (0,-.375).. (-.25,-.25)
				.. controls (.25,0) .. (.25,.25)
				.. controls (.25,.5) and (0,.5) .. (0,1)
				 .. controls (0,1.5) and (.25,1.5)  .. (.25,2);
	\draw[fill=white, color=white] (-.35,-.25) circle (.1cm);
	\draw[vstdhl] (-.25,-.5) node[below]{\small $\lambda$} -- (-.25,2);
	\draw (1.5,-.5) -- (1.5,2);
	\node at (1.75,1) {\tiny $\dots$};
	\draw (2,-.5) -- (2,2);
	\draw[decoration={brace,mirror,raise=-8pt},decorate]  (.4,-.85) -- node {\small $s$} (1.1,-.85);
}
\otimes
\bar 1_{\ell,\rho}
, 
\tikzdiagh[xscale=1.25]{0}{
	\draw (2,-.5) .. controls (2,-.25) .. (-.25,.25) 
		.. controls (1.25,1) .. (1.25,2);
	\draw (.25,-.5) .. controls (.25,-.25) and (.5,-.25) .. (.5,0) 
			.. controls (.5,.5) and (.25,.5) .. (.25,1) node[near start, tikzdot]{}
			.. controls (.25,1.5) and (.5,1.5) .. (.5,2)  node[near end, tikzdot]{};
	\node at (.5,1) {\tiny $\dots$};
	\draw (.75,-.5) .. controls (.75,-.25) and (1,-.25) .. (1,0) 
			.. controls (1,.5) and (.75,.5) .. (.75,1)node[near start, tikzdot]{}
			.. controls (.75,1.5) and (1,1.5) .. (1,2)  node[near end, tikzdot]{};
	\draw (1,-.5) .. controls (1,-.25) and (1.25,-.25) .. (1.25,0) 
			.. controls (1.25,1.5) and (0,1.5) .. (0,2);
	\draw[stdhl] (0,-.5) node[below]{\small $1$}  .. controls (0,-.375).. (-.25,-.25)
				.. controls (.25,0) .. (.25,.25)
				.. controls (.25,.5) and (0,.5) .. (0,1)
				 .. controls (0,1.5) and (.25,1.5)  .. (.25,2);
	\draw[fill=white, color=white] (-.35,-.25) circle (.1cm);
	\draw[vstdhl] (-.25,-.5) node[below]{\small $\lambda$} -- (-.25,2) node[nail, pos = .3]{};
	\draw (1.25,-.5) .. controls (1.25,-.25) and (1.5,-.25) .. (1.5,0) -- (1.5,2);
	\node at (1.75,1) {\tiny $\dots$};
	\draw (1.75,-.5) .. controls (1.75,-.25) and (2,-.25) ..(2,0) -- (2,2);
	\draw[decoration={brace,mirror,raise=-8pt},decorate]  (.15,-.85) -- node {\small $s$} (.85,-.85);
}
\otimes
\bar 1_{\ell,\rho}
 \right)
\end{aligned}
\end{equation}
using \cref{eq:redR2} and \cref{eq:nailslidedcross}.
Then, we conclude by observing that \cref{eq:phizeroimathgenker} follows by applying \cref{lem:dotanddoubledots} on \cref{eq:phizeroimathgenker1}, \cref{eq:phizeroimathgenker2} and \cref{eq:phizeroimathgenker3} together. 
\end{proof}

\begin{lem}
As a left $(T^{\lambda,r},0)$-module, $\ker(u \otimes \gamma_k)$ is generated by the elements 
\begin{equation}\label{eq:kergenelem}
\left(
-\ 
\tikzdiagh{0}{
	\draw (1.25,0) .. controls (1.25,.5) .. (-.5,.75) .. controls (0,.875) .. (0,1);
	\draw[stdhl] (0,0) node[below]{\small $1$} .. controls (0,.1) .. (-.5,.25) .. controls (.5,.7) .. (.5,1);
	\draw[fill=white, color=white] (-.6,.25) circle (.1cm);
	\draw[vstdhl] (-.5,0) node[below]{\small $\lambda$} --(-.5,1) node[pos=.75,nail]{};
	\draw (.5,0) .. controls (.5,.5) and (.75,.5) .. (.75,1);
	\node at (.75,.15){\tiny $\dots$};
	\draw (1,0) .. controls (1,.5) and (1.25,.5)..  (1.25,1);
	\draw[decoration={brace,mirror,raise=-8pt},decorate]  (.4,-.35) -- node { \small $t$} (1.1,-.35);
}
\otimes \bar  1_{\ell+k-1-t,\rho}
,
\tikzdiagh{0}{
  \draw (.5,0) .. controls (.5,.5) and (0,.5) .. (0,1);
  \draw[stdhl] (0,0) node[below]{\small $1$} .. controls (0,.1) .. (-.5,.25) .. controls (.5,.7) .. (.5,1);
  \draw[fill=white, color=white] (-.6,.25) circle (.1cm);
  \draw[vstdhl] (-.5,0) node[below]{\small $\lambda$} --(-.5,1);
}
\otimes \bar  1_{\ell+k-1,\rho}
\right)
\in
 (X\otimes_T {Y'}^0_k) 
\oplus
(X\otimes_T Y^{0,t}_k)
\end{equation}
for all $0 \leq t \leq k-1$. 
\end{lem}

\begin{proof}
Let $K \subset X \otimes_T Y^0_k$ be the submodule generated by the elements in \cref{eq:kergenelem}. A straightforward computation shows that $K \subset \ker(u \otimes \gamma_k)$. Thus, we have a complex
\begin{equation}\label{eq:SESKXTX}
0 \rightarrow K \hookrightarrow X \otimes_T Y_k^0 \overset{u \otimes \gamma_k}{\twoheadrightarrow} \lambda^{-1} X_k \rightarrow 0,
\end{equation}
where the left arrow is an injection and the right arrow is a surjection. Furthermore, by \cref{thm:X0basis}, we have that 
\begin{align*}
 K &\cong \lambda^{-1}   Y_k^1, &
 X \otimes_T Y_k^0 &\cong \lambda^{-1}  Y_k^0.
\end{align*}
Therefore, by \cref{lem:sesX0} we obtain that the sequence in \cref{eq:SESKXTX} is exact. In particular, we have $K = \ker(u \otimes \gamma_k)$. 
\end{proof}

\begin{prop}\label{prop:kerequalsimg1}
We have $\ker(u \otimes \gamma_k) = \Image(\varphi_k^0 + 1\otimes \imath_k)$. 
\end{prop}

\begin{proof}
We will show by backward induction on $t$ that the elements \cref{eq:kergenelem} are all in $\Image(\varphi_k^0 + 1\otimes \imath_k)$. The case $t = k-1$ is \cref{lem:phizeroimathgenker}. The induction step is essentially similar to the proof of \cref{lem:phizeroimathgenker}. 
In particular we want to show that \cref{eq:kergenelem} is in $\bigcup_{t' \geq t} \Image(\varphi_k^0 + 1\otimes \imath_k^{t'})$. 
For this, we write
\begin{align*}
\star &:=
\tikzdiagh[xscale=1.25]{0}{
	\draw[vstdhl] (-.25,0) node[below]{\small $\lambda$} -- (-.25,2);
	\draw (0,0) .. controls (0,1) and (.5,1) .. (.5,2);
	\node at (.25,.25) {\tiny $\dots$}; \node at (.75,1.75) {\tiny $\dots$};
	\draw (.5,0) .. controls (.5,1) and (1,1) .. (1,2);
	\draw (.75,0) .. controls (.75,1) and (0,1) .. (0,2);
	\draw (1,0) .. controls (1,1) and (1.25,1) .. (1.25,2);
	\node at (1.25,.25) {\tiny $\dots$}; \node at (1.5,1.75) {\tiny $\dots$};
	\draw (1.5,0) .. controls (1.5,1) and (1.75,1) .. (1.75,2);
	\draw[stdhl] (2,0) node[below]{\small $1$} 
				 .. controls (2,1) and (.25,1) .. (.25,2);
	\draw[decoration={brace,mirror,raise=-8pt},decorate]  (-.1,-.35) -- node {\small $t$} (.65,-.35);
}
\otimes 
\bar 1_{\ell,\rho},
\end{align*}
Then, we compute using \cref{lem:allcrossingzniszero} and \cref{lem:computezn} 
\[
\varphi_k^{0,t'}(\star) =
\begin{cases}
\quad 0, & \text{if $t' < t$,} \\
\tikzdiagh[xscale=1.25]{0}{
	\draw (.25,-.5) .. controls (.25,-.1) and (0,-.1) .. (0,.15) .. controls (0,.5) and (1.5,.5) .. (1.5,2);
	\draw (.5,-.5) -- (.5,0) .. controls (.5,.5) and (0,.5) .. (0,1) node[tikzdot,pos=1]{}
			 .. controls (0,1.5) and (.5,1.5) .. (.5,2);
	\draw (.75,-.5) -- (.75,0) .. controls (.75,.5) and (.25,.5) .. (.25,1) node[tikzdot,pos=1]{}
			 .. controls (.25,1.5) and (.75,1.5) .. (.75,2);
	\node at (.5,1) {\tiny $\dots$};
	\draw (1.25,-.5) -- (1.25,0) .. controls (1.25,.5) and (.75,.5) .. (.75,1) node[tikzdot,pos=1]{}
			 .. controls (.75,1.5) and (1.25,1.5) .. (1.25,2);
	\draw (1.5,-.5) -- (1.5,0) .. controls (1.5,1.5) and (0,1.5) .. (0,2);
	\draw[stdhl] (0,-.5) node[below]{\small $1$}   .. controls (0,-.25) ..  (-.25,0)
				.. controls (1.5,.5)  .. (1.5,1)
				 .. controls (1.5,1.75) and (.25,1.75) .. (.25,2);
	\draw[fill=white, color=white] (-.35,0) circle (.1cm);
	\draw[decoration={brace,mirror,raise=-8pt},decorate]  (.4,-.85) -- node {\small $k-1$} (1.6,-.85);
	\draw[vstdhl] (-.25,-.5) node[below]{\small $\lambda$} -- (-.25,2);
}
\otimes 
\bar 1_{\ell,\rho},
& \text{if $t' = t$,} \\
\tikzdiagh[xscale=1.25]{0}{
	\draw(.25,-.5) .. controls (.25,0) and (0,0) .. (0,.25) .. controls (0,.5) and (1.75,.5) .. (1.75,.75) .. controls (1.75,1.5) and (2,1.5) .. (2,2);
	\draw (.5,-.5) .. controls (.5,.25) and (0,.25) .. (0,1) .. controls (0,1.5) and (.5,1.5) .. (.5,2) node[pos=0,tikzdot]{};
	\node at (.75,-.4) {\tiny $\dots$}; \node at (.25,1) {\tiny $\dots$};
	\draw (1,-.5) .. controls (1,.25) and (.5,.25) .. (.5,1) .. controls (.5,1.5) and (1,1.5) .. (1,2) node[pos=0,tikzdot]{};
	\draw (1.25,-.5) .. controls (1.25,.25) and (.75,.25) .. (.75,1) .. controls (.75,1.5) and (0,1.5) .. (0,2)  node[pos=0,tikzdot]{};
	\draw (1.5,-.5) .. controls (1.5,.25) and (1,.25) .. (1,1) .. controls (1,1.5) and (1.25,1.5) .. (1.25,2)  node[pos=0,tikzdot]{};
	\node at (1.75,-.4) {\tiny $\dots$}; \node at (1.25,1) {\tiny $\dots$};
	\draw (2,-.5) .. controls (2,.25) and (1.5,.25) .. (1.5,1) .. controls (1.5,1.5) and (1.75,1.5) .. (1.75,2)  node[pos=0,tikzdot]{};
	\draw (2.25,-.5) .. controls (2.25,.25) and (2,.25) .. (2,1) .. controls (2,1.5) and (2.25,1.5) .. (2.25,2);
	\node at (2.5,-.4) {\tiny $\dots$}; \node at (2.5,1.9) {\tiny $\dots$};
	\draw (2.75,-.5) .. controls (2.75,.25) and (2.5,.25) .. (2.5,1) .. controls (2.5,1.5) and (2.75,1.5) .. (2.75,2);
	\filldraw [fill=white, draw=black,rounded corners] (1.625,.75) rectangle (2.625,1.25) node[midway] { $z_{k-t'}$};
	\draw[stdhl] (0,-.5) node[below]{\small $1$}   .. controls (0,-.25) ..  (-.25,0)
				.. controls (2.75,.25)  .. (2.75,1) -- (2.75,1.25)
				 .. controls (2.75,1.75) and (.25,1.75) .. (.25,2);
	\draw[fill=white, color=white] (-.35,0) circle (.1cm);
	\draw[decoration={brace,raise=-8pt},decorate]  (.4,2.35) -- node {\small $t$} (1.1,2.35);
	\draw[decoration={brace,mirror,raise=-8pt},decorate]  (.4,-.85) -- node {\small $t'$} (2.1,-.85);
	\draw[vstdhl] (-.25,-.5) node[below]{\small $\lambda$} -- (-.25,2);
}
\otimes 
\bar 1_{\ell,\rho},
& \text{if $t' > t$,}
\end{cases}
\]
 for $0 \leq t' \leq k-1$, and 
\[
{\varphi'}_k^0(\star) 
= -\ 
\tikzdiagh[xscale=1.25]{0}{
	\draw (1,-.5) .. controls (1,-.4) .. (-.25,.35) .. controls (2,1) .. (2,2);
	\draw (.25,-.5) .. controls (.25,-.25) and (.5,-.25) .. (.5,0) .. controls (.5,.5) and (0,.5) .. (0,1) node[tikzdot,pos=1]{}
			 .. controls (0,1.5) and (.5,1.5) .. (.5,2);
	\node at (.25,1) {\tiny $\dots$};
	\draw (.75,-.5) .. controls (.75,-.25) and (1,-.25) .. (1,0) .. controls (1,.5) and (.5,.5) .. (.5,1) node[tikzdot,pos=1]{}
			 .. controls (.5,1.5) and (1,1.5) .. (1,2);
	\draw (1.25,-.5)  .. controls (1.25,.5) and (.75,.5) .. (.75,1) node[tikzdot,pos=1]{}
			 .. controls (.75,1.5) and (1.25,1.5) .. (1.25,2);
	\node at (1,1) {\tiny $\dots$};
	\draw (1.75,-.5)  .. controls (1.75,.5) and (1.25,.5) .. (1.25,1) node[tikzdot,pos=1]{}
			 .. controls (1.25,1.5) and (1.75,1.5) .. (1.75,2);
	\draw (2,-.5)  .. controls (2,1.75) and (0,1.5) .. (0,2);
	\draw[stdhl] (0,-.5) node[below]{\small $1$}   .. controls (0,-.25) ..  (-.25,0)
				.. controls (2,.5)  .. (2,1)
				 .. controls (2,1.75) and (.25,1.75) .. (.25,2);
	\draw[fill=white, color=white] (-.35,0) circle (.1cm);
	\draw[decoration={brace,mirror,raise=-8pt},decorate]  (.15,-.85) -- node {\small $t$} (.85,-.85);
	\draw[vstdhl] (-.25,-.5) node[below]{\small $\lambda$} -- (-.25,2) node[pos=.335,nail]{};
}
\ \otimes 
\bar 1_{\ell,\rho}
- \sum_{t' = t+1}^{k-1} \ 
\tikzdiagh[xscale=1.25]{0}{
	\draw(2,-.5) .. controls (2,.25)  .. (-.25,.375) .. controls (1.75,.5) .. (1.75,.75) .. controls (1.75,1.5) and (2,1.5) .. (2,2);
	\draw (.25,-.5) .. controls (.25,.25) and (0,.25) .. (0,1) .. controls (0,1.5) and (.5,1.5) .. (.5,2) node[pos=0,tikzdot]{};
	\node at (.5,-.4) {\tiny $\dots$}; \node at (.25,1) {\tiny $\dots$};
	\draw (.75,-.5) .. controls (.75,.25) and (.5,.25) .. (.5,1) .. controls (.5,1.5) and (1,1.5) .. (1,2) node[pos=0,tikzdot]{};
	\draw (1,-.5) .. controls (1,.25) and (.75,.25) .. (.75,1) .. controls (.75,1.5) and (0,1.5) .. (0,2)  node[pos=0,tikzdot]{};
	\draw (1.25,-.5) .. controls (1.25,.25) and (1,.25) .. (1,1) .. controls (1,1.5) and (1.25,1.5) .. (1.25,2)  node[pos=0,tikzdot]{};
	\node at (1.5,-.4) {\tiny $\dots$}; \node at (1.25,1) {\tiny $\dots$};
	\draw (1.75,-.5) .. controls (1.75,.25) and (1.5,.25) .. (1.5,1) .. controls (1.5,1.5) and (1.75,1.5) .. (1.75,2)  node[pos=0,tikzdot]{};
	\draw (2.25,-.5) .. controls (2.25,.25) and (2,.25) .. (2,1) .. controls (2,1.5) and (2.25,1.5) .. (2.25,2);
	\node at (2.5,-.4) {\tiny $\dots$}; \node at (2.5,1.9) {\tiny $\dots$};
	\draw (2.75,-.5) .. controls (2.75,.25) and (2.5,.25) .. (2.5,1) .. controls (2.5,1.5) and (2.75,1.5) .. (2.75,2);
	\filldraw [fill=white, draw=black,rounded corners] (1.625,.75) rectangle (2.625,1.25) node[midway] { $z_{k-t'}$};
	\draw[stdhl] (0,-.5) node[below]{\small $1$}   .. controls (0,-.375) ..  (-.25,-.25)
				.. controls (2.75,0)  .. (2.75,.75) -- (2.75,1.25)
				 .. controls (2.75,1.75) and (.25,1.75) .. (.25,2);
	\draw[fill=white, color=white] (-.35,-.25) circle (.1cm);
	\draw[decoration={brace,raise=-8pt},decorate]  (.4,2.35) -- node {\small $t$} (1.1,2.35);
	\draw[decoration={brace,mirror,raise=-8pt},decorate]  (.15,-.85) -- node {\small $t'$} (1.85,-.85);
	\draw[vstdhl] (-.25,-.5) node[below]{\small $\lambda$} -- (-.25,2) node[pos=.35,nail]{};
}
         \ \otimes 
\bar 1_{\ell,\rho}
\]
Thus, since $t' > t$, by induction hypothesis we know that
\begin{align*}
\varphi_k^0(\star) \equiv \left(
- \ 
\tikzdiagh[xscale=1.25]{0}{
	\draw (1,-.5) .. controls (1,-.4) .. (-.25,.35) .. controls (2,1) .. (2,2);
	\draw (.25,-.5) .. controls (.25,-.25) and (.5,-.25) .. (.5,0) .. controls (.5,.5) and (0,.5) .. (0,1) node[tikzdot,pos=1]{}
			 .. controls (0,1.5) and (.5,1.5) .. (.5,2);
	\node at (.25,1) {\tiny $\dots$};
	\draw (.75,-.5) .. controls (.75,-.25) and (1,-.25) .. (1,0) .. controls (1,.5) and (.5,.5) .. (.5,1) node[tikzdot,pos=1]{}
			 .. controls (.5,1.5) and (1,1.5) .. (1,2);
	\draw (1.25,-.5)  .. controls (1.25,.5) and (.75,.5) .. (.75,1) node[tikzdot,pos=1]{}
			 .. controls (.75,1.5) and (1.25,1.5) .. (1.25,2);
	\node at (1,1) {\tiny $\dots$};
	\draw (1.75,-.5)  .. controls (1.75,.5) and (1.25,.5) .. (1.25,1) node[tikzdot,pos=1]{}
			 .. controls (1.25,1.5) and (1.75,1.5) .. (1.75,2);
	\draw (2,-.5)  .. controls (2,1.75) and (0,1.5) .. (0,2);
	\draw[stdhl] (0,-.5) node[below]{\small $1$}   .. controls (0,-.25) ..  (-.25,0)
				.. controls (2,.5)  .. (2,1)
				 .. controls (2,1.75) and (.25,1.75) .. (.25,2);
	\draw[fill=white, color=white] (-.35,0) circle (.1cm);
	\draw[decoration={brace,mirror,raise=-8pt},decorate]  (.15,-.85) -- node {\small $t$} (.85,-.85);
	\draw[vstdhl] (-.25,-.5) node[below]{\small $\lambda$} -- (-.25,2) node[pos=.335,nail]{};
}
\ \otimes 
\bar 1_{\ell,\rho}
,
\tikzdiagh[xscale=1.25]{0}{
	\draw (.25,-.5) .. controls (.25,-.1) and (0,-.1) .. (0,.15) .. controls (0,.5) and (1.5,.5) .. (1.5,2);
	\draw (.5,-.5) -- (.5,0) .. controls (.5,.5) and (0,.5) .. (0,1) node[tikzdot,pos=1]{}
			 .. controls (0,1.5) and (.5,1.5) .. (.5,2);
	\draw (.75,-.5) -- (.75,0) .. controls (.75,.5) and (.25,.5) .. (.25,1) node[tikzdot,pos=1]{}
			 .. controls (.25,1.5) and (.75,1.5) .. (.75,2);
	\node at (.5,1) {\tiny $\dots$};
	\draw (1.25,-.5) -- (1.25,0) .. controls (1.25,.5) and (.75,.5) .. (.75,1) node[tikzdot,pos=1]{}
			 .. controls (.75,1.5) and (1.25,1.5) .. (1.25,2);
	\draw (1.5,-.5) -- (1.5,0) .. controls (1.5,1.5) and (0,1.5) .. (0,2);
	\draw[stdhl] (0,-.5) node[below]{\small $1$}   .. controls (0,-.25) ..  (-.25,0)
				.. controls (1.5,.5)  .. (1.5,1)
				 .. controls (1.5,1.75) and (.25,1.75) .. (.25,2);
	\draw[fill=white, color=white] (-.35,0) circle (.1cm);
	\draw[decoration={brace,mirror,raise=-8pt},decorate]  (.4,-.85) -- node {\small $k-1$} (1.6,-.85);
	\draw[vstdhl] (-.25,-.5) node[below]{\small $\lambda$} -- (-.25,2);
}
\otimes 
\bar 1_{\ell,\rho}
\right)
&+
\Image(\varphi_k^0 + 1\otimes \imath_k)  \\
&\in (X\otimes_T Y'_0) \oplus (X\otimes_T Y^{0,t}_k).
\end{align*}
Then, by the same arguments as in \cref{lem:phizeroimathgenker}, that is using \cref{eq:dotanddoubledots}, we obtain \cref{eq:kergenelem}.
\end{proof}

\begin{lem}\label{lem:varphi0kinjective}
The map $\varphi^0_k$ is injective.
\end{lem}

\begin{proof}
Since adding black/red crossings is injective, it is enough because of \cref{lem:computezn} to show that the left $T_k^{\lambda,0}$-module map
\begin{equation}\label{eq:phi0injective}
T_k^{\lambda,0} \rightarrow \bigoplus_{t=0}^{k-1} q^{2(k-1-t)} T_k^{\lambda,0},
\quad
\tikzdiag
{
	\draw[vstdhl] (0,0) -- (0,1);
	\draw (.25,0) -- (.25,1);
	\node at(.5,.5) {\tiny $\dots$};
	\draw (.75,0) -- (.75,1);
}
\mapsto 
\left(
\tikzdiag
{
	\draw[vstdhl]node[below]{\small$\lambda$} (0,0) -- (0,1);
	\draw (.25,0) .. controls (.25,.25) and (2,.25) .. (2,.5) .. controls (2,.75) and (1,.75) .. (1,1);
	\draw (.5,0) .. controls (.5,.5) and (.25,.5) .. (.25,1) node[midway, tikzdot]{};
	\node at (.75,.1){\tiny $\dots$};
	\draw (1,0) .. controls (1,.5) and (.75,.5) .. (.75,1) node[midway, tikzdot]{};
	\draw[decoration={brace,mirror,raise=-8pt},decorate]  (.4,-.35) -- node {\small $t$} (1.1,-.35);
	\draw (1.25,0)-- (1.25,1)  node[midway, tikzdot]{};
	\node at(1.5,.1){\tiny $\dots$};
	\draw (1.75,0)-- (1.75,1)  node[midway, tikzdot]{};
}
\right)_{0 \leq t < k}
\end{equation}
is injective. Since $T_k^{\lambda,0}$ is isomorphic to the dg-enhanced nilHecke algebra of \cite{naissevaz2}, we know by the results in \cite[Proposition 2.5]{naissevaz2} that there is a decomposition
\begin{align}\label{eq:Andecomp}
T_k^{\lambda,0} &\cong  \bigoplus_{t' \geq 0} P_{k,t'} ,
&
P_{k,t'} := \bigoplus_{p \geq 0}
\tikzdiag{
	\draw[vstdhl]node[below]{\small$\lambda$} (-.25,0) -- (-.25,2);
	\draw (.25,0) .. controls (.25,.5) and (.5,.5) .. (.5,1) -- (.5,2); 
	\node at(.5,.1){\tiny $\dots$};
	\draw (.75,0) .. controls (.75,.5) and (1,.5) .. (1,1) -- (1,2); 
	\draw[decoration={brace,mirror,raise=-8pt},decorate]  (.15,-.35) -- node {\small $t'$} (.85,-.35);
	\draw (1,0) .. controls (1,.5) and (.25,.5) .. (.25,1) -- (.25,2) node[pos=.75,tikzdot]{}  node[pos=.75,xshift=-1ex,yshift=.75ex]{\small $p$};
	\draw (1.25,0) -- (1.25,2);
	\node at(1.5,.1){\tiny $\dots$};
	\draw (1.75,0) -- (1.75,2);
	\filldraw [fill=white, draw=black] (.375,1.15) rectangle (1.875,1.65) node[midway] { $\nh_{k-1}$};
	\filldraw [fill=white, draw=black] (2.25,1.25) circle (1.5ex) node { $k$};
}
\end{align}
where the box labeled $\nh_{k-1}$ is the nilHecke algebra, and the circle labeled $k$ is the algebra generated by labeled floating dots in the rightmost region (see \cite[\S2.4]{naissevaz2}). These floating dots correspond to combinations of nails, dots and crossings, giving elements that are in the (graded w.r.t. the homological degree) center of $T_k^{\lambda,0}$. 

Furthermore, the map
\[
\nh_{k-1} \rightarrow q^{2k-2} \nh_k, \quad
\tikzdiag{
	\draw(.25,1) -- (.25,2);
	\node at(.75,1.1){\tiny$\dots$};
	\node at(.75,1.9){\tiny$\dots$};
	\draw(1.25,1) -- (1.25,2);
	\filldraw [fill=white, draw=black] (0,1.25) rectangle (1.5,1.75) node[midway] { $\nh_{k-1}$};
}
\mapsto
\tikzdiag{
	\draw (0,0) .. controls (0,.25) and (1.5,.25) .. (1.5,.5) .. controls (1.5,.75) and (0,.75) .. (0,1)  -- (0,2);
	\draw (.5,0) .. controls (.5,.25) and (0,.25) .. (0,.5) node[pos=1,tikzdot]{} .. controls (0,.75) and (.5,.75) .. (.5,1) -- (.5,2);
	\node at(.5,.5){\tiny $\dots$};
	\draw (1.5,0) .. controls (1.5,.25) and (1,.25) .. (1,.5) node[pos=1,tikzdot]{} .. controls (1,.75) and (1.5,.75) .. (1.5,1) -- (1.5,2);
	\filldraw [fill=white, draw=black] (.25,1.25) rectangle (1.75,1.75) node[midway] { $\nh_{k-1}$};
}
\]
is injective (this can be deduced by sliding all dots to the bottom using \cref{eq:nhdotslide}, and then using a basis theorem as for example in \cite[Theorem 2.5]{KL1}  to see the map takes the form of a column echelon matrix with $1$ as pivots).

Then, applying \cref{eq:phi0injective} on $P_{k,t'}$ yields
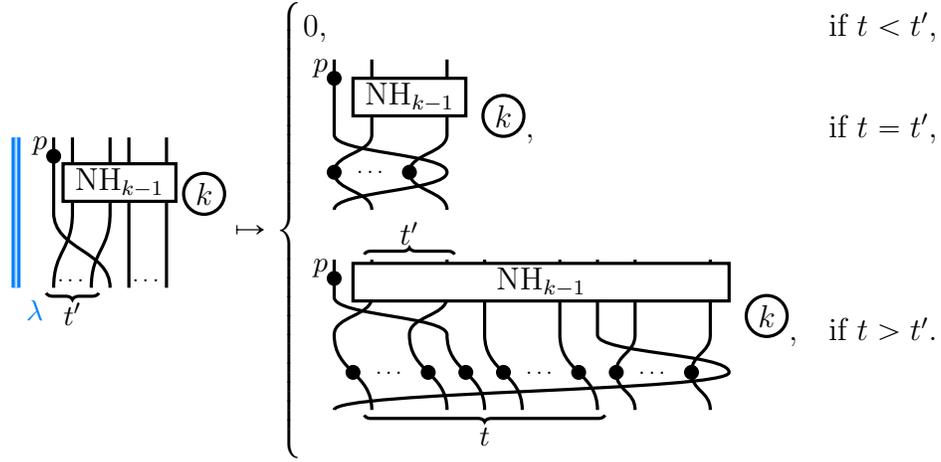
\[
\tikzdiag{
	\draw[vstdhl]node[below]{\small$\lambda$} (-.25,0) -- (-.25,2);
	\draw (.25,0) .. controls (.25,.5) and (.5,.5) .. (.5,1) -- (.5,2); 
	\node at(.5,.1){\tiny $\dots$};
	\draw (.75,0) .. controls (.75,.5) and (1,.5) .. (1,1) -- (1,2); 
	\draw[decoration={brace,mirror,raise=-8pt},decorate]  (.15,-.35) -- node {\small $t'$} (.85,-.35);
	\draw (1,0) .. controls (1,.5) and (.25,.5) .. (.25,1) -- (.25,2) node[pos=.75,tikzdot]{}  node[pos=.75,xshift=-1ex,yshift=.75ex]{\small $p$};
	\draw (1.25,0) -- (1.25,2);
	\node at(1.5,.1){\tiny $\dots$};
	\draw (1.75,0) -- (1.75,2);
	\filldraw [fill=white, draw=black] (.375,1.15) rectangle (1.875,1.65) node[midway] { $\nh_{k-1}$};
	\filldraw [fill=white, draw=black] (2.25,1.25) circle (1.5ex) node { $k$};
}
\mapsto
\begin{cases}
0, & \text{if $t < t'$},\\
\tikzdiag{
	\draw (0,0) .. controls (0,.25) and (1.5,.25) .. (1.5,.5) .. controls (1.5,.75) and (0,.75) .. (0,1)  -- (0,2) node[pos=.75,tikzdot]{}  node[pos=.75,xshift=-1ex,yshift=.75ex]{\small $p$};
	\draw (.5,0) .. controls (.5,.25) and (0,.25) .. (0,.5) node[pos=1,tikzdot]{} .. controls (0,.75) and (.5,.75) .. (.5,1) -- (.5,2);
	\node at(.5,.5){\tiny $\dots$};
	\draw (1.5,0) .. controls (1.5,.25) and (1,.25) .. (1,.5) node[pos=1,tikzdot]{} .. controls (1,.75) and (1.5,.75) .. (1.5,1) -- (1.5,2);
	\filldraw [fill=white, draw=black] (.25,1.25) rectangle (1.75,1.75) node[midway] { $\nh_{k-1}$};
	\filldraw [fill=white, draw=black] (2.25,1.25) circle (1.5ex) node { $k$};
}, &  \text{if $t = t'$}, \\
\tikzdiag{
	\draw (0,0)  .. controls (0,.25) and (5.25,.25) .. (5.25,.5) .. controls (5.25,.75) and (3.5,.75) .. (3.5,1)  -- (3.5,2);
	\draw (.5,0) .. controls (.5,.5) and (0,.5) .. (0,1) node[midway, tikzdot]{} .. controls (0,1.25) and (.5,1.25) .. (.5,1.5)  -- (.5,2);
	\node at(.75,.5){\tiny$\dots$};
	\draw (1.5,0) .. controls (1.5,.5) and (1,.5) .. (1,1) node[midway, tikzdot]{}  .. controls (1,1.25) and (1.5,1.25) .. (1.5,1.5)  -- (1.5,2);
	\draw (2,0) .. controls (2,.5) and (1.5,.5) .. (1.5,1) node[midway, tikzdot]{}  .. controls (1.5,1.25) and (0,1.25) .. (0,1.5) -- (0,2) node[pos=.5,tikzdot]{}  node[pos=.5,xshift=-1ex,yshift=.75ex]{\small $p$};
	\draw (2.5,0) .. controls (2.5,.5) and (2,.5) .. (2,1) node[midway, tikzdot]{} -- (2,2);
	\node at(2.75,.5){\tiny$\dots$};
	\draw (3.5,0) .. controls (3.5,.5) and (3,.5) .. (3,1) node[midway, tikzdot]{}  -- (3,2);
	\draw (4,0) .. controls (4,.25) and (3.75,.25) .. (3.75,.5) .. controls (3.75,.75) and (4,.75) .. (4,1) node[pos=0, tikzdot]{} -- (4,2);
	\node at(4.25,.5){\tiny$\dots$};
	\draw (5,0) .. controls (5,.25) and (4.75,.25) .. (4.75,.5) .. controls (4.75,.75) and (5,.75) .. (5,1)  node[pos=0, tikzdot]{}-- (5,2);
	\filldraw [fill=white, draw=black] (.25,1.45) rectangle (5.25,1.95) node[midway] { $\nh_{k-1}$};
	\filldraw [fill=white, draw=black] (5.75,1.25) circle (1.5ex) node { $k$};
	\draw[decoration={brace,mirror,raise=-8pt},decorate]  (.4,-.35) -- node {\small $t$} (3.6,-.35);
	\draw[decoration={brace,raise=-8pt},decorate]  (.4,2.35) -- node {$t'$} (1.6,2.35); 
}
, & \text{if $t > t'$}.
\end{cases}
\]
Therefore, after decomposing $T^{\lambda,0}_k$,  \cref{eq:phi0injective} yields a column echelon form matrix with injective maps as pivots, and thus is injective.
\end{proof}

\begin{prop}\label{prop:kerequalsimg2}
We have $\ker((1\otimes \imath_k) + \varphi_k^0) = \Image(\varphi_k^1 - u)$.
\end{prop}

\begin{proof}
First, recall that $ \imath_k$ is injective (as explained in \cref{sec:cofX}). Thus, both $(1\otimes \imath_k)$ and $\varphi_k^0$ are injective, and we get
\[
\ker((1\otimes \imath_k) + \varphi_k^0) \cong \Image(1\otimes \imath_k) \cap \Image(\varphi_k^0).
\]
We observe that $\Image(1\otimes \imath_k) \cap \Image(\varphi_k^{0}) \cap (X\otimes_T Y^{0,t}_k)$ is generated by
\[
	\varphi_k^{0,t}\left(
	\tikzdiagh{0}{
		\draw[vstdhl] (-.5,0) node[below]{\small $\lambda$} -- (-.5,1);
		\draw (0,0) .. controls (0,.5) and (.5,.5) .. (.5,1);
		\draw (.25,0) .. controls (.25,.5) and (.75,.5) .. (.75,1);
		\node at (.55,.15){\tiny $\dots$};
		\draw (.75,0) .. controls (.75,.5) and (1.25,.5) .. (1.25,1);
		\draw[stdhl] (1.25,0) node[below]{\small $1$} .. controls (1.25,.5) and (-.25,.5) .. (-.25,1);
	}
	\otimes \bar 1_{\ell,\rho}
	\right)
	=
	(1 \otimes  \imath_k^{t}) \left(
	\tikzdiag{
		\draw (0,-.5) .. controls (0,0) and (1,0) .. (1,.5) -- (1,1);
		\draw (.75,-.5) .. controls (.75,0) and (.25,0) .. (.25,.5)node[pos=.15, tikzdot]{}   -- (.25,1) node[pos=.5, tikzdot]{}  ;
		\node at(1,-.4) {\tiny $\dots$}; \node at(.55,.35) {\tiny $\dots$};
		\draw (1.25,-.5) .. controls (1.25,0) and (.75,0) .. (.75,.5)node[pos=.15, tikzdot]{}  -- (.75,1)  node[pos=.5, tikzdot]{}  ;
		\draw (1.5,-.5) .. controls (1.5,0) and (1.25,0) .. (1.25,.5)node[pos=.15, tikzdot]{}   -- (1.25,1);
		\node at(1.75,-.4) {\tiny $\dots$};  \node at(1.55,.15) {\tiny $\dots$};
		\draw (2,-.5) .. controls (2,0) and (1.75,0) .. (1.75,.5)node[pos=.15, tikzdot]{}   -- (1.75,1);
		\filldraw [fill=white, draw=black,rounded corners] (.875,.4) rectangle (1.875,.9) node[midway] { $z_{k-t}$};
		\draw[decoration={brace,mirror,raise=-8pt},decorate]  (.65,-.85) -- node {\small $t$} (1.35,-.85);
		\draw[stdhl] (.375,-.5) node[below]{\small $1$}   .. controls (.375,-.25) .. (-.5,0)
				.. controls (-.25,.25) ..  (-.25,.5) -- (-.25,1);
		\draw[fill=white, color=white] (-.65,0) circle (.15cm);
		\draw[vstdhl] (-.5,-.5) node[below]{\small $\lambda$} -- (-.5,1)  ; 
	}
	\otimes \bar 1_{\ell,\rho}
	\right)
	=
	\tikzdiag{
		\draw (.25,-.5) .. controls (.25,0) and (1,0) .. (1,.5)node[pos=.15, tikzdot]{}  -- (1,1);
		\draw (.75,-.5) .. controls (.75,0) and (.25,0) .. (.25,.5)node[pos=.15, tikzdot]{}   -- (.25,1) node[pos=.5, tikzdot]{}  ;
		\node at(1,-.4) {\tiny $\dots$}; \node at(.55,.35) {\tiny $\dots$};
		\draw (1.25,-.5) .. controls (1.25,0) and (.75,0) .. (.75,.5)node[pos=.15, tikzdot]{}  -- (.75,1)  node[pos=.5, tikzdot]{}  ;
		\draw (1.5,-.5) .. controls (1.5,0) and (1.25,0) .. (1.25,.5)node[pos=.15, tikzdot]{}   -- (1.25,1);
		\node at(1.75,-.4) {\tiny $\dots$}; \node at(1.55,.15) {\tiny $\dots$};
		\draw (2,-.5) .. controls (2,0) and (1.75,0) .. (1.75,.5)node[pos=.15, tikzdot]{}   -- (1.75,1);
		\filldraw [fill=white, draw=black,rounded corners] (.875,.4) rectangle (1.875,.9) node[midway] { $z_{k-t}$};
		\draw[decoration={brace,mirror,raise=-8pt},decorate]  (.65,-.85) -- node {\small $t$} (1.35,-.85);
		\draw[stdhl] (-.25,-.5) node[below]{\small $1$}   .. controls (-.25,-.25) .. (-.5,0)
				.. controls (-.25,.25) ..  (-.25,.5) -- (-.25,1);
		\draw[fill=white, color=white] (-.65,0) circle (.15cm);
		\draw[vstdhl] (-.5,-.5) node[below]{\small $\lambda$} -- (-.5,1)  ; 
	}
	\otimes \bar 1_{\ell,\rho}.
\]
and by
\[
	\varphi_k^{0,t}\left(
	\tikzdiagh{0}{
		\draw (0,0)  .. controls (0,.125) .. (-.5,.25) .. controls (.5,.75) .. (.5,1);
		\draw (.25,0) .. controls (.25,.5) and (.75,.5) .. (.75,1);
		\node at (.55,.15){\tiny $\dots$};
		\draw (.75,0) .. controls (.75,.5) and (1.25,.5) .. (1.25,1);
		\draw[stdhl] (1.25,0) node[below]{\small $1$} .. controls (1.25,.5) and (-.25,.5) .. (-.25,1);
		\draw[vstdhl] (-.5,0) node[below]{\small $\lambda$} -- (-.5,1) node[pos=.25,nail]{};
	}
	\otimes \bar 1_{\ell,\rho}
	\right)
	=
	(1 \otimes  \imath_k^{t}) \left(
	\tikzdiag{
		\draw (0,-.5) .. controls (0,-.375) .. (-.5,-.25) .. controls (1,0)  .. (1,.5) -- (1,1);
		\draw (.75,-.5) .. controls (.75,0) and (.25,0) .. (.25,.5)node[pos=.15, tikzdot]{}   -- (.25,1) node[pos=.5, tikzdot]{}  ;
		\node at(1,-.4) {\tiny $\dots$}; \node at(.55,.35) {\tiny $\dots$};
		\draw (1.25,-.5) .. controls (1.25,0) and (.75,0) .. (.75,.5)node[pos=.15, tikzdot]{}  -- (.75,1)  node[pos=.5, tikzdot]{}  ;
		\draw (1.5,-.5) .. controls (1.5,0) and (1.25,0) .. (1.25,.5)node[pos=.15, tikzdot]{}   -- (1.25,1);
		\node at(1.75,-.4) {\tiny $\dots$};  \node at(1.55,.15) {\tiny $\dots$};
		\draw (2,-.5) .. controls (2,0) and (1.75,0) .. (1.75,.5)node[pos=.15, tikzdot]{}   -- (1.75,1);
		\filldraw [fill=white, draw=black,rounded corners] (.875,.4) rectangle (1.875,.9) node[midway] { $z_{k-t}$};
		\draw[decoration={brace,mirror,raise=-8pt},decorate]  (.65,-.85) -- node {\small $t$} (1.35,-.85);
		\draw[stdhl] (.375,-.5) node[below]{\small $1$}   .. controls (.375,-.25) .. (-.5,0)
				.. controls (-.25,.25) ..  (-.25,.5) -- (-.25,1);
		\draw[fill=white, color=white] (-.65,0) circle (.15cm);
		\draw[vstdhl] (-.5,-.5) node[below]{\small $\lambda$} -- (-.5,1)  node[pos=.15,nail]{};
	}
	\otimes \bar 1_{\ell,\rho}
	\right)
	=
	\tikzdiag{
		\draw (.25,-.5) .. controls (.25,-.375) .. (-.5,-.25) .. controls (1,0)  .. (1,.5) -- (1,1);
		\draw (.75,-.5) .. controls (.75,0) and (.25,0) .. (.25,.5)node[pos=.15, tikzdot]{}   -- (.25,1) node[pos=.5, tikzdot]{}  ;
		\node at(1,-.4) {\tiny $\dots$}; \node at(.55,.35) {\tiny $\dots$};
		\draw (1.25,-.5) .. controls (1.25,0) and (.75,0) .. (.75,.5)node[pos=.15, tikzdot]{}  -- (.75,1)  node[pos=.5, tikzdot]{}  ;
		\draw (1.5,-.5) .. controls (1.5,0) and (1.25,0) .. (1.25,.5)node[pos=.15, tikzdot]{}   -- (1.25,1);
		\node at(1.75,-.4) {\tiny $\dots$}; \node at(1.55,.15) {\tiny $\dots$};
		\draw (2,-.5) .. controls (2,0) and (1.75,0) .. (1.75,.5)node[pos=.15, tikzdot]{}   -- (1.75,1);
		\filldraw [fill=white, draw=black,rounded corners] (.875,.4) rectangle (1.875,.9) node[midway] { $z_{k-t}$};
		\draw[decoration={brace,mirror,raise=-8pt},decorate]  (.65,-.85) -- node {\small $t$} (1.35,-.85);
		\draw[stdhl] (-.25,-.5) node[below]{\small $1$}   .. controls (-.25,0) .. (-.5,.25)
				.. controls (-.25,.375) ..  (-.25,.5) -- (-.25,1);
		\draw[fill=white, color=white] (-.65,.25) circle (.15cm);
		\draw[vstdhl] (-.5,-.5) node[below]{\small $\lambda$} -- (-.5,1) node[pos=.15,nail]{};
	}
	\otimes \bar 1_{\ell,\rho}.
\]
Moreover, we have
\begin{align*}
\varphi_k^{1,t}\left( 
	\tikzdiagh{0}{
		\draw (0,-.5) .. controls (0,0) and (.25,0) .. (.25,.5);
		\draw (.25,-.5) .. controls (.25,0) and (.5,0) .. (.5,.5);
		\node at(.75,.35) {\tiny $\dots$};
		\draw (.75,-.5) .. controls (.75,0) and (1,0) .. (1,.5);
		\draw[stdhl] (1.25,-.5) node[below]{\small $1$}   .. controls (1.25,-.25) .. (-.5,0)
				.. controls (-.25,.25) ..  (-.25,.5) ;
		\draw[fill=white, color=white] (-.65,0) circle (.15cm);
		\draw[vstdhl] (-.5,-.5) node[below]{\small $\lambda$} -- (-.5,.5)  ;
	}
	\otimes \bar 1_{\ell,\rho}
\right)
\ &= \ 
\tikzdiag{
		\draw (0,-.5) .. controls (0,0) and (1,0) .. (1,.5) -- (1,1);
		\draw (.75,-.5) .. controls (.75,0) and (.25,0) .. (.25,.5)node[pos=.15, tikzdot]{}   -- (.25,1) node[pos=.5, tikzdot]{}  ;
		\node at(1,-.4) {\tiny $\dots$}; \node at(.55,.35) {\tiny $\dots$};
		\draw (1.25,-.5) .. controls (1.25,0) and (.75,0) .. (.75,.5)node[pos=.15, tikzdot]{}  -- (.75,1)  node[pos=.5, tikzdot]{}  ;
		\draw (1.5,-.5) .. controls (1.5,0) and (1.25,0) .. (1.25,.5)node[pos=.15, tikzdot]{}   -- (1.25,1);
		\node at(1.75,-.4) {\tiny $\dots$}; \node at(1.55,.15) {\tiny $\dots$};
		\draw (2,-.5) .. controls (2,0) and (1.75,0) .. (1.75,.5)node[pos=.15, tikzdot]{}   -- (1.75,1);
		\filldraw [fill=white, draw=black,rounded corners] (.875,.4) rectangle (1.875,.9) node[midway] { $z_{k-t}$};
		\draw[decoration={brace,mirror,raise=-8pt},decorate]  (.65,-.85) -- node {\small $t$} (1.35,-.85);
		\draw[stdhl] (.375,-.5) node[below]{\small $1$}   .. controls (.375,-.25) .. (-.5,0)
				.. controls (-.25,.25) ..  (-.25,.5) -- (-.25,1);
		\draw[fill=white, color=white] (-.65,0) circle (.15cm);
		\draw[vstdhl] (-.5,-.5) node[below]{\small $\lambda$} -- (-.5,1)  ; 
	}
	\otimes \bar 1_{\ell,\rho},
\\
u\left(
	\tikzdiagh{0}{
		\draw (0,-.5) .. controls (0,0) and (.25,0) .. (.25,.5);
		\draw (.25,-.5) .. controls (.25,0) and (.5,0) .. (.5,.5);
		\node at(.75,.35) {\tiny $\dots$};
		\draw (.75,-.5) .. controls (.75,0) and (1,0) .. (1,.5);
		\draw[stdhl] (1.25,-.5) node[below]{\small $1$}   .. controls (1.25,-.25) .. (-.5,0)
				.. controls (-.25,.25) ..  (-.25,.5) ;
		\draw[fill=white, color=white] (-.65,0) circle (.15cm);
		\draw[vstdhl] (-.5,-.5) node[below]{\small $\lambda$} -- (-.5,.5)  ;
	}
	\otimes \bar 1_{\ell,\rho}
	\right)
\ &= \ 
\tikzdiagh{0}{
		\draw[vstdhl] (-.5,0) node[below]{\small $\lambda$} -- (-.5,1);
		\draw (0,0) .. controls (0,.5) and (.5,.5) .. (.5,1);
		\draw (.25,0) .. controls (.25,.5) and (.75,.5) .. (.75,1);
		\node at (.55,.15){\tiny $\dots$};
		\draw (.75,0) .. controls (.75,.5) and (1.25,.5) .. (1.25,1);
		\draw[stdhl] (1.25,0) node[below]{\small $1$} .. controls (1.25,.5) and (-.25,.5) .. (-.25,1);
	}
	\otimes \bar 1_{\ell,\rho}.
\end{align*}
The case with a nail is similar, 
concluding the proof. 
\end{proof}

\begin{proof}[Proof of \cref{thm:catdoublebraidphiiso}]
Since $\varphi_k^1 - u$ is injective, and $u \otimes \gamma_k$ is surjective, and by \cref{prop:kerequalsimg1} and \cref{prop:kerequalsimg2}, we conclude that $\cone(\varphi_k)$ is acyclic for all $k$. Consequently,  $\varphi$ is a quasi-isomorphism.
\end{proof}

\subsubsection{The bimodule map $\tilde \varphi$}\label{sec:proofofbimodulemap}

\begin{citethm}{thm:phiisAinfty}
The map $\varphi$ is a map of $\bZ^2$-graded $(T^{\lambda,r},0)$-$(T^{\lambda,r},0)$-$A_\infty$-bimodules. 
\end{citethm}

The goal of this section is to prove \cref{thm:phiisAinfty}. To this end, we first prove that the map $\tilde \varphi^0 : q^2 (T_b^{\lambda,r})[1] \rightarrow X \otimes_T X$ is a map of bimodules.

\begin{prop}\label{prop:tildevarphi0rec}
We have
\[
\tilde \varphi^0 \left( 
\tikzdiag{
	\draw[vstdhl] (0,0) node[below]{\small$\lambda$} --  (0,1);
	\draw (.5,0)  --  (.5,1);
	\node at(1,.5){\tiny $\dots$};
	\draw (1.5,0)  --  (1.5,1);
	\draw[stdhl] (2,0) node[below]{\small$1$} --  (2,1);
	\draw[decoration={brace,mirror,raise=-8pt},decorate]  (.4,-.35) -- node {$k$} (1.6,-.35);
}
\otimes \bar 1_{\ell,\rho} 
\right) =  
(-1)^k
\tikzdiagh{0}{
  \draw[vstdhl] (0,0) node[below]{\small$\lambda$} --  (0,1);
  \draw (.5,0)  --  (.5,1);
  \node at(1,.5){\tiny $\dots$};
  \draw (1.5,0)  --  (1.5,1);
  \draw[stdhl] (2,0) node[below]{\small$1$} --  (2,1);
  \node at(1,.1){\tiny $\dots$};
  \filldraw [fill=white, draw=black] (-.15,.25) rectangle (2.15,.75) node[midway] { $\tilde\varphi(k)$};
  \node at(1,.9){\tiny $\dots$};
}
\otimes \bar 1_{\ell,\rho} ,
\]
where 
\begin{align*}
\tilde\varphi(0) := \ 
\tikzdiagh{0}{
	\draw[vstdhl] (0,0) node[below]{\small$\lambda$} --  (0,1);
	\draw[stdhl] (1,0) node[below]{\small$1$} --  (1,1);
	\filldraw [fill=white, draw=black] (-.15,.25) rectangle (1.15,.75) node[midway] { $\tilde\varphi(0)$};
}
\ &:= \ 0,
\\
 \tilde\varphi(1) := \ 
 \tikzdiagh{0}{
	\draw (.5,0) -- (.5,1);
	\draw[vstdhl] (0,0) node[below]{\small$\lambda$} --  (0,1);
	\draw[stdhl] (1,0) node[below]{\small$1$} --  (1,1);
	\filldraw [fill=white, draw=black] (-.15,.25) rectangle (1.15,.75) node[midway] { $\tilde\varphi(1)$};
}
\ &:= \ 
\tikzdiagh{0}{
	\draw (.5,-.5) .. controls (.5,-.375) .. (0,-.125) .. controls (1,.25) .. (1,.5) .. controls (1,1) and (.5,1) .. (.5,1.5);
	\draw[stdhl] (1,-.5)  node[below]{\small$1$} .. controls (1,0) .. (0,.25) ..controls (.5,.25) and (.5,.75) .. (0,.75) .. controls (1,1) .. (1,1.5);
	\draw[fill=white, color=white] (-.25,.25) circle (.2cm);
	\draw[fill=white, color=white] (-.25,.75) circle (.2cm);
	\draw[vstdhl] (0,-.5) node[below]{\small$\lambda$} --  (0,1.5) node[pos=.175,nail]{};
}
\ - \ 
\tikzdiagh[yscale=-1]{0}{
	\draw (.5,-.5) .. controls (.5,-.375) .. (0,-.125) .. controls (1,.25) .. (1,.5) .. controls (1,1) and (.5,1) .. (.5,1.5);
	\draw[stdhl] (1,-.5) .. controls (1,0) .. (0,.25) ..controls (.5,.25) and (.5,.75) .. (0,.75) .. controls (1,1) .. (1,1.5) node[below]{\small$1$};
	\draw[fill=white, color=white] (-.25,.25) circle (.2cm);
	\draw[fill=white, color=white] (-.25,.75) circle (.2cm);
	\draw[vstdhl] (0,-.5)  --  (0,1.5) node[below]{\small$\lambda$} node[pos=.175,nail]{};
}
\\
\tilde\varphi(t+2) := \ 
\tikzdiagh[xscale=.75]{0}{
	\draw (.5,-.5)  --  (.5,1.5);
	\node at(1,.5){\tiny $\dots$};
	\draw (1.5,-.5)  --  (1.5,1.5);
	\draw (2,-.5)  --  (2,1.5);
	\draw (2.5,-.5)  --  (2.5,1.5);
	\draw[vstdhl] (0,-.5) node[below]{\small$\lambda$} --  (0,1.5);
	\draw[stdhl] (3,-.5) node[below]{\small$1$} --  (3,1.5);
	\node at(1,.1){\tiny $\dots$};
	\filldraw [fill=white, draw=black] (-.15,.25) rectangle (3.15,.75) node[midway] { $\tilde\varphi(t+2)$};
	\node at(1,.9){\tiny $\dots$};
}
\ &:= \ 
\tikzdiagh[xscale=.75]{0}{
	\draw (.5,-.5)  --  (.5,1.5);
	\node at(1,.5){\tiny $\dots$};
	\draw (1.5,-.5)  --  (1.5,1.5);
	\draw (2,-.5) -- (2,0)  --  (2,1) .. controls (2,1.25) and (2.5,1.25) .. (2.5,1.5);
	\draw (2.5, -.5).. controls (2.5,-.25) and(3,-.25) ..  (3,0)  --  (3,1) .. controls (3,1.25) and (2,1.25) .. (2,1.5);
	\draw[vstdhl] (0,-.5) node[below]{\small$\lambda$} --  (0,1.5);
	\draw[stdhl] (3,-.5) node[below]{\small$1$}  .. controls (3,-.25) and (2.5,-.25) ..(2.5,0) --  (2.5,1) .. controls (2.5,1.25) and (3,1.25) .. (3,1.5);
	\node at(1,.1){\tiny $\dots$};
	\filldraw [fill=white, draw=black] (-.15,.25) rectangle (2.65,.75) node[midway] { $\tilde\varphi(t+1)$};
	\node at(1,.9){\tiny $\dots$};
}
\ + \ 
\tikzdiagh[xscale=.75]{0}{
	\draw (.5,-.5)  --  (.5,1.5);
	\node at(1,.5){\tiny $\dots$};
	\draw (1.5,-.5)  --  (1.5,1.5);
	\draw (2.5,-.5) .. controls (2.5,-.25) and (2,-.25) .. (2,0)  --  (2,1) --  (2,1.5);
	\draw (2, -.5).. controls (2,-.25) and(3,-.25) ..  (3,0)  --  (3,1) .. controls (3,1.25) and (2.5,1.25) .. (2.5,1.5);
	\draw[vstdhl] (0,-.5) node[below]{\small$\lambda$} --  (0,1.5);
	\draw[stdhl] (3,-.5) node[below]{\small$1$}  .. controls (3,-.25) and (2.5,-.25) ..(2.5,0) --  (2.5,1) .. controls (2.5,1.25) and (3,1.25) .. (3,1.5);
	\node at(1,.1){\tiny $\dots$};
	\filldraw [fill=white, draw=black] (-.15,.25) rectangle (2.65,.75) node[midway] { $\tilde\varphi(t+1)$};
	\node at(1,.9){\tiny $\dots$};
}
\ +  \ 
\tikzdiagh[xscale=.75]{0}{
	\draw (.5,-.5)  --  (.5,1.5);
	\node at(1,.5){\tiny $\dots$};
	\draw (1.5,-.5)  --  (1.5,1.5);
	\draw (2.5,-.5) .. controls (2.5,-.25) and (2,-.25) .. (2,0)  --  (2,1) .. controls (2,1.25) and (2.5,1.25) .. (2.5,1.5);
	\draw (2, -.5).. controls (2,-.25) and(3,-.25) ..  (3,0)  --  (3,1) node[midway,tikzdot]{} .. controls (3,1.25) and (2,1.25) .. (2,1.5);
	\draw[vstdhl] (0,-.5) node[below]{\small$\lambda$} --  (0,1.5);
	\draw[stdhl] (3,-.5) node[below]{\small$1$}  .. controls (3,-.25) and (2.5,-.25) ..(2.5,0) --  (2.5,1) .. controls (2.5,1.25) and (3,1.25) .. (3,1.5);
	\node at(1,.1){\tiny $\dots$};
	\filldraw [fill=white, draw=black] (-.15,.25) rectangle (2.65,.75) node[midway] { $\tilde\varphi(t+1)$};
	\node at(1,.9){\tiny $\dots$};
} 
\end{align*}
for all $t \geq 0$. 
\end{prop}

\begin{proof}
Recall that $\tilde \varphi^0 := (1 \otimes \gamma) \circ \varphi^0$. Then, we obtain
\begin{align}
\label{eq:gammavarphi0rec}
\begin{split}
(1 \otimes \gamma) \circ \varphi^0 &\left( 
\tikzdiagh{0}{
	\draw[vstdhl] (0,0) node[below]{\small$\lambda$} --  (0,1);
	\draw (.5,0)  --  (.5,1);
	\node at(1,.5){\tiny $\dots$};
	\draw (1.5,0)  --  (1.5,1);
	\draw[stdhl] (2,0) node[below]{\small$1$} --  (2,1);
	\draw[decoration={brace,mirror,raise=-8pt},decorate]  (.4,-.35) -- node {$k$} (1.6,-.35);
}
\otimes \bar 1_{\ell,\rho} 
\right) 
\\
\ &= 
(-1)^k
\sum_{t = 0}^{k-1}
\tikzdiag{
		\draw (.5,-.5) .. controls (.5,0) and (0,0) .. (0,.5) .. controls (0,.75) and (.75,.75) .. (.75,1) -- (.75, 1.5);
		\draw (.75,-.5) .. controls (.75,.25) and (0,.75) .. (0,1) -- (0,1.5) node[midway, tikzdot]{};
		\node at(1,-.5) {\tiny $\dots$};
		\draw (1.25,-.5) .. controls (1.25,.25) and (.5,.75) .. (.5,1) -- (.5,1.5) node[midway, tikzdot]{};
		\draw (1.5,-.5) .. controls (1.5,.25) and (1,.75) .. (1,1) -- (1,1.5);
		\node at(1.75,-.5) {\tiny $\dots$};
		\draw (2,-.5) .. controls (2,.25) and (1.5,.75) .. (1.5,1) -- (1.5,1.5);
		\filldraw [fill=white, draw=black,rounded corners] (.625,.9) rectangle (1.625,1.4) node[midway] { $z_{k-t}$};
		\draw[stdhl] (0,-.5)   .. controls (0,-.25) .. (-.5,0)
				.. controls (2,.25) ..  (2,.5) -- (2,1.5) ;
		\draw[fill=white, color=white] (-.6,0) circle (.1cm);
		\draw[vstdhl] (-.5,-.5) -- (-.5,1.5)  ;
		%
		%
		%
		\draw (0,-2) .. controls (0,-1.5) and (.75,-1.5) .. (.75,-.5);
		\draw (.5,-2) .. controls (.5,-1.5) and (1.25,-1.5) .. (1.25,-.5);
		\draw (.75,-2) .. controls (.75,-1.75) .. (-.5,-1.5) .. controls (.5,-1) ..  (.5,-.5);
		\draw (1,-2) .. controls (1,-1.5) and (1.5,-1.5) .. (1.5,-.5);
		\draw (1.5,-2) .. controls (1.5,-1.5) and (2,-1.5) .. (2,-.5);
		\draw[decoration={brace,mirror,raise=-8pt},decorate]  (-.1,-2.35) -- node {$t$} (.6,-2.35);
		\draw[stdhl] (2,-2) node[below]{\small $1$}  -- (2,-1.5) .. controls (2,-1.25) .. (-.5,-1)
				.. controls (0,-.75) .. (0,-.5);
		\draw[fill=white, color=white] (-.6,-1) circle (.1cm);
		\draw[vstdhl] (-.5,-2)  node[below]{\small $\lambda$} -- (-.5,-.5) node[pos=.33,nail]{};
	}
	\ \otimes \bar 1_{\ell,\rho}
\ - \ 
\tikzdiag{
		\draw (.5,-.5) .. controls (.5,.25) and (0,.25) .. (0,1.5) node[pos=.9,tikzdot]{}; 
		\node at(.75,-.5) {\tiny $\dots$};
		\draw (1,-.5) .. controls (1,.25) and (.5,.25) .. (.5,1.5) node[pos=.9,tikzdot]{}; 
		\draw (1.25,-.5) .. controls (1.25,.25) .. (-.5,.375) .. controls (.75,.5) .. (.75,1.5);
		\draw (1.5,-.5) .. controls (1.5,.5) and (1,.5) .. (1,1.5); 
		\node at(1.75,-.5) {\tiny $\dots$};
		\draw (2,-.5) .. controls (2,.5) and (1.5,.5) .. (1.5,1.5); 
		\filldraw [fill=white, draw=black,rounded corners] (.625,.9) rectangle (1.625,1.4) node[midway] { $z_{k-t}$};
		\draw[stdhl] (0,-.5)   .. controls (0,-.25) .. (-.5,0)
			.. controls (2,.0) ..  (2,1) -- (2,1.5) ;
		\draw[fill=white, color=white] (-.6,0) circle (.1cm);
		\draw[vstdhl] (-.5,-.5) -- (-.5,1.5)  node[pos=.45,nail]{};
		%
		%
		%
		\draw (0,-2) .. controls (0,-1.5) and (.5,-1.5) .. (.5,-.5);
		\draw (.5,-2) .. controls (.5,-1.5) and (1,-1.5) .. (1,-.5);
		\draw (.75,-2) .. controls (.75,-1.5) and (1.25,-1.5)  ..  (1.25,-.5);
		\draw (1,-2) .. controls (1,-1.5) and (1.5,-1.5) .. (1.5,-.5);
		\draw (1.5,-2) .. controls (1.5,-1.5) and (2,-1.5) .. (2,-.5);
		\draw[decoration={brace,mirror,raise=-8pt},decorate]  (-.1,-2.35) -- node {$t$} (.6,-2.35);
		\draw[stdhl] (2,-2) node[below]{\small $1$}  -- (2,-1.5) .. controls (2,-1.25) .. (-.5,-1)
				.. controls (0,-.75) .. (0,-.5);
		\draw[fill=white, color=white] (-.6,-1) circle (.1cm);
		\draw[vstdhl] (-.5,-2)  node[below]{\small $\lambda$} -- (-.5,-.5);
	}
	\ \otimes \bar 1_{\ell,\rho}.
\end{split}
\end{align}
We prove the statement by induction on $k$. The claim is clearly true for $k=0$ and $k=1$. Suppose it is true for $k+1$, and we will show it is true for $k+2$. 

By definition of $\tilde\varphi(k+2)$ and using \cref{eq:nhdotslide}, we have
\begin{equation}\label{eq:altvarphik}
\tikzdiagh[xscale=.75]{0}{
	\draw (.5,-.5)  --  (.5,1.5);
	\node at(1,.5){\tiny $\dots$};
	\draw (1.5,-.5)  --  (1.5,1.5);
	\draw (2,-.5)  --  (2,1.5);
	\draw (2.5,-.5)  --  (2.5,1.5);
	\draw[vstdhl] (0,-.5) node[below]{\small$\lambda$} --  (0,1.5);
	\draw[stdhl] (3,-.5) node[below]{\small$1$} --  (3,1.5);
	\node at(1,.1){\tiny $\dots$};
	\filldraw [fill=white, draw=black] (-.15,.25) rectangle (3.15,.75) node[midway] { $\tilde\varphi(k+2)$};
	\node at(1,.9){\tiny $\dots$};
}
\ = \ 
\tikzdiagh[xscale=.75]{0}{
	\draw (.5,-.5)  --  (.5,1.5);
	\node at(1,.5){\tiny $\dots$};
	\draw (1.5,-.5)  --  (1.5,1.5);
	\draw (2,-.5) -- (2,0)  --  (2,1) .. controls (2,1.25) and (2.5,1.25) .. (2.5,1.5);
	\draw (2.5, -.5).. controls (2.5,-.25) and(3,-.25) ..  (3,0)  --  (3,1) .. controls (3,1.25) and (2,1.25) .. (2,1.5);
	\draw[vstdhl] (0,-.5) node[below]{\small$\lambda$} --  (0,1.5);
	\draw[stdhl] (3,-.5) node[below]{\small$1$}  .. controls (3,-.25) and (2.5,-.25) ..(2.5,0) --  (2.5,1) .. controls (2.5,1.25) and (3,1.25) .. (3,1.5);
	\node at(1,.1){\tiny $\dots$};
	\filldraw [fill=white, draw=black] (-.15,.25) rectangle (2.65,.75) node[midway] { $\tilde\varphi(k+1)$};
	\node at(1,.9){\tiny $\dots$};
}
\ +  \ 
\tikzdiagh[xscale=.75]{0}{
	\draw (.5,-.5)  --  (.5,1.5);
	\node at(1,.5){\tiny $\dots$};
	\draw (1.5,-.5)  --  (1.5,1.5);
	\draw (2.5,-.5) .. controls (2.5,-.25) and (2,-.25) .. (2,0)  --  (2,1) .. controls (2,1.25) and (2.5,1.25) .. (2.5,1.5);
	\draw (2, -.5).. controls (2,-.25) and(3,-.25) ..  (3,0)  --  (3,1) .. controls (3,1.25) and (2,1.25) .. (2,1.5)  node[pos=.8,tikzdot]{};
	\node at(1,.9){\tiny $\dots$};
	\draw[vstdhl] (0,-.5) node[below]{\small$\lambda$} --  (0,1.5);
	\draw[stdhl] (3,-.5) node[below]{\small$1$}  .. controls (3,-.25) and (2.5,-.25) ..(2.5,0) --  (2.5,1) .. controls (2.5,1.25) and (3,1.25) .. (3,1.5);
	\node at(1,.1){\tiny $\dots$};
	\filldraw [fill=white, draw=black] (-.15,.25) rectangle (2.65,.75) node[midway] { $\tilde\varphi(k+1)$};
} 
\end{equation}
Applying the induction hypothesis on \cref{eq:altvarphik}, we get
\begin{align*}
&\tikzdiagh[xscale=.75]{0}{
	\draw (.5,-.5)  --  (.5,1.5);
	\node at(1,.5){\tiny $\dots$};
	\draw (1.5,-.5)  --  (1.5,1.5);
	\draw (2,-.5)  --  (2,1.5);
	\draw (2.5,-.5)  --  (2.5,1.5);
	\draw[vstdhl] (0,-.5) node[below]{\small$\lambda$} --  (0,1.5);
	\draw[stdhl] (3,-.5) node[below]{\small$1$} --  (3,1.5);
	\node at(1,.1){\tiny $\dots$};
	\filldraw [fill=white, draw=black] (-.15,.25) rectangle (3.15,.75) node[midway] { $\tilde\varphi(k+2)$};
	\node at(1,.9){\tiny $\dots$};
}
\\
\ &= \sum_{t=0}^{k-1} \  
\left(
\tikzdiag{
		\draw (2, -2.5) .. controls (2,-2.25) and (2.5,-2.25) .. (2.5,-2) -- (2.5,1.5) .. controls (2.5,1.75) and (1.75,1.75) .. (1.75,2);
		\draw (.5,-.5) .. controls (.5,0) and (0,0) .. (0,.5) .. controls (0,.75) and (.75,.75) .. (.75,1) -- (.75, 1.5) -- (.75,2);
		\draw (.75,-.5) .. controls (.75,.25) and (0,.75) .. (0,1) -- (0,1.5) node[midway, tikzdot]{}  -- (0,2);
		\node at(1,-.5) {\tiny $\dots$};
		\draw (1.25,-.5) .. controls (1.25,.25) and (.5,.75) .. (.5,1) -- (.5,1.5) node[midway, tikzdot]{}  -- (.5,2);
		\draw (1.5,-.5) .. controls (1.5,.25) and (1,.75) .. (1,1) -- (1,1.5)  -- (1,2);
		\node at(1.75,-.5) {\tiny $\dots$};
		\draw (2,-.5) .. controls (2,.25) and (1.5,.75) .. (1.5,1) -- (1.5,1.5)  -- (1.5,2);
		\draw (2.25,-.5) .. controls (2.25,.25) and (1.75,.75) .. (1.75,1) -- (1.75,1.5) .. controls (1.75,1.75) and (2,1.75) .. (2,2);
		\filldraw [fill=white, draw=black,rounded corners] (.625,.9) rectangle (1.875,1.4) node[midway] { $z_{k+1-t}$};
		\draw[stdhl] (0,-.5)   .. controls (0,-.25) .. (-.5,0)
				.. controls (2.25,.25) ..  (2.25,.5) -- (2.25,1.5)
				.. controls (2.25,1.75) and (2.5,1.75) .. (2.5,2) ;
		\draw[fill=white, color=white] (-.6,0) circle (.1cm);
		\draw[vstdhl] (-.5,-.5) -- (-.5,2)  ;
		\draw (0,-2.5) -- (0,-2) .. controls (0,-1.5) and (.75,-1.5) .. (.75,-.5);
		\draw (.5,-2.5) -- (.5,-2) .. controls (.5,-1.5) and (1.25,-1.5) .. (1.25,-.5);
		\draw (.75,-2.5) -- (.75,-2) .. controls (.75,-1.75) .. (-.5,-1.5) .. controls (.5,-1) ..  (.5,-.5);
		\draw (1,-2.5) -- (1,-2) .. controls (1,-1.5) and (1.5,-1.5) .. (1.5,-.5);
		\draw (1.5,-2.5) -- (1.5,-2) .. controls (1.5,-1.5) and (2,-1.5) .. (2,-.5);
		\draw (1.75,-2.5) -- (1.75,-2) .. controls (1.75,-1.5) and (2.25,-1.5) .. (2.25,-.5);
		\draw[decoration={brace,mirror,raise=-8pt},decorate]  (-.1,-2.85) -- node {$t$} (.6,-2.85);
		\draw[stdhl] (2.5,-2.5 )node[below]{\small $1$}  .. controls (2.5,-2.25) and (2.25,-2.25) .. (2.25,-2)  -- (2.25,-1.5) .. controls (2.25,-1.25) .. (-.5,-1)
				.. controls (0,-.75) .. (0,-.5);
		\draw[fill=white, color=white] (-.6,-1) circle (.1cm);
		\draw[vstdhl] (-.5,-2.5)  node[below]{\small $\lambda$}-- (-.5,-2) -- (-.5,-.5) node[pos=.33,nail]{};
	}
\ + \ 
\tikzdiag{
		\draw (1.75, -2.5) .. controls (1.75,-2.25) and (2.5,-2.25) .. (2.5,-2) -- (2.5,1.5) .. controls (2.5,1.75) and (1.75,1.75) .. (1.75,2) node[pos=.8, tikzdot]{};
		\draw (.5,-.5) .. controls (.5,0) and (0,0) .. (0,.5) .. controls (0,.75) and (.75,.75) .. (.75,1) -- (.75, 1.5) -- (.75,2);
		\draw (.75,-.5) .. controls (.75,.25) and (0,.75) .. (0,1) -- (0,1.5) node[midway, tikzdot]{}  -- (0,2);
		\node at(1,-.5) {\tiny $\dots$};
		\draw (1.25,-.5) .. controls (1.25,.25) and (.5,.75) .. (.5,1) -- (.5,1.5) node[midway, tikzdot]{}  -- (.5,2);
		\draw (1.5,-.5) .. controls (1.5,.25) and (1,.75) .. (1,1) -- (1,1.5)  -- (1,2);
		\node at(1.75,-.5) {\tiny $\dots$};
		\draw (2,-.5) .. controls (2,.25) and (1.5,.75) .. (1.5,1) -- (1.5,1.5)  -- (1.5,2);
		\draw (2.25,-.5) .. controls (2.25,.25) and (1.75,.75) .. (1.75,1) -- (1.75,1.5) .. controls (1.75,1.75) and (2,1.75) .. (2,2);
		\filldraw [fill=white, draw=black,rounded corners] (.625,.9) rectangle (1.875,1.4) node[midway] { $z_{k+1-t}$};
		\draw[stdhl] (0,-.5)   .. controls (0,-.25) .. (-.5,0)
				.. controls (2.25,.25) ..  (2.25,.5) -- (2.25,1.5)
				.. controls (2.25,1.75) and (2.5,1.75) .. (2.5,2) ;
		\draw[fill=white, color=white] (-.6,0) circle (.1cm);
		\draw[vstdhl] (-.5,-.5) -- (-.5,2)  ;
		\draw (0,-2.5) -- (0,-2) .. controls (0,-1.5) and (.75,-1.5) .. (.75,-.5);
		\draw (.5,-2.5) -- (.5,-2) .. controls (.5,-1.5) and (1.25,-1.5) .. (1.25,-.5);
		\draw (.75,-2.5) -- (.75,-2) .. controls (.75,-1.75) .. (-.5,-1.5) .. controls (.5,-1) ..  (.5,-.5);
		\draw (1,-2.5) -- (1,-2) .. controls (1,-1.5) and (1.5,-1.5) .. (1.5,-.5);
		\draw (1.5,-2.5) -- (1.5,-2) .. controls (1.5,-1.5) and (2,-1.5) .. (2,-.5);
		\draw (2,-2.5) .. controls (2,-2.25) and (1.75,-2.25) .. (1.75,-2) .. controls (1.75,-1.5) and (2.25,-1.5) .. (2.25,-.5);
		\draw[decoration={brace,mirror,raise=-8pt},decorate]  (-.1,-2.85) -- node {$t$} (.6,-2.85);
		\draw[stdhl] (2.5,-2.5 )node[below]{\small $1$}  .. controls (2.5,-2.25) and (2.25,-2.25) .. (2.25,-2)  -- (2.25,-1.5) .. controls (2.25,-1.25) .. (-.5,-1)
				.. controls (0,-.75) .. (0,-.5);
		\draw[fill=white, color=white] (-.6,-1) circle (.1cm);
		\draw[vstdhl] (-.5,-2.5)  node[below]{\small $\lambda$}-- (-.5,-2) -- (-.5,-.5) node[pos=.33,nail]{};
	}
\right)
\ + \ 
\tikzdiag{
		\draw (1, -2.5) .. controls (1,-2.25) and (1.5,-2.25) .. (1.5,-2) -- (1.5,1.5) .. controls (1.5,1.75) and (.75,1.75) .. (.75,2);
		\draw (.5,-.5) .. controls (.5,0) and (0,0) .. (0,.5) .. controls (0,.75) and (.75,.75) .. (.75,1) -- (.75, 1.5) .. controls (.75,1.75) and (1,1.75) .. (1,2);
		\draw (.75,-.5) .. controls (.75,.25) and (0,.75) .. (0,1) -- (0,1.5) node[midway, tikzdot]{}  -- (0,2);
		\node at(1,-.5) {\tiny $\dots$};
		\draw (1.25,-.5) .. controls (1.25,.25) and (.5,.75) .. (.5,1) -- (.5,1.5) node[midway, tikzdot]{}  -- (.5,2);
		%
		%
		%
		\draw[stdhl] (0,-.5)   .. controls (0,-.25) .. (-.5,0)
				.. controls (1.25,.25) ..  (1.25,.5) -- (1.25,1.5)
				.. controls (1.25,1.75) and (1.5,1.75) .. (1.5,2) ;
		\draw[fill=white, color=white] (-.6,0) circle (.1cm);
		\draw[vstdhl] (-.5,-.5) -- (-.5,2)  ;
		\draw (0,-2.5) -- (0,-2) .. controls (0,-1.5) and (.75,-1.5) .. (.75,-.5);
		\draw (.5,-2.5) -- (.5,-2) .. controls (.5,-1.5) and (1.25,-1.5) .. (1.25,-.5);
		\draw (.75,-2.5) -- (.75,-2) .. controls (.75,-1.75) .. (-.5,-1.5) .. controls (.5,-1) ..  (.5,-.5);
		\draw[decoration={brace,mirror,raise=-8pt},decorate]  (-.1,-2.85) -- node {$k$} (.6,-2.85);
		\draw[stdhl] (1.5,-2.5 )node[below]{\small $1$}  .. controls (1.5,-2.25) and (1.25,-2.25) .. (1.25,-2)  -- (1.25,-1.5) .. controls (1.25,-1.25) .. (-.5,-1)
				.. controls (0,-.75) .. (0,-.5);
		\draw[fill=white, color=white] (-.6,-1) circle (.1cm);
		\draw[vstdhl] (-.5,-2.5)  node[below]{\small $\lambda$}-- (-.5,-2) -- (-.5,-.5) node[pos=.33,nail]{};
	}
\ + \ 
\tikzdiag{
		\draw (.75, -2.5) .. controls (.75,-2.25) and (1.5,-2.25) .. (1.5,-2) -- (1.5,1.5) .. controls (1.5,1.75) and (.75,1.75) .. (.75,2) node[pos=.8, tikzdot]{};
		\draw (.5,-.5) .. controls (.5,0) and (0,0) .. (0,.5) .. controls (0,.75) and (.75,.75) .. (.75,1) -- (.75, 1.5) .. controls (.75,1.75) and (1,1.75) .. (1,2);
		\draw (.75,-.5) .. controls (.75,.25) and (0,.75) .. (0,1) -- (0,1.5) node[midway, tikzdot]{}  -- (0,2);
		\node at(1,-.5) {\tiny $\dots$};
		\draw (1.25,-.5) .. controls (1.25,.25) and (.5,.75) .. (.5,1) -- (.5,1.5) node[midway, tikzdot]{}  -- (.5,2);
		%
		%
		%
		\draw[stdhl] (0,-.5)   .. controls (0,-.25) .. (-.5,0)
				.. controls (1.25,.25) ..  (1.25,.5) -- (1.25,1.5)
				.. controls (1.25,1.75) and (1.5,1.75) .. (1.5,2) ;
		\draw[fill=white, color=white] (-.6,0) circle (.1cm);
		\draw[vstdhl] (-.5,-.5) -- (-.5,2)  ;
		\draw (0,-2.5) -- (0,-2) .. controls (0,-1.5) and (.75,-1.5) .. (.75,-.5);
		\draw (.5,-2.5) -- (.5,-2) .. controls (.5,-1.5) and (1.25,-1.5) .. (1.25,-.5);
		\draw (1,-2.5) .. controls (1,-2.25) and (.75,-2.25) .. (.75,-2) .. controls (.75,-1.75) .. (-.5,-1.5) .. controls (.5,-1) ..  (.5,-.5);
		\draw[decoration={brace,mirror,raise=-8pt},decorate]  (-.1,-2.85) -- node {$k$} (.6,-2.85);
		\draw[stdhl] (1.5,-2.5 )node[below]{\small $1$}  .. controls (1.5,-2.25) and (1.25,-2.25) .. (1.25,-2)  -- (1.25,-1.5) .. controls (1.25,-1.25) .. (-.5,-1)
				.. controls (0,-.75) .. (0,-.5);
		\draw[fill=white, color=white] (-.6,-1) circle (.1cm);
		\draw[vstdhl] (-.5,-2.5)  node[below]{\small $\lambda$}-- (-.5,-2) -- (-.5,-.5) node[pos=.33,nail]{};
	}
\\
&
\phantom{\ = \ }\ - \ \text{(similar terms with the nail above).}
\end{align*}
Applying \cref{eq:defzn} on each pair of terms in the sum (including non-displayed terms) gives the part for $0 \leq t < k-2$ in \cref{eq:gammavarphi0rec} for $k+2$. The last two terms (including non-displayed terms) give $t=k$ and $t=k+1$, since $z_2$ is a single crossing, concluding the proof. 
\end{proof}

Having \cref{prop:tildevarphi0rec}, proving \cref{thm:phiisAinfty} boils down to proving that the left and right action by the same element of $T_b^{\lambda,r}$ on
\[
\sum_{k+\ell+|\rho|=b}  (-1)^k \tilde\varphi(k) \otimes \bar 1_{\ell,\rho}
\]
coincide. 

\begin{lem}\label{lem:varphitdotright}
We have
\[
\tikzdiagh[xscale=.75]{0}{
	\draw (.5,-.5)  --  (.5,1.5);
	\node at(1,.5){\tiny $\dots$};
	\draw (1.5,-.5)  --  (1.5,1.5);
	\draw (2,-.5)  --  (2,1.5) node[near end, tikzdot]{};
	\draw[vstdhl] (0,-.5) node[below]{\small$\lambda$} --  (0,1.5);
	\draw[stdhl] (2.5,-.5) node[below]{\small$1$} --  (2.5,1.5);
	\node at(1,.1){\tiny $\dots$};
	\filldraw [fill=white, draw=black] (-.15,.25) rectangle (2.65,.75) node[midway] { $\tilde\varphi(t+1)$};
	\node at(1,.9){\tiny $\dots$};
}
\ = - \ 
\tikzdiagh[xscale=.75]{0}{
	\draw (.5,-.5)  --  (.5,1.5);
	\node at(1,.5){\tiny $\dots$};
	\draw (1.5,-.5)  --  (1.5,1.5);
	\draw (2,-.5)  .. controls (2,-.25) and (2.5,-.25) .. (2.5,0) -- (2.5,1) .. controls (2.5,1.25) and (2,1.25) ..  (2,1.5);
	\draw[vstdhl] (0,-.5) node[below]{\small$\lambda$} --  (0,1.5);
	\draw[stdhl] (2.5,-.5) node[below]{\small$1$} .. controls (2.5,-.25) and (2,-.25) .. (2,0) -- (2,1) .. controls (2,1.25) and (2.5,1.25) ..  (2.5,1.5);
	\node at(1,.1){\tiny $\dots$};
	\filldraw [fill=white, draw=black] (-.15,.25) rectangle (2.15,.75) node[midway] { $\tilde\varphi(t)$};
	\node at(1,.9){\tiny $\dots$};
}
\ = \ 
\tikzdiagh[xscale=.75]{0}{
	\draw (.5,-.5)  --  (.5,1.5);
	\node at(1,.5){\tiny $\dots$};
	\draw (1.5,-.5)  --  (1.5,1.5);
	\draw (2,-.5)  --  (2,1.5) node[near start, tikzdot]{};
	\draw[vstdhl] (0,-.5) node[below]{\small$\lambda$} --  (0,1.5);
	\draw[stdhl] (2.5,-.5) node[below]{\small$1$} --  (2.5,1.5);
	\node at(1,.1){\tiny $\dots$};
	\filldraw [fill=white, draw=black] (-.15,.25) rectangle (2.65,.75) node[midway] { $\tilde\varphi(t+1)$};
	\node at(1,.9){\tiny $\dots$};
}
\]
for all $t \geq 0$. 
\end{lem}

\begin{proof}
We show the first equality, and the second one follows by symmetry along the horizontal axis of the definition of $\tilde\varphi(t+1)$. 

We prove the statement by induction on $t$. 
The case $t=0$ follows from \cref{eq:dottednailslide}. Suppose the claim is true for $t \geq 0$. We compute using \cref{eq:altvarphik}
\begin{align*}
\tikzdiagh[xscale=.75]{0}{
	\draw (.5,-.5)  --  (.5,1.5);
	\node at(1,.5){\tiny $\dots$};
	\draw (1.5,-.5)  --  (1.5,1.5);
	\draw (2,-.5)  --  (2,1.5);
	\draw (2.5,-.5)  --  (2.5,1.5) node[near end, tikzdot]{};
	\draw[vstdhl] (0,-.5) node[below]{\small$\lambda$} --  (0,1.5);
	\draw[stdhl] (3,-.5) node[below]{\small$1$} --  (3,1.5);
	\node at(1,.1){\tiny $\dots$};
	\filldraw [fill=white, draw=black] (-.15,.25) rectangle (3.15,.75) node[midway] { $\tilde\varphi(t+2)$};
	\node at(1,.9){\tiny $\dots$};
}
\ &= \ 
\tikzdiagh[xscale=.75]{0}{
	\draw (.5,-.5)  --  (.5,1.5);
	\node at(1,.5){\tiny $\dots$};
	\draw (1.5,-.5)  --  (1.5,1.5);
	\draw (2,-.5) -- (2,0)  --  (2,1) .. controls (2,1.25) and (2.5,1.25) .. (2.5,1.5)   node[pos=.8,tikzdot]{};
	\draw (2.5, -.5).. controls (2.5,-.25) and(3,-.25) ..  (3,0)  --  (3,1) .. controls (3,1.25) and (2,1.25) .. (2,1.5);
	\draw[vstdhl] (0,-.5) node[below]{\small$\lambda$} --  (0,1.5);
	\draw[stdhl] (3,-.5) node[below]{\small$1$}  .. controls (3,-.25) and (2.5,-.25) ..(2.5,0) --  (2.5,1) .. controls (2.5,1.25) and (3,1.25) .. (3,1.5);
	\node at(1,.1){\tiny $\dots$};
	\filldraw [fill=white, draw=black] (-.15,.25) rectangle (2.65,.75) node[midway] { $\tilde\varphi(t+1)$};
	\node at(1,.9){\tiny $\dots$};
}
\ +  \ 
\tikzdiagh[xscale=.75]{0}{
	\draw (.5,-.5)  --  (.5,1.5);
	\node at(1,.5){\tiny $\dots$};
	\draw (1.5,-.5)  --  (1.5,1.5);
	\draw (2.5,-.5) .. controls (2.5,-.25) and (2,-.25) .. (2,0)  --  (2,1) .. controls (2,1.25) and (2.5,1.25) .. (2.5,1.5)   node[pos=.8,tikzdot]{};
	\draw (2, -.5).. controls (2,-.25) and(3,-.25) ..  (3,0)  --  (3,1) .. controls (3,1.25) and (2,1.25) .. (2,1.5)  node[pos=.8,tikzdot]{};
	\draw[vstdhl] (0,-.5) node[below]{\small$\lambda$} --  (0,1.5);
	\draw[stdhl] (3,-.5) node[below]{\small$1$}  .. controls (3,-.25) and (2.5,-.25) ..(2.5,0) --  (2.5,1) .. controls (2.5,1.25) and (3,1.25) .. (3,1.5);
	\node at(1,.1){\tiny $\dots$};
	\filldraw [fill=white, draw=black] (-.15,.25) rectangle (2.65,.75) node[midway] { $\tilde\varphi(t+1)$};
	\node at(1,.9){\tiny $\dots$};
}
\\
\ &= - \ 
\tikzdiagh[xscale=.75]{0}{
	\draw (.5,-.5)  --  (.5,1.5);
	\node at(1,.5){\tiny $\dots$};
	\draw (1.5,-.5)  --  (1.5,1.5);
	\draw (2,-.5) -- (2,0)  --  (2,1.5);
	\draw (2.5, -.5).. controls (2.5,-.25) and(3,-.25) ..  (3,0)  --  (3,1) .. controls (3,1.25) and (2.5,1.25) .. (2.5,1.5);
	\draw[vstdhl] (0,-.5) node[below]{\small$\lambda$} --  (0,1.5);
	\draw[stdhl] (3,-.5) node[below]{\small$1$}  .. controls (3,-.25) and (2.5,-.25) ..(2.5,0) --  (2.5,1) .. controls (2.5,1.25) and (3,1.25) .. (3,1.5);
	\node at(1,.1){\tiny $\dots$};
	\filldraw [fill=white, draw=black] (-.15,.25) rectangle (2.65,.75) node[midway] { $\tilde\varphi(t+1)$};
	\node at(1,.9){\tiny $\dots$};
} 
\ - \ 
\tikzdiagh[xscale=.75]{0}{
	\draw (.5,-.5)  --  (.5,1.5);
	\node at(1,.5){\tiny $\dots$};
	\draw (1.5,-.5)  --  (1.5,1.5);
	\draw (2,-.5) .. controls (2,-.25) and (2.5,-.25) ..  (2.5,0)   .. controls (2.5,.5) and (3,.5) ..   (3,1) .. controls (3,1.25) and (2.5,1.25) .. (2.5,1.5);
	\draw (2.5, -.5).. controls (2.5,-.25) and(3,-.25) ..  (3,0)  .. controls (3,.5) and (2.5,.5) .. (2.5,1) .. controls (2.5,1.25) and (2,1.25) .. (2,1.5);
	\draw[vstdhl] (0,-.5) node[below]{\small$\lambda$} --  (0,1.5);
	\draw[stdhl] (3,-.5) node[below]{\small$1$}  .. controls (3,-.25) and (2,-.25) ..(2,0) --  (2,1) .. controls (2,1.25) and (3,1.25) .. (3,1.5);
	\node at(1,.1){\tiny $\dots$};
	\filldraw [fill=white, draw=black] (-.15,.25) rectangle (2.15,.75) node[midway] { $\tilde\varphi(t)$};
	\node at(1,.9){\tiny $\dots$};
}
\ -  \ 
\tikzdiagh[xscale=.75]{0}{
	\draw (.5,-.5)  --  (.5,1.5);
	\node at(1,.5){\tiny $\dots$};
	\draw (1.5,-.5)  --  (1.5,1.5);
	\draw (2,-.5) .. controls (2,-.25) and (2.5,-.25) ..  (2.5,0)  .. controls (2.5,.25) and (3,.25) .. (3,.5) .. controls (3,.75) and (2.5,.75) ..  (2.5,1) node[tikzdot, pos=0]{} .. controls (2.5,1.25) and (2,1.25) .. (2,1.5)   ;
	\draw (2.5, -.5).. controls (2.5,-.25) and(3,-.25) ..  (3,0)   .. controls (3,.25) and (2.5,.25) .. (2.5,.5) .. controls (2.5,.75) and (3,.75) ..  (3,1) .. controls (3,1.25) and (2.5,1.25) .. (2.5,1.5);
	\draw[vstdhl] (0,-.5) node[below]{\small$\lambda$} --  (0,1.5);
	\draw[stdhl] (3,-.5) node[below]{\small$1$}  .. controls (3,-.25) and (2,-.25) ..(2,0) --  (2,1) .. controls (2,1.25) and (3,1.25) .. (3,1.5);
	\node at(1,.1){\tiny $\dots$};
	\filldraw [fill=white, draw=black] (-.15,.25) rectangle (2.15,.75) node[midway] { $\tilde\varphi(t)$};
	\node at(1,.9){\tiny $\dots$};
}
\end{align*}
where the last two terms annihilate each other, concluding the proof. 
\end{proof}

\begin{lem}\label{lem:varphitdotsecondright}
We have
\[
\tikzdiagh[xscale=.75]{0}{
	\draw (.5,-.5)  --  (.5,1.5);
	\node at(1,.5){\tiny $\dots$};
	\draw (1.5,-.5)  --  (1.5,1.5);
	\draw (2,-.5)  --  (2,1.5) node[near end, tikzdot]{};
	\draw (2.5,-.5)  --  (2.5,1.5) ;
	\draw[vstdhl] (0,-.5) node[below]{\small $\lambda$} --  (0,1.5);
	\draw[stdhl] (3,-.5) node[below]{\small $1$} --  (3,1.5);
	\node at(1,.1){\tiny $\dots$};
	\filldraw [fill=white, draw=black] (-.15,.25) rectangle (3.15,.75) node[midway] { $\tilde\varphi(t+2)$};
	\node at(1,.9){\tiny $\dots$};
}
\ = \ 
\tikzdiagh[xscale=.75]{0}{
	\draw (.5,-.5)  --  (.5,1.5);
	\node at(1,.5){\tiny $\dots$};
	\draw (1.5,-.5)  --  (1.5,1.5);
	\draw (2.5,-.5) .. controls (2.5,-.25) and (2,-.25) .. (2,0) --  (2,1) .. controls (2,1.25) and (2.5,1.25) .. (2.5,1.5);
	\draw (2, -.5).. controls (2,-.25) and(3,-.25) ..  (3,.25) node[tikzdot, pos=.2]{}   --  (3,.75) .. controls (3,1.25) and (2,1.25) .. (2,1.5) node[tikzdot, pos=.8]{};
	\draw[vstdhl] (0,-.5) node[below]{\small $\lambda$} --  (0,1.5);
	\draw[stdhl] (3,-.5) node[below]{\small $1$}  .. controls (3,-.25) and (2.5,-.25) ..(2.5,.25) --  (2.5,.75) .. controls (2.5,1.25) and (3,1.25) .. (3,1.5);
	\node at(1,.1){\tiny $\dots$};
	\filldraw [fill=white, draw=black] (-.15,.25) rectangle (2.65,.75) node[midway] { $\tilde\varphi(t+1)$};
	\node at(1,.9){\tiny $\dots$};
} 
\ - \ 
\tikzdiagh[xscale=.75]{0}{
	\draw (.5,-.5)  --  (.5,1.5);
	\node at(1,.5){\tiny $\dots$};
	\draw (1.5,-.5)  --  (1.5,1.5);
	\draw (2,-.5) .. controls (2,-.25) and (2.5,-.25) ..  (2.5,.25)   .. controls (2.5,.5) and (3,.5) ..   (3,.75) .. controls (3,1.25) and (2.5,1.25) .. (2.5,1.5);
	\draw (2.5, -.5).. controls (2.5,-.25) and(3,-.25) ..  (3,.25)  .. controls (3,.5) and (2.5,.5) .. (2.5,.75) .. controls (2.5,1.25) and (2,1.25) .. (2,1.5);
	\draw[vstdhl] (0,-.5) node[below]{\small $\lambda$} --  (0,1.5);
	\draw[stdhl] (3,-.5) node[below]{\small $1$}  .. controls (3,-.25) and (2,-.25) ..(2,0) --  (2,1) .. controls (2,1.25) and (3,1.25) .. (3,1.5);
	\node at(1,.1){\tiny $\dots$};
	\filldraw [fill=white, draw=black] (-.15,.25) rectangle (2.15,.75) node[midway] { $\tilde\varphi(t)$};
	\node at(1,.9){\tiny $\dots$};
}
\ = \ 
\tikzdiagh[xscale=.75]{0}{
	\draw (.5,-.5)  --  (.5,1.5);
	\node at(1,.5){\tiny $\dots$};
	\draw (1.5,-.5)  --  (1.5,1.5);
	\draw (2,-.5)  --  (2,1.5) node[near start, tikzdot]{};
	\draw (2.5,-.5)  --  (2.5,1.5) ;
	\draw[vstdhl] (0,-.5) node[below]{\small$\lambda$} --  (0,1.5);
	\draw[stdhl] (3,-.5) node[below]{\small$1$} --  (3,1.5);
	\node at(1,.1){\tiny $\dots$};
	\filldraw [fill=white, draw=black] (-.15,.25) rectangle (3.15,.75) node[midway] { $\tilde\varphi(t+2)$};
	\node at(1,.9){\tiny $\dots$};
}
\]
for all $t \geq 0$. 
\end{lem}

\begin{proof}
By \cref{eq:altvarphik}, we have
\begin{align*}
\tikzdiagh[xscale=.75]{0}{
	\draw (.5,-.5)  --  (.5,1.5);
	\node at(1,.5){\tiny $\dots$};
	\draw (1.5,-.5)  --  (1.5,1.5);
	\draw (2,-.5)  --  (2,1.5) node[near start, tikzdot]{};
	\draw (2.5,-.5)  --  (2.5,1.5) ;
	\draw[vstdhl] (0,-.5) node[below]{\small $\lambda$} --  (0,1.5);
	\draw[stdhl] (3,-.5) node[below]{\small $1$} --  (3,1.5);
	\node at(1,.1){\tiny $\dots$};
	\filldraw [fill=white, draw=black] (-.15,.25) rectangle (3.15,.75) node[midway] { $\tilde\varphi(t+2)$};
	\node at(1,.9){\tiny $\dots$};
}
\ &= \ 
\tikzdiagh[xscale=.75]{0}{
	\draw (.5,-.5)  --  (.5,1.5);
	\node at(1,.5){\tiny $\dots$};
	\draw (1.5,-.5)  --  (1.5,1.5);
	\draw (2,-.5) -- (2,0)  node[pos=1, tikzdot]{}  --  (2,1) .. controls (2,1.25) and (2.5,1.25) .. (2.5,1.5);
	\draw (2.5, -.5).. controls (2.5,-.25) and(3,-.25) ..  (3,0)  --  (3,1) .. controls (3,1.25) and (2,1.25) .. (2,1.5);
	\draw[vstdhl] (0,-.5) node[below]{\small $\lambda$} --  (0,1.5);
	\draw[stdhl] (3,-.5) node[below]{\small $1$}  .. controls (3,-.25) and (2.5,-.25) ..(2.5,0) --  (2.5,1) .. controls (2.5,1.25) and (3,1.25) .. (3,1.5);
	\node at(1,.1){\tiny $\dots$};
	\filldraw [fill=white, draw=black] (-.15,.25) rectangle (2.65,.75) node[midway] { $\tilde\varphi(k+1)$};
	\node at(1,.9){\tiny $\dots$};
}
\ +  \ 
\tikzdiagh[xscale=.75]{0}{
	\draw (.5,-.5)  --  (.5,1.5);
	\node at(1,.5){\tiny $\dots$};
	\draw (1.5,-.5)  --  (1.5,1.5);
	\draw (2.5,-.5) .. controls (2.5,-.25) and (2,-.25) .. (2,0)  --  (2,1) .. controls (2,1.25) and (2.5,1.25) .. (2.5,1.5);
	\draw (2, -.5).. controls (2,-.25) and(3,-.25) ..  (3,0)   node[pos=.2,tikzdot]{}  --  (3,1) .. controls (3,1.25) and (2,1.25) .. (2,1.5)  node[pos=.8,tikzdot]{};
	\draw[vstdhl] (0,-.5) node[below]{\small $\lambda$} --  (0,1.5);
	\draw[stdhl] (3,-.5) node[below]{\small $1$}  .. controls (3,-.25) and (2.5,-.25) ..(2.5,0) --  (2.5,1) .. controls (2.5,1.25) and (3,1.25) .. (3,1.5);
	\node at(1,.1){\tiny $\dots$};
	\filldraw [fill=white, draw=black] (-.15,.25) rectangle (2.65,.75) node[midway] { $\tilde\varphi(k+1)$};
	\node at(1,.9){\tiny $\dots$};
}
\end{align*}
We conclude by applying \cref{lem:varphitdotright}.
\end{proof}

\begin{lem}\label{lem:varphitcrossingright}
We have
\[
\tikzdiagh[xscale=.75]{0}{
	\draw (.5,-.5)  --  (.5,1.5);
	\node at(1,.5){\tiny $\dots$};
	\draw (1.5,-.5)  --  (1.5,1.5);
	\draw (2,-.5)  --  (2,1) .. controls (2,1.25) and (2.5,1.25) .. (2.5,1.5) ;
	\draw (2.5,-.5)  --  (2.5,1) .. controls (2.5,1.25) and (2,1.25) .. (2,1.5);
	\draw[vstdhl] (0,-.5) node[below]{\small$\lambda$} --  (0,1.5);
	\draw[stdhl] (3,-.5) node[below]{\small$1$} --  (3,1.5);
	\node at(1,.1){\tiny $\dots$};
	\filldraw [fill=white, draw=black] (-.15,.25) rectangle (3.15,.75) node[midway] { $\tilde\varphi(t+2)$};
	\node at(1,.9){\tiny $\dots$};
}
\ = \ 
\tikzdiagh[xscale=.75]{0}{
	\draw (.5,-.5)  --  (.5,1.5);
	\node at(1,.5){\tiny $\dots$};
	\draw (1.5,-.5)  --  (1.5,1.5);
	\draw (2.5,-.5) .. controls (2.5,-.25) and (2,-.25) .. (2,0) --  (2,1) .. controls (2,1.25) and (2.5,1.25) .. (2.5,1.5);
	\draw (2, -.5).. controls (2,-.25) and(3,-.25) ..  (3,.25)   --  (3,.75) .. controls (3,1.25) and (2,1.25) .. (2,1.5);
	\draw[vstdhl] (0,-.5) node[below]{\small$\lambda$} --  (0,1.5);
	\draw[stdhl] (3,-.5) node[below]{\small$1$}  .. controls (3,-.25) and (2.5,-.25) ..(2.5,.25) --  (2.5,.75) .. controls (2.5,1.25) and (3,1.25) .. (3,1.5);
	\node at(1,.1){\tiny $\dots$};
	\filldraw [fill=white, draw=black] (-.15,.25) rectangle (2.65,.75) node[midway] { $\tilde\varphi(t+1)$};
	\node at(1,.9){\tiny $\dots$};
} 
\ = \ 
\tikzdiagh[xscale=.75]{0}{
	\draw (.5,-.5)  --  (.5,1.5);
	\node at(1,.5){\tiny $\dots$};
	\draw (1.5,-.5)  --  (1.5,1.5);
	\draw (2.5,-.5)  .. controls (2.5,-.25) and (2,-.25) .. (2,0) --(2,1.5) ;
	\draw (2,-.5)  .. controls (2,-.25) and (2.5,-.25) ..  (2.5,0) -- (2.5,1.5);
	\draw[vstdhl] (0,-.5) node[below]{\small$\lambda$} --  (0,1.5);
	\draw[stdhl] (3,-.5) node[below]{\small$1$} --  (3,1.5);
	\node at(1,.1){\tiny $\dots$};
	\filldraw [fill=white, draw=black] (-.15,.25) rectangle (3.15,.75) node[midway] { $\tilde\varphi(t+2)$};
	\node at(1,.9){\tiny $\dots$};
}
\]
for all $t \geq 0$. 
\end{lem}

\begin{proof}
This is immediate by applying \cref{eq:nhR2andR3} on the definition of $\tilde\varphi(t+2)$. 
\end{proof}

\begin{lem}\label{lem:varphitcrossingsecondright}
We have
\[
\tikzdiagh[xscale=.75]{0}{
	\draw (.5,-.5)  --  (.5,1.5);
	\node at(1,.5){\tiny $\dots$};
	\draw (1.5,-.5)  --  (1.5,1.5);
	\draw (2,-.5)  --  (2,1) .. controls (2,1.25) and (2.5,1.25) .. (2.5,1.5) ;
	\draw (2.5,-.5)  --  (2.5,1) .. controls (2.5,1.25) and (2,1.25) .. (2,1.5);
	\draw (3,-.5) -- (3,1.5);
	\draw[vstdhl] (0,-.5) node[below]{\small$\lambda$} --  (0,1.5);
	\draw[stdhl] (3.5,-.5) node[below]{\small$1$} --  (3.5,1.5);
	\node at(1,.1){\tiny $\dots$};
	\filldraw [fill=white, draw=black] (-.15,.25) rectangle (3.65,.75) node[midway] { $\tilde\varphi(t+3)$};
	\node at(1,.9){\tiny $\dots$};
}
\ = \ 
\tikzdiagh[xscale=.75]{0}{
	\draw (.5,-.5)  --  (.5,1.5);
	\node at(1,.5){\tiny $\dots$};
	\draw (1.5,-.5)  --  (1.5,1.5);
	\draw (2.5,-.5)  .. controls (2.5,-.25) and (2,-.25) .. (2,0) --(2,1.5) ;
	\draw (2,-.5)  .. controls (2,-.25) and (2.5,-.25) ..  (2.5,0) -- (2.5,1.5);
	\draw (3,-.5) -- (3,1.5);
	\draw[vstdhl] (0,-.5) node[below]{\small$\lambda$} --  (0,1.5);
	\draw[stdhl] (3.5,-.5) node[below]{\small$1$} --  (3.5,1.5);
	\node at(1,.1){\tiny $\dots$};
	\filldraw [fill=white, draw=black] (-.15,.25) rectangle (3.65,.75) node[midway] { $\tilde\varphi(t+3)$};
	\node at(1,.9){\tiny $\dots$};
}
\]
for all $t \geq 0$. 
\end{lem}

\begin{proof}
By \cref{eq:altvarphik} we have
\begin{align*}
\tikzdiagh[xscale=.75]{0}{
	\draw (.5,-.5)  --  (.5,1.5);
	\node at(1,.5){\tiny $\dots$};
	\draw (1.5,-.5)  --  (1.5,1.5);
	\draw (2,-.5)  --  (2,1) .. controls (2,1.25) and (2.5,1.25) .. (2.5,1.5) ;
	\draw (2.5,-.5)  --  (2.5,1) .. controls (2.5,1.25) and (2,1.25) .. (2,1.5);
	\draw (3,-.5) -- (3,1.5);
	\draw[vstdhl] (0,-.5) node[below]{\small$\lambda$} --  (0,1.5);
	\draw[stdhl] (3.5,-.5) node[below]{\small$1$} --  (3.5,1.5);
	\node at(1,.1){\tiny $\dots$};
	\filldraw [fill=white, draw=black] (-.15,.25) rectangle (3.65,.75) node[midway] { $\tilde\varphi(t+3)$};
	\node at(1,.9){\tiny $\dots$};
}
\ &= \ 
\tikzdiagh[xscale=.75]{0}{
	\draw (.5,-.5)  --  (.5,1.5);
	\node at(1,.5){\tiny $\dots$};
	\draw (1.5,-.5)  --  (1.5,1.5);
	\draw (2,-.5) -- (2,1) .. controls (2,1.25) and (2.5,1.25) .. (2.5,1.5);
	\draw  (2.5,-.5) --  (2.5,1) .. controls (2.5,1.25) and (3,1.25) .. (3,1.5);
	\draw (3, -.5).. controls (3,-.25) and(3.5,-.25) ..  (3.5,.25)   --  (3.5,.75) .. controls (3.5,1.25) and (2,1.25) .. (2,1.5);
	\draw[vstdhl] (0,-.5) node[below]{\small$\lambda$} --  (0,1.5);
	\draw[stdhl] (3.5,-.5) node[below]{\small$1$}  .. controls (3.5,-.25) and (3,-.25) ..(3,.25) --  (3,.75) .. controls (3,1.25) and (3.5,1.25) .. (3.5,1.5);
	\node at(1,.1){\tiny $\dots$};
	\filldraw [fill=white, draw=black] (-.15,.25) rectangle (3.15,.75) node[midway] { $\tilde\varphi(t+2)$};
	\node at(1,.9){\tiny $\dots$};
} 
\ + \ 
\tikzdiagh[xscale=.75]{0}{
	\draw (.5,-.5)  --  (.5,1.5);
	\node at(1,.5){\tiny $\dots$};
	\draw (1.5,-.5)  --  (1.5,1.5);
	\draw (2,-.5) -- (2,1) .. controls (2,1.25) and (2.5,1.25) .. (2.5,1.5);
	\draw (3,-.5) .. controls (3,-.25) and (2.5,-.25) .. (2.5,0) --  (2.5,1) .. controls (2.5,1.25) and (3,1.25) .. (3,1.5);
	\draw (2.5, -.5).. controls (2.5,-.25) and(3.5,-.25) ..  (3.5,.25)   --  (3.5,.75) .. controls (3.5,1.25) and (2,1.25) .. (2,1.5) node[pos=.6,tikzdot]{} ;
	\draw[vstdhl] (0,-.5) node[below]{\small$\lambda$} --  (0,1.5);
	\draw[stdhl] (3.5,-.5) node[below]{\small$1$}  .. controls (3.5,-.25) and (3,-.25) ..(3,.25) --  (3,.75) .. controls (3,1.25) and (3.5,1.25) .. (3.5,1.5);
	\node at(1,.1){\tiny $\dots$};
	\filldraw [fill=white, draw=black] (-.15,.25) rectangle (3.15,.75) node[midway] { $\tilde\varphi(t+2)$};
	\node at(1,.9){\tiny $\dots$};
} 
\end{align*}
Then, we compute
\[
\tikzdiagh[xscale=.75]{0}{
	\draw (.5,-.5)  --  (.5,1.5);
	\node at(1,.5){\tiny $\dots$};
	\draw (1.5,-.5)  --  (1.5,1.5);
	\draw (2,-.5) -- (2,1) .. controls (2,1.25) and (2.5,1.25) .. (2.5,1.5);
	\draw  (2.5,-.5) --  (2.5,1) .. controls (2.5,1.25) and (3,1.25) .. (3,1.5);
	\draw (3, -.5).. controls (3,-.25) and(3.5,-.25) ..  (3.5,.25)   --  (3.5,.75) .. controls (3.5,1.25) and (2,1.25) .. (2,1.5);
	\draw[vstdhl] (0,-.5) node[below]{\small$\lambda$} --  (0,1.5);
	\draw[stdhl] (3.5,-.5) node[below]{\small$1$}  .. controls (3.5,-.25) and (3,-.25) ..(3,.25) --  (3,.75) .. controls (3,1.25) and (3.5,1.25) .. (3.5,1.5);
	\node at(1,.1){\tiny $\dots$};
	\filldraw [fill=white, draw=black] (-.15,.25) rectangle (3.15,.75) node[midway] { $\tilde\varphi(t+2)$};
	\node at(1,.9){\tiny $\dots$};
}  
\ = \ 
\tikzdiagh[xscale=.75]{0}{
	\draw (.5,-.5)  --  (.5,1.5);
	\node at(1,.5){\tiny $\dots$};
	\draw (1.5,-.5)  --  (1.5,1.5);
	\draw (2,-.5) -- (2,1) .. controls (2,1.25) and (3,1.25) .. (3,1.5);
	\draw  (2.5,-.5) .. controls (2.5,-.25) and (3,-.25) ..(3,.25) .. controls (3,.5) and (3.5,.5) .. (3.5,.75) .. controls (3.5,1.25) and (2.5,1.25) .. (2.5,1.5);
	\draw (3, -.5).. controls (3,-.25) and(3.5,-.25) ..  (3.5,.25)  .. controls (3.5,.5) and (3,.5) ..  (3,.75) .. controls (3,1.25) and (2,1.25) .. (2,1.5);
	\draw[vstdhl] (0,-.5) node[below]{\small$\lambda$} --  (0,1.5);
	\draw[stdhl] (3.5,-.5) node[below]{\small$1$}  .. controls (3.5,-.25) and (2.5,0) ..(2.5,.25) --  (2.5,.75) .. controls (2.5,1) and (3.5,1.25) .. (3.5,1.5);
	\node at(1,.1){\tiny $\dots$};
	\filldraw [fill=white, draw=black] (-.15,.25) rectangle (2.65,.75) node[midway] { $\tilde\varphi(t+1)$};
	\node at(1,.9){\tiny $\dots$};
} 
\ + \ 
\tikzdiagh[xscale=.75]{0}{
	\draw (.5,-.5)  --  (.5,1.5);
	\node at(1,.5){\tiny $\dots$};
	\draw (1.5,-.5)  --  (1.5,1.5);
	\draw (2.5,-.5) .. controls (2.5,-.25) and (2,-.25) .. (2,0) -- (2,1) .. controls (2,1.25) and (2.5,1.25) .. (2.5,1.5);
	\draw  (2,-.5) .. controls (2,-.25) and (3,-.25) ..(3,.25) .. controls (3,.5) and (3.5,.5) .. (3.5,.75) .. controls (3.5,1.25) and (3,1.25) .. (3,1.5);
	\draw (3, -.5).. controls (3,-.25) and(3.5,-.25) ..  (3.5,.25)  .. controls (3.5,.5) and (3,.5) ..  (3,.75) .. controls (3,1.25) and (2,1.25) .. (2,1.5);
	\draw[vstdhl] (0,-.5) node[below]{\small$\lambda$} --  (0,1.5);
	\draw[stdhl] (3.5,-.5) node[below]{\small$1$}  .. controls (3.5,-.25) and (2.5,0) ..(2.5,.25) --  (2.5,.75) .. controls (2.5,1) and (3.5,1.25) .. (3.5,1.5);
	\node at(1,.1){\tiny $\dots$};
	\filldraw [fill=white, draw=black] (-.15,.25) rectangle (2.65,.75) node[midway] { $\tilde\varphi(t+1)$};
	\node at(1,.9){\tiny $\dots$};
}
\ + \ 
\tikzdiagh[xscale=.75]{0}{
	\draw (.5,-.5)  --  (.5,1.5);
	\node at(1,.5){\tiny $\dots$};
	\draw (1.5,-.5)  --  (1.5,1.5);
	\draw (2.5,-.5) .. controls (2.5,-.25) and (2,-.25) .. (2,0) -- (2,1) .. controls (2,1.25) and (3,1.25) .. (3,1.5);
	\draw  (2,-.5) .. controls (2,-.25) and (3,-.25) ..(3,.25) node[tikzdot, pos=1]{} .. controls (3,.5) and (3.5,.5) .. (3.5,.75) .. controls (3.5,1.25) and (2.5,1.25) .. (2.5,1.5);
	\draw (3, -.5).. controls (3,-.25) and(3.5,-.25) ..  (3.5,.25)  .. controls (3.5,.5) and (3,.5) ..  (3,.75) .. controls (3,1.25) and (2,1.25) .. (2,1.5);
	\draw[vstdhl] (0,-.5) node[below]{\small$\lambda$} --  (0,1.5);
	\draw[stdhl] (3.5,-.5) node[below]{\small$1$}  .. controls (3.5,-.25) and (2.5,0) ..(2.5,.25) --  (2.5,.75) .. controls (2.5,1) and (3.5,1.25) .. (3.5,1.5);
	\node at(1,.1){\tiny $\dots$};
	\filldraw [fill=white, draw=black] (-.15,.25) rectangle (2.65,.75) node[midway] { $\tilde\varphi(t+1)$};
	\node at(1,.9){\tiny $\dots$};
}
\]
and
\[
\tikzdiagh[xscale=.75]{0}{
	\draw (.5,-.5)  --  (.5,1.5);
	\node at(1,.5){\tiny $\dots$};
	\draw (1.5,-.5)  --  (1.5,1.5);
	\draw (2,-.5) -- (2,1) .. controls (2,1.25) and (2.5,1.25) .. (2.5,1.5);
	\draw (3,-.5) .. controls (3,-.25) and (2.5,-.25) .. (2.5,0) --  (2.5,1) .. controls (2.5,1.25) and (3,1.25) .. (3,1.5);
	\draw (2.5, -.5).. controls (2.5,-.25) and(3.5,-.25) ..  (3.5,.25)   --  (3.5,.75) .. controls (3.5,1.25) and (2,1.25) .. (2,1.5) node[pos=.6,tikzdot]{} ;
	\draw[vstdhl] (0,-.5) node[below]{\small$\lambda$} --  (0,1.5);
	\draw[stdhl] (3.5,-.5) node[below]{\small$1$}  .. controls (3.5,-.25) and (3,-.25) ..(3,.25) --  (3,.75) .. controls (3,1.25) and (3.5,1.25) .. (3.5,1.5);
	\node at(1,.1){\tiny $\dots$};
	\filldraw [fill=white, draw=black] (-.15,.25) rectangle (3.15,.75) node[midway] { $\tilde\varphi(t+2)$};
	\node at(1,.9){\tiny $\dots$};
} 
\ = \ 
\tikzdiagh[xscale=.75]{0}{
	\draw (.5,-.5)  --  (.5,1.75);
	\node at(1,.5){\tiny $\dots$};
	\draw (1.5,-.5)  --  (1.5,1.75);
	\draw (2,-.5) -- (2,1) .. controls (2,1.25) and (3,1.25) .. (3,1.75);
	\draw  (2.5,-.5) .. controls (2.5,-.25) and (3,-.25) ..(3,.25) .. controls (3,.5) and (3.5,.5) .. (3.5,.75) .. controls (3.5,1.5)  and (2,1.5) ..(2,1.75) node[tikzdot,pos=.6]{};
	\draw (3, -.5).. controls (3,-.25) and(3.5,-.25) ..  (3.5,.25)  .. controls (3.5,.5) and (3,.5) ..  (3,.75) .. controls (3,1.25) and (2,1.25) .. (2,1.5) .. controls (2,1.625) and (2.5,1.625) .. (2.5,1.75);
	\draw[vstdhl] (0,-.5) node[below]{\small$\lambda$} --  (0,1.75);
	\draw[stdhl] (3.5,-.5) node[below]{\small$1$}  .. controls (3.5,-.25) and (2.5,0) ..(2.5,.25) --  (2.5,.75) .. controls (2.5,1) and (3.5,1.25) .. (3.5,1.5) -- (3.5,1.75);
	\node at(1,.1){\tiny $\dots$};
	\filldraw [fill=white, draw=black] (-.15,.25) rectangle (2.65,.75) node[midway] { $\tilde\varphi(t+1)$};
	\node at(1,.9){\tiny $\dots$};
}
\ + \ 
\tikzdiagh[xscale=.75]{0}{
	\draw (.5,-.5)  --  (.5,1.5);
	\node at(1,.5){\tiny $\dots$};
	\draw (1.5,-.5)  --  (1.5,1.5);
	\draw (3,-.5) .. controls (3,-.25) and (2,-.25) .. (2,0) -- (2,1) .. controls (2,1.25) and (2.5,1.25) .. (2.5,1.5);
	\draw  (2,-.5) .. controls (2,-.25) and (3,-.25) ..(3,.25) .. controls (3,.5) and (3.5,.5) .. (3.5,.75) .. controls (3.5,1.25) and (3,1.25) .. (3,1.5);
	\draw (2.5, -.5).. controls (2.5,-.25) and(3.5,-.25) ..  (3.5,.25)  .. controls (3.5,.5) and (3,.5) ..  (3,.75) node[tikzdot,pos=1]{} .. controls (3,1.25) and (2,1.25) .. (2,1.5);
	\draw[vstdhl] (0,-.5) node[below]{\small$\lambda$} --  (0,1.5);
	\draw[stdhl] (3.5,-.5) node[below]{\small$1$}  .. controls (3.5,-.25) and (2.5,0) ..(2.5,.25) --  (2.5,.75) .. controls (2.5,1) and (3.5,1.25) .. (3.5,1.5);
	\node at(1,.1){\tiny $\dots$};
	\filldraw [fill=white, draw=black] (-.15,.25) rectangle (2.65,.75) node[midway] { $\tilde\varphi(t+1)$};
	\node at(1,.9){\tiny $\dots$};
}
\ + \ 
\tikzdiagh[xscale=.75]{0}{
	\draw (.5,-.5)  --  (.5,1.75);
	\node at(1,.5){\tiny $\dots$};
	\draw (1.5,-.5)  --  (1.5,1.75);
	\draw (3,-.5) .. controls (3,-.25) and (2,-.25) .. (2,0) --  (2,1) .. controls (2,1.25) and (3,1.25) .. (3,1.75);
	\draw  (2,-.5) .. controls (2,-.25) and (3,-.25) ..(3,.25) --  (3,.75) node[midway, tikzdot]{} .. controls (3,1.25) and (2,1.25) .. (2,1.5) .. controls (2,1.625) and (2.5,1.625) .. (2.5,1.75);
	\draw (2.5, -.5).. controls (2.5,-.25) and(3.5,-.25) ..  (3.5,.25) -- (3.5,.75) .. controls (3.5,1.5)  and (2,1.5) ..(2,1.75) node[tikzdot,pos=.6]{};
	\draw[vstdhl] (0,-.5) node[below]{\small$\lambda$} --  (0,1.75);
	\draw[stdhl] (3.5,-.5) node[below]{\small$1$}  .. controls (3.5,-.25) and (2.5,0) ..(2.5,.25) --  (2.5,.75) .. controls (2.5,1) and (3.5,1.25) .. (3.5,1.5) -- (3.5,1.75);
	\node at(1,.1){\tiny $\dots$};
	\filldraw [fill=white, draw=black] (-.15,.25) rectangle (2.65,.75) node[midway] { $\tilde\varphi(t+1)$};
	\node at(1,.9){\tiny $\dots$};
}
\]
Furthermore, we compute mainly using \cref{eq:nhR2andR3} and \cref{eq:nhdotslide},
\[
\tikzdiagh[xscale=.75]{0}{
	\draw (.5,-.5)  --  (.5,1.75);
	\node at(1,.5){\tiny $\dots$};
	\draw (1.5,-.5)  --  (1.5,1.75);
	\draw (2,-.5) -- (2,1) .. controls (2,1.25) and (3,1.25) .. (3,1.75);
	\draw  (2.5,-.5) .. controls (2.5,-.25) and (3,-.25) ..(3,.25) .. controls (3,.5) and (3.5,.5) .. (3.5,.75) .. controls (3.5,1.5)  and (2,1.5) ..(2,1.75) node[tikzdot,pos=.6]{};
	\draw (3, -.5).. controls (3,-.25) and(3.5,-.25) ..  (3.5,.25)  .. controls (3.5,.5) and (3,.5) ..  (3,.75) .. controls (3,1.25) and (2,1.25) .. (2,1.5) .. controls (2,1.625) and (2.5,1.625) .. (2.5,1.75);
	\draw[vstdhl] (0,-.5) node[below]{\small$\lambda$} --  (0,1.75);
	\draw[stdhl] (3.5,-.5) node[below]{\small$1$}  .. controls (3.5,-.25) and (2.5,0) ..(2.5,.25) --  (2.5,.75) .. controls (2.5,1) and (3.5,1.25) .. (3.5,1.5) -- (3.5,1.75);
	\node at(1,.1){\tiny $\dots$};
	\filldraw [fill=white, draw=black] (-.15,.25) rectangle (2.65,.75) node[midway] { $\tilde\varphi(t+1)$};
	\node at(1,.9){\tiny $\dots$};
}
\ = - \  
\tikzdiagh[xscale=.75]{0}{
	\draw (.5,-.5)  --  (.5,1.5);
	\node at(1,.5){\tiny $\dots$};
	\draw (1.5,-.5)  --  (1.5,1.5);
	\draw (2,-.5) -- (2,1) .. controls (2,1.25) and (3,1.25) .. (3,1.5);
	\draw  (2.5,-.5) .. controls (2.5,-.25) and (3,-.25) ..(3,.25) .. controls (3,.5) and (3.5,.5) .. (3.5,.75) .. controls (3.5,1.25) and (2.5,1.25) .. (2.5,1.5);
	\draw (3, -.5).. controls (3,-.25) and(3.5,-.25) ..  (3.5,.25)  .. controls (3.5,.5) and (3,.5) ..  (3,.75) .. controls (3,1.25) and (2,1.25) .. (2,1.5);
	\draw[vstdhl] (0,-.5) node[below]{\small$\lambda$} --  (0,1.5);
	\draw[stdhl] (3.5,-.5) node[below]{\small$1$}  .. controls (3.5,-.25) and (2.5,0) ..(2.5,.25) --  (2.5,.75) .. controls (2.5,1) and (3.5,1.25) .. (3.5,1.5);
	\node at(1,.1){\tiny $\dots$};
	\filldraw [fill=white, draw=black] (-.15,.25) rectangle (2.65,.75) node[midway] { $\tilde\varphi(t+1)$};
	\node at(1,.9){\tiny $\dots$};
} 
\]
and
\[
\tikzdiagh[xscale=.75]{0}{
	\draw (.5,-.5)  --  (.5,1.75);
	\node at(1,.5){\tiny $\dots$};
	\draw (1.5,-.5)  --  (1.5,1.75);
	\draw (3,-.5) .. controls (3,-.25) and (2,-.25) .. (2,0) --  (2,1) .. controls (2,1.25) and (3,1.25) .. (3,1.75);
	\draw  (2,-.5) .. controls (2,-.25) and (3,-.25) ..(3,.25) --  (3,.75) node[midway, tikzdot]{} .. controls (3,1.25) and (2,1.25) .. (2,1.5) .. controls (2,1.625) and (2.5,1.625) .. (2.5,1.75);
	\draw (2.5, -.5).. controls (2.5,-.25) and(3.5,-.25) ..  (3.5,.25) -- (3.5,.75) .. controls (3.5,1.5)  and (2,1.5) ..(2,1.75) node[tikzdot,pos=.6]{};
	\draw[vstdhl] (0,-.5) node[below]{\small$\lambda$} --  (0,1.75);
	\draw[stdhl] (3.5,-.5) node[below]{\small$1$}  .. controls (3.5,-.25) and (2.5,0) ..(2.5,.25) --  (2.5,.75) .. controls (2.5,1) and (3.5,1.25) .. (3.5,1.5) -- (3.5,1.75);
	\node at(1,.1){\tiny $\dots$};
	\filldraw [fill=white, draw=black] (-.15,.25) rectangle (2.65,.75) node[midway] { $\tilde\varphi(t+1)$};
	\node at(1,.9){\tiny $\dots$};
}
\ = \ 
\tikzdiagh[xscale=.75]{0}{
	\draw (.5,-.5)  --  (.5,1.5);
	\node at(1,.5){\tiny $\dots$};
	\draw (1.5,-.5)  --  (1.5,1.5);
	\draw (3,-.5) .. controls (3,-.25) and (2,-.25) .. (2,0) -- (2,1) .. controls (2,1.25) and (3,1.25) .. (3,1.5);
	\draw  (2,-.5) .. controls (2,-.25) and (3,-.25) ..(3,.25) .. controls (3,.5) and (3.5,.5) .. (3.5,.75)  node[tikzdot,pos=1]{}  .. controls (3.5,1.25) and (2.5,1.25) .. (2.5,1.5);
	\draw (2.5, -.5).. controls (2.5,-.25) and(3.5,-.25) ..  (3.5,.25)  .. controls (3.5,.5) and (3,.5) ..  (3,.75) node[tikzdot,pos=1]{} .. controls (3,1.25) and (2,1.25) .. (2,1.5);
	\draw[vstdhl] (0,-.5) node[below]{\small$\lambda$} --  (0,1.5);
	\draw[stdhl] (3.5,-.5) node[below]{\small$1$}  .. controls (3.5,-.25) and (2.5,0) ..(2.5,.25) --  (2.5,.75) .. controls (2.5,1) and (3.5,1.25) .. (3.5,1.5);
	\node at(1,.1){\tiny $\dots$};
	\filldraw [fill=white, draw=black] (-.15,.25) rectangle (2.65,.75) node[midway] { $\tilde\varphi(t+1)$};
	\node at(1,.9){\tiny $\dots$};
}
\]
In conclusion, we get
\begin{align*}
&\tikzdiagh[xscale=.75]{0}{
	\draw (.5,-.5)  --  (.5,1.5);
	\node at(1,.5){\tiny $\dots$};
	\draw (1.5,-.5)  --  (1.5,1.5);
	\draw (2,-.5)  --  (2,1) .. controls (2,1.25) and (2.5,1.25) .. (2.5,1.5) ;
	\draw (2.5,-.5)  --  (2.5,1) .. controls (2.5,1.25) and (2,1.25) .. (2,1.5);
	\draw (3,-.5) -- (3,1.5);
	\draw[vstdhl] (0,-.5) node[below]{\small$\lambda$} --  (0,1.5);
	\draw[stdhl] (3.5,-.5) node[below]{\small$1$} --  (3.5,1.5);
	\node at(1,.1){\tiny $\dots$};
	\filldraw [fill=white, draw=black] (-.15,.25) rectangle (3.65,.75) node[midway] { $\tilde\varphi(t+3)$};
	\node at(1,.9){\tiny $\dots$};
} 
\\
\ &= \  
\tikzdiagh[xscale=.75]{0}{
	\draw (.5,-.5)  --  (.5,1.5);
	\node at(1,.5){\tiny $\dots$};
	\draw (1.5,-.5)  --  (1.5,1.5);
	\draw (2.5,-.5) .. controls (2.5,-.25) and (2,-.25) .. (2,0) -- (2,1) .. controls (2,1.25) and (2.5,1.25) .. (2.5,1.5);
	\draw  (2,-.5) .. controls (2,-.25) and (3,-.25) ..(3,.25) .. controls (3,.5) and (3.5,.5) .. (3.5,.75) .. controls (3.5,1.25) and (3,1.25) .. (3,1.5);
	\draw (3, -.5).. controls (3,-.25) and(3.5,-.25) ..  (3.5,.25)  .. controls (3.5,.5) and (3,.5) ..  (3,.75) .. controls (3,1.25) and (2,1.25) .. (2,1.5);
	\draw[vstdhl] (0,-.5) node[below]{\small$\lambda$} --  (0,1.5);
	\draw[stdhl] (3.5,-.5) node[below]{\small$1$}  .. controls (3.5,-.25) and (2.5,0) ..(2.5,.25) --  (2.5,.75) .. controls (2.5,1) and (3.5,1.25) .. (3.5,1.5);
	\node at(1,.1){\tiny $\dots$};
	\filldraw [fill=white, draw=black] (-.15,.25) rectangle (2.65,.75) node[midway] { $\tilde\varphi(t+1)$};
	\node at(1,.9){\tiny $\dots$};
}
\ + \ 
\tikzdiagh[xscale=.75]{0}{
	\draw (.5,-.5)  --  (.5,1.5);
	\node at(1,.5){\tiny $\dots$};
	\draw (1.5,-.5)  --  (1.5,1.5);
	\draw (2.5,-.5) .. controls (2.5,-.25) and (2,-.25) .. (2,0) -- (2,1) .. controls (2,1.25) and (3,1.25) .. (3,1.5);
	\draw  (2,-.5) .. controls (2,-.25) and (3,-.25) ..(3,.25) node[tikzdot, pos=1]{} .. controls (3,.5) and (3.5,.5) .. (3.5,.75) .. controls (3.5,1.25) and (2.5,1.25) .. (2.5,1.5);
	\draw (3, -.5).. controls (3,-.25) and(3.5,-.25) ..  (3.5,.25)  .. controls (3.5,.5) and (3,.5) ..  (3,.75) .. controls (3,1.25) and (2,1.25) .. (2,1.5);
	\draw[vstdhl] (0,-.5) node[below]{\small$\lambda$} --  (0,1.5);
	\draw[stdhl] (3.5,-.5) node[below]{\small$1$}  .. controls (3.5,-.25) and (2.5,0) ..(2.5,.25) --  (2.5,.75) .. controls (2.5,1) and (3.5,1.25) .. (3.5,1.5);
	\node at(1,.1){\tiny $\dots$};
	\filldraw [fill=white, draw=black] (-.15,.25) rectangle (2.65,.75) node[midway] { $\tilde\varphi(t+1)$};
	\node at(1,.9){\tiny $\dots$};
}
\ + \ 
\tikzdiagh[xscale=.75]{0}{
	\draw (.5,-.5)  --  (.5,1.5);
	\node at(1,.5){\tiny $\dots$};
	\draw (1.5,-.5)  --  (1.5,1.5);
	\draw (3,-.5) .. controls (3,-.25) and (2,-.25) .. (2,0) -- (2,1) .. controls (2,1.25) and (2.5,1.25) .. (2.5,1.5);
	\draw  (2,-.5) .. controls (2,-.25) and (3,-.25) ..(3,.25) .. controls (3,.5) and (3.5,.5) .. (3.5,.75) .. controls (3.5,1.25) and (3,1.25) .. (3,1.5);
	\draw (2.5, -.5).. controls (2.5,-.25) and(3.5,-.25) ..  (3.5,.25)  .. controls (3.5,.5) and (3,.5) ..  (3,.75) node[tikzdot,pos=1]{} .. controls (3,1.25) and (2,1.25) .. (2,1.5);
	\draw[vstdhl] (0,-.5) node[below]{\small$\lambda$} --  (0,1.5);
	\draw[stdhl] (3.5,-.5) node[below]{\small$1$}  .. controls (3.5,-.25) and (2.5,0) ..(2.5,.25) --  (2.5,.75) .. controls (2.5,1) and (3.5,1.25) .. (3.5,1.5);
	\node at(1,.1){\tiny $\dots$};
	\filldraw [fill=white, draw=black] (-.15,.25) rectangle (2.65,.75) node[midway] { $\tilde\varphi(t+1)$};
	\node at(1,.9){\tiny $\dots$};
}
\ + \ 
\tikzdiagh[xscale=.75]{0}{
	\draw (.5,-.5)  --  (.5,1.5);
	\node at(1,.5){\tiny $\dots$};
	\draw (1.5,-.5)  --  (1.5,1.5);
	\draw (3,-.5) .. controls (3,-.25) and (2,-.25) .. (2,0) -- (2,1) .. controls (2,1.25) and (3,1.25) .. (3,1.5);
	\draw  (2,-.5) .. controls (2,-.25) and (3,-.25) ..(3,.25) .. controls (3,.5) and (3.5,.5) .. (3.5,.75)  node[tikzdot,pos=1]{}  .. controls (3.5,1.25) and (2.5,1.25) .. (2.5,1.5);
	\draw (2.5, -.5).. controls (2.5,-.25) and(3.5,-.25) ..  (3.5,.25)  .. controls (3.5,.5) and (3,.5) ..  (3,.75) node[tikzdot,pos=1]{} .. controls (3,1.25) and (2,1.25) .. (2,1.5);
	\draw[vstdhl] (0,-.5) node[below]{\small$\lambda$} --  (0,1.5);
	\draw[stdhl] (3.5,-.5) node[below]{\small$1$}  .. controls (3.5,-.25) and (2.5,0) ..(2.5,.25) --  (2.5,.75) .. controls (2.5,1) and (3.5,1.25) .. (3.5,1.5);
	\node at(1,.1){\tiny $\dots$};
	\filldraw [fill=white, draw=black] (-.15,.25) rectangle (2.65,.75) node[midway] { $\tilde\varphi(t+1)$};
	\node at(1,.9){\tiny $\dots$};
}
\end{align*}
which is symmetric with respect to taking the mirror image along the horizontal axis. Therefore, we get the same a crossing at the bottom of $\tilde \varphi(t+3)$, finishing the proof.
\end{proof}

\begin{lem}\label{lem:varphitredblackcrossing}
We have
\[
\tikzdiagh[xscale=.75]{0}{
	\draw (.5,-.5)  --  (.5,1.5);
	\node at(1,.5){\tiny $\dots$};
	\draw (1.5,-.5)  --  (1.5,1.5);
	\draw (2,-.5)  --  (2,1) .. controls (2,1.25) and (2.5,1.25) .. (2.5,1.5);
	\draw[vstdhl] (0,-.5) node[below]{\small$\lambda$} --  (0,1.5);
	\draw[stdhl] (2.5,-.5) node[below]{\small$1$} --  (2.5,1) .. controls (2.5,1.25) and (2,1.25) .. (2,1.5);
	\node at(1,.1){\tiny $\dots$};
	\filldraw [fill=white, draw=black] (-.15,.25) rectangle (2.65,.75) node[midway] { $\tilde\varphi(t+1)$};
	\node at(1,.9){\tiny $\dots$};
}
\ = - \ 
\tikzdiagh[xscale=.75]{0}{
	\draw (.5,-.5)  --  (.5,1.5);
	\node at(1,.5){\tiny $\dots$};
	\draw (1.5,-.5)  --  (1.5,1.5);
	\draw (2,-.5)  ..controls (2,-.25) and (2.5,-.25) ..  (2.5,0) --(2.5,1.5);
	\draw[vstdhl] (0,-.5) node[below]{\small$\lambda$} --  (0,1.5);
	\draw[stdhl] (2.5,-.5) node[below]{\small$1$} .. controls (2.5,-.25) and (2,-.25) ..  (2,0) -- (2,1.5);
	\node at(1,.1){\tiny $\dots$};
	\filldraw [fill=white, draw=black] (-.15,.25) rectangle (2.15,.75) node[midway] { $\tilde\varphi(t)$};
	\node at(1,.9){\tiny $\dots$};
}
\]
for all $t \geq 0$. 
\end{lem}

\begin{proof}
We prove the statement by induction on $t$.
The case $t=0$ follows from \cref{eq:nailslidedcross}. We suppose the claim is true for $t \geq 0$. We compute using the mirror of \cref{eq:altvarphik},
\begin{align*}
\tikzdiagh[xscale=.75]{0}{
	\draw (.5,-.5)  --  (.5,1.5);
	\node at(1,.5){\tiny $\dots$};
	\draw (1.5,-.5)  --  (1.5,1.5);
	\draw (2,-.5) -- (2,1.5);
	\draw (2.5,-.5)  --  (2.5,1) .. controls (2.5,1.25) and (3,1.25) .. (3,1.5);
	\draw[vstdhl] (0,-.5) node[below]{\small$\lambda$} --  (0,1.5);
	\draw[stdhl] (3,-.5) node[below]{\small$1$} --  (3,1) .. controls (3,1.25) and (2.5,1.25) .. (2.5,1.5);
	\node at(1,.1){\tiny $\dots$};
	\filldraw [fill=white, draw=black] (-.15,.25) rectangle (3.15,.75) node[midway] { $\tilde\varphi(t+2)$};
	\node at(1,.9){\tiny $\dots$};
}
\ &= \ 
\tikzdiagh[xscale=.75]{0}{
	\draw (.5,-.5)  --  (.5,1.75);
	\node at(1,.5){\tiny $\dots$};
	\draw (1.5,-.5)  --  (1.5,1.75);
	\draw (2.5,-.5) .. controls (2.5,-.25) and (2,-.25) .. (2,0) --  (2,1) -- (2,1.75);
	\draw (2, -.5).. controls (2,-.25) and(3,-.25) ..  (3,.25)  --  (3,.75) .. controls (3,1) and (2.5,1) .. (2.5,1.25) .. controls (2.5,1.5) and (3,1.5) ..(3,1.75);
	\draw[vstdhl] (0,-.5) node[below]{\small$\lambda$} --  (0,1.75);
	\draw[stdhl] (3,-.5) node[below]{\small$1$}  .. controls (3,-.25) and (2.5,-.25) ..(2.5,.25) --  (2.5,.75) .. controls (2.5,1) and (3,1) .. (3,1.25) .. controls (3,1.5) and (2.5,1.5) .. (2.5,1.75);
	\node at(1,.1){\tiny $\dots$};
	\filldraw [fill=white, draw=black] (-.15,.25) rectangle (2.65,.75) node[midway] { $\tilde\varphi(t+1)$};
	\node at(1,.9){\tiny $\dots$};
}
\ + \ 
\tikzdiagh[xscale=.75]{0}{
	\draw (.5,-.5)  --  (.5,1.75);
	\node at(1,.5){\tiny $\dots$};
	\draw (1.5,-.5)  --  (1.5,1.75);
	\draw (2.5,-.5) .. controls (2.5,-.25) and (2,-.25) .. (2,0) --  (2,1) .. controls (2,1.25) and (3,1.25) .. (3,1.75);
	\draw (2, -.5).. controls (2,-.25) and(3,-.25) ..  (3,.25) node[tikzdot, pos=.2]{}  --  (3,.75) .. controls (3,1.25) and (2,1.25) .. (2,1.75);
	\draw[vstdhl] (0,-.5) node[below]{\small$\lambda$} --  (0,1.75);
	\draw[stdhl] (3,-.5) node[below]{\small$1$}  .. controls (3,-.25) and (2.5,-.25) ..(2.5,.25) --  (2.5,.75) .. controls (2.5,1) and (3,1) .. (3,1.25) .. controls (3,1.5) and (2.5,1.5) .. (2.5,1.75);
	\node at(1,.1){\tiny $\dots$};
	\filldraw [fill=white, draw=black] (-.15,.25) rectangle (2.65,.75) node[midway] { $\tilde\varphi(t+1)$};
	\node at(1,.9){\tiny $\dots$};
} 
\\
\ &= \ 
\tikzdiagh[xscale=.75]{0}{
	\draw (.5,-.5)  --  (.5,1.75);
	\node at(1,.5){\tiny $\dots$};
	\draw (1.5,-.5)  --  (1.5,1.75);
	\draw (2.5,-.5) .. controls (2.5,-.25) and (2,-.25) .. (2,0) --  (2,1) -- (2,1.75);
	\draw (2, -.5).. controls (2,-.25) and(3,-.25) ..  (3,.25)  --  (3,.75) --(3,1.75) node[tikzdot,midway]{};
	\draw[vstdhl] (0,-.5) node[below]{\small$\lambda$} --  (0,1.75);
	\draw[stdhl] (3,-.5) node[below]{\small$1$}  .. controls (3,-.25) and (2.5,-.25) ..(2.5,.25) --  (2.5,.75) --(2.5,1.75);
	\node at(1,.1){\tiny $\dots$};
	\filldraw [fill=white, draw=black] (-.15,.25) rectangle (2.65,.75) node[midway] { $\tilde\varphi(t+1)$};
	\node at(1,.9){\tiny $\dots$};
}
\ + \ 
\tikzdiagh[xscale=.75]{0}{
	\draw (.5,-.5)  --  (.5,1.75);
	\node at(1,.5){\tiny $\dots$};
	\draw (1.5,-.5)  --  (1.5,1.75);
	\draw (2.5,-.5) .. controls (2.5,-.25) and (2,-.25) .. (2,0) --  (2,1) .. controls (2,1.25) and (3,1.25) .. (3,1.75);
	\draw (2, -.5).. controls (2,-.25) and(3,-.25) ..  (3,.25) node[tikzdot, pos=.2]{}  --  (3,.75) .. controls (3,1.25) and (2,1.25) .. (2,1.75);
	\draw[vstdhl] (0,-.5) node[below]{\small$\lambda$} --  (0,1.75);
	\draw[stdhl] (3,-.5) node[below]{\small$1$}  .. controls (3,-.25) and (2.5,-.25) ..(2.5,.25) --  (2.5,.75) .. controls (2.5,1) and (2,1) .. (2,1.25) .. controls (2,1.5) and (2.5,1.5) .. (2.5,1.75);
	\node at(1,.1){\tiny $\dots$};
	\filldraw [fill=white, draw=black] (-.15,.25) rectangle (2.65,.75) node[midway] { $\tilde\varphi(t+1)$};
	\node at(1,.9){\tiny $\dots$};
} 
\ - \ 
\tikzdiagh[xscale=.75]{0}{
	\draw (.5,-.5)  --  (.5,1.75);
	\node at(1,.5){\tiny $\dots$};
	\draw (1.5,-.5)  --  (1.5,1.75);
	\draw (2.5,-.5) .. controls (2.5,-.25) and (2,-.25) .. (2,0) --  (2,1.75);
	\draw (2, -.5).. controls (2,-.25) and(3,-.25) ..  (3,.25) node[tikzdot, pos=.2]{}  --  (3,1.75) ;
	\draw[vstdhl] (0,-.5) node[below]{\small$\lambda$} --  (0,1.75);
	\draw[stdhl] (3,-.5) node[below]{\small$1$}  .. controls (3,-.25) and (2.5,-.25) ..(2.5,.25) --  (2.5,.75) -- (2.5,1.75);
	\node at(1,.1){\tiny $\dots$};
	\filldraw [fill=white, draw=black] (-.15,.25) rectangle (2.65,.75) node[midway] { $\tilde\varphi(t+1)$};
	\node at(1,.9){\tiny $\dots$};
} 
\end{align*}
Then, we have
\[
\tikzdiagh[xscale=.75]{0}{
	\draw (.5,-.5)  --  (.5,1.75);
	\node at(1,.5){\tiny $\dots$};
	\draw (1.5,-.5)  --  (1.5,1.75);
	\draw (2.5,-.5) .. controls (2.5,-.25) and (2,-.25) .. (2,0) --  (2,1) .. controls (2,1.25) and (3,1.25) .. (3,1.75);
	\draw (2, -.5).. controls (2,-.25) and(3,-.25) ..  (3,.25) node[tikzdot, pos=.2]{}  --  (3,.75) .. controls (3,1.25) and (2,1.25) .. (2,1.75);
	\draw[vstdhl] (0,-.5) node[below]{\small$\lambda$} --  (0,1.75);
	\draw[stdhl] (3,-.5) node[below]{\small$1$}  .. controls (3,-.25) and (2.5,-.25) ..(2.5,.25) --  (2.5,.75) .. controls (2.5,1) and (2,1) .. (2,1.25) .. controls (2,1.5) and (2.5,1.5) .. (2.5,1.75);
	\node at(1,.1){\tiny $\dots$};
	\filldraw [fill=white, draw=black] (-.15,.25) rectangle (2.65,.75) node[midway] { $\tilde\varphi(t+1)$};
	\node at(1,.9){\tiny $\dots$};
} 
\ = - \ 
\tikzdiagh[xscale=.75]{0}{
	\draw (.5,-.5)  --  (.5,1.5);
	\node at(1,.5){\tiny $\dots$};
	\draw (1.5,-.5)  --  (1.5,1.5);
	\draw  (2.5,-.5).. controls (2.5,-.375) and (2,-.375) .. (2,-.25)  ..controls (2,0) and (2.5,0) ..  (2.5,.25) --(2.5,1) .. controls (2.5,1.25) and (3,1.25) .. (3,1.5);
	\draw (2,-.5) .. controls (2,-.25) and (3,-.25) .. (3,.25) -- (3,1) .. controls (3,1.25) and (2,1.25) .. (2,1.5);
	\draw[vstdhl] (0,-.5) node[below]{\small$\lambda$} --  (0,1.5);
	\draw[stdhl] (3,-.5) node[below]{\small$1$} .. controls (3,-.25) and (2,-.25) ..  (2,.25) -- (2,1) .. controls (2,1.25) and (2.5,1.25) .. (2.5,1.5);
	\node at(1,.1){\tiny $\dots$};
	\filldraw [fill=white, draw=black] (-.15,.25) rectangle (2.15,.75) node[midway] { $\tilde\varphi(t)$};
	\node at(1,.9){\tiny $\dots$};
}
\ = 0,
\]
by induction  hypothesis.
Finally, we obtain
\[
\tikzdiagh[xscale=.75]{0}{
	\draw (.5,-.5)  --  (.5,1.75);
	\node at(1,.5){\tiny $\dots$};
	\draw (1.5,-.5)  --  (1.5,1.75);
	\draw (2.5,-.5) .. controls (2.5,-.25) and (2,-.25) .. (2,0) --  (2,1) -- (2,1.75);
	\draw (2, -.5).. controls (2,-.25) and(3,-.25) ..  (3,.25)  --  (3,.75) --(3,1.75) node[tikzdot,midway]{};
	\draw[vstdhl] (0,-.5) node[below]{\small$\lambda$} --  (0,1.75);
	\draw[stdhl] (3,-.5) node[below]{\small$1$}  .. controls (3,-.25) and (2.5,-.25) ..(2.5,.25) --  (2.5,.75) --(2.5,1.75);
	\node at(1,.1){\tiny $\dots$};
	\filldraw [fill=white, draw=black] (-.15,.25) rectangle (2.65,.75) node[midway] { $\tilde\varphi(t+1)$};
	\node at(1,.9){\tiny $\dots$};
}
\ - \ 
\tikzdiagh[xscale=.75]{0}{
	\draw (.5,-.5)  --  (.5,1.75);
	\node at(1,.5){\tiny $\dots$};
	\draw (1.5,-.5)  --  (1.5,1.75);
	\draw (2.5,-.5) .. controls (2.5,-.25) and (2,-.25) .. (2,0) --  (2,1.75);
	\draw (2, -.5).. controls (2,-.25) and(3,-.25) ..  (3,.25) node[tikzdot, pos=.2]{}  --  (3,1.75) ;
	\draw[vstdhl] (0,-.5) node[below]{\small$\lambda$} --  (0,1.75);
	\draw[stdhl] (3,-.5) node[below]{\small$1$}  .. controls (3,-.25) and (2.5,-.25) ..(2.5,.25) --  (2.5,.75) -- (2.5,1.75);
	\node at(1,.1){\tiny $\dots$};
	\filldraw [fill=white, draw=black] (-.15,.25) rectangle (2.65,.75) node[midway] { $\tilde\varphi(t+1)$};
	\node at(1,.9){\tiny $\dots$};
} 
\ = - \ 
\tikzdiagh[xscale=.75]{0}{
	\draw (.5,-.5)  --  (.5,1.5);
	\node at(1,.5){\tiny $\dots$};
	\draw (1.5,-.5)  --  (1.5,1.5);
	\draw (2,-.5) -- (2,1.5);
	\draw (2.5,-.5)  .. controls (2.5,-.25) and (3,-.25) ..  (3,0) -- (3,1.5);
	\draw[vstdhl] (0,-.5) node[below]{\small$\lambda$} --  (0,1.5);
	\draw[stdhl] (3,-.5) node[below]{\small$1$} .. controls (3,-.25) and (2.5,-.25) ..  (2.5,0) --(2.5,1.5);
	\node at(1,.1){\tiny $\dots$};
	\filldraw [fill=white, draw=black] (-.15,.25) rectangle (2.65,.75) node[midway] { $\tilde\varphi(t+1)$};
	\node at(1,.9){\tiny $\dots$};
}
\]
finishing the proof.
\end{proof}

\begin{prop}
The map $\tilde\varphi^0$ is a map of dg-bimodules.
\end{prop}

\begin{proof}
As already mentioned above, it is enough to show that the left and right action by the same element of $T_b^{\lambda,r}$ on
\[
\sum_{k+\ell+|\rho|=b}  (-1)^k \tilde\varphi(k) \otimes \bar 1_{\ell,\rho}
\]
coincide. We obtain commutation with dots and crossings by induction on $k$, 
using Lemmas~\ref{lem:varphitdotright}--\ref{lem:varphitredblackcrossing}.
The commutation with a nail also  comes from a straightforward induction on $k$, where the base case is immediate by \cref{eq:relNail}.
\end{proof}




\section{Homological toolbox}\label{sec:dgcat}

The goal of this section is to recall and briefly explain the tools from homological algebra which we use in this paper. 
The main references for this section are \cite{keller}, \cite{toen} and \cite{asympK0} (see also \cite{toenlectures}, \cite{kellersurvey} and \cite[Appendix A]{naissevaz3}).

\subsection{Derived category}

Let $(A, d_A)$ be a $\bZ^n$-graded dg-algebra (with the same conventions as in \cref{sec:conventions}). 

The \emph{derived category $\cD(A,d_A)$} of $(A,d_A)$ is the localization of the category $(A,d_A)\amod$ of  $\bZ^n$-graded (left) $(A,d_A)$-dg-modules along quasi-isomorphisms. It is a triangulated category with translation functor induced by the homological shift functor $[1]$, and distinguished triangles are equivalent to 
\[
(M,d_N) \xrightarrow{f} (N,d_N) \xrightarrow{\imath_N} \cone(f) \xrightarrow{\pi_{M[1]}} (M,d_N)[1],
\]
for every maps of dg-modules $f : (M,d_M) \rightarrow (N,d_N)$. 

\subsubsection{(Co)fibrant replacements}

A \emph{cofibrant} dg-module $(P,d_P)$ is a dg-module such that $P$ is projective as graded $A$-module. 
Equivalently, it is a dg-module $(P,d_P)$ such that for every surjective quasi-isomorphism $(L,d_L) \xrightarrowdbl{\simeq} (M,d_M)$, every morphism $(P,d_P) \rightarrow (M,d_M)$ factors through $(L,d_L)$.
For any dg-module $(N, d_N)$, we have
\begin{align*}
\Hom_{\cD(A,d_A)}\bigl((P,d_P), (N,d_N)\bigr) \cong H^0_0 \left(\HOM_{(A,d_A)}\bigl((P,d_P), (N,d_N) \bigr) \right).
\end{align*}
Moreover, tensoring with a cofibrant dg-module preserves quasi-isomorphisms.

Given a left (resp. right) dg-module $(M,d_M)$, there exists a cofibrant replacement $(\br M , d_{\br M })$ (resp. $( M\rb, d_{M\rb})$) together with a surjective quasi-isomorphism $\pi_M : \br M \xrightarrowdbl{\simeq}  M$ (resp. $\pi'_M :  M \rb \xrightarrowdbl{\simeq} M$). Moreover, the assignment $M \mapsto \br M$ (resp. $M \mapsto  M \rb$) is natural. Thus, we can compute $\Hom_{\cD(A,d_A)}\bigl((M,d_M), (N,d_N)\bigr)$ by taking 
\[
H^0_0\left(\HOM_{(A,d_A)}\bigl((\br M,d_{\br M}), (N,d_N) \bigr) \right) \cong \Hom_{\cD(A,d_A)}\bigl((M,d_M), (N,d_N)\bigr).
\] 

\smallskip

A  dg-module $(I,d_I)$ is \emph{fibrant} if for every injective quasi-isomorphism $(L,d_L) \xhookrightarrow{\simeq} (M, d_M)$, every morphism $(L,d_L) \rightarrow (M,d_M)$ extends to $(M,d_M)$.  Then, we have
\begin{align*}
\Hom_{\cD(A,d_A)}\bigl((M,d_M), (I,d_I)\bigr) \cong H^0_0\left(\HOM_{(A,d_A)}\bigl((M,d_M), (I,d_I) \bigr) \right).
\end{align*}
Again, for every dg-module $(M,d_M)$ there exists a fibrant replacement $(\fr M, d_{\fr M})$ with an injective quasi-isomorphism $ \imath_M : (M,d_M) \xhookrightarrow{\simeq} (\fr M, d_{\fr M})$.

\subsubsection{Strongly projective modules}\label{sec:stronglyproj}

Let $R$ be a unital commutative ring. 
The following was introduced in  \cite{moore}, but we use the definition given in \cite{sixdgmodels}.

\begin{defn}[{\cite[Definition 8.17]{sixdgmodels}}] \label{def:stronglyproj}
A dg-module $(P,d_P)$ over a dg-$R$-algebra $(A,d_A)$ is \emph{strongly projective} if it is a direct summand of some dg-module $(A, d_A) \otimes_R (Q, d_Q)$ where $(Q,d_Q)$ is a $(R,0)$-dg-module such that both $H(Q,d_Q)$ and $\Image (d_Q)$ are projective $R$-modules.
\end{defn}

\begin{prop}[{\cite[Lemma 8.23]{sixdgmodels}}]\label{prop:stronglyproj}
Let $(P,d_P)$ be a strongly projective left dg-module. For any right dg-module $(M,d_M)$, we have an isomorphism
\[
H\left( (M,d_M) \otimes_{(A,d_A)} (P,d_P) \right) \cong H(M,d_M) \otimes_{H(A,d_A)} H(P,d_P).
\]
\end{prop}

\subsubsection{$A_\infty$-action}\label{sec:Ainftyaction}

Let $(B,d_B)$ be a dg-bimodule over a pair of dg-algebras $(S,d_S)$-$(R,d_R)$. 
As explained in \cite[\S 2.3]{MW}, there is (in general) no right $(R,d_R)$-action on $\br B_i$ compatible with the left $(S,d_S)$-action. 
However, there is an induced $A_\infty$-action (defined uniquely up to homotopy), so that the quasi-isomorphism $\pi_B : \br B \xrightarrowdbl{\simeq} B$ can be upgraded to a map of $A_\infty$-bimodules. 

\begin{lem}
  \label{lem:indAinftymap}
Let $(A,d_A)$ be a dg-algebra, and let $U$ and $V$ be dg-(bi)modules over $(A,d_A)$, with a fixed cofibrant replacement $\pi_V: \br V \rightarrow V$. Suppose $(1 \otimes p_V) : U \otimes_{(A,d_A)} \br V \xrightarrow{\simeq} U \otimes_{(A,d_A)} V$ is a quasi-isomorphism. If $f : Z \rightarrow U \otimes_{(A,d_A)} \br V$ is a map of complexes of graded $\Bbbk$-spaces, and $f \circ (1 \otimes p_V)$ is a map of dg-(bi)modules, then there is an induced map $\overline f : Z \rightarrow U \Lotimes V$ of $A_\infty$-(bi)modules whose degree zero part is $f$.   
\end{lem}

\begin{proof}
We take $\overline f := (1 \otimes p_V)^{-1} \circ (f \circ  (1 \otimes p_V))$, as a composition of maps of $A_\infty$-(bi)modules, since any map of (bi)module can be considered as a map of $A_\infty$-(bi)modules with no higher composition. 
\end{proof}

Note that the equivalent statement also holds for a cofibrant replacement $U \rb \rightarrow U$ such that $(\pi_U \otimes 1) : U \rb \otimes_{(A,d_A)} V \xrightarrow{\simeq} U \otimes_{(A,d_A)} V$ is a quasi-isomorphism.

\subsection{Dg-derived categories}\label{sec:dgdercat}

One of the issues with triangulated categories is that the category of functors between triangulated categories is in general not triangulated. To fix this, we work with a dg-enhancement of the derived category. In particular, this allows us to talk about distinguished triangles of dg-functors. 

Recall that a dg-category is a category where the hom-spaces are dg-modules over $(\Bbbk,0)$, and compositions are compatible with this structure (see \cite[\S1.2]{keller} for a precise definition). Given such a dg-category $\cC$ with hom-spaces $\Hom_{\cC}(X,Y) = (\bigoplus_{h \in \bZ} \Hom^{h}(X,Y), d_{X,Y})$, we can consider its \emph{underlying category} $Z^0(\cC)$, which is given by the same objects as $\cC$ and hom-spaces
\[
\Hom_{Z^0(\cC)}(X,Y) := \ker \left( d_{X,Y} : \Hom^0(X,Y)  \rightarrow \Hom^{-1}(X,Y) \right).
\]
Similarly, the \emph{homotopy category $H^0(\cC)$} is given by
\[
\Hom_{H^0(\cC)}(X,Y) := H^0(\Hom_\cC(X,Y)). 
\]
A \emph{dg-enhancement} of a category $\cC_0$ is a dg-category $\cC$ such that $H^0(\cC) \cong \cC_0$. 

\smallskip

The \emph{dg-derived category $\cD_{dg}(A,d_A)$} of a $\bZ^n$-graded dg-algebra $(A,d_A)$ is the $\bZ^n$-graded dg-category with objects being cofibrant dg-modules over $(A,d_A)$, and hom-spaces being subspaces of the graded dg-spaces $\HOM_{(A,d_A)}$ from \cref{eq:dghom}, given by maps that preserve the $\bZ^n$-grading:
\[
\Hom_{\cD_{dg}(A,d_A)}(M,N) := \HOM_{(A,d_A)}(M,N)_0,
\]
for $(M,d_M)$ and $(N,d_N)$ cofibrant dg-modules. 
By construction, we have $H^0(\cD_{dg}(A,d_A)) \cong \cD(A,d_A)$. 
Moreover, $\cD_{dg}(A,d_A)$ is a dg-triangulated category, meaning its homotopy category is canonically triangulated (see \cite{toen} for a precise definition,  or \cite[Appendix A]{naissevaz3} for a summary oriented toward categorification), and this triangulated structure matches with the usual one on $\cD(A,d_A)$. 

\subsubsection{Dg-functors}\label{sec:dgfunctors}

A \emph{dg-functor} between dg-categories is a functor commuting with the differentials. Given a dg-functor $F : \cC \rightarrow \cC'$, it induces a functor on the homotopy categories $[F] : H^0(\cC) \rightarrow H^0(\cC')$. 
We say that a dg-functor is a \emph{quasi-equivalence} if it gives quasi-isomorphisms on the hom-spaces, and induces an equivalence on the homotopy categories. 
We want to consider dg-category up to quasi-equivalence. Let $\Hqe$ be the homotopy category of dg-categories up to quasi-equivalence , and we write $\cRHom_{\Hqe}$ for the dg-space of quasi-functors between dg-categories (see \cite{toen}, \cite{toenlectures}, or \cite[Appendix A]{naissevaz3}). These quasi-functors induce honest functors on the homotopy categories. 
Whenever $\cC'$ is dg-triangulated, then $\cRHom_{\Hqe}(\cC,\cC')$ is dg-triangulated.

\begin{rem}
The space of quasi-functors is equivalent to the space of strictly unital $A_\infty$-functors. 
\end{rem}

It is in general hard to understand the space of quasi-functors. However, by the results of Toen~\cite{toen}, if $\Bbbk$ is a field and $(A,d_A)$ and $(A',d_{A'})$ are dg-algebras, then it is possible to compute the space of `coproduct preserving' quasi-functors $\cRHom_{\Hqe}^{cop}(\cD_{dg}(A,d_A),\cD_{dg}(A',d_{A'}))$, in the same way as the category of coproduct preserving functors between categories of modules is equivalent to  the category of bimodules. Indeed, we have a quasi-equivalence
\begin{equation}\label{eq:quasifunctequiv}
\cRHom_{\Hqe}^{cop}(\cD_{dg}(A,d_A),\cD_{dg}(A',d_{A'})) \cong \cD_{dg}((A',d_{A'}), (A,d_A)),
\end{equation}
where $ \cD_{dg}((A',d_{A'}), (A,d_A))$ is the dg-derived category of dg-bimodules. Composition of functors becomes equivalent to derived tensor product. 
Then, understanding the triangulated structure of $\cRHom_{\Hqe}^{cop}(\cD_{dg}(A,d_A),\cD_{dg}(A',d_{A'}))$ becomes as easy as to understand $\cD((A,d_A), (A',d_{A'}))$. 
In particular, a short exact sequence of dg-bimodules gives a distinguished triangle of dg-functors. 

\subsection{Derived hom and tensor dg-functors}\label{sec:deriveddghomtensor}

Let $(R,d_R)$ and $(S,d_S)$ be dg-algebras. Let $M$ and $N$ be $(R,d_R)$-module and $(S,d_S)$-module respectively. Let $B$ be a dg-bimodule over $(S,d_S)$-$(R,d_R)$. 
Then, the \emph{derived tensor product} is
\[
B \Lotimes_{(R,d_R)} M := B \otimes \br M,
\]
and the \emph{derived hom space} is
\[
\RHOM_{(S,d_S)}(B, N) := \HOM_{(S,d_S)}(B, \fr N).
\]
Note that we have quasi-isomorphisms as dg-spaces $B \Lotimes_{(R,d_R)}  M \cong B \rb \otimes_{(R,d_R)}\br M \cong B \rb \otimes_{(R,d_R)} M$, and $\RHOM_{(S,d_S)}(B, N)  \cong \HOM_{(S,d_S)}(\br B,\fr N) \cong \HOM_{(S,d_S)}(\br B, N)$. 

\smallskip

This defines in turns triangulated dg-functors
\begin{align*}
B \Lotimes_{(R,d_R)} (-) &: \cD_{dg}(R,d_R) \rightarrow \cD_{dg}(S,d_S),
\intertext{and} 
\RHOM_{(S,d_S)}(B, -) &: \cD_{dg}(S,d_S) \rightarrow \cD_{dg}(R,d_R).
\end{align*}
They induce a pair of adjoint functors $B \Lotimes_{(R,d_R)} (-)  \vdash \RHOM_{(S,d_S)}(B, -)$ between the derived categories $\cD_{dg}(R,d_R)$ and $\cD_{dg}(S,d_S)$.

\subsubsection{Computing units and counits}\label{sec:unitandcounit}

The natural bijection
\[
\bar \Phi_{M,N}^{B} :  \Hom_{\cD(S,d_S)}( B \Lotimes_{(R,d_R)} M, N  ) \xrightarrow{\simeq} \Hom_{\cD(R,d_R)}(M, \RHOM_{(S,d_S)}(B,N)), 
\]
is obtained by making the following diagram commutative:
\[
\begin{tikzcd}
\Hom_{\cD(S,d_S)}( B \Lotimes_{(R,d_R)} M, N  ) \ar{r}{\bar \Phi_{M,N}^B}
\ar[sloped]{d}{\simeq}
&
\Hom_{\cD(R,d_R)}(M, \RHOM_{(S,d_S)}(B,N))
\\
\Hom_{(S,d_S)}( B \otimes_{(R,d_R)} \br M, \fr N ) 
\ar[swap]{r}{\Phi_{\br M,\fr N}^B}
&
\Hom_{(R,d_R)}(\br M, \HOM_{(S,d_S)}(B,\fr N)).
\ar[sloped]{u}{\simeq}
\end{tikzcd}
\] 
where $\Phi$ is defined in \cref{eq:homtensajd}. 

\smallskip

For the sake of keeping notations short, we will write $\HOM$  instead of  $\HOM_{(S,d_S)}$, and $\otimes$ instead of $\otimes_{(R,d_R)}$, and similarly for the derived versions. 

We are interested in computing the unit
\[
\eta_M : M \rightarrow  \RHOM(B, B \Lotimes M),  
\]
which is given by $\eta_M = \bar \Phi_{M, B \Lotimes M}^B(\id_{B \Lotimes M})$. 
Composing with the isomorphisms $\RHOM(B, B \Lotimes M) \cong \HOM(B, \fr (B \otimes \br M))$ and $\br M \cong M$, we can compute $\eta_M$ as 
\[
\eta_M' =  \Phi_{\br M,\fr(B \Lotimes M)}^B( \imath_{B \Lotimes M}) : \br M \rightarrow  \HOM(B, \fr (B \otimes \br M)),
\]
which gives
\[
\eta'_M(m) = (b \mapsto \imath_{B \otimes \br M}(b \otimes m)).
\]
Using the quasi-isomorphisms
\[
\begin{tikzcd}
\HOM(\br B, B \Lotimes M) 
\ar["\imath_{B \Lotimes M} \circ -", "\simeq"']{r}
& 
\HOM(\br B, \fr(B \Lotimes M))
&
\ar["- \circ \pi_B"', "\simeq"]{l}
\HOM(B, \fr(B \Lotimes M)),
\end{tikzcd}
\]
we can compute $\eta'_M$ through
\begin{align*}
\eta_M'' &: \br M \rightarrow \HOM(\br B, B \otimes \br M), 
&
\eta_M''(m) &:= (b \mapsto \pi_B(b) \otimes m). 
\end{align*}
This is particularly useful, since it means we do not have to compute any fibrant replacement to understand $\eta_M$. 

\smallskip

Similarly, for the counit
\[
\varepsilon_M : B \Lotimes \RHOM(B,M) \rightarrow M,
\]
we have $\varepsilon_M =  (\bar\Phi^{B}_{\RHOM(B,M),M})^{-1}(\id_{\RHOM(B,M)})$. 
We rewrite it as
\[
\varepsilon_M ' = \Phi_{\br \HOM(B, \fr M), \fr M}^{-1}(\pi_{\HOM(B, \fr M)})
: B \otimes \br \HOM(B, \fr M) \rightarrow \fr M,
\]
with $\varepsilon_M'(b \otimes f) = \bigl(\pi_{\HOM(B, \fr M)}(f)\bigr)(b)$.  
We consider the quasi-isomorphisms
\[
\begin{tikzcd}[column sep=10ex]
B \otimes \br \HOM(B, \fr M) 
&
\ar["\pi_B \otimes 1"',"\simeq"]{l}
B \rb \otimes \br \HOM(B, \fr M) 
\ar["1 \otimes \pi_{\HOM(B,\fr M)}","\simeq"']{r}
&
B \rb \otimes \HOM(B, \fr M).
\end{tikzcd}
\]
Therefore, we can compute $\varepsilon_M$ as
\begin{align*}
\varepsilon''_M &: B \rb \otimes \HOM(B, \fr M) \rightarrow \fr M,
&
\varepsilon''_M(b \otimes f) &:= f(\pi'_B(b)),
\end{align*}
where $\pi'_B : B \rb \xrightarrow{\simeq} B$. 

If in addition $B$ is already cofibrant as left dg-module, then we can suppose $\br B = B$ and $\pi_B = \id_B$, and we obtain a commutative diagram
\[
\begin{tikzcd}
B\rb \otimes \HOM(B, \fr M)
\ar{r}{\varepsilon''_M}
\ar[sloped]{d}{\simeq}
\ar[swap]{d}{1 \otimes (- \circ \pi_B)}
&
\fr M
\ar[equals]{d}
\\
B\rb \otimes \HOM(\br B, \fr M) 
\ar[sloped]{d}{\simeq}
\ar{r}
&
\fr M
\\
B\rb \otimes \HOM(\br B, M)
\ar[swap]{r}{\varepsilon'''_M}
\ar{u}{1 \otimes (\imath_M \circ -)}
&
\ar[sloped]{u}{\simeq}
\ar[swap]{u}{\imath_M}
M
\end{tikzcd}
\]
where
\begin{align*}
\varepsilon_M''' &: B \rb \otimes \HOM(B, M) \rightarrow M,
&
\varepsilon_M'''(b \otimes f) := f(\pi'(b)).
\end{align*}
This is useful, since it means we can compute $\varepsilon_M$ using $\varepsilon'''_M$, which does not require any fibrant replacement. 
\color{black}

\subsection{Asymptotic Grothendieck group} \label{sec:asympK0}

The usual definition of the Grothendieck group of a triangulated category does not take into consideration relations coming from infinite iterated extensions. 
When $\cC$ is a triangulated subcategory of a triangulated category $\cT$ admitting countable products and coproducts, and these preserves distinguished triangles, then there exists a notion of \emph{asymptotic Grothendick group $\bKO^\Delta(\cC)$} of $\cC$, given by modding out relations obtained from Milnor (co)limits (see \cref{sec:cblfitext} below) in the usual Grothendieck group $K_0(\cC)$ (see \cite[\S 8]{asympK0} for a precise definition).

\subsubsection{Ring of Laurent series}

We follow the construction of the ring of formal Laurent series given in~\cite{laurent} (see also \cite[\S5]{asympK0}). 
The \emph{ring of formal Laurent series} $\Bbbk\pp{x_1,\dots, x_n}$ is given by first choosing a total additive order $\prec$ on $\bZ^n$. 
One says that a cone $C := \{\alpha_1 v_1 + \cdots + \alpha_n v_n | \alpha_i \in \bR_{\geq 0} \} \subset \bR^n$ is compatible with $\prec$ whenever $0 \prec v_i$ for all $i \in \{1,\dots,n\}$. 
 Then, we set 
\[
\Bbbk\pp{x_1,\dots,x_n} := \bigcup_{\be \in \bZ^n} x^{\be} \Bbbk_{\prec}\llbracket x_1,\dots, x_n \rrbracket,
\]
where $\Bbbk_{\prec}\llbracket x_1,\dots, x_n \rrbracket$ consists of formal Laurent series in $\Bbbk \llbracket x_1,\dots, x_n\rrbracket$ such that the terms are contained in a cone compatible with $\prec$. 
It forms a ring when we equip $\Bbbk\pp{x_1,\dots,x_n}$ with the usual addition and multiplication of series. 

\subsubsection{C.b.l.f. structures}\label{sec:cblfstruct}

We fix an arbitrary additive total order $\prec$ on $\bZ^n$. 
We say that a $\bZ^n$-graded $\Bbbk$-vector space $M = \bigoplus_{ \bigoplus_{\bg \in \bZ^n}} M_\bg$ is \emph{c.b.l.f. (cone bounded, locally finite) dimensional}  if
\begin{itemize}
\item $\dim M_\bg < \infty$ for all $\bg \in \bZ^n$;
\item there exists a cone $C_M \subset \bR^n$ compatible with $\prec$ and $\be \in \bZ^n$ such that $M_\bg = 0$ whenever $\bg - \be \notin C_M$. 
\end{itemize}

Let $(A,d_A)$ be a $\bZ^n$-graded dg-algebra. Suppose that $(A,d)$ is concentrated in non-negative homological degrees, that is $A_\bg^h = 0$ whenever $h < 0$. 
The \emph{c.b.l.f. derived category $\cD^{cblf}(A,d_A)$} of $(A,d_A)$ is the triangulated full subcategory of $\cD(A,d_A)$ given by dg-modules having homology being c.b.l.f. dimensional for the $\bZ^n$-grading. There exists also a dg-enhanced version $\cD_{dg}^{cblf}(A,d_A)$. 
We write $\bKO^\Delta(A,d) := \bKO^{\Delta}(\cD^{cblf}(A,d_A))$.

\begin{defn}\label{def:positivecblfdgalg}
We say that $(A,d)$ is a \emph{positive c.b.l.f. dg-algebra} if 
\begin{enumerate}
\item $A$ is c.b.l.f. dimensional for the $\bZ^n$-grading;
\item $A$ is non-negative for the homological grading;
\item $A_0^0$ is semi-simple;
\item $A_0^h = 0$ for $h >0$; 
\item$(A,d_A)$ decomposes a direct sum of shifted copies of modules $P_i := A e_i$ for some idempotent $e_i \in A$, such that $P_i$ is non-negative for the $\bZ^n$-grading.
\end{enumerate}
\end{defn}

In a $\bZ^n$-graded triangulated category $\cC$, we define the notion of \emph{c.b.l.f. direct sum} as follows: 
\begin{itemize}
\item take a a finite collection of objects $\{K_1,\dots, K_m\}$ in $\cC$; 
\item consider a direct sum of the form
\begin{align*}
&\bigoplus_{\bg \in \bZ^n} x^{\bg} (K_ {1,\bg} \oplus \cdots \oplus K_{m,\bg}), &
&\text{ with }&
K_{i,\bg} &= \bigoplus_{j = 1}^{k_{i,\bg}} K_i[h_{i,j,\bg}],
\end{align*}
where $k_{i,\bg} \in \bN$ and $h_{i,j,\bg} \in \bZ$ such that:
\item there exists a cone $C$ compatible with $\prec$, and $\be \in \bZ^n$ such that for all $j$ we have $k_{j,\bg} = 0$ whenever $\bg -  \be \notin C$;
\item there exists $h \in \bZ$ such that $h_{i,j,\bg} \geq h$ for all $i,j,\bg$. 
\end{itemize}

If $\cC$ admits arbitrary c.b.l.f. direct sums, then $K_0^{\Delta}(\cC)$ has a natural structure of $\bZ\pp{x_1,\dots, x_n}$-module with
\[
\sum_{\bg \in C} a_\bg x^{\be + \bg} [X] := [\bigoplus_{\bg \in C} x^{\bg + \be} X^{\oplus a_\bg}],
\]
where $X^{\oplus a_\bg} = \bigoplus_{\ell = 1}^{|a_\bg|} X[\alpha_\bg]$ and $\alpha_\bg = 0$ if $a_\bg \geq 0$ and $\alpha_\bg =1$ if $a_\bg < 0$.

\begin{thm}[{\cite[\thmasympKO]{asympK0}}]\label{thm:triangtopK0genbyPi}
Let $(A,d)$ be a positive c.b.l.f. dg-algebra, and let $\{P_j\}_{j \in J}$ be a complete set of indecomposable cofibrant $(A,d)$-modules that are pairwise non-isomorphic (even up to degree shift). Let $\{S_j\}_{j \in J}$ be the set of corresponding simple modules. 
There is an isomorphism
\begin{align*}
\bKO^\Delta(A,d)  & \cong  \bigoplus_{j \in J} \bZ\pp{x_1, \dots, x_\ell} [P_j], 
\end{align*}
and $\bKO^\Delta(A,d)$ is also freely generated by the classes of $\{[S_j]\}_{j \in J}$.
\end{thm}

\begin{prop}[{\cite[\propcblfbim]{asympK0}}]\label{prop:cblfbiminduceKO}
Let $(A,d)$ and $(A',d')$ be two c.b.l.f. positive dg-algebras. Let $B$ be a c.b.l.f. dimensional $(A',d')$-$(A,d)$-bimodule. The derived tensor product functor
\[
F : \cD^{cblf}(A,d) \rightarrow \cD^{cblf}(A',d'), \quad F(X) := B \Lotimes_{(A,d)} X,
\]
induces a map
\[
[F] : \bKO^\Delta(A,d)  \rightarrow \bKO^\Delta(A',d'),
\]
sending $[X]$ to $[F(X)]$. 
\end{prop}

\subsubsection{C.b.l.f. iterated extensions}\label{sec:cblfitext}

Recall that the \emph{Milnor colimit $\mcolim_{r  \geq 0} (f_r) $} (using the terminology of~\cite{kellerweight}) of a collection of arrows $\{X_r \xrightarrow{f_r} X_{r+1}\}_{r \in \bN}$ in a triangulated category $\cT$ is the mapping cone fitting inside the following distinguished triangle
\[
\coprod_{r \in \bN} X_r \xrightarrow{1-f_\bullet} \coprod_{r \in \bN} X_r \rightarrow \mcolim_{r  \geq 0} (f_r) \rightarrow 
\]
where the left arrow is given by the infinite matrix
\[
1-f_\bullet := 
\begin{pmatrix}
1      & 0       &  0 & 0 & \cdots \\
-f_0 & 1       & 0 & 0 & \cdots  \\
0      & -f_1  & 1 & 0  & \cdots \\
\vdots & \ddots & \ddots & \ddots & \ddots
\end{pmatrix}
\] 

\begin{defn}
Let $\{K_1, \dots, K_m\}$ be a finite collection of objects in $\cC$, and let $\{E_r\}_{r \in \bN}$ be a family of direct sums of $\{K_1, \dots, K_m\}$ such that $\bigoplus_{r \in \bN} E_r$ is a c.b.l.f. direct sum of $\{K_1, \dots, K_m\}$. Let $\{M_r\}_{r \in \bN}$  be a collection of objects in $\cC$ with $M_0 = 0$, such that they fit in distinguished triangles
\[
M_r \xrightarrow{f_r} M_{r+1} \rightarrow E_r \rightarrow 
\]
Then, we say that an object $M \in \cC$ such that $M \cong_{\cT} \mcolim_{r\geq 0} (f_r)$ in $\cT$ is a \emph{c.b.l.f. iterated extension of  $\{K_1, \dots, K_m\}$}. 
\end{defn}

Note that under the conditions above, we have
\[
[M] = \sum_{r \geq 0} [E_r],
\]
in the asymptotic Grothendieck group $\bKO(\cC)$. 

\begin{defn}\label{def:cblfgenerated}
Let $\cT$ be a $\bZ^n$-graded (dg-)triangulated (dg-)category, and $\{X_j\}_{j \in J}$ be a collection of objects in $\cT$. 
The subcategory of $\cT$ \emph{c.b.l.f. generated by $\{X_j\}_{j \in J}$} is the triangulated full subcategory $\cC \subset \cT$ given by all objects $Y \in \cT$ such that there exists a finite subset $\{X_k\}_{k \in K}$ such that $Y$ is isomorphic to a c.b.l.f. iterated extension of $\{X_k\}_{k \in K}$ in $\cT$. 
\end{defn}

Thus, under the conditions above,  $\bKO(\cC)$ is generated as a $\bZ\pp{x_1,\dots, x_n}$-module by the classes of $\{[X_j]\}_{j \in J}$. 

\subsubsection{Dg-functors}\label{sec:cblfdgfunctors}

Let $(R,d_R)$ and $(S,d_S)$ be ($\bZ^n$-graded) dg-algebras. 
The situation of \cref{eq:quasifunctequiv} in \cref{sec:dgfunctors} restricts to the c.b.l.f. version $\cD_{dg}^{cblf}$ of \cref{sec:cblfstruct}, so that 
\[
\cRHom_{\Hqe}^{cop}(\cD_{dg}^{cblf}(R,d_R),\cD_{dg}^{cblf}(S,d_{S})) \cong \cD_{dg}^{cblf}((S,d_{S}), (R,d_R)).
\] 
Then, we obtain an induced map
\begin{equation} \label{eq:K0homK0}
\bKO^\Delta(\cRHom_{\Hqe}^{cop}( \cD_{dg}^{cblf}(R,d_R), \cD_{dg}^{cblf}(S,d_{S}))) \rightarrow \Hom_{\bZ\pp{x_1,\dots,x_n}}(\bKO^\Delta(R,d_R),  \bKO^\Delta(S,d_{S})),
\end{equation}
by using \cref{prop:cblfbiminduceKO}.



\bibliographystyle{bibliography/habbrv}



\end{document}